\documentclass[10pt,a4paper,reqno]{amsart}
\pdfoutput=1
\usepackage{Kosanovic-embcalc-gropes}

\title{Embedding calculus and grope cobordism of knots}
\author{Danica Kosanovi\'c}

\address{Max-Planck-Institut f\"ur Mathematik, Bonn}
\email{danica@mpim-bonn.mpg.de}
\curraddr{ETH Z\"urich, Department of Mathematics, R\"amistrasse 101, 8092 Z\"urich, Switzerland}
\email{danica.kosanovic@math.ethz.ch}

\newcommand{\red}[1]{{\color{black}#1}}

\begin{document}
\begin{abstract}
  We show that embedding calculus invariants $ev_n$ are \emph{surjective} for long knots in an arbitrary $3$-manifold. This solves some remaining open cases of Goodwillie--Klein--Weiss connectivity estimates, and at the same time confirms one half of the conjecture that for classical knots $ev_n$ are universal additive Vassiliev invariants over the integers. In addition, we give a sufficient condition for this conjecture to hold over a coefficient group, which is by recent results of Boavida de Brito and Horel fulfilled for the rationals and for the $p$-adic integers in a range. Therefore, \emph{embedding calculus invariants are strictly more powerful than the Kontsevich integral}. 
  
  Furthermore, our work shows they are more computable as well. Namely, the main theorem computes \emph{the first possibly non-vanishing invariant $ev_n$ of a knot which is grope cobordant to the unknot} to be precisely equal to \emph{the equivalence class of the underlying decorated tree of the grope in the associated graph complex}. Actually, our techniques apply beyond dimension $3$, offering a description of the layers in embedding calculus for long knots in a manifold of any dimension, and suggesting that certain generalised gropes realise the corresponding graph complex classes.
\end{abstract}

\maketitle

\thispagestyle{empty}

\setcounter{section}{-1}

\section{Introduction}\label{sec:intro}

In contrast to the study of spaces $\Map(P,M)$ of maps between topological spaces, which gave rise to numerous techniques of homotopy theory, spaces $\Emb(P,M)$ of smooth embeddings of manifolds seem less tractable from the algebraic viewpoint. Here usually purely {geometric} arguments are used, especially at the level of components, that is, for the study of isotopy classes of embeddings. 

In this paper we attempt to reconcile these two viewpoints in the case of {long knots in a compact oriented $3$-manifold $M$ with boundary}. Namely, \red{for $I\coloneqq[0,1]$ 
}
we study the space
\begin{equation}\label{eq-def:knots}
  \Knots(M)\coloneqq\Emb_\partial(I,M)\coloneqq\big\{\,K\colon I\hra M \;|\; K\equiv\U\text{ near }\partial I\,\big\},
\end{equation}
where $\U\colon I\hra M$ is a fixed neat embedding, i.e.\ transverse to $\partial M$ and with $\U^{-1}(\partial M)=\{0,1\}$. Our results on one hand relate previous such homotopy theoretic approaches to knot theory, and on the other hand, offer an explicit geometric interpretation of the involved algebraic tools. 

The first approach, \emph{Vassiliev's theory of finite type knot invariants}~\cite{Vass90,Vass-arb-mfld}, starts from the observation that having understood the space of smooth maps, we can try to describe its subspace of embeddings by studying homotopy types of the strata of the complement $\Map_\partial(I,M)\sm\Knots(M)$. In its more geometric version~\cite{Gusarov,Habiro,CT1} this theory gives a sequence of equivalence relations $\sim_n$ for $n\geq1$ on the set $\KK(M)\coloneqq\pi_0\Knots(M)$, defined in terms of either \emph{claspers or gropes}.

The second approach, \emph{the embedding calculus of Goodwillie and Weiss}~\cite{WeissI,GWII}, builds on the idea 
that, since we understand embeddings of disjoint unions of disks, we could use them to approximate the space $\Emb(P,M)$; \red{this is related to Smale--Hirsch immersion theory. 
}
The outcome is a tower of spaces \red{$\dots\to\pT_{n+1}(M)\to\pT_n(M)\to\cdots$ 
}
and maps $\ev_n\colon\Knots(M)\to\pT_n(M)$, for $n\geq1$. 

See Section~\ref{subsec-intro:brief-background} for a brief survey of both approaches, and Sections~\ref{app:finite-type} and~\ref{sec:gropes-punc-knots} for more details.
It was suggested early on~\cite{GWII,GKW} that these theories should be closely related for $M=I^3$, in which case $\KK(I^3)\cong\pi_0\Emb(\S^1,\R^3)$ is precisely the abelian monoid of classical knots. The study of the relationship was initiated in~\cite{BCSS}, where the following conjecture was stated, and proven in the first non-trivial degree $n=3$ (see Example~\ref{ex:n-3-classical} below for another proof).
\begin{conj}[\cite{BCSS}]\label{conj:universal}
  For each $n\geq 1$ the map $\pi_0\ev_n\colon\KK(I^3)\to\pi_0\pT_n(I^3)$ is a universal additive Vassiliev invariant of type $\leq n-1$ over $\Z$.
\end{conj}
As explained \red{in Section~\ref{subsec:geom-finite-type-theory}, the conjecture equivalently says 
}
that $\pi_0\ev_n$ is a monoid homomorphism which factors \red{through 
}
${\KK(I^3)/\sim_n}$ and the induced map \red{${\KK(I^3)/\sim_n}\to\pi_0\pT_n(I^3)$ 
}
is \emph{an isomorphism of groups.} In other words, paths in the space $\pT_n(I^3)$ should precisely encode the \red{Gusarov--Habiro
}
$n$-equivalence relation $\sim_n$.

\red{Some evidence comes from the fact that the same \emph{graph complexes} appear in both theories. Namely, Conant~\cite{Conant} studied the spectral sequence $E^*_{-n,t}(I^3)$ for the homotopy groups of the tower of fibrations $p_{n+1}\colon\pT_{n+1}(I^3)\to\pT_n(I^3)$ and showed that the terms $E^2_{-(n+1),n+1}(I^3)$ on the diagonal of the second page are isomorphic to the group of \emph{Jacobi trees} $\A^T_n$ (see Definition~\ref{def:jac-trees} and Theorem~\ref{thm:conant-d1}). This implies that the kernel of $\pi_0p_{n+1}$ is isomorphic to a quotient of $\A^T_n$. On the other hand, $\A^T_n$ is related to Vassilev's chord diagrams and graph complexes, see Remark~\ref{rem:graph-complexes}. Voli\'c ~\cite{Volic} and Turchin~\cite{Turchin-otherside} showed similar results for the corresponding spectral sequences in (co)homology.%
}

More recently a part of the conjecture was confirmed in~\cite{BCKS}. They show that each $\pi_0\ev_n$ is an \emph{additive invariant of type $\leq n-1$}, namely, $\pi_0\pT_n(I^3)$ has a group structure so that $\pi_0\ev_n$ is homomorphism of monoids, and the mentioned factorisation exists. \red{Our Theorem~\ref{thm:KST} 
}
reproves the latter claim from the perspective of gropes.

One of the main results of the present paper is the proof of the `surjectivity part' of Conjecture~\ref{conj:universal}.
\begin{theorem}\label{cor:surj-classical}
  For each $n\geq 1$ the homomorphism $\pi_0\ev_n\colon\KK(I^3)\to\pi_0\pT_n(I^3)$ is surjective.
\end{theorem}
This is just a special case of Theorem~\ref{thm:surj} which says that the same holds for an arbitrary $3$-manifold.

Starting from a geometric viewpoint, we use different models than the mentioned thread of work: \emph{capped gropes} for finite type theory~\cite{CT1} and Goodwillie's \emph{punctured knots model} for embedding calculus (see~\cite{Sinha-cosimplicial}). Moreover, in~\cite{Volic} Voli\'{c} asks `Can one in general understand the geometry of finite type invariants using the evaluation map?', and we make a step forward in that direction.
Namely,
our main Theorem~\ref{the-main} implies that (see \red{\eqref{eq:key-statement} and 
}
Corollary~\ref{cor:detect-tree} for the precise statement)
\begin{equation}\label{key-statement}
    \textit{the evaluation map detects the underlying tree of a grope/clasper in the graph complex.}
\end{equation}
Analogous results hold for universal rational Vassiliev knot invariants of Kontsevich~\cite{Kont-Vassiliev} and Bott--Taubes~\cite{Bott-Taubes,Altschuler-Freidel}, as well as for similar invariants of families of diffeomorphisms of $\S^4$, shown by Watanabe~\cite{Watanabe} to detect the underlying graph of a family constructed using similar claspers. However, a crucial difference is that while all these invariants use integrals over configuration spaces -- and so can provide results only in characteristic zero -- embedding calculus is a purely topological technique for studying homotopy types of embedding spaces themselves.

Indeed, Theorem~\ref{the-main} can be used both to confirm Conjecture~\ref{conj:universal} rationally and show that Kontsevich and Bott--Taubes integrals factor through the tower (see Remark~\ref{rem:integrals-factor}), and also more generally:
\begin{cor}\label{cor:collapse-implies-universal}
    Let $A$ be a torsion-free abelian group. If the spectral sequence $E^*_{-n,t}(I^3)\otimes A$ for the \red{homotopy groups 
    }
    of the Taylor tower of $\Knots(I^3)$ collapses at the $E^2$-page along the diagonal, then $\pi_0\ev_n$ is a universal additive Vassiliev invariant of type $\leq n-1$ over $A$, meaning that
    \[\begin{tikzcd}
        \pi_0\ev_n\otimes A\colon\:\faktor{\KK(I^3)}{\sim_n}\otimes A\arrow[]{r}{\cong} & \pi_0\pT_n(I^3)\otimes A.
    \end{tikzcd}
    \]
\end{cor}
\red{The mentioned collapse condition 
}
is equivalent to the claim that \emph{the canonical projection from the group
$E^2_{-(n+1),n+1}\cong\A^T_n$ of Jacobi trees to $E^\infty_{-(n+1),n+1}=\ker(\pi_0p_{n+1})$ is an isomorphism over $A$}.
This has already been confirmed in some cases\red{, so we obtain the following corollaries.
}
\begin{cor}\label{cor:rational-universal} \hfill
\begin{enumerate}
    \item $\pi_0\ev_n$ is a universal additive Vassiliev invariant of type $\leq n-1$ over $\Q$.
    \item\label{item:2} For any prime $p$, the evaluation map $\pi_0\ev_n$ is a universal additive Vassiliev invariant of type $\leq n-1$ over the $p$-adic integers $\Z_{p}$ if $n\leq p+2$.
    \item\label{item:3} $\pi_0\ev_n$ is a universal additive Vassiliev invariant of type $\leq n-1$ over $\Z$ if $n\leq 7$.
\end{enumerate}
\end{cor}
Namely, Boavida de Brito and Horel~\cite{BH} show vanishing of higher differentials in a range for this spectral sequence in positive characteristic, implying \textsl{(2)}. Their results also imply that $E^*_{-n,t}(I^3)\otimes\Q$ collapses at the whole $E^2$-page (this could also be deduced from~\cite{FTW}). Moreover, some existing low-degree computations show that the group $\A^T_n$ is torsion-free, giving \textsl{(3)}. For proofs of both corollaries and further details see Appendix~\ref{app:finite-type}.
\begin{theorem}\label{thm:surj}
  For a compact oriented $3$-manifold $M$ the map of sets $\pi_0\ev_n\colon\KK(M)\to\pi_0\pT_n(M)$ is surjective for any $n\geq 1$.
\end{theorem}
This was expected to hold by analogy to the famous Goodwillie–Klein connectivity formula (see Theorem~\ref{thm:GK}), which predicts that for a $1$-dimensional source manifold and a $3$-dimensional target the map $\ev_n$ is $0$-connected. For a connected source this is precisely Theorem~\ref{thm:surj}, and in future work we plan to investigate if our methods can be extended to links.
\begin{remm}
   The corollaries of Theorem~\ref{the-main} stated above apply to some extent also for knots in a general $3$-manifold $M$, using analogous groups $\A^T_n(M)$ as studied in~\cite{K-thesis}; see also Remark~\ref{rem:comparison-general-M}.
\end{remm}
Let us now introduce some notation, so that we can state Theorem~\ref{the-main} and deduce Theorem~\ref{thm:surj}. We will actually study the punctured knots model $\pT_n(M)$ even more generally: for $M$ any \emph{connected compact smooth manifold of dimension $d\geq3$ with non-empty boundary}. The space $\pT_n(M)$ is defined as a homotopy limit over a certain finite category and we will study its properties in detail in Section~\ref{sec:punc-knot-model}. We only restrict to oriented $3$-manifolds when we later consider gropes.

Let us pick \emph{an arbitrary knot} $\U\in\Knots(M)$ as our basepoint and call \red{$\U$ 
}
the \emph{unknot}, and let $\ev_n(\U)$ be the basepoint of $\pT_n(M)$. There is a natural map $p_{n+1}\colon\pT_{n+1}(M)\to\pT_n(M)$ which is a surjective fibration and satisfies $p_{n+1}\circ\ev_{n+1}=\ev_n$, so preserves basepoints. The fibres $\pF_{n+1}(M)\coloneqq\fib_{\ev_n\U}(p_{n+1})$ are called the \emph{layers of the Taylor tower}. Moreover, we also consider the \emph{homotopy fibre}
\begin{align*}
  \H_n(M)&\coloneqq\hofib_{\ev_n\U}(\ev_n)\\
  &\coloneqq\big\{\,(K,\gamma) \in \Knots(M)\times \Map([0,1],\pT_n(M))\;|\; \gamma(0)=\ev_n(K),\gamma(1)=\ev_n(\U)\,\big\}.
\end{align*}
and in Section~\ref{subsec:pn-surj-fib} we construct a map $\emap_{n+1}$ making the following diagram commute:
\begin{equation}\label{diag:hofib-hn}
    \begin{tikzcd}[column sep=0.7cm]
     \H_n(M)\arrow{d}{}\arrow{rr}{\emap_{n+1}} && \pF_{n+1}(M)\arrow[hook]{d} \\
    \Knots(M)\arrow{d}[swap]{\ev_n} \arrow{rr}{\ev_{n+1}} && \pT_{n+1}(M)\arrow[two heads]{d}{\;p_{n+1}} \\
    \pT_n(M)\arrow[equal]{rr} && \pT_nM)
\end{tikzcd}
\end{equation}
Remarkably, both $\pF_{n+1}(M)$ and $\H_n(M)$ will be related -- for a priori different reasons -- to the set $\Tree_{\pi_1M}(n)$ of \emph{$\pi_1M$-decorated trees}, where $\Gamma^{g_{\ul{n}}}\in\Tree_{\pi_1M}(n)$ consists of a rooted planar binary tree $\Gamma$ with enumerated leaves which are also decorated by elements $g_i\in\pi_1(M)$, for $i\in\ul{n}\coloneqq\{1,\dots,n\}$. See Section~\ref{subsec-prelim:trees} for all definitions related to trees. 
For example, 
\[
\Gamma^{g_{\ul{3}}}\;\coloneqq\quad\begin{tikzpicture}[baseline=2ex,scale=0.35,every node/.style={font=\bfseries}]
        \clip (-2.4,-0.21) rectangle (2.4,4.1);
        \draw[fill=black!30!white]
            (-0.15, -0.2) rectangle ++ (0.3,0.2);
        \draw[thick]
            (0,0) -- (0,1) --
            (-1,2) -- (-2,3) node[pos=1,above]{$1$} node[Decoration, below=0.1cm]{$g_1$}
            (-1,2) -- (0,3)    node[pos=1,above]{$3$} node[Decoration, below=0.1cm]{$g_3$}
                    (0,1) -- (2,3)  node[pos=1,above]{$2$} node[Decoration, below=0.1cm]{$g_2$};
    \end{tikzpicture}\quad\in\: \Tree_{\pi_1M}(3).
\]
Indeed, \textit{on one hand}, the set $\pi_0\pF_{n+1}(M)$ will be isomorphic to the group $\Lie_{\pi_1M}(n)$ of Lie trees, defined as the quotient of the $\Z$-span of $\Tree_{\pi_1M}(n)$ by the antisymmetry $(AS)$ and Jacobi relations $(IHX)$. \textit{On the other hand}, from a grope $\TG$ in a $3$-manifold $M$ we will construct a point $\psi(\TG)\in\H_n(M)$, and the underlying combinatorics of $\TG$ will also be described by $\ut(\TG)\in\Tree_{\pi_1M}(n)$. \red{Then Theorem~\ref{the-main} says that $[\emap_{n+1}(\psi(\TG))]\in\pi_0\pF_{n+1}(M)$ is precisely the class of $\ut(\TG)$ modulo $AS,IHX$.%

To state this theorem, let $\B(\ul{n})$ be an additive 
}
basis for the free Lie algebra $\L(x^i:i\in\ul{n})$ \red{(more precisely, a Hall basis, see Appendix~\ref{subsec:HM}), 
}
and $\N\B(\ul{n})\subseteq\B(\ul{n})$ the subset of \red{Lie 
}
words in which each letter $x^i$ for $i\in\ul{n}$ appears at least once. Let $l_w$ be the word length of $w\in\N\B(\ul{n})$. Further, for $k\geq1$ denote by $M^{\times k}$ the $k$-fold product of $M$ with itself and by $\Omega M^{\times k}$ the space of based loops in it; for a space $X$ let $X_+\coloneqq X\sqcup\{*\}$ and $\Sigma^k(X_+)$ its $k$-fold reduced suspension. The weak product $\prod^{\circ}$ is defined as the filtered colimit of products over finite subsets.
\begin{introthm}\label{the-main}
    \red{For any smooth manifold $M$ of dimension $d\geq3$ with boundary there is an explicit zig-zag of 
    }
    homotopy equivalences
    \[
        \mathsf{F}_{n+1}(M)\simeq\Omega^{n}\sideset{}{^{\circ}}\prod_{w\in \N\B(\ul{n})}\Omega\Sigma^{1 + l_w(d-2)}(\Omega M^{\times l_w})_+\:.
    \]
    It follows that $\mathsf{F}_{n+1}(M)$ is $(n(d-3)-1)$-connected and $\pi_{n(d-3)}\mathsf{F}_{n+1}(M)\cong\mathsf{Lie}_{\pi_1M}(n)$. Moreover, for $d=3$ capped gropes in $M$ give a map of sets
    \[
        \rho_n\colon\Z[\Tree_{\pi_1M}(n)] \to\pi_{n(d-3)}\H_n(M)
    \]
    such that $\pi_{n(d-3)} \emap_{n+1}\circ\rho_n$ is the canonical quotient map by the relations $AS$ and $IHX$.
\end{introthm}
In other words, the two roles of trees, homotopy theoretic for $\pF_{n+1}(M)$ and geometric for $\H_n(M)$, are compatible: we let $\rho_n(\Gamma^{ g_{\ul{n}} })\coloneqq[\psi(\TG)]$ for \red{any 
}
grope $\TG$ with the underlying tree $\ut(\TG)=\Gamma^{\ul{g}}$ and show that $[\emap_{n+1}\psi(\TG)]\in\pi_0\pF_{n+1}(M)\cong\Lie_{\pi_1M}(n)$ is precisely the class of $\ut(\TG)$ modulo $AS$, $IHX$.
\begin{remm}\label{rem:operadic}
  There is an isomorphism $\Lie_{\pi_1M}(n)\cong\big(\Z[\pi_1M^n]\big)^{(n-1)!}$. If $M$ is simply connected, we obtain $\Lie(n)\cong\Z^{(n-1)!}$, which is the arity $n$ of the Lie operad. Interestingly, in Goodwillie's homotopy calculus the $n$-th Taylor layer of a functor $F\colon\mathsf{Top}_*\to\mathsf{Top}_*$ is computed in terms of a spectrum $\partial_{n}(\Id)$, which turn out to form an operad $\partial_*(\Id)$ with homology precisely the Lie operad.
\end{remm}
The following is an immediate corollary of Theorem~\ref{the-main}, and implies Theorem~\ref{thm:surj}.
\begin{cor}\label{cor:surj-fibre}
    For $d=3$ and $n\geq1$, $\pi_0\emap_{n+1}\colon\pi_0\H_n(M)\to\pi_0\pF_{n+1}(M)$ is a surjection of sets.
\end{cor}
\begin{proof}[Proof of Theorem~\ref{thm:surj} assuming Theorem~\ref{the-main}]
It is a standard fact (see~\eqref{eq:deg1-lim-holim} below) that $\pT_1(M)$ is homotopy equivalent to the loop space $\Omega(\S M)$ on the unit tangent bundle $\S M$ of $M$. Thus, $\pi_0\pT_1(M)\cong\pi_1M$, and since each class here can be represented by an embedded loop (as $d=3$), the map $\pi_0\ev_1$ is surjective. Assume by induction that $\pi_0\ev_n$ is surjective for some $n\geq1$.

Let us pick $x\in\pT_{n+1}(M)$ and show it is in the image of $\pi_0\ev_{n+1}$. Denote $y\coloneqq p_{n+1}(x)\in\pT_n(M)$ and the corresponding fibres $\pF_{n+1}^y(M)\coloneqq\fib_y(p_{n+1})$ and $\H_n^y(M)\coloneqq\hofib_y(\ev_n)$. Since by definition $x\in\pF_{n+1}^y(M)$ and $\emap^y_{n+1}\colon\H_n^y(M)\to\pF_{n+1}^y(M)$ as in \eqref{diag:hofib-hn}, it suffices to prove $\pi_0\emap^y_{n+1}$ is surjective.

However, it is instead enough to check that $\pi_0\emap^{\ev_nK}_{n+1}$ is surjective, where $K\in\Knots(M)$ is any knot such that there is a path $\gamma$ in $\pT_n(M)$ from $\ev_n(K)$ to $y$ (this exists by the induction hypothesis). Indeed, $\emap_{n+1}$ is equivalent to the map induced on the homotopy fibres, and changing the basepoint on both sides using $\gamma$ induces homotopy equivalences which commute with $\emap_{n+1}$. As our choice of $\U$ was arbitrary, we can take $\U\coloneqq K$, so $\emap^{\ev_nK}_{n+1}=\emap_{n+1}$. Now apply Corollary~\ref{cor:surj-fibre}.
\end{proof}

\red{
Thus, the only results from this introduction whose proofs we owe in the rest of the paper are Theorem~\ref{the-main} and Corollaries~\ref{cor:collapse-implies-universal} and~\ref{cor:rational-universal}. 
}

Finally, let us point out that our techniques easily extend to show that for the space of arcs $\Knots(M)$ with $\dim(M)=d\geq4$, the analogous realisation map $\rho_n$ is inverse to $\pi_{n(d-3)}\ev_{n+1}$, see also Remark~\ref{rem:other-d}. \red{This is discussed in work in preparation~\cite{K-families}. 
}
Remarkably, these spaces \red{of arcs in codimension bigger than two 
}
also show relevant in low-dimensional topology: \red{in joint work with Peter Teichner%
}
~\cite{KT-highd,KT-4dLBT} we show that the space of certain properly embedded $2$-disks in a $4$-manifold is homotopy equivalent to \red{the loop space 
}
$\Omega\Knots(M)$, allowing us to classify such disks up to isotopy using $\rho_1$. This fully answers questions posed by David Gabai~\cite{Gabai-disks} and reproves~\cite{Gabai-spheres}. In fact, $\ev_n$ are defined for all spaces of embeddings and diffeomorphisms, and can be expected to detect analogous clasper-like geometric constructions.

\subsubsection*{Acknowledgements}
My deepest thanks go to my thesis advisor Peter Teichner for his guidance, most inspiring conversations and selfless sharing of time, knowledge, support and advice. I benefited greatly from our joint work with Yuqing Shi and I thank her for interesting discussions. Thanks also to Geoffroy Horel and Greg Arone for useful remarks, and Pedro Boavida de Brito, Thomas Goodwillie, Gijs Heuts, Dev Sinha, Victor Turchin and Tadayuki Watanabe for showing interest in this work. I also wish to acknowledge the generous support of the Max Planck Institute for Mathematics in Bonn throughout my doctoral studies, during which these results were obtained. Thank you also to the referees for useful comments and suggestions.

\subsubsection*{The outline}
    \red{The paper has six sections and three appendices. Sections~\ref{sec:intermediate-results} and~\ref{sec:prelim} contain further results and preliminary discussions. In Sections~\ref{sec:punc-knot-model},~\ref{sec:delooping} and ~\ref{sec:htpy-type} we study the punctured knots model of the Taylor tower for arcs in an arbitrary $d$-manifold $M$ with boundary, $d\geq3$. In Section~\ref{sec:gropes-punc-knots} we define gropes in $3$-manifolds and points that they give in the corresponding Taylor tower. Finally, in Section~\ref{sec:main} we prove the main result, Theorem~\ref{the-main}. We now give a more detailed outline.
    }

    In Section~\ref{sec:intermediate-results}, after giving a brief survey of the literature \red{(Section~\ref{subsec-intro:brief-background}), 
    }
    we rephrase Theorem~\ref{the-main} as several intermediate results. Namely, Theorem~\ref{thm:W-iso} \red{(Section~\ref{subsec-intro:careful-study}) 
    }
    provides for any $d\geq3$ an explicit description of the homotopy groups of $\pF_{n+1}(M)$, and Theorem~\ref{thm:KST} \red{(Section~\ref{subsec-intro:gropes}) 
    }
    constructs points $\psi(\TG)\in\H_n(M)$ for $d=3$. The compatibility of the two roles of decorated trees is shown in Theorem~\ref{thm:main-thm} and Theorem~\ref{thm:main-extended} \red{(Section~\ref{subsec-intro:merge}).
    }
    Along the way, we give proof ingredients which may be of independent interest, and some corollaries. We end Section~\ref{subsec-intro:merge} with a proof of Theorem~\ref{the-main}.

    \red{Next, in Section~\ref{subsec-prelim:trees}
    we define several variants of trees, and in Section~\ref{subsec-prelim:holims} we define homotopy limits and total homotopy fibres, and give proofs of their properties which are used in the paper.
    }

    In Section~\ref{sec:punc-knot-model} we study the punctured knots model $\pT_n(M)$; \red{in particular, in Section~\ref{subsec:fn-hn} we give a working model for the layers $\pF_{n+1}(M)$ and the homotopy fibres $\H_n(M)$.
    }
    In Section~\ref{sec:delooping} we show that each $\pF_{n+1}(M)$ is an iterated loop space, while in Section~\ref{sec:htpy-type} we determine its homotopy type, proving Theorem~\ref{thm:W-iso}. We describe geometrically the generators of the lowest non-trivial homotopy group \red{$\pi_{n(d-3)}\pF_{n+1}(M)$ 
    }
    in Section~\ref{subsec:gen-maps}, and in Section~\ref{subsec:strategy} we outline the proof of Theorem~\ref{thm:main-thm}.

    In Section~\ref{sec:gropes-punc-knots} we are concerned with $3$-manifolds. Firstly, Section~\ref{subsec:gropes} \red{defines several variants of gropes that we use. 
    }
    Then in Section~\ref{subsec:grope-paths} we prove Theorem~\ref{thm:KST} and its extension for grope forests $\forest$, and in Section~\ref{subsec:grope-points} we describe points $\emap_{n+1}\psi(\forest)\in\pF_{n+1}(M)$ explicitly.
    Finally, main Theorems~\ref{thm:main-thm} and~\ref{thm:main-extended} are proven in Section~\ref{sec:main}; the proof of Theorem~\ref{thm:main-thm} is by induction, using two auxiliary lemmas.

    The construction of an explicit homotopy equivalence $\chi$ (from total fibres of certain cubes to iterated loop spaces, \red{used for delooping $\pF_{n+1}(M)$) 
    }
    is deferred to Appendix~\ref{app:proofs}, while Appendix~\ref{app:samelson} provides background on Samelson products and contains two important lemmas about their inductive behaviour. In Appendix~\ref{app:finite-type} we survey finite type theory, \red{by explaining how Vassiliev's singular knots relate to claspers and gropes, and the meaning of a universal Vassiliev invariant. 
    }
    In Section~\ref{subsec:further-cor} we prove Corollaries~\ref{cor:collapse-implies-universal} and~\ref{cor:rational-universal} \red{related to graph complexes. 
    }

\subsubsection*{Reader's guide}
    Throughout the paper we aim to make both homotopy theory and geometry accessible without assuming much background. \red{In particular, Sections~\ref{subsec-prelim:holims} and Appendices~\ref{app:proofs} and~\ref{app:samelson} give a self-contained account of homotopy theory needed in this paper, and Section~\ref{subsec:gropes} is a self-contained account of gropes used in this paper. 

    The paper unfortunately contains a lot of notation. Let us point out that important maps are displayed and labelled (in brackets), and important spaces are depicted in figures, and we encourage the reader to refer to them using cross-references. Furthermore, at the end of the paper there is an extensive table of notation, see~\nameref{Notation}.

    Moreover, it could be helpful to first understand  
    }
    the lowest degree computation, $n=1$, that is the induction base in the proof of Theorem~\ref{thm:main-thm}. It is explained in the mutually related Examples \ref{ex:cube-deg-2}, \ref{ex:n=2-part1}, \ref{ex:n=2-part2}, \ref{ex:n=2-part3}, \ref{ex:n=2-part4}, \ref{ex:grope-path-deg-1}, and Figure~\ref{fig:deg-1-layer-pt}. The induction step is outlined in Example~\ref{ex:proof-main}, for $n=2$.

\begin{spacing}{0.1}
    \tableofcontents
\end{spacing}


\section{Intermediate results}\label{sec:intermediate-results}
\subsection{Brief overview of background}\label{subsec-intro:brief-background}
\subsubsection*{Finite type theory}
Vassiliev's study of the strata of the discriminant $\Map_\partial(I,I^3)\sm\Knots(I^3)$ gave rise to the filtration $\VV^*_n(A)$ by type $n\geq 1$ of the group $H^0(\Knots(I^3);A)$ of knot invariants with values in an abelian group $A$, as formulated by~\cite{Birman-Lin}. A new, very active field emerged: it was shown that quantum invariants give rise to invariants of finite type~\cite{Lin} (for example, for the Jones polynomial $J(q)$ the coefficient next to $h^n$ in $J(e^h)$ is of type $\leq n$); that for $A=\Q$ there is a universal such invariant -- the Kontsevich integral~\cite{Kont-Vassiliev,Le-Murakami}; and a comprehensive treatment of its target, the rational Hopf algebra of chord diagrams, was given in~\cite{Bar-Natan}.

A geometric approach to the field was introduced by Gusarov~\cite{Gusarov} and Habiro~\cite{Habiro} independently, as a sequence of knot operations called \emph{surgeries on claspers} (or variations) of degree $n\geq1$. This gives a sequence of equivalence relations $\sim_n$ on the monoid $\KK(I^3)\coloneqq\pi_0\Knots(I^3)$ and a decreasing filtration $\KK_n(I^3)\coloneqq\{K\in\KK(I^3): K\sim_n\U\}$ by submonoids. Stanford~\cite{Stanford98} exhibits a close connection of this filtration with the lower central series of the pure braid group.

By the work of Conant and Teichner~\cite{CT1,CT2} one can instead of claspers equivalently use \emph{capped gropes}, and this is the approach we take. Gropes first appeared in the theory of topological $4$-manifolds, and can be viewed as a tool for detecting `embedded commutators'~\cite{Teichner-what-is}.

Notably, it has been shown \cite{Habiro} that the map $\KK(I^3)\to\Z[\KK(I^3)]=H_0(\Knots(I^3);\Z)$ defined by $K\mapsto K-\U$, takes $\KK_n(I^3)$ into $\VV_n(\Z)$, the dual of the Vassiliev filtration for $A=\Z$. Hence, this indeed gives a geometric version of the theory (or its primitive/additive part), since one works with knots instead of their linear combinations or invariants. See Appendix~\ref{app:finite-type} for a survey and comparison of these two approaches, and Section~\ref{subsec:gropes} for background on gropes; see Remark~\ref{rem:claspers-gropes} for an advantage of using gropes over claspers.

Additionally, the quotient of $\KK(I^3)$ by $\sim_n$ is actually \emph{an abelian group} and the projection
\begin{equation}\label{eq:KK-sim-n}
    \begin{tikzcd}
    \nu_n\colon\KK(I^3)\arrow[two heads]{r}{} & \faktor{\KK(I^3)}{\sim_n}
\end{tikzcd}
\end{equation}
is a \emph{universal additive invariant of type $\leq n-1$ \red{over $\Z$}
}
~\cite[Thm.~6.17]{Habiro} -- meaning that any additive invariant $v\colon\KK(I^3)\to A$ of type $\leq n-1$ factors through $\nu_n$. However, the target ${\KK(I^3)}/{\sim_n}$ is a mysterious group and \red{the goal is to 
}
have something combinatorial instead, \red{e.g.\ the 
}
primitive part of the mentioned algebra of chord diagrams or the group of Jacobi trees (see Section~\ref{subsec:jacobi-trees}).

\subsubsection*{Embedding calculus}
The pioneering approach of Goodwillie and Weiss~\cite{WeissI,GWII} for studying embedding spaces\footnote{One can take compact manifolds with a fixed boundary condition for all embeddings, or closed manifolds.} $\Emb_\partial(P,M)$ produces a tower of spaces, called \emph{the Taylor tower},
\[
\cdots\to\T_{n+1}\Emb_\partial(P,M)\to\T_n\Emb_\partial(P,M)\to\cdots\to\T_1\Emb_\partial(P,M)
\]
and the \emph{evaluation maps} $\ev_n\colon\Emb_\partial(P,M)\to\T_n\Emb_\partial(P,M)$, \red{defined in Section~\ref{sec:punc-knot-model} below for $P=I$. If $\dim P<\dim M$ then this starts 
}
with $\T_1\Emb_\partial(P,M)\simeq\Imm_\partial(P,M)$, the space of immersions. Since the definition of these objects is homotopy theoretic -- analogously to the description of immersions due to Hirsch and Smale -- we obtain an inductive way for studying the homotopy type of $\Emb_\partial(P,M)$, using a variety of tools. Indeed, a fundamental result in the field is the following theorem of Goodwillie and Klein (announced in~\cite{GWII}). Recall that a map is $k$-connected if it induces an isomorphism on homotopy groups below degree $k$ and a surjection on $\pi_k$.
\begin{theorem}[\cite{GKmultiple}]\label{thm:GK}
  The map $\ev_n$ is $\big(1-\dim P + n(\dim M-\dim P-2)\big)$-connected, except if $\dim P=1$, $\dim M=3$. Hence, the induced map $\ev_\infty\colon\Emb_\partial(P,M)\to\lim\T_n\Emb_\partial(P,M)$ is a weak homotopy equivalence if $\dim M-\dim P>2$ (`the tower converges to the embedding space').
\end{theorem}
This result inspired a great deal of research on Taylor towers for various pairs $(P,M)$.

To mention just a few, in~\cite{LTV-homology,AT14} the rational homology of spaces $\Emb_\partial(\D^k,\D^{k+c})$ of disks of codimension $c>2$ was expressed as the homology of certain graph complexes, and similarly for the rational homotopy groups~\cite{ALTV,AT15,FTW}. The spaces $\T_n\Emb_\partial(\D^k,\D^{k+c})$ were shown to be iterated loop spaces in~\cite{DH,Turchin-Delooping,BW2,Duc-Turchin}. Another (cosimplicial) model for $\T_n\Emb_\partial(I,M)$ was constructed in~\cite{Sinha-cosimplicial} and studied in~\cite{Sinha-operads,Scannell-Sinha,LT-GC} for $M=I^d$, $d\geq4$.

Note that the excluded case in Theorem~\ref{thm:GK} is precisely the setting of knot theory. Moreover, by an argument of Goodwillie the tower for classical long knots $\Knots(I^3)\coloneqq\Emb_\partial(I,I^3)$ does not converge (see Proposition~\ref{prop:non-conv}). Nevertheless, it still remains a source of interesting knot invariants: taking path components gives a tower of sets to which the monoid $\KK(I^3)$ maps.

Furthermore, the delooping results of~\cite{Turchin-Delooping,BW2} apply in this case as well: for $n\geq 2$ each $\T_n\Knots(I^3)$ is weakly equivalent to a double loop space, so each $\pi_0\T_n\Knots(I^3)$ is actually \emph{an abelian group} (for $n=1$ trivial as $\T_1\Knots(I^3)\simeq\Omega\S^2$). Moreover,~\cite{Tikhon} showed that $\ev^{BW}_n\colon\Knots(I^3)\to \T_n\Knots(I^3)$, the model from~\cite{BW2}, is a map of $H$-spaces. Thus, $\pi_0\ev^{BW}_n$ is a monoid map (an additive invariant).

\subsubsection*{Relating the two theories}
A different approach by~\cite{BCKS} uses the model $AM_n$ for $\T_n\Knots(I^3)$ from~\cite{Sinha-cosimplicial} to equip $\pi_0\T_n\Knots(I^3)$ directly with an abelian group structure, so that the corresponding $\pi_0ev_n$ is also a monoid map, as predicted by Conjecture~\ref{conj:universal} of~\cite{BCSS}. It has been an open problem whether the group structures of~\cite{BW2} and~\cite{BCKS} on $\pi_0\T_n\Knots(I^3)$ agree. We confirm this is indeed the case, see Proposition~\ref{prop:group-structures}.

The authors of~\cite{BCKS} also show that $\pi_0ev_n$ is of Vassiliev type $\leq n-1$, and we reprove this below for any $3$-manifold $M$. See Remark~\ref{rem:comparison-general-M}.

\subsection{A careful study of the layers in the Taylor tower}\label{subsec-intro:careful-study}
In the homotopy theoretic part of the paper we study the space $\pF_{n}(M)$ for $M$ \emph{any smooth manifold with non-empty boundary and dimension $\dim(M)=d\geq 3$}. The upshot is the following theorem, which reformulates the first part of Theorem~\ref{the-main}.
\begin{introthm}\label{thm:W-iso}
  If $n\geq 1$ there is a weak equivalence $\pF_{n+1}(M)\simeq\Omega^n\textstyle\prod^{\circ}_{w\in \N\B(\ul{n})}\Omega\Sigma^{1 + l_w(d-2)}(\Omega M^{\times l_w})_+$. Moreover, this space is $(n(d-3)-1)$-connected and there are explicit isomorphisms
\[\begin{tikzcd}[column sep=large]
    \Lie_{\pi_1M}(n)\arrow[]{r}{W}[swap]{\cong} & \pi_{n(d-2)}\tofib\big(\Omega(M\vee\S_{\bull}),\Omega\coll\big)\arrow[]{rr}{(\retr\circ\deriv\circ\chi)_*^{-1}}[swap]{\cong} && \pi_{n(d-3)}\pF_{n+1}(M).
\end{tikzcd}
\]
\end{introthm}
Let us give more details. \emph{Firstly}, $\pF_{n+1}(M)$ in Section~\ref{subsec:fn-hn} is described as the total homotopy fibre
\[
    \pF_{n+1}(M)=\tofib_{S\subseteq\ul{n}}\big(\FF^{n+1}_{S},r\big)
\]
of the $n$-cube\footnote{An $n$-cube $X_{\bull}$ consists of a space $X_S$ for each $S\subseteq\ul{n}$ and a compatible collection of maps $x^k_S\colon X_S\to X_{Sk}$ for $k\notin S$. The total homotopy fibre $\tofib(X_{\bull},x)$ generalises the notion of a homotopy fibre of a $1$-cube, see Section~\ref{subsec-prelim:holims}.}
of spaces\footnote{Of course, $\FF^{n+1}_S$ depends on $M$, but we omit it from the notation.}
$\FF^{n+1}_S\coloneqq\Emb_\partial([0,1],M_{0S})$, for $S\subseteq\ul{n}$. Here $M_{0S}\subseteq M$
is obtained by removing $|S|{+}1$ $d$-dimensional balls from $M$ \red{and embeddings satisfy a certain fixed boundary condition $\{0,1\}\hra\partial M_{0S}$. 
}
The map $r^k_S$ is induced from the inclusion $\rho^k_S\colon M_{0S}\hra M_{0Sk}$, which adds certain material between two balls (see Figure~\ref{fig:M-0S})\red{; the collection $(M_{0\bull},\rho)$ is also a cube of spaces. 
}
\emph{Secondly}, in Section~\ref{sec:delooping} we show that $\pF_1(M)\simeq\Omega\S M$ and for $n\geq1$:
\begin{equation*}\label{eq:for-the-main}
\begin{tikzcd}[column sep=scriptsize]
    \pF_{n+1}(M)\arrow{r}{\chi}[swap]{\sim} & \Omega^{n}\tofib\big(\FF^{n+1}_{\bull},l\big)\arrow{r}{\deriv}[swap]{\sim} & \Omega^{n}\tofib\big(\Omega M_{\bull}, \Omega\lambda\big)\arrow{r}{\retr}[swap]{\sim} & \Omega^{n}\tofib\big(\Omega\big(M\vee\S_{\bull}\big),\Omega\coll\big).
\end{tikzcd}
\end{equation*}
\red{Let us explain these objects and maps.
}
\begin{itemize}
    \item     
        For $\chi$ and its inverse see Theorem~\ref{thm:delooping-layer}: the map $l^k_S$ is defined using the left homotopy inverse $\lambda^k_S\colon M_{0Sk}\to M_{0S}$ for $\rho^k_S$ \red{(this means $\lambda^k_S\circ\rho^k_S\simeq\Id$), 
        }
        which adds back the $k$-th ball and rescales (see Figure~\ref{fig:M-erase}). \red{The collections $(M_{0\bull},\lambda)$ and $(\FF^{n+1}_{\bull},l)$ are contravariant cubes of spaces.
        }
    \item    
        In Theorem~\ref{thm:final-delooping} taking unit derivatives is shown to give a homotopy equivalence of contravariant cubes $\deriv_{\bull}\colon(\FF^{n+1}_{\bull},l)\to(\Omega\S M_{\bull},\Omega\S\lambda)$, where $M_S\supseteq M_{0S}$ are obtained by gluing in a ball \red{(so $(M_{\bull},\lambda)$ is a subcube of $(M_{0\bull},\lambda)$).
        }
        In the total homotopy fibre of the cube \red{$(\Omega\S M_{\bull},\Omega\S\lambda)$ 
        }
        the unit tangent data can be omitted, giving \red{the map 
        }
        $\deriv$ as displayed above.
    \item    
        Finally, \red{we have a contravariant cube $(M\vee\S_{\bull},\coll)$ 
        }
        with $M\vee\S_S\coloneqq M\vee\bigvee_{i\in S}\S^{d-1}_i$ and the map $\coll^k_S\colon M\vee\S_S\to M\vee\S_{S\sm k}$ given by collapsing a sphere. There are deformation retractions $\retr_S\colon M_S\xrightarrow{\sim} M\vee\S_S$, such that \red{$\retr_{S\sm k}\circ\lambda^k_{S\sm k}=\coll^k_S\circ \retr_S$, 
        }
        so we have an equivalence of contravariant cubes $\retr\colon(\Omega M_{\bull}, \Omega\lambda\big)\to(\Omega\big(M\vee\S_{\bull}\big),\Omega\coll)$.
\end{itemize}
\emph{Thirdly}, in Section~\ref{subsec:proof-ThmB} we describe the homotopy type of $\tofib\big(\Omega(M\vee\S_{\bull}),\Omega\coll\big)$, using a generalisation of the Hilton--Milnor theorem due to Gray~\cite{Gray} and Spencer~\cite{Spencer} \red{(this is reviewed in Appendix~\ref{subsec:HM}). 
}
Namely, we find a weak equivalence (see Theorem~\ref{thm:thmB-1}):
\begin{equation}\label{eq:tofib-hm-announce}
\begin{tikzcd}
    \sideset{}{^{\circ}}\prod_{w\in \N\B(\ul{n})}\Omega\Sigma^{1 + l_w(d-2)}(\Omega M^{\times l_w})_+ \arrow[]{rr}{\mu\circ hm}[swap]{\sim} && \tofib\Big(\Omega(M\vee\S_{\bull}),\Omega\coll\Big).
\end{tikzcd}
\end{equation}
Here $M^{\times l_w}$ is the cartesian product of $M$ with itself $l_w$ times, and $X_+$ means addition of a disjoint basepoint to a space $X$.
This proves the first statement of Theorem~\ref{thm:W-iso}. \red{We note that 
}
this formula differs slightly from the one we gave in~\cite{K-thesis}, but that one can be recovered using $\Sigma^k(X\times Y)_+\simeq\S^k\vee\Sigma^k(X\times Y)\simeq\S^k\vee\Sigma^kX\vee\Sigma^kY\vee\Sigma^k(X\wedge Y)$, and once again applying the Hilton--Milnor theorem.
\begin{remm}
    If $M\simeq\Sigma Y$ is homotopy equivalent to a suspension, the homotopy type of $\pF_{n+1}(M)$ was calculated in~\cite{GWII}; we can recover their result using the James splitting $\Sigma\Omega\Sigma Y\simeq\bigvee_{i=1}^\infty \Sigma Y^{\wedge i}$. See also~\cite{WeissI} for other description of the layers, and~\cite{BCKS} for $M=I^3$.
    
    However, in neither of those approaches could we understand the comparison map, which is crucial for the proof of Theorem~\ref{thm:main-thm}; we hope that our equivalence $\chi$ might be of independent interest.
\end{remm}
For the second statement \red{of Theorem~\ref{thm:W-iso}, 
}
we easily find that $\tofib(M\vee\S_{\bull})$
is $(n(d-2)-1)$-connected in Section~\ref{subsec:gen-maps}, and then construct an isomorphism
\[
    W\colon\Lie_{\pi_1M}(n)\to \pi_{n(d-2)}\tofib(\Omega(M\vee\S_{\bull}),\Omega\coll).
\]
It takes a decorated tree $\Gamma^{g_{\ul{n}}}$ to a certain Samelson product $\Gamma(x_i^{g_i})\colon\S^{n(d-2)}\to\Omega(M\vee\S_{\ul{n}})$, according to the word described by $\Gamma$ and using classes $g_{\ul{n}}\in(\pi_1M)^{n}$, \red{of the inclusions $x_i\colon \S^{d-2}\hra M\vee\S_{\ul{n}}$.
}

\red{Finally, we have the induced isomorphism
\[
    (\retr\deriv\chi)_*\colon \pi_{n(d-3)}\pF_{n+1}(M)\to\pi_{n(d-2)}\tofib(\Omega(M\vee\S_{\bull}),\Omega\coll),
\]
and we would like to make its inverse explicit. In other words, we would like to explicitly describe maps $\S^{n(d-3)}\to\pF_{n+1}$ that represent the generators $(\retr\deriv\chi)_*^{-1}W(\Gamma^{g_{\ul{n}}})$. For $d=3$ we will have such a map (point) $\ev_{n+1}(\psi(\TG))\in\pF_{n+1}(M)$ obtained from a grope $\TG$ (see Section~\ref{subsec-intro:gropes} below), and we will directly check that this satisfies
\begin{equation}\label{eq:key-statement}
    W^{-1}(\retr\deriv\chi)_*[\ev_{n+1}(\psi(\TG))]=\ut(\TG),
\end{equation}
the underlying tree of $\TG$, modulo $AS,IHX$ (see Section~\ref{subsec-intro:merge} below).
In~\cite{K-families} we will define gropes in all dimensions $d\geq3$ and show that the proof given here still applies.

However, for any $d\geq3$ we can at least directly invert $\retr$, i.e.\ construct maps that represent the generating classes $\retr_*^{-1}W(\Gamma^{g_{\ul{n}}})$. Namely, 
}
there is an obvious map $m_i\colon\S^{d-2}\to\Omega(M\sm\ball^d_i)\to\Omega M_{\ul{n}}$ satisfying $\retr\circ m_i\simeq x_i$, which simply `swings a lasso' around the missing ball, see Figure~\ref{fig:deg-2-m-i}. 
\red{The following is then proven in Section~\ref{subsec:gen-maps} and will be useful in the proof of the main Theorem~\ref{thm:main-thm}. 
}
\begin{manualthm}{\ref*{thm:W-iso}'}{\label{lem:W-iso'}}
    The generators of $\pi_{n(d-2)}\tofib(\Omega M_{\bull},\Omega\lambda)$ are \red{represented 
    }
    by the canonical extensions to the total homotopy fibre of the Samelson products $\Gamma(m_i^{\varepsilon_i\gamma_i})\colon\S^{n(d-2)}\to\Omega M_{\ul{n}}$ of the maps $m_i^{\varepsilon_i\gamma_i}\colon \S^{d-2}\to\Omega M_{\ul{n}}$ for $\varepsilon_i\in\{\pm1\}$ and $\gamma_i\in\Omega M$. \red{Explicitly, 
    }
    $\Gamma(m_i^{\varepsilon_i\gamma_i})(\vec{t})\coloneqq\gamma_i\cdot m_i(\vec{t})^{\varepsilon_i}\cdot\gamma_i^{-1}$. 
\end{manualthm}

\subsubsection{About the convergence}

\vspace{-5pt}

It is worth pointing out that our results are obtained from scratch, starting with the definition of the punctured knots model and assuming only the Hilton--Milnor--Gray--Spencer theorem. In particular, independently of the rest of the literature, we have reproved the following.
\begin{cor}\label{cor:tower-converge}
  The Taylor tower for the space $\Knots(M)$ of long knots in a $d$-manifold $M$ converges if $d\geq4$, i.e.\ the connectivity of $p_{n+1}\colon\pT_{n+1}(M)\to\pT_n(M)$ increases with $n\geq1$.
\end{cor}
\red{However, 
}
Goodwillie--Klein Theorem~\ref{thm:GK} \red{says in addition that 
}
the \emph{tower converges precisely to $\Knots(M)$} for $d\geq4$: the homotopy groups of $\H_n(M)$ in degrees below $(n+1)(d-3)$ agree with those of $\pF_{n+1}(M)$. We believe that our map $\rho_n$ will provide a geometric inverse to the isomorphism $\pi_{n(d-3)}\ev_n$, see Remark~\ref{rem:other-d}.

For a tower of (surjective) based fibrations there is an associated Bousfield--Kan (non-`fringed') spectral sequence built out of long exact sequences for the homotopy groups of a fibration~\cite[Sec.~IX.4]{BK}: the first page is given by $E^1_{-n,t}\coloneqq\pi_{t-n}\pF_n(M)$ and the differential is the composite
\begin{equation}\label{eq-def:sp-seq}
\begin{tikzcd}
    d^1_{-n,t}\colon E^1_{-n,t}(M)=\:\pi_{t-n}\pF_n(M)\arrow[]{r}{} & \pi_{t-n}\pT_n(M)\arrow[]{r}{}  & \pi_{t-n-1}\pF_{n+1}(M)\coloneqq E^1_{-(n+1),t}(M)
\end{tikzcd}
\end{equation}
of the fibre inclusion and the connecting map $\delta$ for $p_{n+1}$. A vanishing slope for this spectral sequence was determined in~\cite[Thm.~7.1]{Sinha-cosimplicial}; see also~\cite{Scannell-Sinha} for $M=I^d$. By Theorem~\ref{thm:W-iso} we have
\[
  E^1_{-(n+1),t}\coloneqq\pi_{t-n-1}\pF_{n+1}(M)\,\cong\bigoplus_{w\in\N\B(n)} \pi_t\Sigma^{1+l_w(d-2)}(\Omega M^{\times l_w})_+\;.
\]
\begin{cor}\label{cor:ss}
    The group $E^1_{-(n+1),t}(M)$ vanishes for $t\leq n(d-2)$, and for each $l=n,n+1,\dots$ all entries in the strip $1+l(d-2)\leq t\leq (l+1)(d-2)$ are generated by Samelson products using words of length at most $l$ in which all letters appear. In particular, the first non-vanishing slope is
\[
    E^1_{-(n+1),1+n(d-2)}(M)\,\cong\,\Lie_{\pi_1M}(n).
\]
\end{cor}

\subsubsection{About configuration spaces}

\vspace{-5pt}

\red{Embedding calculus is closely related to configuration spaces $\Conf_S(M)\coloneqq\Emb(S,M)$ where $S$ is a finite; see~\cite{Sinha-cosimplicial,BW2,FTW} to mention just a few references. For example, 
}
Sinha~\cite{Sinha-cosimplicial} uses certain compactifications of $\Conf_S(M)$ to construct the mentioned model $AM_n(M)$ for $\T_n\Knots(M)$, then employed in~\cite{BCKS}; see Remark~\ref{rem:bcks} for a comparison to our approach. Configuration spaces were also used by Koschorke~\cite{Koschorke-milnor} to construct invariants of link maps in arbitrary dimensions. His results are very similar in spirit to ours, showing that certain invariants related to Samelson products agree with Milnor invariants for classical links. These space are behind the scenes in our approach as well, and we record related corollaries. 
\begin{cor}\label{cor:conf-spaces}
    For a compact $d$-manifold $M$ with $\partial M\neq\emptyset$ there is an additive isomorphism
    \[
      \pi_*\Conf_n(M)\cong (\pi_*M)^n\oplus\bigoplus_{i=0}^{n-1}\bigoplus_{w\in\B(\ul{i})}\pi_{*+1}\big(\Sigma^{1+l_w(d-2)}(\Omega M^{\times l_w})_+\big).
    \]
\end{cor}
\begin{proof}
    Since $\partial M\neq\emptyset$ there is an isomorphism $\pi_*\Conf_n(M)\cong \bigoplus_{i=0}^{n-1}\pi_*(M\sm\ul{i})$, where $\ul{i}\coloneqq\{1,\dots,i\}$ is a set of $i$ distinct points in $M$~\cite{Levitt}. By Lemma~\ref{lem:retr} there is a retraction $\retr\colon M\sm S\to M\vee \S_S$ for any finite set $S$ (recall that $\S_S\coloneqq\bigvee_{S}\S^{d-1}$), and $\Omega(M\vee\S_S)\simeq\Omega M\times\prod^{\circ}_{w\in\B(S)}\Omega\Sigma^{1+l_w(d-2)}(\Omega M^{\times l_w})_+$ by \eqref{eq:gray-cor} and \eqref{eq:hm-cor} from Section~\ref{subsec:proof-ThmB}.
\end{proof}
\begin{cor}
    For a compact $d$-manifold $M$ with $\partial M\neq\emptyset$, if $\mc{s}^k_S\colon \Conf_S(M)\to\Conf_{S\sm k}(M)$ denotes the map forgetting the $k$-th point in the configuration, $k\in S\subseteq\ul{n}$, then
    \[
    \Omega\tofib\big(\Conf_{\bull}(M),\mc{s}\big)\simeq\sideset{}{^{\circ}}\prod_{w\in \N\B(\ul{n-1})}\Omega\Sigma^{1 + l_w(d-2)}(\Omega M^{\times l_w})_+\;.
    \]
    Hence, the first non-trivial homotopy group is $\pi_{(n-1)(d-2)+1}\tofib(\Conf_{\bull}(M),\mc{s})\cong\Lie_{\pi_1M}(n-1)$.
\end{cor}
\begin{proof}
    Each map $\mc{s}^n_{S\cup n}$ for $S\subseteq\ul{n-1}$ is a fibre bundle whose fibre is homeomorphic to $M\sm S\simeq M\vee\S_S$. By taking fibres first in the direction of $\mc{s}^n$-maps, the total fibre of the $n$-cube $(\Omega\Conf_{\bull}(M),\Omega\mc{s})$ is equivalent to that of the $(n-1)$-cube $\tofib(\Omega(M\vee\S_{\bull}),\Omega\coll)$, and this was computed in \eqref{eq:tofib-hm-announce}.
\end{proof}
    Note that one can show that the contravariant $n$-cube $(\Conf_{\bull}(M),\mc{s})$ is $((n-1)(d-2)+1)$-cartesian using the Blakers--Massey theorem \red{(see for instance~\cite[Ex.6.2.9]{MV}), 
    }
    but we could not find a computation of the whole homotopy type in the literature.

\subsection{Gropes give points in the layers of the Taylor tower}\label{subsec-intro:gropes}
In this geometric part we specialise to $d=3$ (but this restriction is not essential, see Remark~\ref{rem:other-d}). \red{For any oriented $3$-manifold $M$ we construct 
}
points in $\H_n(M)$ from \emph{(simple capped genus one) grope cobordisms} of degree $n$ of~\cite{CT1}.

\subsubsection*{Gropes}
In Section~\ref{subsec:gropes} we discuss grope cobordisms, which are certain $2$-complexes in $M$ modelled on trees, that `witness' $n$-equivalence of the two knots on `the boundary of a cobordism'.

More precisely, one first defines (see Definition~\ref{def:abstract-grope}) the \emph{abstract grope $G_\Gamma$} modelled on an undecorated tree $\Gamma\in\Tree(n)$ as a $2$-complex built by inductively attaching surface stages according to $\Gamma$: each leaf contributes a disk (called a \emph{cap}), and each trivalent vertex a \red{punctured torus (genus one surface with one boundary component). The upper stages are attached along the two simple closed curves of the punctured torus that generate its fundamental group, see the left part of Figure~\ref{fig:grope-intro}. 
}
We define the boundary $\partial G_\Gamma$ as the boundary of the bottom stage, which is the stage corresponding to the root of $\Gamma$; we also fix an oriented subarc $a_0\subseteq\partial G_\Gamma=\S^1$. Moreover, there is a canonical embedding $\Gamma\hra G_\Gamma$ of the tree into this $2$-complex.

A \emph{grope cobordism \red{(or a grope for short) 
}
on a knot $K\in\Knots(M)$ modelled on $\Gamma$} is a map $\G\colon G_\Gamma\to M$ which embeds all stages mutually disjointly and disjointly from $K$ except that $\G(a_0)\subseteq K$ and for $i\in\ul{n}$ the $i$-th cap intersects $K$ transversely in a point $p_i$, so that $\G(a_0)<p_1<\dots<p_n$ in $K$. The \emph{degree of $\G$} is $n$, the number of its caps. See Definition~\ref{def:grope-cob} for details and orientation conventions.

\begin{figure}[!htbp]
    \centering
    \includegraphics[width=0.8\linewidth]{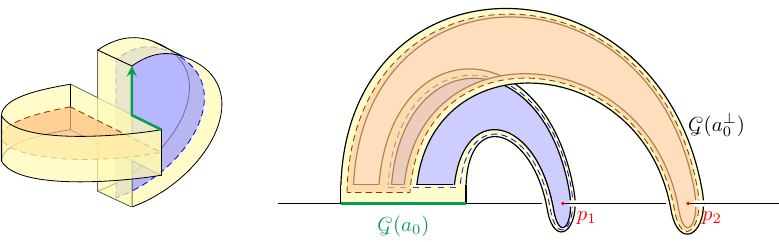}
    \vspace{-10pt}

    \caption[An abstract grope and a grope cobordism.]{\textit{Left}: The abstract grope $G_\Gamma$ modelled on $\Gamma=\tree$ is the union of the yellow punctured torus and the two disks, orange and blue. \textit{Right}: A grope cobordism $\G\colon G_\Gamma\to I^3$ on $K=\U$, the horizontal line. The knot $\partial^\perp\G$ is the union of $\U\sm\G(a_0)$ and the long black arc $\G(a_0^\perp)$, and is isotopic to the trefoil. \red{Although our pictures sometimes seem not smooth, the corners are there only for drawing convenience.
    }
    }
    \label{fig:grope-intro}
\end{figure}
A simple example of a grope cobordism of degree $2$ is shown on the right of Figure~\ref{fig:grope-intro}. Note how the two caps or `arms' could instead be twisted and tied into knots, producing non-isotopic grope cobordisms on $K$ which are all modelled on the same tree $\Gamma$. 

Moreover, we define $a_0^\perp\coloneqq\partial G_\Gamma\sm a_0$ and the \emph{output knot of $\G$} by
\[
    \partial^\perp\G\coloneqq(K\sm\G(a_0))\cup\G(a_0^\perp).
\]
Thus, a grope describes a modification of the knot $K$ by replacing its arc $\G(a_0)\subseteq K$ by $\G(a_0^\perp)$. We say that \red{the knot $K$ is $n$-equivalent to the knot $\partial^\perp\G$, and write
}
\begin{equation}\label{eq-def:n-equiv}
    K\sim_n\partial^\perp\G.
\end{equation}
In general, \emph{two knots are $n$-equivalent} if there is a finite sequence of grope cobordisms of degree $n$ from one knot to the other. This gives the variant due to~\cite{CT1} of the Gusarov--Habiro filtration $\KK_n(M;\U)\coloneqq\{K\in\KK(M):K\sim_n\U\}$ mentioned \red{at the start of Section~\ref{sec:intermediate-results}; 
}
see also Section~\ref{subsec:geom-finite-type-theory}.

Another important notion related to a grope cobordism $\G$ modelled on $\Gamma\in\Tree(n)$ is its \emph{signed decoration} $(\varepsilon_{i},\gamma_{i})_{i\in\ul{n}}$, where $\varepsilon_i\in\{\pm1\}$ is the sign of the intersection point $p_i\in K$ of $K$ and the $i$-th cap of $\TG$, and $\gamma_i\in\Omega M$ is the loop from $K(0)$ to $p_0$ first following the unique path in the tree $\G(\Gamma)$, then going back along $K$ (see Definition~\ref{def:underlying-decor-tree}). Finally, let $\upvarepsilon\coloneqq\sqcap_i\varepsilon_i$ and $g_i=[\gamma_i]\in\pi_1M$ and define the \emph{underlying decorated tree}
\[
    \ut(\G)\coloneqq\upvarepsilon\Gamma^{g_{\ul{n}}}\:
    \in\{\pm1\}\times\Tree_{\pi_1M}(n).
\]
\red{
Any $\pm\Gamma^{g_{\ul{n}}}\in\Z[\Tree_{\pi_1M}(n)]$ is realised by many mutually non-isotopic gropes, see Proposition~\ref{prop:uf-map}.%
}

\subsubsection*{Thickened gropes}
For the relation to the Taylor tower it is convenient to consider thickenings (tubular neighbourhoods) of grope cobordisms which we call \emph{thickened gropes}, see Definition~\ref{def:thick}. They can be described as embeddings $\TG\colon\ball_\Gamma\hra M$ of a certain model ball $\ball_\Gamma\cong\ball^3$ which contains $G_\Gamma$, so that $\TG(G_\Gamma)$ is a grope cobordism on $K$.
\begin{introthm}\label{thm:KST}
  If $\TG$ is a thickened grope of degree $n\geq1$ in $M$ on a knot $K$, then there is a path $\Psi^\TG\colon I\to\pT_n(M)$ from $\ev_n(\partial^\perp\TG)$ to $\ev_n(K)$. \red{In particular, for $K=\U$ we obtain a point
  \[
  \psi(\TG) \coloneqq (\partial^\perp\TG,\Psi^{\TG})\quad\in\H_n(M). 
  \]
  }
\end{introthm}
We give the proof in Section~\ref{subsec:grope-paths}\red{, following an idea for $M=I^3$ first discussed with Peter Teichner and Yuqing Shi, and which appeared in~\cite{Shi}. 
}
This uses the crucial \emph{isotopy between the two surgeries on a capped torus} (Lemma~\ref{lem:sym-surg-isotopy}). Namely, combining these isotopies for each stage of $G_\Gamma$ gives an $(n-1)$-parameter family of disks $\D_u$ contained in the model ball $\ball_\Gamma$ and with $\partial\D_u=\partial G_\Gamma$. Moreover, the interior of each disk $\TG(\D_u)$ intersects the knot $K$ only inside of certain subarcs of $K$, and so that the homotopy of $\TG(a_0^\perp)$ back to $\TG(a_0)$ across $\TG(\D_u)$ defines a path in $\pT_n(M)$.

The theorem immediately implies that there is a factorisation
\begin{equation}\label{eq:kst-factor}
\begin{tikzcd}[row sep=small]
    \KK(M)\arrow{rr}{\pi_0\ev_n}\arrow{dr} && \pi_0\pT_n(M)\\
    &\faktor{\KK(M)}{\sim_n}\arrow{ru}[swap]{}&
\end{tikzcd}
\end{equation}
Indeed, if $K\sim_n K'$, there is a sequence of thickened gropes witnessing it, so concatenation of the corresponding paths in $\pT_n(M)$ is a path from $\ev_nK$ to $\ev_nK'$. In particular, as mentioned in the discussion after Conjecture~\ref{conj:universal}, for $M=I^3$ this is \emph{equivalent} to the claim that $\pi_0\ev_n$ is a Vassiliev invariant of type $\leq n-1$ (this was first shown by~\cite{BCKS}). \red{The remaining part of Conjecture~\ref{conj:universal} says that if $\ev_nK$ is in the path component of $\ev_n\U$, then there exists a sequence of paths between them induced from thickened gropes (or a grope forest).
}
\begin{remm}\label{rem:comparison-general-M}
  This equivalence to Vassiliev's theory is discussed in Section~\ref{subsec:geom-finite-type-theory}. In contrast, it is an open problem (see Remark~\ref{conj:habiro-vassiliev}) if such an equivalence holds for any oriented $3$-manifold $M$: the factorisation \eqref{eq:kst-factor} just says that $\pi_0\ev_n$ is an invariant of $n$-equivalence of knots in $M$.
\end{remm}


\subsubsection*{Grope forests}
We will also need to realise arbitrary linear combinations of decorated trees. \red{On the geometric side this corresponds to mentioned sequences of grope cobordisms, or a grope cobordism of `higher genus', see \cite{CT2}. 
}
We instead define a slightly different notion, called a \emph{grope forest} (see Definition~\ref{def:forest}). This is an embedding
\[
\forest\colon\bigsqcup_{1\leq l\leq N}\ball_{\Gamma_l}\hra M
\]
such that $\forest|_{\ball_{\Gamma_l}}$ for $1\leq l\leq N$ are mutually disjoint thickened gropes on $K$ whose arcs $\forest|_{\ball_{\Gamma_l}}(a_0)\subseteq K$ appear in the order of their label $l$. 
Taking underlying trees of all the thickened gropes in a grope forest gives the \emph{underlying decorated tree map}
\[\begin{tikzcd}
  \ut\colon\;\{\text{degree $n$ grope forests on $\U$ in $M$}\}\arrow[two heads]{r}{} & \Z[\Tree_{\pi_1M}(n)].
\end{tikzcd}
\]
This is a \emph{surjection} of sets: any linear combination of $\pi_1M$-decorated trees is realised by a grope forest on $\U$ (see Proposition~\ref{prop:uf-map}). Furthermore, we extend the map $\psi$ from Theorem~\ref{thm:KST} to grope forests to obtain points $\psi(\forest)\in \H_n(M)$ (see Proposition~\ref{prop:extended-KST}).

\begin{remm}\label{rem:other-d}
  One can generalise gropes to any $d\geq3$ by simply replacing the model $3$-ball $\ball_\Gamma$ by a $d$-dimensional ball obtained as a $(d-2)$-thickening of the $2$-complex $G_\Gamma$. Then a thickened grope in $M$ is again an embedding $\ball_\Gamma\hra M$ so that for each $1\leq i\leq n$ the neighbourhood of the $i$-th cap ($\cong\D^2\times\D^{d-2}$) intersects $K$ in a neighbourhood of a single point $p_i\in K$.

  One can similarly construct maps $\psi(\TG)\colon\S^{n(d-3)}\to\H_n(M)$, using that $(d-1)$-dimensional normal disks to $K$ at $p_i$'s give an $n(d-2)$-family of arcs. This gives points $\emap_{n+1}\psi(\TG)\in\Omega^{n(d-3)}\pF_{n+1}(M)$ and we believe the proofs of our main theorems below readily extend to show that $\emap_{n+1}$ is a surjection onto the first non-trivial group $\pi_{n(d-3)}\pF_{n+1}(M)$ for any $d\geq3$. We will explore this in future work.
\end{remm}

\subsection{The key step: the underlying tree is detected in the Taylor tower}\label{subsec-intro:merge}
The first step on the journey relating the homotopy theory of punctured knots and the geometry of gropes was to connect them both to the language of decorated trees: they generate the group of components of the layers and also underlie gropes. It remains to show their compatibility via
\[\begin{tikzcd}
  \{\text{degree $n$ grope forests on $\U$ in $M$}\}\arrow[]{r}{\psi} & \H_n(M)\arrow[]{r}{\emap_{n+1}} & \pF_{n+1}(M).
\end{tikzcd}
\]
\begin{introthm}[Main Theorem]\label{thm:main-thm}
For a thickened grope $\TG$ of degree $n$ on $\U$ in $M$ the connected component of the point
$\emap_{n+1}\psi(\TG)\in\pF_{n+1}(M)$ is given by the class of its underlying tree:
\[
    [\emap_{n+1}\psi(\TG)]=[\ut(\TG)]\in\Lie_{\pi_1M}(n).
\]
\end{introthm}
This is the statement~\eqref{key-statement} from the introduction, see also Corollary~\ref{cor:detect-tree}. See Section~\ref{sec:main} for the proof of this key step and Section~\ref{subsec:gen-maps} for its outline. We now give a very brief sketch.

More explicitly, for a thickened grope $\TG\colon\ball_\Gamma\to M$ on $\U$ with the underlying tree $\upvarepsilon\Gamma^{g_{\ul{n}}}$, we claim
\[
    \big[\emap_{n+1}\psi(\TG)\big] =\big[\upvarepsilon\Gamma^{g_{\ul{n}}}\big] \quad\in\pi_0\pF_{n+1}(M)\cong\Lie_{\pi_1M}(n).
\]
We prove this in Section~\ref{sec:main} using the above Lemma~\ref{lem:W-iso'}: it is enough to show that the Samelson product $\Gamma(m_i^{\varepsilon_i g_i})\colon \S^n\to \Omega M_{\ul{n}}$ is homotopic to the map (the initial coordinate of $\deriv\big(\chi\emap_n\psi(\TG)\big)$):
\[
    \deriv\big(\chi\emap_n\psi(\TG)\big)^{\ul{n}}\colon
    \:\S^n\to \Omega M_{\ul{n}}.
\]
The maps $\deriv$ and $\chi$ were constructed in Theorem~\ref{thm:W-iso}. The idea of the proof is to use inductive descriptions of both Samelson products (see Lemma~\ref{lem:samelson-inductive}) and thickened gropes, to reduce to checking that $\deriv\big(\emap_n\psi(\TG)\big)^{\ul{n}}$ is homotopic to a certain pointwise commutator map.

A crucial step for this reduction is our description of the map $\chi$ in Appendix~\ref{app:proofs} as
\[
  \big(\chi f\big)^{\ul{n}}=\glueOp_{S\subseteq\ul{n}}(f^{\ul{n}})^{h^S}.
\]
This is a map on $\S^n$ obtained by gluing together along faces $2^n$ different maps $(f^{\ul{n}})^{h^S}$, each defined on $I^n$ as a certain `$h^S$-reflection' of the original map $f^{\ul{n}}$ across a face of $I^n$.

Hence, every generator of $\pi_0\pF_{n+1}(M)$ is in the image of $\pi_0\emap_{n+1}$. However, to prove that this map of sets is surjective we must also realise linear combinations. Recall that they arise as underlying trees of grope forests, so the following is all we need (for a proof see the end of Section~\ref{sec:main}).
\begin{introthm}\label{thm:main-extended}
  For any grope forest $\forest$ of degree $n$ on $\U$ in $M$ we have
  \[
    [\emap_{n+1}\psi(\forest)]=[\ut(\forest)]\in\Lie_{\pi_1M}(n).
  \]
\end{introthm}
\begin{proof}[Proof of Theorem~\ref{the-main}]
  The statements about $\pF_{n+1}$ are contained in Theorem~\ref{thm:W-iso}. It only remains to construct the map $\rho_n$ making the lower triangle in the following diagram commute:
  \begin{equation}\label{eq:main-extended}
      \begin{tikzcd}[column sep=large,row sep=large]
        \{\text{degree $n$ grope forests on $\U$ in $M$}\}\arrow[]{d}[swap]{\psi}\arrow[two heads]{r}{\ut} & \Z[\Tree_{\pi_1M}(n)]\arrow[two heads]{d}{\text{mod}~(AS,IHX)}\arrow[dashed]{dl}[swap]{\rho_n}\\
      \pi_0\H_n(M)\arrow[]{r}{\pi_0\emap_{n+1}} & \pi_0\pF_{n+1}(M)
  \end{tikzcd}
  \end{equation}
  Observe that Theorem~\ref{thm:main-extended} is equivalent to the commutativity of the outer square. Thus, it is enough to pick any set-theoretic section of $\ut$. In other words, \red{for $T\in\Z[[\Tree_{\pi_1M}(n)]$ let $\rho_n(T)\coloneqq[\psi(\forest)]$ for any grope forest $\forest$ on $\U$
  }
  with the underlying tree $\ut(\forest)=T$.
\end{proof}

\subsection{Further consequences and examples}
The rational collapse along the diagonal \red{of the spectral sequence from Corollary~\ref{cor:rational-universal} 
}
is used in the following argument of Goodwillie, which we include for completeness.
\begin{prop}[Goodwillie]\label{prop:non-conv}
 The set $\pi_0\pT_\infty(I^3)$ is uncountable, so the Taylor tower does not converge to $\Knots(I^3)$, i.e. the map $\ev_\infty\colon\Knots(I^3)\to\pT_\infty(I^3)\coloneqq\lim\pi_0\pT_n(I^3)$ is \emph{not} a weak equivalence.
\end{prop}
\begin{proof}
We claim that $\pi_0\pT_\infty(I^3)$ is uncountable, whereas $\pi_0\Knots(I^3)$ is countable. We have a surjection of sets $\pi_0\pT_\infty(I^3)\twoheadrightarrow\lim\pi_0\pT_n(I^3)$, so it is enough to show that $\lim\pi_0\pT_n(I^3)$ is uncountable.

Indeed, $\pi_0p_{n+1}\colon\pi_0\pT_{n+1}(I^3)\twoheadrightarrow\pi_0\pT_n(I^3)$ is surjective for every $n\geq1$, but not injective since we saw in Corollary~\ref{cor:rational-universal} that $\ker(\pi_0p_{n+1})\otimes\Q\cong\A^T_n\otimes\Q$, and these groups are non-trivial for all $n\geq2$.
\end{proof}
Nevertheless, if we denote the Gusarov-Habiro completion by $\wh{\KK(I^3)}_\Z\coloneqq\lim\faktor{\KK(I^3)}{\sim_n}$ (see Section~\ref{subsec:geom-finite-type-theory}), we have an exact sequence
\begin{equation}\label{eq:limits}
    \begin{tikzcd}
        \lim\big(\ker\pi_0\ev_n\big) \arrow[hook]{r}{} & \wh{\KK(I^3)}_\Z\arrow[]{r}{\ol{\ev}_\infty} & \lim\pi_0\pT_n(I^3)\arrow[two heads]{r}{} & \textstyle\lim^1\big(\ker\pi_0\ev_n\big).
    \end{tikzcd}
\end{equation}
This is obtained by taking limits in \eqref{eq:two-ses}, using Milnor's $\lim^1$-sequence and $\lim^1\big(\faktor{\KK(I^3)}{\sim_n}\big)=0$, as the maps in that tower are surjective. Thus, if Conjecture~\ref{conj:universal} is true then $\ol{\ev}_\infty$ is an isomorphism.

\begin{prop}\label{prop:group-structures}
Two group structures on the set of components of Taylor stages constructed by~\cite{BW2} and~\cite{BCKS} are equivalent, i.e. $\pi_0\T_n^{BW}\cong\pi_0 AM_n$ as abelian groups. More generally, any two group structures on $\pi_0\T_n\Knots(I^3)$ respecting the connected sum of knots must agree.
\end{prop}
\begin{proof}
  The models are weakly equivalent, so there is a bijection of sets $f\colon\pi_0 AM_n\to\pi_0\T_n^{BW}$ so that $f\circ\pi_0ev_n=\pi_0\ev^{BW}_n$. Since $\pi_0ev_n$ is a \emph{surjective} monoid map, for $x_i\in\pi_0AM_n$ we find $K_i\in\KK(I^3)$ with $\pi_0ev_n(K_i)=x_i$; then $x_1+x_2=\pi_0ev_n(K_1\# K_2)$. Using that $\pi_0\ev^{BW}_n$ is also a monoid homomorphism we have
  \[
    f(x_1+x_2)=\pi_0\ev^{BW}_n(K_1\# K_1) =\pi_0\ev^{BW}_n(K_1)+\pi_0\ev^{BW}_n(K_2)=f(x_1)+f(x_2).\qedhere
  \]
\end{proof}

\begin{figure}[!htbp]
        \centering
        \includegraphics[width=0.9\linewidth]{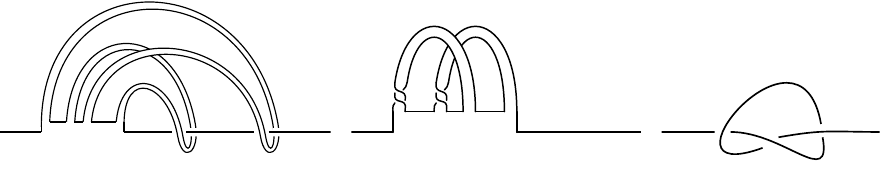}
    \caption[The trefoil.]{
        \textit{Left}: The knot $\partial^\perp\G$ for $\G$ from Figure~\ref{fig:grope-intro}. \textit{Middle}: `Swinging' the bands of $\partial^\perp\G$ is an isotopy which introduces twists into the bands giving a projection of the right-handed trefoil, cf.\ Example~\ref{ex:n-3-classical}. \textit{Right}: The standard diagram of the right-handed trefoil.
    }
    \label{fig:trefoil}
\end{figure}
\begin{example}\label{ex:n-2-classical}
Let $n=2$. The grope cobordism $\G$ on the unknot $\U$ from Figure~\ref{fig:grope-intro} is modelled on the unique tree of degree $2$, and has $K=\partial^\perp\G$ the right handed trefoil $(RHT)$, see Figure~\ref{fig:trefoil}. Thus, $RHT$ is 2-equivalent to the unknot, $RHT\sim_2\U$, \red{as defined in~\eqref{eq-def:n-equiv}. 
}
Actually, every knot is $2$-equivalent to the unknot (see e.g.~\cite{Murakami-Nakanishi}), so
\[
    \faktor{\KK(I^3)}{\sim_2}=\big\{[\U]\big\}.
\]
Moreover,~\cite{BCSS} show that $\pT_2(I^3)\simeq*$,
confirming Conjecture~\ref{conj:universal} in degree $n=2$.
\end{example}

\begin{example}\label{ex:n-3-classical}
\red{From the previous example it follows that 
}
the first non-trivial classical knot invariant from embedding calculus is
\[
  \pi_0\ev_3\colon\;\KK(I^3)\to\pi_0\pT_3(I^3)\cong\pi_0\pF_3(I^3)\cong\Lie(2)\cong\Z.
\]
Using the linking of certain `colinearity submanifolds' of configuration spaces,~\cite{BCSS} show that $\pi_0\ev_3$ agrees with the unique Vassiliev invariant $v_2$ of type $\leq2$ taking value $1$ on $RHT$. Classically, $v_2$ is given as the second coefficient of the Conway polynomial (lifting the Arf invariant) and induces
\[
v_2\colon\;\faktor{\KK(I^3)}{\sim_3}\xrightarrow{\cong} \Z.
\]
Our approach not only recovers $\pi_0\ev_3=v_2$ but also lifts this computation to the fibres via the map $\emap_3\colon\H_2(I^3)\to\pF_3(I^3)$. Namely, by Example~\ref{ex:n-2-classical} for any $K\in\KK(I^3)$ there exists a grope forest $\forest$ of degree~$2$ from $K$ to $\U$. By the extension of Theorem~\ref{thm:KST} for grope forests we get a point $\psi(\forest)\in\H_2(I^3)$ and by definition $\pi_0\ev_3(K)=[\emap_3\psi(\forest)]\in\Lie(2)$. Now, our main Theorem~\ref{thm:main-extended} says that this element is the class of the underlying trees of $\forest$. Actually, in Example~\ref{ex:f-for-n=3} we do this computation for the case $K=RHT$, as a warm-up problem for the proof of that theorem.

We get precisely $\pi_0\ev_3(RHT)=\tree$, implying $\pi_0\ev_3(K)=v_2(K)\cdot\tree$ by the uniqueness of $v_2$. If we then for computing the coefficient $v_2(K)$ use the Hopf invariant $\Lie(2)\subseteq\pi_3(\S^2\vee\S^2)\to\Z$ given as the linking number $\mathrm{lk}(f^{-1}(p_1),f^{-1}(p_2))$ for a suitable representative $f\colon \S^3\to\S^2\vee\S^2$ of the desired homotopy class, then we are in the colinearity story of~\cite{BCSS}.
\end{example}


\section{Preliminaries}\label{sec:prelim}
\subsection{Trees}\label{subsec-prelim:trees}

\subsubsection*{Lie trees} Fix a finite non-empty set $S$ and an integer $d\geq2$.
\begin{defn}\label{def:trees}
    A \textsf{(vertex-oriented uni-trivalent) tree} is a connected simply connected graph with vertices of valence three or one and with cyclic order of the edges incident to each trivalent vertex, called the \textsf{vertex orientation}. In the pictures this is specified by the positive orientation of the plane.
    
    A \textsf{rooted tree} $\Gamma\in\Tree(S)$ is a tree with one distinguished univalent vertex (the \textsf{root}) and all other univalent vertices (the \textsf{leaves}) labelled in bijective manner by the set $S$. \red{The degree of $\Gamma$ is the cardinality $|S|$ of $S$. 
    }
    The \textsf{grafting} of two rooted trees $\Gamma_{\myj}\in \Tree(S_{\myj})$ for $\myj=1,2$ is the rooted tree
    \[
    \grafted\in\Tree(S_1\sqcup S_2)
    \]
    obtained by gluing the two roots together and `sprouting' a new edge with a new root.
    
    Define the \textsf{group of Lie trees} $\Lie(S)\coloneqq\faktor{\Z\left[\Tree(S)\right]}{AS,IHX}$ as the quotient of the free abelian group on the set of rooted trees by the following relations (dots represent the remaining unchanged part of an arbitrary tree):
    \begin{equation}\label{eq-def:AS-IHX}
    \begin{aligned}
        \includegraphics{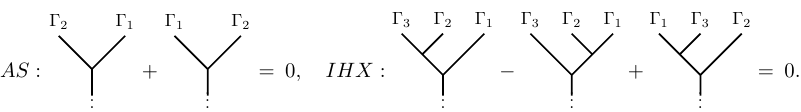}
    \end{aligned}
    \end{equation}
\end{defn}
We now relate Lie trees to words (Lie monomials) in the free Lie algebra.

\begin{defn}\label{def:lie-algebras}
    Let $\L_d(S)=\L(x^k:k\in S)$ be the free $\mathbb{N}_{0}$-\emph{graded} Lie algebra over $\Z$, with each $x^k$ having degree $|x^k|=d-2$. Thus, the degree of a word $w\in\L_d(S)$ is $|w|=l_w(d-2)$ where $l_w$ is the length of $w$, that is, the total number of letters in $w$.
    
    The \textsf{normalised Lie algebra} $\N\L_d(S)$ is the Lie subalgebra of $\L_d(S)$ generated by the words in which every letter appears at least once.\footnote{That is, $\N\L_d(S)\coloneqq\bigcap_{k\in S} \ker(\mc{s}^k\colon\L_d(S)\to\L_d(S\sm k))$, where $\mc{s}^k$ replaces each appearance of $x^k$ by zero.} Let $\Lie_d(S)\subseteq\N\L_d(S)$ be its subgroup generated by the words in which each letter appears exactly once. This is precisely the part of degree $|S|(d-2)$ of $\N\L_d(S)$.
\end{defn}
If $d=2$ one can assign to a tree $\Gamma\in\Lie(S)$ a Lie word $\omega_2(\Gamma)\in\Lie_2(S)$ using vertex orientations, so that the grafting of trees precisely corresponds to the Lie bracket. As relations in~\eqref{eq-def:AS-IHX} correspond to the antisymmetry and Jacobi relations in $\Lie_2(S)$ we get an isomorphism $\Lie(S)\cong\Lie_2(S)$.

However, Lie words for a general $d\geq2$ satisfy the \emph{graded} antisymmetry and Jacobi relations:
\begin{align}
    &[w_1,w_2]+(-1)^{|w_1||w_2|}[w_2,w_1]=0,\nonumber\\
    &[w_1,[w_2,w_3]]-[[w_1,w_2],w_3]-(-1)^{|w_1||w_2|}[w_2,[w_1,w_3]]=0,\label{eq:graded-jacobi}
\end{align}
whereas the relations~\eqref{eq-def:AS-IHX}, which are inspired by applications in geometric topology (see~\cite{CST} for instance), never involve graded signs. Nevertheless, as we learned from~\cite{Conant} and~\cite{Robinson-part-cx} a correspondence can be obtained as follows. Let us denote
\[
    (1|2)_d\coloneqq(d-2)\cdot\left|\big\{(i_1,i_2)\in S_1\times S_2: i_1>i_2 \big\}\right|.
\]
\begin{lemma}\label{lem:lie-tree}
    If $S$ is ordered, there is an isomorphism of abelian groups $\omega_d\colon\Lie(S)\xrightarrow{\cong} \Lie_d(S)$ defined inductively on $|S|$ by $\ichord{i}\:\mapsto\: x^i$ and for $S=S_1\sqcup S_2$ and $\Gamma_{\myj}\in\Tree(S_{\myj})$ by
    \[
        \grafted\:\mapsto\: (-1)^{(1|2)_d}\big[\omega_d(\Gamma_1),\omega_d(\Gamma_2)\big].
    \]
\end{lemma}
Hence, one can think of graded Lie words also as Lie trees, but keeping in mind that for odd $d$ the isomorphism $\omega_d$ introduces a sign. \red{We prove this lemma in Appendix~\ref{subsec:lem:lie-tree}.
}

For $\ul{n}\coloneqq\{1,\dots,n\}$ we write $\Tree(n)\coloneqq\Tree(\ul{n})$ and $\Lie(n)\coloneqq\Lie(\ul{n})$. Their elements can alternatively be drawn in the plane as in Figure~\ref{fig:Lie-trees}: the root and leaves are attached to a fixed horizontal line according to their increasing label, with the root labelled by $0$ (the edges might intersecting, but this is not part of the data). The vertex orientation is still induced from the plane.
\begin{figure}[!htbp]
    \centering
    \includegraphics{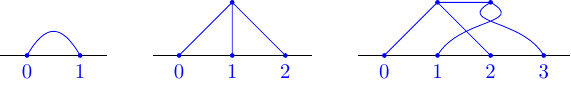}
    \caption[Drawing Lie trees in the plane.]{Trees $\omega\colon\ichord{i}\mapsto x^i,\tree\mapsto [x^1,x^2],\begin{tikzpicture}[baseline=0.2ex,scale=0.21,every node/.style={scale=0.78,font=\bfseries}]
        \clip (-2.2,-0.21) rectangle (2.2,4.13);
        \draw[fill=black!30!white]
            (-0.15, -0.2) rectangle ++ (0.3,0.2);
        \draw[thick]
            (0,0) -- (0,1) --
            (-1,2) -- (-2,3) node[pos=1,above]{$1$}
            (-1,2) -- (0,3)    node[pos=1,above]{$3$}
                    (0,1) -- (2,3)  node[pos=1,above]{$2$};
    \end{tikzpicture}\mapsto(-1)^{d-2}[x^2,(-1)^{d-2}[x^3,x^1]]$ drawn in the plane.}
    \label{fig:Lie-trees}
\end{figure}

Using the $AS$ and $IHX$ relations repeatedly one can show that $\Lie(n)\cong\Z^{(n-1)!}$, with a basis given by trees from Figure~\ref{fig:tree-left-normed} (ignoring red decorations for now) for various permutations $\sigma\in\mc{S}_{n-1}$, corresponding to left-normed Lie words $[x^{\sigma(1)},[x^{\sigma(2)},[\dots[x^{\sigma(n-1)},x^n]\dots]]]$. However, $\Lie(n)$ is also an interesting $\mc{S}_n$-representation, by permuting the leaf labels, giving the arity $n$ of the \emph{Lie operad}.

\subsubsection*{Decorated trees}
For manifolds with non-trivial fundamental group we need more general trees. Let $S$, $d$ be as above, and $\pi$ be a set. We write $\pi^S\coloneqq\Map(S,\pi)$ (so $\pi^{\ul{n}}=\pi^n$ is the cartesian product), and if $\pi=\pi_1M$, then we assume $d\coloneqq\dim M$ by convention.
\begin{defn}\label{def:decorated-lie}
    Define the abelian group $\Lie_{\pi}(S)$ of \textsf{$\pi$-decorated Lie trees} as
    \[
        \Lie_{\pi}(S)\coloneqq\Lie(S)\otimes\Z[\pi^S].
    \]
\end{defn}
If we let $\Tree_\pi(S)\coloneqq\Tree(S)\times \pi^S$, then $\Lie_{\pi}(S)$ is the quotient of $\Z[\Tree_{\pi}(S)]$ by the relations analogous to $AS,IHX$ from~\eqref{eq-def:AS-IHX}, which respect decorations in the natural way. We denote elements of $\Tree_\pi(S)$ by $\Gamma^{g_S}$, where $\Gamma\in\Tree(S)$ and $g_S\coloneqq(g_i)_{i\in S}\in\pi^S$, and call them \textsf{$\pi$-decorated trees}. Namely, these are rooted trees whose leaves are labelled bijectively by $S$, and additionally for each $i\in S$ the edge incident to the leaf $i$ is assigned an element $g_i\in \pi$, called a \textsf{decoration}.
\begin{figure}[!htbp]
    \centering
    \includegraphics{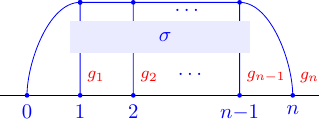}
    \caption[A left-normed decorated tree.]{A left-normed tree $\Gamma^{g_{\ul{n}}}\in\Tree_{\pi_1M}(n)$, with $g_{\ul{n}}\coloneqq(g_1,\dots,g_n)\in(\pi_1M)^{n}$ and $\sigma\in\mc{S}_{n-1}$.}
    \label{fig:tree-left-normed}
\end{figure}

When $\pi=\pi_1M$ we view a $\pi_1M$-decorated tree $\Gamma^{g_{S}}$ as a \emph{homotopy class of a map} $x\colon\Gamma\to M$ which takes the root and leaves to some fixed arc $K$ in $M$. \red{In our applications $K$ is a long knot in $M$, and the tree will be the underlying tree of a grope cobordism $\G\colon G_\Gamma\hra M$ on $K$; recall from Section~\ref{subsec-intro:gropes} that $\Gamma$ can be embedded into $G_\Gamma$, see Definition~\ref{def:underlying-decor-tree} for details. 
}
The decoration $g_i\in\pi_1M$ for the $i$-th leaf \red{will then be 
}
the homotopy class of the loop $\gamma_i$, which goes from $x(0)$ (the basepoint) to the leaf $x(i)$ along the unique path in the tree connecting the $i$-th leaf and the root, then back from $x(i)$ to $x(0)$ along $K$. More precisely, the basepoint of $M$ is fixed in some $p_0\in K$ and the described loop $\gamma_i$ should be conjugated by the piece of $K$ between $p_0$ and $x(i)$.

\red{
    The $\pi$-decorated Lie trees will for us exactly arise in this manner: we will fix an embedding of a tree $\Gamma$ into the $3$-ball $\ball^3$, and consider certain embeddings of this model ball (denoted by $\ball^3_\Gamma$) into our $3$-manifold $M$ with a particular intersection pattern with a fixed knot $K$ in $M$, so that we will obtain elements $g_i\in\pi_1M$ as explained in the previous paragraph.
}

\begin{remm}
The group $\Lie_{\pi}(n)$ is equal to $\Lambda_{n-1}(\pi,n+1)$ from the work of Schneiderman and Teichner~\cite{Schneiderman-Teichner14}, where the more general groups $\Lambda_n(\pi,m)$ were used as targets for obstruction invariants for pulling apart $m$ surfaces in $M$; see~\cite[Lem.~2.1]{Schneiderman-Teichner14} for this identification.
\end{remm}

\subsection{Homotopy limits}\label{subsec-prelim:holims}
A \textsf{diagram} over a small category $\mc{C}$ is a functor $X_{\bull}\colon \mc{C}\to\mathsf{Top}_*$ to the category of based topological spaces. Let $\mathsf{Top}_*^{\mc{C}}$ denote the category of diagrams over $\mc{C}$. The categories we will be using are the cube $\mc{C}=\Cube(S)$ and punctured cube categories,  $\mc{C}=\PCube(S)$, see below.

We will need the notion of a \textsf{homotopy limit} of a diagram $X_{\bull}\in\mathsf{Top}_*^{\mc{C}}$; for \red{details see~\cite[Ch.~XI]{BK} or~\cite[Ch.~8]{MV}. 
}
This is the space $\holim(X_{\bull})\in\mathsf{Top}_*$, also written $\holim_{c\in\mc{C}}X_c$, defined as the mapping space
\[
    \holim(X_{\bull})\coloneqq\Map_{\mathsf{Top}_*^{\mc{C}}}\big(\,|\mc{C}\downarrow\bull|,\,X_{\bull} \big).
\]
Firstly, $|\mc{C}\downarrow\bull|\in\mathsf{Top}_*^{\mc{C}}$ is the diagram which sends $c\in\mc{C}$ to the classifying space of the category $(\mc{C}\downarrow c)$, called the overcategory, whose objects are morphisms $c'\to c$ in $\mc{C}$, and arrows are triangles over $c$ in $\mc{C}$. The classifying space $|\mc{D}|$ of a category $\mc{D}$ is the geometric realisation of the nerve of $\mc{D}$ -- the simplicial set whose $k$-simplices are sets of $k$-composable arrows in $\mc{D}$. Finally, the mapping space between two objects in $\mathsf{Top}_*^{\mc{C}}$ is defined as the set of natural transformations between the two diagrams and is seen as a subspace of $\prod_{c\in\mc{C}} \Map\big(|\mc{C}\downarrow c|,X_c\big)$ from which it inherits the topology. 

In other words, a point $f\in\holim(X_{\bull})$ consists of a collection of maps $f^c\colon|\mc{C}\downarrow c|\to X_c$ \red{for every $c\in\mc{C}$ 
}
which are compatible with respect to the morphisms in $\mc{C}$. \red{As mentioned, the cases relevant to us are the cube and punctured cube categories, and we describe the respective homotopy limits more explicitly in~\eqref{eq-def:holim-cube} and in~\eqref{eq-def:holim-punc-cube} below.
}

The crucial property of a homotopy limit is its \emph{homotopy invariance}: if $X_{\bull}\to Y_{\bull}$ is a map of diagrams such that each $X_c\to Y_c$ is a weak equivalence, then the induced map $\holim X_{\bull}\to\holim Y_{\bull}$ is a weak equivalence as well.

For a finite non-empty set $T$ we consider the cube category $\mathcal{C}=\Cube(T)$, which is the poset of all subsets of $T$, and the punctured cube category $\mc{C}=\PCube(T)$, which is the poset of all non-empty subsets of $T$. In this paper we will use both $T=[n]\coloneqq\{0,1,2,\dots,n\}$ and $T=\ul{n}\coloneqq\{1,2,\dots,n\}$.

\subsubsection*{Punctured cubes}
A \textsf{punctured $|T|$-cube} is a diagram over $\PCube(T)$ and will be denoted by
\[
    (X_{\bull},x)\colon\PCube(T)\to\mathsf{Top}_*.
\]
It is given by spaces $X_S$ for $\emptyset\neq S\subseteq T$ and mutually compatible maps $x^k_S\colon X_S\to X_{S\cup k}$ for $k\in T\sm S$. Note that this diagram indeed has a shape of a cube of dimension $|T|$, with one vertex removed (corresponding to the initial object of $\Cube(T)$), see the first and third pictures in Figure~\ref{fig:simplex}. Moreover, $|\PCube(S)|$ is the geometric realisation of the simplicial set obtained by barycentric subdivision of the standard simplex $\Delta^S$, see Figure~\ref{fig:simplex}. We denote by $\Delta^S$ the standard simplex on the vertex set $S$, which has dimension $|S|-1$.

\begin{figure}[!htbp]
    \centering
    \includegraphics[width=0.9\linewidth]{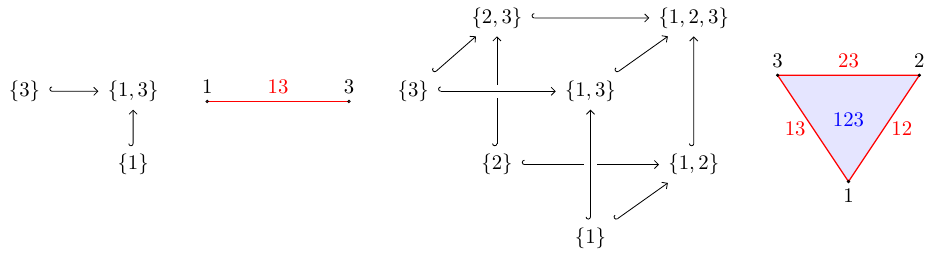}
    \caption[Examples of punctured cubes and barycentrically subdivided simplex.]{Examples of $\PCube(S)\cong\PCube(T)\downarrow S$ and $\Delta^S$ for $S=\{1,3\}$ and $S=\ul{3}$.}
  \label{fig:simplex}
\end{figure}
Let us describe a point in $\holim(X_{\bull},x)$. First observe that the overcategory $\PCube(T)\downarrow S$ is isomorphic to $\PCube(S)$. Thus, there is a levelwise homeomorphism of punctured cubes
\begin{equation}\label{eq:punc-cube-simplex}
    |\PCube(T)\downarrow \bull|\cong\Delta^{\bull}.
\end{equation}
Indeed, the maps in $\PCube(T)\downarrow \bull$ precisely correspond to inclusions $\iota_S^k\colon\Delta^S\hra\Delta^{S\cup k}$ for $k\notin S$ of the face whose $k$-the barycentric coordinate is zero. Hence, we have
\begin{equation}\label{eq-def:holim-punc-cube}
    \holim(X_{\bull},x)
    =\Map_{\mathsf{Top}_*^{\PCube(T)}} \big(\Delta^{\bull},X_{\bull} \big).
\end{equation}
\begin{cor}\label{cor:holim-pt}
    For a punctured $|T|$-cube $X_{\bull}$ a point $f\in\holim(X_{\bull},x)$ consists of maps $f^S\colon\Delta^S\to X_S$ such that for all $k\notin S\subseteq T$ the following diagram commutes
\[\begin{tikzcd}
    \Delta^S\arrow[]{r}{\iota_S^k}\arrow[]{d}[swap]{f^S} &\Delta^{S\cup k}\arrow[]{d}{\;f^{S\cup k}}\\
    X_S\arrow[]{r}{x^k_S} & X_{S\cup k}
\end{tikzcd}
\]
\end{cor}
For $T=[1]=\{0,1\}$ the limit of a diagram over $\PCube(T)$ is called the \textsf{pullback}, and its homotopy limit is the \textsf{homotopy pullback}. We present
\begin{equation}\label{eq:holim-pic}
f=(f^0,f^1,f^{01})\in\holim\left(\begin{tikzcd}
    {\color{blue}X_1} \arrow{r}{x^0_1} & {\color{green!50!black}X_{01}} \\
    {} & {\color{orange}X_0}\arrow{u}[swap]{x^1_0}
\end{tikzcd}\right)
\quad
\text{ pictorially by }
\begin{tikzpicture}[scale=0.8,baseline=1.1cm,every node/.style={scale=0.9}]
\clip (-1.5,-0.2) rectangle (3.2,3.6);
    \draw[thick,blue]
        (-0.9,3) circle (0.9pt) node[left]{$f^1$};
    \draw[thick,orange!90!black]
        (2,0.2) circle (0.9pt) node[right]{$f^0$};
    \draw[green!50!black]
        (0,3) circle (0.9pt) node[above]{$x_1^0(f^{1})$}
        -- (2,3)  node[right]{$f^{01}$}
        -- (2,1) circle (0.9pt) node[right]{$x_0^1(f^{0})$};
\end{tikzpicture}
\quad.
\end{equation}

\subsubsection*{Cubes}
Similarly, a \textsf{$|T|$-cube} is a diagram
\[
    (Y_{\bull},y)\colon\Cube(T)\to\mathsf{Top}_*,
\]
which consists of spaces $Y_S$ for $S\subseteq T$ and mutually compatible maps $y^k_S\colon Y_S\to Y_{S\cup k}$ for $k\in T\sm S$.
\begin{remm}\label{rem:contravar-cubes}
    We will also use diagrams over $\Cube(T_1)\times\Cube^{op}(T_2)$; in particular, we call a diagram over $\Cube^{op}(T)$ a contravariant $|T|$-cube. However, all these categories are isomorphic to $\Cube(T_2\sqcup T_2)$ if we appropriately rename the vertices. For example,
    \[\begin{tikzcd}[column sep=large]
        \emptyset\times\{1\}  & \{1\}\times\{2\}\lar[hook][swap]{Id_1\times(i^2)^{op}} \\
        \emptyset\times\emptyset \uar[hook]{Id_\emptyset\times i^1} & \emptyset\times\{2\}\uar[hook][swap]{i^1\times Id_2}\lar[hook]{Id_\emptyset\times(i^2)^{op}}
    \end{tikzcd}\quad\cong\quad
    \begin{tikzcd}
        \{a,b\}  & \{b\}\lar[hook][swap]{i^a} \\
        \{a\} \uar[hook]{i^b} & \emptyset\uar[hook]{i^b} \lar[hook]{i^a}
    \end{tikzcd}
    \]
\end{remm}

Similarly as before, we have $\Cube(T)\downarrow S\cong\Cube(S)$, and there is a levelwise homeomorphism of cubes
\begin{equation}\label{eq:cube-cube}
    |\Cube(T)\downarrow \bull|\cong I^{\bull},
\end{equation}
since $|\Cube(T)\downarrow S|\cong I^S$ is an $|S|$-dimensional cube whose coordinates are indexed by $S$, and the map $\iota_S^k\colon I^S\hra I^{S\cup k}$ for $k\notin S\subseteq T$ is the inclusion of the face whose $k$-the coordinate is zero.  Hence,
\begin{equation}\label{eq-def:holim-cube}
    \holim(Y_{\bull},y)
    =\Map_{\mathsf{Top}_*^{\Cube(T)}} \big(I^{\bull},Y_{\bull} \big).
\end{equation}
Observe now that since $\emptyset\in\Cube(T)$ is the initial object, all maps $I^S\to Y_S$ can be contracted to $I^\emptyset\to Y_\emptyset\to Y_S$ (see also the proof of Lemma~\ref{lem:goo-tofib}), so the homotopy limit of $Y_{\bull}$ is uninteresting:
\begin{equation}\label{eq:holim-cube-contr}
    \holim(Y_{\bull})\simeq Y_{\emptyset}.
\end{equation}
However, it turns out to be useful to instead take the homotopy limit of the punctured $|T|$-cube obtained from $Y_{\bull}$ by omitting $Y_{\emptyset}$, and compare this to $Y_{\emptyset}$. This is called the total homotopy fibre of $Y_{\bull}$ and is discussed next.

\begin{example}\label{ex:mapping-path-space}
    Let us compute the homotopy limit of a map $y_\emptyset^1\colon Y_\emptyset\to Y_1$, so for a diagram $Y_{\bull}$ over the category $\mc{C}=\Cube\{1\}=\emptyset\to\{1\}$. Since $\mc{C}\downarrow\emptyset$ has a single object with a single (identity) morphism, $|\mc{C}\downarrow\emptyset|=\{0\}$. On the other hand, $\mc{C}\downarrow \{1\}$ has two objects, $\emptyset\to\{1\}$ and $\{1\}\to\{1\}$, and apart from their identity morphisms, exactly one morphism between them. Thus, $|\mc{C}\downarrow \{1\}|=I$.
    
    Then, a point $f=(f^\emptyset,f^1)$ in $\holim(Y_{\bull})\subseteq\Map(\{0\},Y_\emptyset) \times\Map(I,Y_1)$ consists of a point $f^\emptyset\in Y_\emptyset$ and a path $f^1\colon I\to Y_1$ that starts at $f^1(0)=y_\emptyset^1(f^\emptyset)$. Thus, $\holim Y_{\bull}$ is precisely the \textsf{mapping path space} (also called mapping cocylinder) of the map $y$:
    \[
        E_{y_\emptyset^1}\coloneqq\{\,(K,\gamma)\in Y_\emptyset\times \Path Y_1\;|\; \gamma(0)=y_\emptyset^1(K)\,\}.
    \]
    Since the natural map $p\colon E_{y_\emptyset^1}\to Y_1$, $p(K,\gamma)=\gamma(1)$, is a fibration, and $E_{y_\emptyset^1}\simeq Y_\emptyset$, we replaced $y_\emptyset^1$ by a fibration $p$. Define the \textsf{homotopy fibre} of $y_\emptyset^1$ over $\U\in Y_1$ as the fibre of $p$: $\hofib_\U(y_\emptyset^1)\coloneqq\fib_\U(p)$. Thus, a point in this space is a pair $(K,\gamma)$ such that $\gamma(0)=y_\emptyset^1(K)$ and $\gamma(1)=\U$. 
\end{example}

\subsubsection*{Total homotopy fibres}
\begin{defn}\label{def:tofib}
    The \textsf{total homotopy fibre} of a $|T|$-cube $(Y_{\bull},y)\colon\Cube(T)\to \mathsf{Top}_*$ is the space
\[
    \tofib(Y_{\bull},y)\coloneqq\hofib_{c(\U_\emptyset)} \Big(Y_\emptyset\xrightarrow{c} \holim (Y|_{\PCube(T)}) \Big).
\]
\end{defn}
Here $Y|_{\PCube(T)}$ denotes the \red{restriction} to the subdiagram over $\PCube(T)\subseteq\Cube()$ $\U_{S}\in Y_S$ are the basepoints and $c$ is the natural map sending $K\in Y_\emptyset$ to the collection of constant maps $c(K)^S\colon\Delta^S\to Y_S$, each equal to the image of $K$ under $y^S_\emptyset\colon Y_\emptyset\to Y_S$. This factors as
\[\begin{tikzcd}[row sep=small]
    Y_\emptyset=\lim(Y_{\bull},y) \arrow{d}[swap]{\const}{\simeq}\arrow{r}{c}  & \holim(Y|_{\PCube(T)}) \\
    \holim(Y_{\bull},y) \arrow{ru}[swap]{c'} &
\end{tikzcd}
\]
where $c'$ is the restriction map to the homotopy limit of a subdiagram and $\const$ is the canonical map from the limit to the homotopy limit. Since $\Cube(T)$ has an initial object, $\const$ is a weak equivalence, so $\hofib(c)$ and $\hofib(c')$ are weakly equivalent. In fact, something stronger is true:

\begin{lemma}\label{lem:goo-tofib}
    The map $c'$ is a fibration and there is a {homeomorphism}  
    \[
        \fib_{c(\U_\emptyset)}(c')\cong\hofib_{c(\U_\emptyset)}(c)\eqqcolon\tofib(Y_{\bull},y).
    \]
\end{lemma}
We will use this to describe in Corollary~\ref{cor:tofib-pt} a point in the total homotopy fibre of a cube, analogously to what we had in Corollary~\ref{cor:holim-pt} for homotopy limits of punctured cubes.
\begin{proof}
    See~\cite{GooII} for several descriptions of total homotopy fibres and inspiration for this proof. Consider the mapping path space $E_{c}$ and the natural projection $p\colon E_c\to\holim(Y|_{\PCube(T)})$, which is a fibration (see Example~\ref{ex:mapping-path-space}). 
    Then by definitions $\fib_{c(\U_\emptyset)}(p)=\hofib_{c(\U_\emptyset)}(c)=\tofib(Y_{\bull},y)$.
    
    We will construct a \emph{homeomorphism} $q\colon E_c\to\holim(Y_{\bull},y)$ and a commutative diagram:
    \[\begin{tikzcd}[row sep=tiny]
    &Y_\emptyset \arrow[hook]{dl}[swap]{\simeq}\arrow{rrd}{c}\arrow{dd} && \\
    E_c \arrow[crossing over]{rrr}[xshift=3pt,yshift=-3pt]{p}\arrow[dotted]{dr}[swap]{q} &&& \holim(Y|_{\PCube(T)})\\
    & \holim(Y_{\bull},y) \arrow{rru}[swap]{c'} &&
    \end{tikzcd}
    \]
    It will immediately follow that $c'$ is a fibration as well (as the composition of a fibration and a homeomorphism) with the fibre homeomorphic to $\tofib(Y_{\bull},y)$. Note that this also implies~\eqref{eq:holim-cube-contr}.
    
    Let $(K,\gamma)\in E_{c}$, so $K\in Y_\emptyset$ and $\gamma\colon I\to\holim Y|_{\PCube(T)}$ with $\gamma(0)=c(K)$. Equivalently, $\gamma$ is a collection $\gamma(-)^S\colon I\times\Delta^S\to Y_S$ for $S\neq\emptyset$, which is on $\{0\}\times\Delta^S$ constantly equal to $y^S_\emptyset(K)$.
    
    Hence, $\gamma(-)^S$ factors through the quotient of $I\times\Delta^S$ by $\{0\}\times\Delta^S$, so is a map on the cone $\mathsf{C}\Delta^S$. Let us set $\mathsf{C}\Delta^\emptyset\coloneqq I^0$ and $\gamma(-)^\emptyset=K\colon \mathsf{C}\Delta^\emptyset\to Y_\emptyset$. Then our point gives a map of cubical diagrams
    \[
    \gamma(-)^{\bull}\colon \mathsf{C}\Delta_{bar}^{\bull}\to Y_{\bull}
    \]
    where we define $\mathsf{C}\Delta_{bar}^S$ as the simplicial set obtained as the cone on the barycentric subdivision of $\Delta^S$, and the maps in the cubical diagram $\mathsf{C}\Delta_{bar}^{\bull}$ are given by face inclusions.\footnote{\label{fnote:cone-bary}If we would not subdivide, we would have $\mathsf{C}\Delta^{\bull}\cong\Delta^{\bull\cup \alpha}$, where $\alpha$ is a new index labelling the cone point, and this forms not a cube but a punctured cube, cf.\ Figure~\ref{fig:simplex} and Lemma~\ref{lem:iterated}.}

    On the other hand, by \eqref{eq-def:holim-cube} a point in $\holim(Y_{\bull},y)$ is a compatible collection of maps $I^S\to Y_S$. 
    By Lemma~\ref{lem:bar-cone-is-cube} below, there is a homeomorphism $\mc{h}^{\bull}\colon I^{\bull}\to \mathsf{C}\Delta_{bar}^{\bull}$ of cubes. Thus, we can define
    \[
      q\colon E_c\to\holim(Y_{\bull},y),\quad q(K,\gamma)\coloneqq\gamma(-)^{\bull}\circ \mc{h}^{\bull}.
    \]
    This is also a homeomorphism, and makes the diagram above commute.
\end{proof}

\begin{lemma}\label{lem:bar-cone-is-cube}
  There is a levelwise homeomorphism of cubes
  \[
      \mc{h}^{\bull}\colon I^{\bull}\to \mathsf{C}\Delta_{bar}^{\bull}.
  \]
\end{lemma}
\begin{proof}[Sketch of proof]
    We saw in \eqref{eq:punc-cube-simplex} that the nerve of the category $\PCube(T)\downarrow S$ is the simplicial set $\Delta^S_{bar}$. Similarly, $|\Cube(T)\downarrow S|\cong I^S$ by \eqref{eq:cube-cube}. Moreover, $\Cube(T)\downarrow S$ is obtained from $\PCube(T)\downarrow S$ by adding the object $\emptyset$ and an arrow to every other object, which precisely corresponds to coning (with $\emptyset$ corresponding to the cone point). 
    We omit writing out explicit formulae for $\mc{h}^S$.
\end{proof}

Here is the promised corollary.
\begin{cor}\label{cor:tofib-pt}
    A point $f\in\tofib(Y_{\bull},y)\cong\fib_{c(\U_\emptyset)}(c')$ consists of maps $f^S\colon I^S\to Y_S$, $S\subseteq T$:
    \begin{itemize}
        \item which are compatible on the $0$-faces $y^k_s\circ f^S=f^{Sk}|_{t_k=0}$,
        \item which send the $1$-faces $\{(t_i)_{i\in S}\in I^S:\exists k\in S,t_k=1\}\subseteq I^S$ to the basepoint $\U_S\in Y_S$.
    \end{itemize}
\end{cor}
\begin{proof}
    By definition, $f\in\fib_{c(\U_\emptyset)}(c')$ is a point $f\in\holim(Y_{\bull},y)$ which maps to $c(\U_\emptyset)$ under $c'$. By the definition~\eqref{eq-def:holim-cube} of the homotopy limit of a cubical diagram, $f$ consists of maps $f^S\colon I^S\to Y_S$. The map $c'$ forgets the data $f^\emptyset$ and restricts each $f^S$ to the union of faces which have at least one coordinate equal to one. This is homeomorphic to $\Delta^S$, see the proof of Lemma~\ref{lem:goo-tofib}. 
\end{proof}
For $T=[1]$ we present
\begin{equation}\label{eq:tofib-pic}
f\in\tofib\left(\begin{tikzcd}
    {\color{blue}Y_1} \arrow{r}{y^0_1} & {\color{green!50!black}Y_{01}} \\
    {\color{red}Y_\emptyset} \arrow{u}{y^1_\emptyset} \arrow{r}{y^0_\emptyset} & {\color{orange}Y_0}\arrow{u}[swap]{y^1_0}
\end{tikzcd}\right)
\quad \text{ pictorially by }
\begin{tikzpicture}[scale=0.8,baseline=1.1cm,every node/.style={scale=0.9}]
\clip (-2.7,-0.5) rectangle (3.7,3.6);
    \fill[thick,red,draw=red]
        (-1.1,0) circle (0.9pt) node[left]{$f^\emptyset$};
    \draw[thick,blue]
        (-1.1,1) circle (0.9pt) node[left]{\footnotesize$y^1(f^\emptyset)$}
        -- (-1.1,3) circle (0.9pt) node[pos=0.5,left]{$f^1$} node[left]{\footnotesize$\U_1$};
    \draw[thick,orange!90!black]
        (0,0) circle (0.8pt) node[below]{\footnotesize$y^0(f^\emptyset)$}
        -- (2,0) circle (0.8pt) node[pos=0.5, below]{$f^0$} node[below]{\footnotesize$\U_0$};
    \draw[fill=green!10!white,text=green!40!black]
        (0,1) circle (1pt) node[below]{\footnotesize$yyf^\emptyset$}
        -- (2,1) node[pos=0.5, below]{$y^1f^0$}
        -- (2,3) node[pos=0.5,right]{$\const_{\U_{01}}$}
        -- (0,3) node[pos=0.5, above]{$\const_{\U_{01}}$}
        -- (0,1) node[pos=0.5,left=0.01cm]{\small$y^0f^1$} node[pos=0.5, right=0.5cm, green!30!black]{$f^{01}$};
\end{tikzpicture}.
\end{equation}

\subsubsection*{Iterative descriptions}
\begin{lemma}\label{lem:iterated}
    For any finite set $T$ and $k\in T$ there are homeomorphisms
    \begin{equation}\label{eq:iterated-holim}
    \holim_{S\in\PCube(T)}X_S\cong\holim\left(
            \begin{tikzcd}
            X_k \arrow{r}{c} &   \holim_{S\in\PCube(T\sm\{k\})}X_{S\cup\{k\}}  \\
             &  \holim_{S\in\PCube(T\sm\{k\})}X_S\arrow{u}[swap]{x^k_*}
        \end{tikzcd}
        \right),
    \end{equation}
    \begin{equation}\label{eq:iterated-tofib}
        \tofib_{S\in\Cube(T)}Y_S
        \cong \hofib
        \left(
        \begin{tikzcd}
        \tofib_{S\in\Cube(T\sm\{k\})}Y_S\rar{y_*^k} &
        \tofib_{S\in\Cube(T\sm\{k\})}Y_{S\cup\{k\}}
        \end{tikzcd}
        \right).
    \end{equation}
\end{lemma}
\begin{proof}
     The homeomorphism \eqref{eq:iterated-holim} uses $\Delta^{S\cup\{k\}}\cong \mathsf{C}\Delta^S$, which assemble into a homeomorphism $\mathsf{C}\Delta^{\bull}\to\Delta^{\bull}$, cf.\ Lemma~\ref{lem:bar-cone-is-cube} and Footnote~\ref{fnote:cone-bary}. Therefore, similarly to the proof of Lemma~\ref{lem:goo-tofib}, a collection $f^S\colon\Delta^S\to X_S$ is determined by a point $f_k\in X_k$, and for $S\subseteq T\sm\{k\}$ the collections $f^S\colon\Delta^S\to X_S$ and $f^{S\cup\{k\}}\colon\mathsf{C}\Delta^S\to X_{S\cup\{k\}}$. The last is equivalent to the data of $I\to\holim_{S\in\PCube(T\sm\{k\})}X_{S\cup\{k\}}$, and one can see that together this exactly gives a point in the homotopy pullback on the right hand side of \eqref{eq:iterated-holim}. See also~\cite{GooII} or~\cite[Lemma 5.3.6]{MV}.

     We then have that \eqref{eq:iterated-tofib} follows from  \eqref{eq:iterated-holim}, using Definition~\ref{def:tofib}. Alternatively, use the description from Corollary~\ref{cor:tofib-pt} and homeomorphisms $I^{S\cup\{k\}}=I\times I^S$.
\end{proof}

In other words,~\eqref{eq:iterated-holim} says that a homotopy limits over a punctured cube can be computed as an iterated homotopy pullback. Similarly,~\eqref{eq:iterated-tofib} says that total homotopy fibres of cubes can be computed as iterated homotopy fibres.

\begin{remm}\label{rem:tofib-hofib}    
    One can interpret the last result as saying that homotopy limits commute, or that total homotopy fibres and homotopy fibres commute. For example, if $\Psi_{\bull}\colon (Y_{\bull},y)\to (Z_{\bull},z)$ is a map of cubes, and $yz_S\colon\hofib\Psi_S\to\hofib\Psi_{Sk}$ is the induced map, then we have
    \[
        \tofib\big(\hofib(\Psi_{\bull}),yz\big)\simeq \hofib\big(\tofib(Y_{\bull},y)\to\tofib(Z_{\bull},z)\big).
    \]
    This follows by applying Lemma~\ref{lem:iterated} to the map $(Y_{\bull},y)\to (Z_{\bull},z)$ viewed as an $(n+1)$-cube.
    In particular, we have that `total homotopy fibres commute with loop spaces':
    \[
        \tofib(\Omega X_{\bull},\Omega x)\simeq \Omega\tofib(X_{\bull},x).
    \]
\end{remm}


\section{The punctured knots model}\label{sec:punc-knot-model}
Throughout Sections~\ref{sec:punc-knot-model}--\ref{sec:htpy-type} of this paper $M$ is a fixed \emph{connected compact smooth manifold of dimension $d\geq3$ with non-empty boundary}.
We fix \red{a neat embedding $\U\colon I=[0,1]\hra M$ and consider the space of smooth neat embeddings (with the Whitney $C^\infty$ topology)
}
\begin{equation*}
  \Knots(M)\coloneqq\Emb_\partial(I,M)\coloneqq\big\{\,K\colon I\hra M \;|\; K|_{I\sm[L_0,R_\infty]}=\U|_{I\sm[L_0,R_\infty]}\,\big\}
\end{equation*}
whose elements we call \emph{knots}. \red{Here we fix two points $L_0,R_\infty\in I$ and require that all embeddings agree with $\U$ outside $[L_0,R_\infty]\subseteq I$, so on a fixed collar of $\partial I$. This gives a space homotopy equivalent to the one we previously gave in~\eqref{eq-def:knots}.
}
We \red{take
}
$\U\in\Knots(M)$ for the basepoint. 

\subsection{The stages}\label{subsec:stages}
The $n$-th Taylor approximation of $\Knots(M)$, for $n\geq 0$, is defined as the homotopy limit
\[
    \T_n\Knots(M)\coloneqq\holim_{V\in\mc{O}_n(I)^{op}}\Emb_\partial(V,M).
\]
Here the category $\mc{O}_n(I)$ is the poset of those open subsets $V$ of $I$ which are homeomorphic to the union of the collar of $\partial I$ and at most $n$ open intervals. Equivalently, $V$ is of the shape $I\sm V'$, where $V'\subseteq I\sm\partial I$ consists of at most $n+1$ closed subintervals. The space $\Emb_\partial(V,M)$ consists of embeddings $V\hra M$, which near $\partial I$ agree with $\U$, that is, on $I\sm[L_0,R_\infty]$, and the maps in the diagram are restrictions of embeddings to submanifolds.

Now, as observed by Goodwillie, computing this homotopy limit over a certain \emph{finite subposet} gives a homotopy equivalent space $\pT_n(M)$ which we define next; see~\cite[Example 10.2.18]{MV} for a proof. Namely, the desired subposet contains only sets $V=I\sm J_S$ for $\emptyset\neq S\subseteq[n]\coloneqq\{0,1,\dots,n\}$, where $J_S\coloneqq\bigsqcup_{i\in S}J_i$ for a fixed collection of disjoint closed subintervals $J_i=[L_i,R_i]\subseteq I$ \red{with $0<L_i<R_i<L_{i+1}$ and the sequence $R_i$ converges to $R_\infty<1$.

Since the poset of such $V=I\sm J_S$ is equivalent to the opposite of the punctured cube category $\PCube[n]$ (whose objects are $\emptyset\neq S\subseteq[n]\coloneqq\{0,1,\dots,n\}$ and morphisms are inclusions of sets),
}
we have a punctured cubical diagram $\big(\Emb_\partial(I\sm J_{\bull},M),r\big)$, where for $k\notin S\subseteq[n]$ the restriction map
\[
    r^k_S\colon\Emb_\partial(I\sm J_S,M)\to \Emb_\partial(I\sm J_{S\cup \{k\}},M)
\]
introduces a `puncture' at $J_k$. The space
\begin{equation}\label{eq-def:pt-n}
    \pT_n(M)\coloneqq\holim_{S\in\PCube[n]}\Emb_\partial(I\sm J_S,M)
\end{equation}
is called the \textsf{punctured knots model} for $\T_n\Knots(M)$. Indeed, by the definition of homotopy limits from Section~\ref{subsec-prelim:holims}, a point $f\coloneqq\{f^S\}_{\emptyset\neq S\subseteq[n]}\in\pT_n(M)$ consists of a compatible collection of maps $f^S\colon\Delta^S\to\Emb_\partial(I\sm J_S,M)$. That is, each $f^S$ is a $\Delta^S$-family of knots punctured at $J_i$, for $i\in S$.
\begin{example}\label{ex:cube-deg-2}\renewcommand{\qedsymbol}{}
In degree $n=2$ the space $\pT_2(M)$ is the homotopy limit of the punctured $3$-cube
\begin{equation*}
\begin{tikzcd}[column sep=tiny,row sep=small]
        &
        \Emb_\partial(I\sm J_{12},M) \arrow{rr}\arrow[from=dd] &&
        \Emb_\partial(I\sm J_{012},M)
        \\
        \Emb_\partial(I\sm J_2,M)\arrow[crossing over]{rr}\arrow{ur} &&
        \Emb_\partial(I\sm J_{02},M) \arrow{ur}
        \\
        &
        \Emb_\partial(I\sm J_1,M) \arrow{rr} &&
        \Emb_\partial(I\sm J_{01},M)  \arrow{uu}
        \\
     &&
        \Emb_\partial(I\sm J_0,M) \arrow{ur} \arrow[crossing over]{uu}
    \end{tikzcd}
\end{equation*}
Thus, $f\in\pT_2(M)$ consists of three once-punctured knots, for each two of them an isotopy between their restrictions to twice-punctured knots, and three two-parameter isotopies of thrice-punctured knots connecting restrictions of respective isotopies of twice-punctured knots.
\end{example}
\begin{notation}\label{notat:punctured-model}
\red{
To save space we introduce the following notation; the ambient manifold $M$ is fixed as above and mostly omitted from the notation. 

For a knot $K\in\Knots(M)$ and $S\in\PCube[n]$ we denote $K$ punctured at each $J_i$ for $i\in S$ by
\begin{equation}\label{eq-def:K-punctured}
    K_{\wh{S}}\coloneqq K|_{I\sm J_S}\in\mc{E}_S.
\end{equation}
For example, since $J_i=[L_i,R_i]$ for $\emptyset\neq S=\{i_1<i_2<\dots<i_d\}\subseteq [n]$ we have 
\[
    \U_{\wh{S}}=\U|_{[0,L_0]}\sqcup \U_{[R_0,L_{i_1}]}\sqcup  \U_{[R_{i_1},L_{i_2}]}\sqcup\dots\sqcup  \U_{[R_{i_n},L_n]}\sqcup \U|_{[R_n,1]},
\]
see Figure~\ref{fig:punctured-U}. 
For any $f\in\mc{E}_S$ we assume $f|_{[0,L_0]}=\U_{[0,L_0]}$ and $f|_{[R_\infty,1]}=\U_{[R_\infty,1]}$. Let us denote by $u_{i}\in M$ the image under $\U$ of the midpoint of the interval $[R_i,L_{i+1}]$.

Next, we let
\begin{equation}\label{eq-def:E-S}
    \mc{E}_{S}\coloneqq\Emb_\partial(I\sm J_S,M)
    \quad\text{and}\quad 
    \mc{E}_{S k}\coloneqq\mc{E}_{S\cup \{k\}}\;.
\end{equation}
We equip each $\mc{E}_S$ with the basepoint $\U_{\wh{S}}$, so $r^k_S\colon\mc{E}_S\to\mc{E}_{S k}$ with $k\notin S\subseteq[n]$ is a based map. Then $(\mc{E}^n_{\bull},r)=(\mc{E}_S,r^k_S)_{\emptyset\neq S\subseteq[n],k\notin S}$ is a punctured $(n+1)$-cube, and by definition~\eqref{eq-def:pt-n} we have
\begin{equation}\label{eq-def:E-cube}
    \pT_n(M)\coloneqq\holim(\mc{E}^n_{\bull},r).
\end{equation}
}
Using $[n]\subseteq[n+1]$ we consider two inclusions $\PCube[n]\hra \PCube[n+1]$ \red{of posets 
}
given by $\Id\colon S\mapsto S$ and $\Id\cup n{+}1\colon S\mapsto S\cup \{n+1\}$. Note that the punctured $(n+1)$-cube $\mc{E}^{n+1}_{\bull}\circ \Id$ is precisely  $\mc{E}^n_{\bull}$.
 Let us denote the other \red{punctured $(n+1)$-cube by
\[
    \mc{E}_{\bull\cup n{+}1}^n\coloneqq\mc{E}^{n+1}_{\bull}\circ (\Id\cup n{+}1)= (\Emb_\partial(I\sm J_{S n+1},M),r^k_S)_{\emptyset\neq S\subseteq[n],k\notin S}\;.
\]}
Hence, the punctured $(n+2)$-cube $\mc{E}^{n+1}_{\bull}$ decomposes as
\begin{equation}\label{eq:cube-subcubes}
    \mc{E}^{n+1}_{\bull}=
    \begin{tikzcd}
        \mc{E}_{n+1}\arrow{r}{r_{n+1}} & \mc{E}_{\bull\cup n{+}1}^n  \\
        & \mc{E}^n_{\bull}\arrow{u}[swap]{r^{n+1}_{\bull}}
    \end{tikzcd}
\end{equation}
\red{and the upper row forms an $(n+1)$-cube, which we will denote by $\mc{E}_{\bull\cup n{+}1}$.
}
\end{notation}
\begin{figure}[!htbp]
    \centering
    \includegraphics{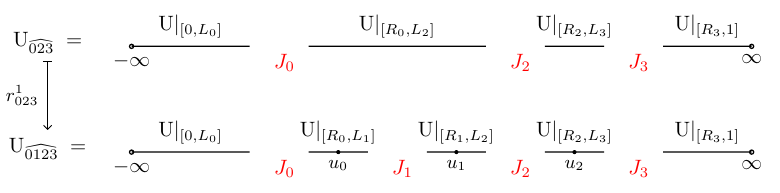}
    \caption[Examples of the punctured unknot and the restriction map.]{Examples of $\U_{\wh{S}}$ and the restriction map $r^j_S$ for $n=3,S=\{0,2,3\},j=1$.}
    \label{fig:punctured-U}
\end{figure}

For $K\in\Knots(M)$ the restrictions $K_{\wh{S}}$ are mutually compatible, so assemble to a map $\Knots(M)\to\lim\mc{E}_{\bull}$. Composing it with the canonical map from the limit to the homotopy limit gives
\begin{equation}\label{eq-def:ev-n}
\begin{tikzcd}[column sep=large]
    \Knots(M)\arrow{r}{\cong}\arrow[bend left=16pt]{rr}{\ev_n} & \lim \mc{E}^n_{\bull}\arrow{r}{\const} & \holim \mc{E}^n_{\bull} = \pT_n(M).
\end{tikzcd}
\end{equation}
called \textsf{evaluation map}.
More explicitly, for $S\in\PCube[n]$ this is given as the constant family
\[
    \ev_n(K)^S\colon\Delta^S\to \mc{E}_S,\quad \vec{t}\mapsto K_{\wh{S}}\;.
\]
Actually, for $n\geq 2$ there is a homeomorphism $\Knots(M)\cong\lim\mc{E}^n_{\bull}$. This follows since having at least three different punctures ensures that all $f^S\in\mc{E}_{S}$ are pairwise disjoint, apart from agreeing on intersections. However, for $n=1$ this does not hold, since $f^{\{0\}}|_{J_1}$ and $f^{\{1\}}|_{J_0}$ potentially intersect. Instead, we shall see \red{in~\eqref{eq:deg1-lim-holim} below that there is a homotopy equivalence
}
$\lim\mc{E}^1_{\bull}\simeq\pT_1(M)$.
\begin{remm}\label{rem:palais}
By a family version of the isotopy extension theorem, $r^k_S$ is a locally trivial fibre bundle~\cite{Palais}. In particular, it is a Serre fibration, whose fibre is the space of embeddings of $J_k$ into $M$ which miss the punctured unknot $\U_{\wh{S k}}$, namely $\fib(r^k_S)=\Emb_\partial(J_k,M\sm \U_{\wh{Sk}})$.
\end{remm}

\subsubsection*{The zeroth and first Taylor stages}\label{subsec:zero-first-stage}
Given a point $\mc{E}_i$ for some $i\geq 0$ we can start `shortening' its both ends until only the parts \red{$\U|_{[0,L_0]}$ and $\U|_{[R_\infty,1]}$ remain. 
}
In other words, there is a deformation retraction of $\mc{E}_i=\Emb_\partial(I\sm J_i,M)$ onto the point $\U_{\wh{i}}\in\mc{E}_i$. So $\pT_0(M)=\mc{E}_0$ is contractible, and also:
\begin{equation}\label{eq:E-i-contractible}
    \red{\mc{E}_i\simeq *.
    }
\end{equation}
Next, as we mentioned above, the limit of the diagram
\begin{equation}\label{eq:diag-first-stage}
    \mc{E}^1_{\bull}=
    \begin{tikzcd}[column sep=scriptsize, row sep=scriptsize]
        \Emb_\partial(I\sm J_1,M) \arrow{r} & \Emb_\partial(I\sm J_{01},M)  \\
        & \Emb_\partial(I\sm J_0,M) \arrow{u}
    \end{tikzcd}
\end{equation}
is not homeomorphic to the space of knots. Instead, it is given as
\begin{equation}\label{eq:special-immersions}
    \lim\mc{E}^1_{\bull}=\{\,(f^0,f^1)\;|\; f^0_{I\sm J_{01}}=f^1_{I\sm J_{01}}\,\}=\{K\in\Imm_\partial(I,M)\;|\; K_{\wh{0}}, K_{\wh{1}}\text{ are embeddings}\}
\end{equation}
-- the space of those immersions which are embeddings when restricted to $I\sm J_0$ or $I\sm J_1$. Actually, since \emph{both maps in the diagram $\mc{E}^1_{\bull}$ are fibrations} (by Remark~\ref{rem:palais}), the limit is equivalent to the homotopy limit.\footnote{\label{fnote:holim-lim}\red{See \cite[Prop.~3.2.13]{MV} for a proof. 
}
Warning: although all maps in cubes for higher $n$ are also fibrations, one cannot conclude that $\holim\mc{E}^n_{\bull}\simeq\lim\mc{E}^n_{\bull}$ as for $n=1$.} Hence we have the upper row in the commutative diagram
\begin{equation}\label{eq:deg1-lim-holim}
\begin{tikzcd}[row sep=tiny]
\Knots(M)\arrow{r}\arrow[hook]{dd}[swap]{\iota}\arrow[bend left=20pt]{rr}{\ev_1} & \lim \mc{E}^1_{\bull} \arrow{r}{\const}[swap]{\simeq}\arrow[hook]{dd}[swap]{\iota} & \holim \mc{E}^1_{\bull} \eqqcolon\pT_1(M)\arrow{dd}{\simeq} \arrow{rd}{\simeq} & \\
&&& \Omega(\S M)\\
\ImmK(I,M)\arrow[equals]{r} & \lim\mc{I}^1_{\bull}\arrow{r}{\simeq}[swap]{Smale} &  \holim\mc{I}^1_{\bull}\arrow{ru} &
\end{tikzcd}
\end{equation}
Let us now explain the rest of the diagram. As mentioned in \red{Section~\ref{subsec-prelim:holims}, 
}
the homotopy limit can be computed from any levelwise homotopy equivalent diagram:
\begin{equation}\label{eq:emb-imm}
\mc{E}^1_{\bull}\quad\:\simeq\:
        \begin{tikzcd}
        \Imm_\partial(I\sm J_1,M)\arrow{r}{} & \Imm_\partial(I\sm J_{01},M)  \\
        & \Imm_\partial(I\sm J_0,M) \arrow{u}
    \end{tikzcd}\:\simeq\: \begin{tikzcd}[column sep=large]
        \Path_*(\S M) \arrow{r}{|_{\mathrm{U}|_{[R_0,L_1]}}} & \Path (\S M)  \\
         & \Path_*(\S M) \arrow{u}[swap]{|_{\mathrm{U}|_{[R_0,L_1]}}}
    \end{tikzcd}
\end{equation}
The first equivalence is induced from the weak equivalences $\Emb(V,M)\to \Imm(V,M)$ for $V$ a disjoint union of disks, see for instance~\cite{Cerf}. The second equivalence is induced from the unit derivative maps, giving paths in the unit tangent bundle $\S M$, with $\Path\S M\simeq\S M$ (both endpoints free) and $\Path_{*}\S M\simeq*$ (one endpoint fixed). One can check that the homotopy limit of the rightmost diagram is \red{the loop space 
}
$\Omega(\S M)$, so one has the triangle in \eqref{eq:deg1-lim-holim}.

On the other hand, the strict limit in the middle diagram is clearly $\lim\mc{I}^1_{\bull}\cong\Imm_\partial(I,M)$. The fact that this is also \emph{its homotopy limit} is non-trivial: by a theorem of Smale~\cite{Smale} the restriction maps for immersions are \emph{also fibrations}. This also implies that immersions form a polynomial functor of degree at most $1$, that is, $\T_n\ImmK(I,M)\simeq\ImmK(I,M)$ for all $n\geq 1$. Here we similarly define $\T_n\ImmK(I,M)\coloneqq\holim\mc{I}^n_{\bull}$, using $\mc{I}_S\coloneqq\Imm_\partial(I\sm J_S,M)$. See~\cite{weiss-survey,GWII}.

However, to obtain $\pT_1(M)\simeq\Omega(\S M)$ we did not need Smale's result. Finally, observe that as a consequence of this discussion the inclusion $\iota\colon\lim\mc{E}^1_{\bull}\hra \lim\mc{I}^1_{\bull}$ of `special' immersions \eqref{eq:special-immersions} into all immersions is -- maybe surprisingly -- a weak equivalence.


\subsection{The projection and evaluation maps}\label{subsec:pn-surj-fib}
Fix $n\geq 1$ and let
\begin{equation}\label{eq-def:p-n+1}
    p_{n{+}1}\colon\pT_{n{+}1}(M)\to\pT_n(M)
\end{equation}
be the map induced \emph{by forgetting the last puncture} $J_{n{+}1}$, i.e. the map induced on homotopy limits from the inclusion of diagrams $\mc{E}^n_{\bull}\subseteq\mc{E}^{n+1}_{\bull}$ (see Notation~\ref{notat:punctured-model}). We clearly have $p_{n{+}1}\circ\ev_{n{+}1}=\ev_n$, so $p_{n{+}1}$ respects the basepoints $p_{n{+}1}(\ev_{n{+}1}\U)=\ev_n\U$.
\begin{prop}\label{prop:pn-fibration}
    The map $p_{n{+}1}\colon\pT_{n{+}1}(M)\to\pT_n(M)$ is a fibration. Furthermore, its fibre
    \begin{equation}\label{eq-def:F-n+1}
        \pF_{n{+}1}(M)\coloneqq\fib_{\ev_n\U}(p_{n{+}1})
    \end{equation}
    is {homeomorphic} to the total homotopy fibre of the $(n+1)$-cube $\mc{E}_{\bull\cup n{+}1}$:
    \begin{equation}\label{eq:F-n+1-tofib}
        \pF_{n{+}1}(M)\cong\tofib_{\Cube[n]}(\mc{E}_{\bull\cup n{+}1}).
    \end{equation}
\end{prop}

\begin{proof}
For the definition of a total homotopy fibre and its properties see Section~\ref{subsec-prelim:holims}.

    Using the decomposition of $\mc{E}^{n+1}_{\bull}$ into subcubes from \eqref{eq:cube-subcubes} and \red{Lemma~\ref{lem:iterated} saying that homotopy limits can be computed iteratively, 
    }
    we obtain a homeomorphism:
     \begin{equation}\label{eq:T-n+1-iteratively}
         \pT_{n{+}1}(M)=\holim\left(
        \begin{tikzcd}
            \mc{E}_{n+1} \arrow{r}{r^{\bull}_{n+1}} &  \mc{E}_{\bull\cup n{+}1}^n  \\
            &  \mc{E}^n_{\bull}\arrow{u}[swap]{r^{n+1}_{\bull}}
        \end{tikzcd}\right) \,\cong\, \holim\left(
            \begin{tikzcd}
            \mc{E}_{n+1} \arrow{r}{c} &   \holim_{\PCube[n]}(\mc{E}_{\bull\cup n{+}1}^n)  \\
             &  \pT_n(M)\arrow{u}[swap]{r^{n+1}_*}
        \end{tikzcd}
        \right)
     \end{equation}
    
    The map $c$ is an analogue of $\ev_n$ but for $J_{n+1}$-punctured knots $\mc{E}_{n{+}1}\coloneqq\Emb_\partial(I\sm J_{n+1},M)$, while $r^{n+1}_*\colon\pT_n(M)\to\holim(\mc{E}_{\bull\cup n{+}1}^n)$ is the induced map on homotopy limits from the maps $r^{n+1}_S$, so it punctures at $J_{n+1}$ every punctured knot in the family.
    
    The homotopy pullback is homeomorphic to the \emph{pullback} of the same diagram with $c$ replaced by a fibration \red{(see Footnote~\ref{fnote:holim-lim}), 
    }
    for example the fibration $c'$ from Lemma~\ref{lem:goo-tofib}, so we have a (strict) pullback square:
     \begin{equation}\label{diag:pn-as-pullback}
        \begin{tikzcd}
            \holim_{\Cube[n]}(\mc{E}_{\bull\cup n{+}1}) \arrow{r}{c'} &  \holim_{\PCube[n]}(\mc{E}_{\bull\cup n{+}1}^n)  \\
            \pT_{n{+}1}(M)\arrow[dashed]{r}{p_{n{+}1}}\arrow[dashed]{u} &  \pT_n(M)\arrow{u}[swap]{r_*^{n+1}}
        \end{tikzcd}
     \end{equation}
    Since $c'$ is a fibration, $p_{n+1}$ is also\footnote{\red{In other words, `pullbacks preserve fibrations': it is elementary to check that if $c'$ satisfies homotopy lifting property for a space then $p_{n+1}$ does as well.
    }
    }, and the fibres are homeomorphic (check from definitions):
    \[
        \pF_{n+1}(M)\coloneqq\fib_{\ev_n\U}(p_{n+1})\:\cong\: \fib_{c(\U_{\wh{n+1}})}(c')\cong \tofib(\mc{E}_{\bull\cup n{+}1}),
    \]
    where the last homeomorphism is by Lemma~\ref{lem:goo-tofib}.
\end{proof}
The diagram \eqref{eq:T-n+1-iteratively} from the proof can be seen as an inductive definition of the Taylor tower. Moreover, if we denote $\B\pF_{n+1}(M)\coloneqq\holim(\mc{E}_{\bull\cup n{+}1}^n)$, and use $\mc{E}_{n+1}\simeq *$ from \eqref{eq:E-i-contractible} of Section~\ref{sec:punc-knot-model}, we get $\pT_{n+1}(M)\simeq\holim(* \to \B\pF_{n+1}(M) \leftarrow \pT_n(M))$
and $\pF_{n+1}(M)\simeq\hofib(* \to \B\pF_{n+1}(M))$. Thus, $\B\pF_{n+1}(M)$ is indeed a delooping of $\pF_{n+1}(M)$, explaining the notation. Moreover, $\B\pF_{n+1}(M)$ is connected, so $p_{n+1}$ is surjective, see~\cite{K-thesis}.
\begin{remm}
    In Section~\ref{sec:delooping} we will see that $\B\pF_{n+1}(I^3)$ is an $n$-fold loop space, so one can try to show that $\pT_{n+1}(I^3)=\hofib(\pT_n(I^3)\to\B\pF_{n+1}(I^3))$ is a double loop space by induction on $n\geq1$ and showing that $r_*^{n+1}$ is a map of double loop spaces. Such deloopings were shown in other models by~\cite{Turchin-Delooping} and~\cite{BW2}, but in this approach one could check if they also exist for other $M$.
\end{remm}


Next, define $\H_n(M)\coloneqq\hofib_{\ev_n\U}(\ev_n))$ as the homotopy fibre of the map $\ev_n\colon\Knots(M)\to\pT_n(M)$ over the basepoint $\ev_n\U$.
Since $\Knots(M)=\Emb_\partial(I,M)=\mc{E}_\emptyset$ and $\pT_n(M)\coloneqq\holim_{\PCube[n]}(\mc{E}^n_{\bull})$, the homotopy fibre of $\ev_n$ is the total homotopy fibre by definition (see Definition~\ref{def:tofib}):
\begin{equation}\label{eq-def:H-n}
    \H_n(M)=\tofib_{\Cube[n]}(\mc{E}^n_{\bull}),
\end{equation}
\red{Using the identification $\pF_{n+1}(M)\cong\tofib_{\Cube[n]}(\mc{E}_{\bull\cup n+1})$ 
}
from Proposition~\ref{prop:pn-fibration}, let us define
\begin{equation}\label{eq-def:emap-n+1}
    \emap_{n+1}\colon\H_n(M)\to\pF_{n+1}(M)
\end{equation}
as the map induced on total homotopy fibres from the map
$r^{n+1}_{\bull}\colon\mc{E}^n_{\bull}\to \mc{E}_{\bull\cup n+1}$ of $(n+1)$-cubes. This again `punctures at $J_{n+1}$ every punctured knot in the family'. The following completes the diagram~\eqref{diag:hofib-hn} from the introduction.
\begin{lemma}
    The composite
\[
    \H_n(M)\coloneqq\hofib(\ev_n)\xrightarrow{\emap_{n+1}}
    \hofib(c)\coloneqq\fib(c')\cong\fib(p_{n+1})\eqqcolon\pF_{n+1}(M)
    \hra\hofib(p_{n+1})
\]
    is homotopic to the canonical map $\hofib(\ev_n)\to\hofib(p_{n+1})$ given by $(K,\eta)\mapsto(\ev_{n+1}K,\eta)$.
\end{lemma}
\begin{proof}
     We use Proposition~\ref{prop:pn-fibration} to identify $\pT_{n+1}(M)$ as a homotopy pullback, and write its elements using the notation~\eqref{eq:holim-pic} as triples $f=(\{f^S\}_{S\in\PCube[n]},f^{n+1},\delta)$, with $\{f^S\}_{S\in\PCube[n]}\in\pT_n(M)=\holim\mc{E}^n_{\bull}$, $f^{n+1}\in\mc{E}_{n+1}$ and $\delta\colon I\to\holim(\mc{E}^n_{\bull\cup n+1})$, so that $r^{n+1}_*(\{f^S\})=\delta(0)$ and $c(f^{n+1})=\delta(1)$.
     
     For $(K,\eta)\in\hofib(\ev_n\colon\mc{E}_\emptyset\to\holim\mc{E}^n_{\bull})$ we have
    \[
        \emap_{n+1}(K,\eta)=(r^{n+1}K,r^{n+1}_*\eta)\in\hofib(c\colon\mc{E}_{n+1}\to\holim\mc{E}^n_{\bull\cup n+1}).
    \]
    Under $\hofib(c)\coloneqq\fib(c')\cong\fib(p_{n+1})=\pF_{n+1}(M)\subseteq\pT_{n+1}(M)$ this corresponds to the triple
    \[
        (\ev_n\U,r^{n+1}K,r^{n+1}_*\eta).
    \]
    On the other hand, since $p_{n+1}$ is a fibration, the inclusion $\fib(p_{n+1})\hra\hofib(p_{n+1})$ is a homotopy equivalence. Moreover, we can identify an explicit homotopy inverse: it sends a point $((\{f^S\},f^{n+1},\delta),\eta) \in\hofib(p_{n+1}\colon\pT_{n+1}(M)\to\pT_n(M))$ to the triple $(\ev_n\U,f^{n+1},\delta\cdot (r^{n+1}_*\eta))$ in $\fib(p_{n+1})\subseteq\pT_{n+1}$, where the dot denotes the concatenation of paths. It is straightforward to construct homotopies proving that this is a homotopy inverse.

    The canonical map takes our point $(K,\eta)\in\hofib(\ev_n)$ to $(\ev_{n+1}K,\eta)\in\hofib(p_{n+1})$ for $\ev_{n+1}K=(\ev_nK,r^{n+1}K,\{\const_{r^SK}\}_{S\in\PCube[n]})$, where $\const_{r^SK}\colon\Delta^S\to\mc{E}_S$ is the constant map equal to $r^SK=K_{\wh{S}}$, see~\eqref{eq-def:ev-n}. Our homotopy equivalence then maps this to 
    \[
        (\ev_n\U,r^{n+1}K,\{\const_{r^SK}\}\cdot r^{n+1}_*\eta)\in\fib(p_{n+1}),
    \]
    which agrees with the triple displayed above (up to reparametrisation).
\end{proof}

\subsection{Another description of the (homotopy) fibres}\label{subsec:fn-hn}
\red{We will now simplify the descriptions $\pF_{n+1}(M)=\tofib\mc{E}_{\bull\cup n{+}1}$ and $\H_n(M)=\tofib\mc{E}^n_{\bull}$.
}

Recall from Lemma~\ref{lem:iterated} that total homotopy fibre of an $(n+1)$-cube can also be computed `iteratively', by first taking homotopy fibres in one arbitrary direction and then finding the total fibre of the resulting $n$-cube. For the first direction we choose the \red{restriction map $r^0_S$ 
}
which `punctures at zero', that is, we take homotopy fibres of $r^0_S$ \red{and then the total homotopy fibre of the resulting cube. In fact, since 
}
by Remark~\ref{rem:palais} the restriction maps are fibrations, we can instead take their actual fibres.
\begin{defn}\label{def:FF}
    For each $S\subseteq\ul{n}\coloneqq\{1,2,\dots,n\}$ define
    \begin{align*}
    \FF_S&\coloneqq\fib\big(r^0_{S}\colon\mc{E}_{S}\to\mc{E}_{0S}\big)\:\quad\text{and}\quad\:
    \FF^{n+1}_S\coloneqq\fib\big(r^0_{Sn+1}\colon\mc{E}_{Sn+1}\to\mc{E}_{0Sn+1}\big).
    \end{align*}
    The map $r^{n+1}_S\colon\FF_S\to\FF^{n+1}_S$ gives the right vertical map in the commutative diagram
    \[
    \begin{tikzcd}
        \H_n(M)\arrow[shift right=2.5cm]{d}[swap]{\emap_{n+1}}  \cong\tofib_{S\subseteq[n]}(\mc{E}_{S})\cong \tofib_{S\subseteq\ul{n}}\Big(\hofib(r^0_{S})\Big)& \tofib_{S\subseteq\ul{n}}\Big(\FF_S\Big) \arrow[hook]{l}{\simeq}\arrow[]{d}{}
        \\
        \pF_{n+1}(M)\cong\tofib_{S\subseteq [n]}(\mc{E}_{Sn+1})  \cong \tofib_{S\subseteq\ul{n}}\big(\hofib(r^0_{Sn+1})\big) & \tofib_{S\subseteq\ul{n}}\big(\FF^{n+1}_S\big)\arrow[hook]{l}{\simeq}
    \end{tikzcd}\qedhere
    \]
\end{defn}
The basepoint of $\mc{E}_{S}$ is $\U_{\wh{S}}\coloneqq\U|_{I\sm J_{S}}$, so writing the fibre \red{of $r^0_S$ 
}
out, we get
\begin{align}\label{eq-def:FF-S}
    \FF_S&=\fib_{\U_{\wh{0S}}}\Big(r^0_{S}\colon\Emb_\partial(I\sm J_{S},M) \to \Emb_\partial(I\sm J_{0S},M)\Big) \nonumber\\
    &=(r^0_{S})^{-1}(\U_{\wh{0S}})=\big\{\; K:I\sm J_{S}\hra M \;|\; K_{\wh{0S}}=\U_{\wh{0S}}\;\big\}
    \cong \Emb_\partial(J_0,M\sm \U_{\wh{0S}}).
\end{align}
Thus, $\FF_S$ is the space of embeddings of the arc $J_0$ into the complement in $M$ of the punctured unknot $\U_{\wh{0S}}$ with condition that near the boundary $\partial J_0$ they agree with $\U_0$, see \red{the top of Figure~\ref{fig:element-of-ffs}. The maps $r^k_{S}\colon\FF_S\to\FF_{Sk}$ in the $n$-cube $(\FF_{\bull},r)$ are induced by the restriction maps: they are clearly given as inclusions of an embedding $J_0\hra M\sm\U_{\wh{0S}}$ into the manifold $M\sm\U_{\wh{0Sk}}=(M\sm\U_{\wh{0S}})\cup\U(J_k)$.
}

Similarly, the $n$-cube $(\FF^{n+1}_{\bull},r)$ can be given by
\begin{equation}\label{eq-def:FF-n+1-S}
    \FF^{n+1}_S\cong\Emb_\partial(J_0,M\sm \U_{\wh{0Sn+1}})
\end{equation}
with the restriction maps $r^k_{Sn+1}\colon\FF^{n+1}_S\to\FF^{n+1}_{Sk}$. See \red{the bottom of 
}
Figure~\ref{fig:element-of-ffs}. 

\red{Moreover, the map $\emap_{n+1}$ is induced on total homotopy fibres by the similar maps $r^{n+1}_S\colon\FF_S\to\FF^{n+1}_S$, that postcompose with $M\sm\U_{\wh{0S}}\hra M\sm\U_{\wh{0Sn+1}}$, cf.\ the top and bottom of Figure~\ref{fig:element-of-ffs}.
}
\begin{figure}[!htbp]
    \centering
    \includegraphics[width=\linewidth]{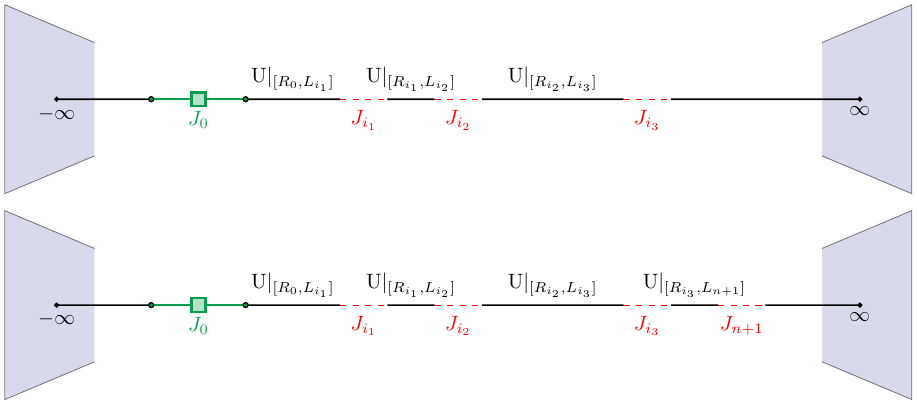}
    \caption[Points in \protect$\FF_S$ and \protect$\FF^{n+1}_S$.]{ Points in $\FF_S$ and $\FF^{n+1}_S$ respectively, for $S=\{i_1,i_2,i_3\}\subseteq\ul{n}$. The small square on $J_0$ denotes that there some knotting, as well as linking with the rest of the black open intervals, can occur. Dashed intervals are punctures, i.e.\ the embedded $J_0$ can intersect them.}
    \label{fig:element-of-ffs}
\end{figure}
\red{
Observe that an embedding $J_0\hra M\sm\U_{\wh{0T}}$ for $T\subseteq\ul{n+1}$ misses also a tubular neighbourhood of $\U_{\wh{0T}}$, by compactness of $J_0$. We now fix one such neighbourhood.
\begin{notation}\label{notat:M-0S}
    Firstly, for $i\geq0$ let $\S_i\coloneqq\S^{d-1}_i\subseteq M$ be the sphere with the diameter $\mathrm{U}|_{[R_i,L_{i+1}]}$ and centre $u_i$, see Figure~\ref{fig:nbhd-of-U-S}.
\begin{figure}[!htbp]
    \centering
    \includegraphics[width=\linewidth]{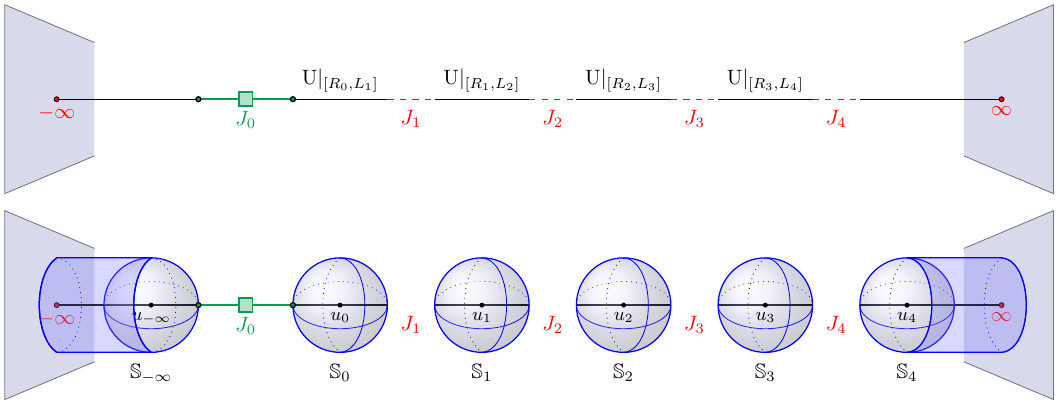}
    \caption[The punctured knot $\U_{\wh{01234}}$ and its chosen neighbourhood.]{The punctured knot $\U_{\wh{01234}}$ and its chosen neighbourhood, respectively.}
    \label{fig:nbhd-of-U-S}
\end{figure}

    Next, fix $n\geq0$ and a set $T=\{i_0<i_1<\dots<i_q\}\subseteq[n+1]$ for some $q\geq 1$. By convention, let $i_{-1}=-\infty$ and $i_{q+1}=i_1$. We now define certain $(d-1)$-dimensional spheres $\S_{i_pi_{p+1}}$ in $M$ for $-1\leq p\leq q$, see Figure~\ref{fig:M-0S}.
  Firstly, for $0\leq p\leq q-1$ let
  \begin{equation}\label{eq-def:S-ii}
    \S_{i_pi_{p+1}}\subseteq M
  \end{equation}
  be the codimension 1 submanifold consisting of the cylinder $[u_{i_p},u_{i_{p+1}-1}]\times\S^{d-2}$ together with the west hemisphere of $\S_{i_p}$ and the east of $\S_{i_{p+1}-1}$. Moreover, let $\S_{-\infty i_0}$ (respectively $\S_{i_q\infty}$) be the codimension 1 submanifold consisting of the cylinder $[-\infty,u_{i_0}]\times\S^{d-2}$ (respectively $[u_{i_q},\infty]\times\S^{d-2}$) together with a tubular neighbourhood of $-\infty\in\partial M$ (respectively $\infty\in\partial M$). 
  
  Let $\ball_{i_pi_{p+1}}$ be the region interior to the submanifold $\S_{i_pi_{p+1}}$, for $-1\leq p\leq q$. Thus, this is homeomorphic to a $d$-ball and contains $\mathrm{U}|_{[R_{i_p},L_{i_{p+1}}]}$. 

  Finally, for $S=\{i_1,\dots,i_q\}\subseteq\ul{n}$ let us denote by
  \begin{equation}\label{eq-def:S-S}
    \S_S\coloneqq\S_{i_1i_2}\vee\dots\vee\S_{i_qn+1}
  \end{equation}
  the wedge of $|S|$ copies of the $(d-1)$-dimensional sphere, each corresponding to one of our embedded sphere in the manifold. 
\end{notation}

\begin{cor}
    We have homotopy equivalences $\H_n(M)\simeq\tofib\FF_{\bull}$ and $\pF_{n+1}(M)\simeq\tofib\FF^{n+1}_{\bull}$ where for $S=\{i_1<\dots<i_q\}$ we have
\begin{align*}
    \FF_S &\cong\Embp\Big(J_0,M\sm\big(\ball_{-\infty0}\sqcup\ball_{0i_1}\sqcup\dots\sqcup \ball_{i_{q-1}i_q} \sqcup \ball_{i_q\infty}\big)\Big),\\
    \FF^{n+1}_S &\cong\Embp\Big(J_0,M\sm\big(\ball_{-\infty0}\sqcup\ball_{0i_1}\sqcup\dots\sqcup \ball_{i_{q-1}i_q} \sqcup \ball_{i_qn+1}\sqcup \ball_{n+1\infty}\big)\Big).
\end{align*}
\end{cor}
For a point in $\FF_{\ul{4}}$ as in the top of Figure~\ref{fig:nbhd-of-U-S} the corresponding embedding is in the bottom of Figure~\ref{fig:nbhd-of-U-S}. As another example,  the point in $\FF^{n+1}_S$ that was depicted at the bottom of Figure~\ref{fig:element-of-ffs} correspond to the bottom of Figure~\ref{fig:M-0S}, with $n+1=6$ and $S=\{2,3,5\}$. 

Moreover, observe that for $S=\emptyset$ we have
\[
    \FF_\emptyset
    \cong\Embp(J_0,M\sm(\ball_{-\infty0}\sqcup\ball_{0\infty})
    \cong\Emb_{\partial_0=-\infty,\partial_1=\infty}(I,M)\eqqcolon\Knots(M)
\]
since the removed space is a half-tubular neighbourhood of two points in the boundary, see the top of Figure~\ref{fig:M-emptyset}. In other words, the cube computing $\H_n(M)$ has \emph{the space of knots itself} as a vertex. 
\begin{figure}[!htbp]
    \centering
    \includegraphics[width=\linewidth]{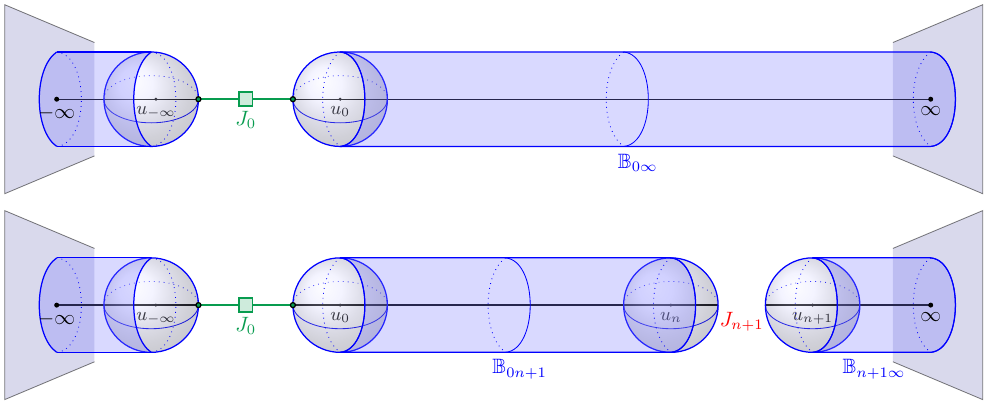}
    \caption[An example of a manifold $M$ minus our chosen neighbourhood of $\U_{\wh{0S}}$.]{For $S=\emptyset$ the spaces $\FF_S$ and $\FF^{n+1}_S$ consist of embeddings of the green arc into the complement of the blue submanifolds at the top and bottom pictures respectively.}
    \label{fig:M-emptyset}
\end{figure}

This is in contrast with the cube $\FF^{n+1}_{\bull}$ computing $\pF_{n+1}(M)$, for which at least one ball in the interior of $M$ is always removed. For example,
\[
    \FF^{n+1}_\emptyset
    \cong \Embp(J_0,M\sm(\ball_{-\infty0}\sqcup\ball_{0n+1}\sqcup\ball_{n+1\infty})
    \cong \Emb_{\partial_0=-\infty,\partial_1=\U(R_0)}(I,M\sm\ball_{0n+1}).
\]
Precisely this will allow us to compute the homotopy type of $\pF_{n+1}(M)$ in Section~\ref{sec:htpy-type}. In fact, in the following Section~\ref{subsec:delooping-vertices} we first describe the homotopy type of the spaces $\FF^{n+1}_S$ (and $\FF_S$, $S\neq\emptyset$).
}
\begin{example}[n=1]\label{ex:n=2-part1}
    The punctured cube $\mc{E}^2_{\bull}$ computing $\pT_2(M)$ was displayed in Example~\ref{ex:cube-deg-2}. On one hand, $\pF_2(M)\coloneqq\fib(p_2\colon\pT_2(M)\to\pT_1(M))$ is the total homotopy fibre of the top square:
\[\begin{tikzcd}[column sep=tiny,row sep=small]
        & \FF^2_{1}\arrow[dashed]{rr} &&
        \Emb_\partial(I\sm J_{12},M) \arrow{rr}[font={\footnotesize}]{r^0_1} &&
        \Emb_\partial(I\sm J_{012},M)
        \\
        \FF^2_{\emptyset}\arrow[dashed]{rr} && \Emb_\partial(I\sm J_2,M)\arrow[crossing over]{rr}[font={\footnotesize}]{r^0_\emptyset}
        \arrow[near start]{ur}[font={\footnotesize}]{r^1_{2}} &&
        \Emb_\partial(I\sm J_{02},M) \arrow{ur}[swap,font={\footnotesize}]{r^1_{02}}
    \end{tikzcd}
\]
Equivalently, $\pF_2(M)$ is the (total) homotopy fibre of the induced map $r^1_\emptyset\colon\FF^2_\emptyset\to\FF^2_1$.

On the other hand, if we complete $\mc{E}^2_{\bull}$ with the initial vertex $\Knots(M)\coloneqq\Emb_\partial(I,M)$, then the space $\H_1\coloneqq\hofib(\Knots(M)\to\pT_1(M))$ is the total fibre of the bottom square:
\[\begin{tikzcd}[column sep=tiny,row sep=small]
        & \FF_{1}\arrow[dashed]{rr} &&
        \Emb_\partial(I\sm J_{1},M) \arrow{rr}[font={\footnotesize}]{r^0_1} &&
        \Emb_\partial(I\sm J_{01},M)
        \\
        \FF_{\emptyset}\arrow[dashed]{rr} && \Emb_\partial(I,M)\arrow[crossing over]{rr}[font={\footnotesize}]{r^0_\emptyset}
        \arrow[near start]{ur}[font={\footnotesize}]{r^1_{\emptyset}} &&
        \Emb_\partial(I\sm J_{0},M) \arrow{ur}[swap,font={\footnotesize}]{r^1_{0}}
  \end{tikzcd}
\]
Equivalently, $\H_1(M)$ is the (total) homotopy fibre of the induced map $r^1_\emptyset\colon\FF_\emptyset\to\FF_1$.

Moreover, the map $\emap_2\colon\H_1(M)\to\pF_2(M)$ is the obvious upward map \red{between the displayed squares. Therefore, we have%
}
\[
\begin{tikzcd}
    \pF_2(M)\rar[dashed] & \FF^2_\emptyset\rar & \FF^2_1\\
    \H_1(M)\uar{\emap_2}\rar[dashed] & \FF_\emptyset\rar\uar & \FF_1\uar
\end{tikzcd}
\]
\end{example}

\subsection{Delooping the vertices}\label{subsec:delooping-vertices}
\red{
For a fixed $n\geq1$ (that will be clear from context) and any $S=\{i_1<\dots<i_q\}\subseteq\ul{n}$ define 
\begin{equation}\label{eq-def:M-0S}
    M_{0S}\coloneqq M\sm\big(\ball_{-\infty0}\sqcup\ball_{0i_1}\sqcup\dots\sqcup \ball_{i_{q-1}i_q}\sqcup \ball_{i_qn+1} \sqcup \ball_{n+1\infty}\big).
\end{equation}
Thus, we have a homeomorphism $\FF^{n+1}_S\cong\Emb_\partial(J_0,M_{0S})$,
under which the restriction map $r^k_S$ corresponds to the postcomposition with the obvious inclusion
\begin{equation}\label{eq-def:rho}
      \rho^k_S\colon M_{0S}\hra M_{0Sk}=M_{0S}\cup \big(\ball_{k-1,k+1}\sm (\ball_{k-1,k}\sqcup\ball_{k,k+1})\big).
\end{equation}
Roughly speaking, this adds the material corresponding to the $k$-the puncture, see Figure~\ref{fig:M-0S} (we omit drawing $\ball_{-\infty0}$ and $\ball_{n+1\infty}$ from now on). 
\begin{figure}[!htbp]
    \centering
    \includegraphics[width=\linewidth]{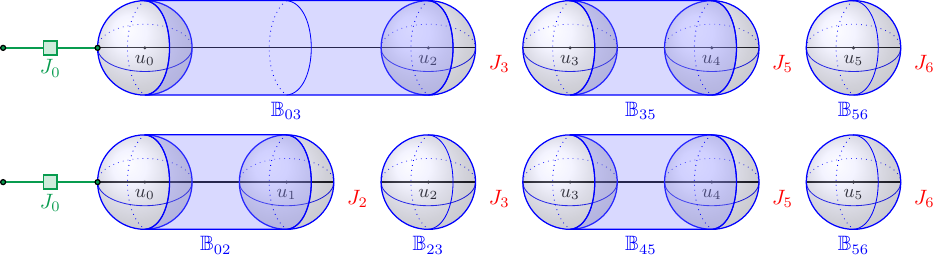}
    \caption[An example of a manifold $M_{0S}$.]{ For $n+1=6$ the manifold $M_{035}$ (resp.\ $M_{0235}$) is the complement of the three (resp.\ four) balls. The map $\rho^2_{04}\colon M_{035}\hra M_{0235}$ from top to bottom adds $\ball_{13}\sm (\ball_{12}\cup\ball_{23})=\ball_{03}\sm (\ball_{02}\cup\ball_{23})$.}
    \label{fig:M-0S}
\end{figure}

Moreover, the spaces $M_{0S}$ and maps $\rho^k_S$ obviously form an $n$-cube. From this and the discussion of the previous section we conclude the following. 
\begin{cor}\label{cor:r-rho} 
    The space $\pF_{n+1}(M)$ is homotopy equivalent to the total homotopy fibre of the $n$-cube $(\FF^{n+1}_{\bull},r)$, which is levelwise homeomorphic to the cube obtained from the $n$-cube $(M_{0S},\rho^k_S)$ by applying the functor $\Emb_\partial(J_0,-)$.
\end{cor}
}
\red{We will determine the homotopy type of $\FF^{n+1}_S$, 
}
motivated by the following observation for $n=1$.
\begin{example}
    Consider the following pullback square from \eqref{eq:diag-first-stage} in Section~\ref{sec:punc-knot-model}:
\[\begin{tikzcd}
        \Emb_\partial(I\sm J_1,M)\arrow{r}{r^0_1} & \Emb_\partial(I\sm J_{01},M)  \\
        \pT_1(M)\coloneqq\lim\mc{E}^1\arrow[dotted]{r}{p}\arrow[dotted]{u} & \pT_0(M) \arrow{u}{r^1_0}
    \end{tikzcd}
\]
    \red{The fibres of the two horizontal maps are homeomorphic, so by the definition \eqref{eq-def:FF-n+1-S} of $\FF^1_\emptyset$ we have 
    }
\[
    \pF_1(M)\coloneqq\fib(p)=\fib(r^0_1)\cong \Emb_\partial(J_0,M\sm \U_{\wh{01}})\simeq\Emb_\partial(J_0,M_{01})\cong\FF^1_\emptyset.
\]
    \red{Using the equivalences \eqref{eq:emb-imm}, we can compare $r^0_1$ to the analogous restriction map for immersions: 
    }
\begin{equation}\label{eq:final-delooping-n0}
    \begin{tikzcd}
        \pF_1(M)\arrow[dashed]{r}\arrow[hook]{d} & \Emb_\partial(I\sm J_1,M)\arrow{r}{r^0_1} \arrow[hook]{d}[swap]{\simeq} & \Emb_\partial(I\sm J_{01},M) \arrow[hook]{d}[swap]{\simeq} \\
        \Imm_\partial(J_0,M)\arrow[dashed]{r}\arrow{d} & \Imm_\partial(I\sm J_1,M)\arrow{r}{r^0_1} \arrow{d}{\deriv_{[L_0,w_{01}]}}[swap]{\simeq} & \Imm_\partial(I\sm J_{01},M) \arrow{d}{\deriv_{w_{01}}}[swap]{\simeq} \\
        \Omega\S M\arrow[dashed]{r} & \Path_*\S M\arrow{r} & \S M
    \end{tikzcd}
\end{equation}
    \red{Namely, for $\fib(r^0_1)$ in the case of immersions the disjointness condition with $\U_{\wh{01}}$ is lost and we just have $\Imm_\partial(J_0,M)$. 
    }
    The resulting equivalence $\pF_1(M)\simeq\Imm_\partial(J_0,M)\simeq\Omega\S M$ takes the unit derivative $\deriv f\colon I\to\S M$ of $f\colon J_0\hra M$ and makes it into a closed loop based at $(R_0,\vec{e})$, by concatenating it with the unit derivative of any arc $\wt{\U_0}$, which agrees with $\U_0$ except near endpoints, at which its derivative is $-\vec{e}$ instead (the endpoints of $\deriv f$ are $(L_0,\vec{e})$ and $(R_0,\vec{e})$, see Figure~\ref{fig:M-i}).
\end{example}
\begin{figure}[!htbp]
    \centering
    \includegraphics[width=0.95\linewidth]{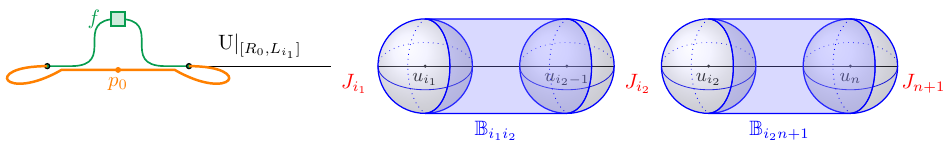}
    \caption[An example of a manifold $M_S$.]{The manifold $M_{i_1i_2}\coloneqq M\sm\big(\ball_{-\infty0}\sqcup\ball_{i_1i_2}\sqcup\ball_{i_2n+1}\sqcup\ball_{n+1\infty}\big)$. An example of $\wt{\U_0}$ is in orange.}
    \label{fig:M-i}
\end{figure}
Let us now analogously determine the homotopy type of each $\FF^{n+1}_S\coloneqq\fib(r^0_S)\cong\Emb_\partial(J_0,M_{0S})$ for $S=\{i_1<\dots<i_q\}\subseteq\ul{n}$ (possibly empty, for $q=0$). Namely, recall the space $M_{0S}$ from \eqref{eq-def:M-0S} and define
\begin{equation}\label{eq-def:M-S}
     \red{M_S\coloneqq  M\sm\big(\ball_{-\infty0}\sqcup\ball_{i_1i_2}\sqcup\dots\sqcup \ball_{i_qn+1}\sqcup\ball_{n+1\infty}\big) =M_{0S}\cup \ball_{0i_1}. 
     }
\end{equation}
We analogously have the composite
\begin{equation}\label{eq:FFS-Imm}
  \begin{tikzcd}[column sep=small]
    \deriv_S\colon\FF^{n+1}_S\cong\Emb_\partial(J_0,M_{0S})\arrow[hook]{r} & \Imm_\partial(J_0, M_{0S})\arrow[hook]{r} & \Imm_\partial(J_0, M_S)\arrow{r} & \Omega \big(\S M_S\big)
\end{tikzcd}
\end{equation}
of the inclusions which forget embeddedness and the disjointness between $J_0$ and $\U_{\wh{0n+1}}$, and a similar derivative map $f\mapsto (\deriv\wt{\U_0}) \cdot (\deriv f)_{1-t}$. To see that $\deriv_S$ is a weak equivalence we can simply replace $M$ by $M_S$ in the diagram \eqref{eq:final-delooping-n0}. This is actually a homotopy equivalence since the spaces are of the homotopy type of $CW$ complexes; for an alternative proof see~\cite{K-thesis}, where $\deriv_S$ is shown to arise also as a homotopy equivalence $\chi$ for a certain left homotopy inverse of $r^0_S$.

\begin{cor}\label{cor:delooping-vertices}
    There is a homotopy equivalence $\deriv_S\colon\FF^{n+1}_S\to\Omega \big(\S M_S\big)$.
\end{cor}

At this point, one must wonder if the whole cube $(\FF^{n+1}_S,\rho^k_S)$ can be delooped, i.e.\ if there are maps $\Omega\S M_S\to\Omega\S M_{Sk}$ which would correspond to our maps $r^k_S\colon\FF^{n+1}_S\to\FF^{n+1}_{Sk}$. We will see in Remark~\ref{rem:no-rho-cube} that this is not possible. However, in the next section we will be able to deloop this cube, and in fact $(n+1)$-times!
\red{
\begin{example}\label{ex:n=2-part2}
    In Example~\ref{ex:n=2-part1} we saw $\pF_2(M)\simeq\hofib(r^2_\emptyset\colon\FF^2_\emptyset\to\FF^2_1)$, and Corollary~\ref{cor:delooping-vertices} gives $\FF^2_\emptyset\simeq\Omega\S M_\emptyset\simeq\Omega\S M$ and $\FF^2_1\simeq\Omega\S M_1\simeq\Omega \S(M\sm\ball_{12})$. However, there is no map $M\to M\sm\ball_{12}$ which would correspond to $r^2_\emptyset$. Instead, we will show that $\pF_2(M)\simeq\Omega\hofib(l^2_\emptyset\colon\FF^2_1\to\FF^2_\emptyset)$ and $l^2_\emptyset$ corresponds to the map induced by $\lambda^2_\emptyset\colon M\sm\ball_{12}\to M$, homotopic to the inclusion. See Example~\ref{ex:n=2-part3}.
\end{example}
}


\section{Delooping the layers}\label{sec:delooping}
\red{In this section we prove that $\pF_{n+1}(M)$ is homotopy equivalent to an explicit $(n+1)$-fold loop space. We do this in several steps: first in Theorem~\ref{thm:delooping-layer} we will deloop $n$ times, then in Theorem~\ref{thm:final-delooping} once more, and then in Corollary~\ref{cor:gather-h-e} we will simplify the resulting space.
}

\begin{theorem}\label{thm:delooping-layer}
    For the $(n+1)$-st layer $\pF_{n+1}(M)\simeq\tofib(\FF^{n+1}_{\bull},r)$ of the Taylor tower for $\Knots(M)$, $n\geq0$, there is a contravariant $n$-cube $(\FF^{n+1}_{\bull},l)$ and an explicit homotopy equivalence
    \[\begin{tikzcd}
        \chi\colon\tofib_{\Cube(\ul{n})}\big(\FF^{n+1}_{\bull},r\big) \arrow{r}{\sim} & \Omega^n\tofib_{\Cube(\ul{n})^{\mathsf{op}}}\big(\FF^{n+1}_{\bull},l\big).
    \end{tikzcd}
    \]
\end{theorem}
\red{
Recall from Remark~\ref{rem:contravar-cubes} that a contravariant $n$-cube is a functor $\Cube(\ul{n})^{\mathsf{op}}\to\mathsf{Top}_*$ and that there is an isomorphism $\Cube(\ul{n})^{\mathsf{op}}\cong\Cube(\ul{n})$. 
We first prove in Proposition~\ref{prop:cube-retraction} that a homotopy equivalence $\chi$ as in the theorem exists for any cube which has an \emph{$n$-fold left homotopy inverse}. In Section~\ref{subsec:left-htpy-inv} 
we define this notion and state that proposition, and in Section~\ref{subsec:delooping-initial} we proceed to construct maps $l^k_S$ giving such an inverse $(\FF^{n+1}_{\bull},l)$ for our cube $(\FF^{n+1}_{\bull},r)$, proving Theorem~\ref{thm:delooping-layer}. All proofs about left homotopy inverses (including the proof of Proposition~\ref{prop:cube-retraction}) are deferred to Appendix~\ref{app:proofs}.

In Section~\ref{subsec:delooping-final} we will then observe that, in contrast to $r^k_S$, the maps $l^k_S$ are compatible with the homotopy equivalences $\deriv_S$ from Corollary~\ref{cor:delooping-vertices}. That is, we will have commutative diagrams
}
\[
    \begin{tikzcd}[column sep=large]
        \FF^{n+1}_S\arrow{d}[swap]{\deriv_S} & \FF^{n+1}_{Sk} \arrow{l}[swap]{l^k_{S}} \arrow{d}{\deriv_{Sk}}\\
        \Omega(\S M_S) & \Omega\S(M_{Sk}) \arrow{l}[swap]{\Omega(\S\lambda^k_S)}
    \end{tikzcd}
\]
\red{
for any $k\notin S\subseteq\ul{n}$, where $\lambda^k_S\colon M_{Sk}\to M_S$ is a certain smooth map and $\Omega(\S\lambda^k_S)$ is the induced map on the loop space of the unit sphere bundle. We will then conclude the following.
}

\begin{theorem}\label{thm:final-delooping}
  The map $\deriv_{\bull}\colon\big(\FF^{n+1}_{\bull},l\big)\to \big(\Omega\S M_{\bull},\Omega\S\lambda\big)$ is a homotopy equivalence of cubes. Moreover, this gives a homotopy equivalence
  \[\begin{tikzcd}
  \deriv\colon\tofib\big(\FF^{n+1}_{\bull},l\big)\arrow{r}{\sim} & \tofib\big(\Omega M_{\bull},\Omega\lambda\big).
  \end{tikzcd}
  \]
\end{theorem}
\red{
Finally, we will observe that the manifold $M_S$ retracts onto the wedge $M\vee\S_S$, of $M$ with $|S|$ copies of the sphere $\S^{d-1}$. Then in Lemma~\ref{lem:retr} we show that $l^k_S$ corresponds under these retractions to the map $\coll^k_S$, which collapses the sphere corresponding to the index $k$ to the wedge point (and is the identity on all other wedge factors). In other words, there are commutative diagrams
}
\[
    \begin{tikzcd}[column sep=large]
        M_S\arrow{d}[swap]{\retr_S} & M_{Sk} \arrow{l}[swap]{l^k_{S}} \arrow{d}{\retr_{Sk}}\\
        M\vee\S_S & M\vee\S_{Sk} \arrow{l}[swap]{\coll^k_S}
    \end{tikzcd}
\]
\red{
Note that in the last total homotopy fibre in Theorem~\ref{thm:final-delooping} the unit sphere bundle is omitted. Moreover, recall from Remark~\ref{rem:tofib-hofib} that $\tofib(\Omega M_{\bull},\Omega\lambda)\simeq \Omega\tofib(M_{\bull},\lambda)$. Collecting all these results we have the following.
}
\begin{cor}\label{cor:gather-h-e}
  There are homotopy equivalences $\pF_1(M)\simeq\Omega(\S M)$ and for $n\geq 1$:
  \[\begin{tikzcd}
    \pF_{n+1}(M)\arrow{r}{\chi}[swap]{\sim} & \Omega^{n}\tofib(\FF^{n+1}_{\bull},l)\arrow[]{r}{\deriv}[swap]{\sim} & \Omega^{n+1}\tofib(M_{\bull},\lambda)\arrow[]{r}{\retr}[swap]{\sim} & \Omega^{n+1}\tofib(M\vee\S_{\bull},\coll).
  \end{tikzcd}
  \]
\end{cor}

\subsection{About left homotopy inverses}\label{subsec:left-htpy-inv}
\red{
    Let $(X,*_X)$ and $(Y,*_Y)$ be based spaces, and let us consider only based maps in what follows. 
}
A \textsf{left homotopy inverse} (a retraction up to homotopy) for a map $r\colon X\to Y$ is a map $l\colon Y\to X$ such that $l\circ r\simeq \Id_X$. For our purposes it is crucial to specify a homotopy $h$ from $\Id$ to $l\circ r$:
    \begin{equation}\label{data:retraction}
    \begin{tikzcd}[labels={font=\scriptsize}]
     & X\\
    X \arrow[""{name=U, below}]{ur}{\Id} \arrow{r}[swap]{r} & |[alias=D]| Y\arrow{u}[swap]{l}
    \arrow[Rightarrow, from=U, to=D, "h"]
    \end{tikzcd}
    \end{equation}
Moreover, we have $\pi_{*-1}\hofib(r)\cong\pi_*\hofib(l)\cong\ker(l_*)$ and split short exact sequences
\begin{equation}\label{eq:htpy-gps-retraction}
    \begin{tikzcd}
0\arrow[]{r}{} & \pi_*X\arrow[shift left]{r}{r_*} &\pi_*Y\arrow[]{r}{}\arrow[shift left]{l}{l_*} & \pi_{*-1}\hofib(r)\arrow[]{r}{} & 0,
    \end{tikzcd}
\end{equation}
since $l_*$ is a section in the long exact sequence of homotopy groups for $r$. Actually, more is true.
\begin{lemma}\label{lem:retractions}
    Given the data of \eqref{data:retraction} there are explicit inverse homotopy equivalences
    \[\begin{tikzcd}
        \chi\colon\hofib(r)\arrow[shift left=3pt]{rr}[swap]{\sim} && \Omega\hofib(l)\arrow[shift left=3pt]{ll}{}\colon\chi^{-1}\;,
    \end{tikzcd}
    \]
    \red{where $\chi^{-1}$ is a forgetful map, taking $(y_t,x_{s,t})\in\Omega\hofib(l)$ to $(*_X,y_t)\in\hofib(r)$.
    }
\end{lemma}
\red{
    We abuse notation and use $y_t$ to denote a path $y\colon I\to Y, t\mapsto y_t$, and $x_{s,t}$ a 2-parameter family $x\colon I^2\to X$. Thus $(y_t,x_{s,t})\in\Omega\hofib(l)$ means that $y_0=y_1$ and $x_{0,t}=y_t$ and $x_{1,t}=*_X$.  See the proof of this lemma in Appendix~\ref{app:proofs} for a formula for $\chi$. 
}

One can generalise Lemma~\ref{lem:retractions} from $1$-cubes (maps), to diagrams over $\Cube(\ul{m})$ for $m\geq1$ as follows.
\begin{defn}\label{def:left-inverse-cube}
    Let $R_{\bull}=C_{\bull\notni m}\xrightarrow{r^m} C_{\bull\ni m}$ be an $m$-cube with $m\geq 1$, seen as a $1$-cube of $(m-1)$-cubes. A \textsf{left homotopy inverse} for $R_{\bull}$ is the data of a diagram:
    \begin{equation}\label{data:1-fold-retraction}
    \begin{tikzcd}[labels={font=\scriptsize}]
     & C_{\bull\notni m}\\
    C_{\bull\notni m} \arrow[""{name=U, below}]{ur}{\Id} \arrow{r}[swap]{r^m} 
    & |[alias=D]| C_{\bull\ni m}\arrow{u}[swap]{l^m}
    \arrow[Rightarrow, from=U, to=D, "h^m"]
\end{tikzcd}
\end{equation}
In more detail, it consists of
    \begin{enumerate}[topsep=-0.5em]
        \item an $m$-cube $L_{\bull}=C_{\bull\ni m}\xrightarrow{l^m} C_{\bull\notni m}$ and
        \item for each $S\subseteq\ul{m-1}$ a homotopy
    $h^m_S\colon\Id_{C_S}\squig l^m_S\circ r^m_S$,
    that are mutually compatible in the sense that $h^m_{\bull}(t)\colon C_{\bull\notni m}\to C_{\bull\notni m}$ is an $m$-cube for each fixed $0\leq t\leq 1$.
    \end{enumerate}
\end{defn}
\begin{lemma}\label{lem:first-cube-retraction}
    Given the data of \eqref{data:1-fold-retraction}, there are inverse homotopy equivalences
    \[\begin{tikzcd}
        \chi_m\colon\tofib(R_{\bull})\arrow[shift left=3pt]{rr}{}[swap]{\sim} 
        && \Omega\tofib(L_{\bull})\arrow[shift left=3pt]{ll}:(\chi_m)^{-1}.
    \end{tikzcd}
    \]
\end{lemma}
To repeat this procedure and get a homotopy equivalence from the total fibre of an $m$-cube to an $m$-fold loop space, we need a left homotopy inverse $l^k_S$ for each $r^k_S$. We also need suitable conditions for homotopies $h^k_S$, in order to avoid obtaining cubes which are commutative only up to homotopy.

\begin{defn}\label{def:m-left-inverse-cube}
An \textsf{$m$-fold left homotopy inverse} for an $m$-cube $D^m\coloneqq(C_{\bull},r)$ is given as follows.
\begin{enumerate}
    \item For each $S\subseteq\ul{m}$ and $k\in\ul{m}\sm S$ a map $l^k_S\colon C_{S k}\to C_S$ is given such that
    \begin{align}
    l^k_{S}\circ l^i_{Sk}= l^i_S\circ l^k_{S i},\quad &\forall i\notin S,\;i\neq k,\label{cond:ll}\\
    l^k_{Si}\circ r^i_{Sk}= r^i_S\circ l^k_S,\quad &\forall i\notin S,\; k> Si\coloneqq S\cup\{i\}.\label{cond:lr}
    \end{align}

\noindent These equations ensure that for $0\leq k\leq m$ there is a well-defined $m$-cube $D^k$ obtained from $D^m$ by replacing the arrows $r^i_S$ by $l^i_S$ for $k+1\leq i\leq m$.\footnote{We see $D^k$ as an $(m-k)$-cube of $k$-cubes; maps in $k$-cubes are $r$-maps, while maps between them are $l$-maps.} In particular, $D^0=(C_{\bull},l)$.

\item For each $0\leq k\leq m$ and $t\in[0,1]$ a map of diagrams
    \begin{equation}\label{cond:htpies-of-diagrams}
    h^k_{\bull}(t)\colon D^k_{\bull\notni k} \to D^k_{\bull\notni k}
    \end{equation}
    is given, such that $h^k_{\bull}(0)=\Id$ and $h^k_{\bull}(1)=l^k_{\bull}\circ r^k_{\bull}$.
\end{enumerate}
\end{defn}
For $k=m,\dots,0$ \red{if we write the cube $D^k$ as $D^k_{\bull\notni k}\xrightarrow{r^k_{\bull}}D^k_{\bull\ni k}$, then we have
}
$D^{k-1}\coloneqq D^k_{\bull\ni k}\xrightarrow{l^k_{\bull}} D^k_{\bull\notni k}$. Hence $D^k$ is a diagram over $\Cube(\ul{m-k})\times\Cube^{op}(\ul{m}\sm\ul{m-k})$; in particular, $D^0$ is a contravariant cube. However, \red{recall from Remark~\ref{rem:contravar-cubes} that all these categories are isomorphic to $\Cube(\ul{m})$, and we will not distinguish among them. 
}
Thus, an $m$-fold left homotopy inverse for $D^m$ is given by the following data for each $0\leq k\leq m$:
\begin{equation}\label{data:m-fold-retraction}
    \begin{tikzcd}[labels={font=\scriptsize}]
     & D^k_{\bull\notni k}\\
    D^k_{\bull\notni k} \arrow[""{name=U, below}]{ur}{\Id} \arrow{r}[swap]{r^k_{\bull}} & |[alias=D]| D^k_{\bull\ni k} \arrow{u}[swap]{l^k_{\bull}}
    \arrow[Rightarrow, from=U, to=D, "h^k_{\bull}"]
    \end{tikzcd}
\end{equation}
\begin{prop}\label{prop:cube-retraction}
    Given the data of \eqref{data:m-fold-retraction} there is a sequence of homotopy equivalences
\begin{equation*}
    \begin{tikzcd}
    \chi\colon\tofib(D^m)\arrow{r}{\chi_m}[swap]{\sim} & \Omega\tofib(D^{m-1})\arrow{r}{\chi_{m-1}}[swap]{\sim} & \Omega^2\tofib(D^{m-2})\arrow{r}{\chi_{m-2}}[swap]{\sim} & \cdots\arrow{r}{\chi_{1}}[swap]{\sim} & \Omega^{m}\tofib(D^{0}).
 \end{tikzcd}
\end{equation*}
Moreover, the homotopy inverse is given by $\chi^{-1}\colon\Omega^m\tofib(D^0)\xrightarrow{\Omega^m\forg}\Omega^mC_{\ul{m}}\hra\tofib(D^m)$, where
\[
\forg\colon\tofib(D^0)\ra C_{\ul{m}},\quad\quad
\big\{f^{\ul{m}\sm S}\colon I^S\to C_{\ul{m}\sm S}\big\} \mapsto f^{\ul{m}}\;.
\]
\end{prop}
Taking entrywise homotopy groups of such $D^0=(C_{\bull},l)$ gives a contravariant $m$-cube in graded groups $\pi_*D^0=(\pi_*C_{\bull},\pi_*l)$ (with $*>0$), which has a \emph{right $m$-fold inverse}. Thus, there is an injection
\begin{equation}\label{eq:htpy-gps-tofib}
    \begin{tikzcd}
    \pi_*\Big(\tofib(D^0)\Big)\cong \bigcap_{k\in\ul{m}}\ker\big(\pi_*l^k_{\ul{m}\sm k}\big)\arrow[hook]{rr}{\forg} && \pi_*C_{\ul{m}}
    \end{tikzcd}
\end{equation}
analogously to \red{the one in
}
\eqref{eq:htpy-gps-retraction}. For proofs of all these results see Appendix~\ref{app:proofs}.

\subsection{A left homotopy inverse for the layer}\label{subsec:l-h-i-for-layer}\label{subsec:delooping-initial}
Let us now turn to applying the preceding discussion in our situation.

\begin{proof}[Proof of Theorem~\ref{thm:delooping-layer}]
    \red{By Proposition~\ref{prop:cube-retraction} it suffices to 
    }
    construct an $n$-fold left homotopy inverse $(\FF^{n+1}_{S},l^k_S)$ for $(\FF^{n+1}_S,r^k_S)$. 
    Recall that by Corollary~\ref{cor:r-rho} the latter is obtained by applying $\Emb_\partial(J_0,-)$ to the $n$-cube $(M_{0S},\rho^k_S)$. 
    
    \red{
    By Theorem~\ref{thm:lambda-cube} below we have a left homotopy inverse $n$-cube $(M_{0S},\lambda^k_S)$ for $(M_{0S},\rho^k_S)$. 
    Then we get the desired cube $(\FF^{n+1}_{S},l^k_S)$ by applying $\Emb_\partial(J_0,-)$ to $(M_{0S},\lambda^k_S)$. More precisely, letting $l^k_S=\lambda^k_S\circ-$ and $h^k_{\bull}\circ-$ gives a cube that also satisfies conditions of Definition~\ref{def:m-left-inverse-cube}.
    }
\end{proof}

\red{
Therefore, it remains to show Theorem~\ref{thm:lambda-cube}, i.e.\ that such a left homotopy inverse cube $(M_{0S},\lambda^k_S)$ exists.
We first define for each map $\rho^k_S$ a left homotopy inverse  $\lambda^k_S$ 
}
in the sense of~\eqref{data:retraction}, and then revisit the construction to ensure that all conditions of Definition~\ref{def:m-left-inverse-cube} are satisfied, so that the whole cube $(M_{0S},\lambda^k_S)$ is a left homotopy inverse, giving Theorem~\ref{thm:lambda-cube}.
\begin{lemma}\label{lem:defn-of-lambda}
    For $k\notin S\subseteq \ul{n}$ the map $\rho^k_S$ has a left homotopy inverse $\lambda^k_S\colon M_{0Sk}\to M_{0S}$.
\end{lemma}
\begin{proof}
Let $S>k\coloneqq\{j\in S: j>k\}$ and let $i_{p+1}\coloneqq\min\{(S>k)\cup \{n+1\}\}$ (this is the smallest index in $S$ which is bigger than $k$, or $n+1$ if that set is empty). Consider the inclusion map
\[
e_{ki_{p+1}}\colon M_{0Sk}\hra M_{0Sk}\cup \ball_{ki_{p+1}}
\]
which adds back the ball $\ball_{ki_{p+1}}$. We visualise this by erasing $\ball_{ki_{p+1}}$ as in Figure~\ref{fig:M-erase}.

Observe that $M_{0Sk}\cup \ball_{ki_{p+1}}$ and $M_{0S}$ are isotopic as submanifolds of $M$ by an ambient isotopy
\[
    \drag_{S>k}^{k}(t)\colon M\to M,\quad \drag_{S>k}^{k}(0)=\Id_M,\quad \drag_{S>k}^{k}(1)(M_{0Sk}\cup \ball_{ki_{p+1}})=M_{0S},
\]
which, loosely speaking, \emph{elongates $\ball_{i_pk}$ by gradually dragging the right hemisphere of $\S_{k-1}$ to the right, until it equals the right hemisphere of $\S_{i_{p+1}-1}$}. We will define a specific parametrisation in the proof of Theorem~\ref{thm:lambda-cube} below (it will indeed depend on $S>k$ and not only on $i_{p+1}$).

Now let $d_S^{k}(t)\coloneqq\drag_{S>k}^{k}(t)|_{M_{0Sk}\cup \ball_{ki_{p+1}}}$ and $\lambda^k_S\coloneqq d_S^{k}(1)\circ e_{ki_{p+1}}\colon M_{0Sk}\to M_{0S}$.

\begin{figure}[!htbp]
    \centering
    \includegraphics[width=0.9\linewidth]{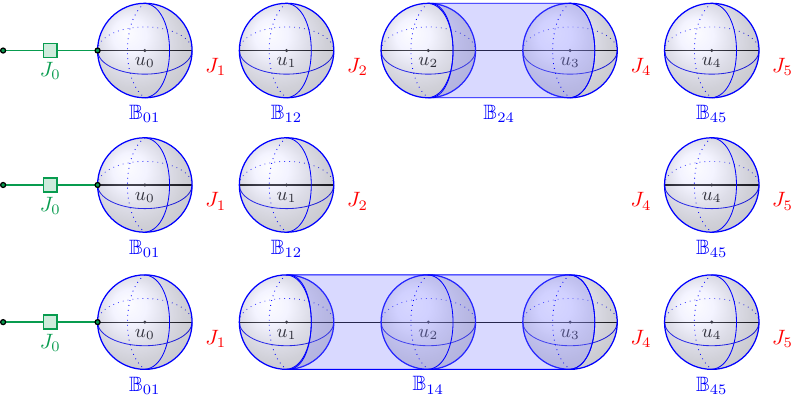}
    \caption[The map $\lambda^k_S$.]{ The map $\lambda^k_S$ for $n=5,S=\{1,4\},k=2,i_p=1,i_{p+1}=4$ takes top $M_{0124}$ to the bottom $M_{014}$. From the top to the middle apply $e_{24}$, and from the middle to the bottom apply $d^2_{1,4}(1)$.}
    \label{fig:M-erase}
\end{figure}
It remains to provide a homotopy $h^k_S$ between $\Id_{M_{0S}}$ and the composite
\[\begin{tikzcd}[row sep=tiny,column sep=small]
    \lambda^k_S\circ \rho^k_S\colon\: M_{0S}\arrow[hook]{r}{\rho^k_S}
    & M_{0Sk}\arrow[equals]{d} \arrow[hook]{r}{e_{ki_{p+1}}}
    & M_{0Sk}\cup \ball_{ki_{p+1}}\arrow[equals]{d} \arrow{rrrr}{d_S^{k}(1)}
    &&&& M_{0S}.
    \\
    & M_{0S}\cup \big(\ball_{i_pi_{p+1}}\sm (\ball_{i_pk}\cup\ball_{ki_{p+1}}\big) \arrow{r}[font={\tiny}]{\cup\ball_{ki_{p+1}}}
    & M_{0S}\cup \big(\ball_{i_pi_{p+1}}\sm\ball_{i_pk}\big)
    &&&
\end{tikzcd}
\]
The composition of the first two maps adds to $M_{0S}$ the material $\ball_{i_pi_{p+1}}\sm\ball_{i_pk}$, which is diffeomorphic to a ball (note that $\ball_{k-1,k+1}\sm (\ball_{k-1}\cup\ball_{k})=\ball_{i_pi_{p+1}}\sm (\ball_{i_pk}\cup\ball_{ki_{p+1}})$). Now adding this material gradually gives an isotopy $\mathsf{add}_t\colon M_{0S}\hra M$ such that $\im (\mathsf{add}_0)=M_{0S}$ and $\im (\mathsf{add}_1)=M_{0Sk}\cup\ball_{ki_{p+1}}$. We can parametrise this so that $\im(\mathsf{add}_t)=\im(d_S^{k}(t))$ for each $t\in[0,1]$, so the two isotopies can be composed into the desired homotopy
\[\begin{tikzcd}[column sep=large]
    h^k_S(t)\colon M_{0S}\arrow[hook]{r}[font={\footnotesize}]{\mathsf{add}_t}
    & \im(\mathsf{add}_t) \arrow{rr}[font={\footnotesize}]{d_S^{k}(t)}
    && M_{0S}.
\end{tikzcd}\qedhere
\]
\end{proof}

\begin{theorem}\label{thm:lambda-cube}
    The $n$-cube $(M_{0S},\rho^k_S)$ has an $n$-fold left homotopy inverse $(M_{0S},\lambda^k_S)$.
\end{theorem}
\begin{proof}
We now ensure that the maps $\lambda^k_S$ and $h^k_S$ constructed in the previous proof satisfy conditions of Definition~\ref{def:m-left-inverse-cube}. We are still free to \emph{specify a particular parametrisation} of the ambient isotopy $\drag_{S>k}^{k}(t)\colon M\to M$, which we roughly described as a `dragging move', acting non-trivially only in a tubular neighbourhood of $\mathrm{U}|_{[R_{i_p},L_{i_{p+1}}]}$.

Firstly, for $S\subseteq\ul{n}\sm\{i,k\}$  the conditions \eqref{cond:lr} and \eqref{cond:ll} are respectively equivalent to having that the following left diagram commute for $k>Si$ and the right diagram for, say, $i<k$:
\[\begin{tikzcd}
        M_{0Si} & M_{0Sik} \arrow{l}[swap]{\lambda^k_{Si}}\\
        M_{0S}\arrow{u}[swap]{\rho^i_S} & M_{0Sk} \arrow{l}[swap]{\lambda^k_S} \arrow{u}{\rho^i_{Sk}}
    \end{tikzcd}\qquad\quad \begin{tikzcd}
        M_{0Si}\arrow{d}[swap]{\lambda^i_S} & M_{0Sik} \arrow{l}[swap]{\lambda^k_{Si}} \arrow{d}{\lambda^i_{Sk}}\\
        M_{0S} & M_{0Sk} \arrow{l}[swap]{\lambda^k_S}
    \end{tikzcd}
\]
For the left diagram this is clear. Indeed, $\rho^i_S$ and $\rho^i_{Sk}$ both add the same material $\ball_{i-1,i+1}\sm (\ball_{i-1}\cup\ball_{i})$ independently of the location of the other punctures, and as $k>S i$, both $\lambda^k_S$ and $\lambda^k_{Si}$ erase the ball $\ball_{kn}$ and then use the same flow $\drag_{S>k}^{k}=\drag_{Si>k}^{k}=\drag_\emptyset^k$. On the other hand, the commutativity of the right diagram will follow if we ensure that
\begin{equation}\label{diag:cond-ll}
  \drag_{S{>i}}^{i}(1)\circ\drag_{Si{>k}}^{k}(1)=\drag_{S{>k}}^{k}(1)\circ\drag_{Sk{>i}}^{i}(1).
\end{equation}
Lastly, the condition \eqref{cond:htpies-of-diagrams} is equivalent to having that for each $t\in[0,1]$ the following left square commute if $i>k$ and the right square if $i<k$:
\[\begin{tikzcd}
        M_{0S}\arrow{d}[swap]{h^i_S(t)}\arrow{r}{\rho^k_{S}} & M_{0Sk} \arrow{d}{h^i_{Sk}(t)}\\
        M_{0S} \arrow{r}{\rho^k_S} & M_{0Sk}
    \end{tikzcd}\qquad\quad
    \begin{tikzcd}
        M_{0S}\arrow{d}[swap]{h^i_S(t)} & M_{0Sk} \arrow{l}[swap]{\lambda^k_{S}} \arrow{d}{h^i_{Sk}(t)}\\
        M_{0S} & M_{0Sk} \arrow{l}[swap]{\lambda^k_S}
    \end{tikzcd}
\]
Again, the case on the left is clear since $\drag_{Sk>i}^{i}(t)=\drag_{S>i}^{i}(t)$), and for the right one we should ensure that
\begin{equation}\label{diag:cond-htpies}
    \drag_{S{>i}}^{i}(t)\circ\drag_{S{>k}}^{k}(1)=\drag_{S{>k}}^{k}(1)\circ\drag_{Sk{>i}}^{i}(t) \quad\quad \forall t\in[0,1].
\end{equation}
Note that \eqref{diag:cond-ll} follows from \eqref{diag:cond-htpies} by putting $t=1$ and using $\{iS{>k}\}=\{S{>k}\}$, since $i<k$.

We now define parametrisations of $\drag_{S>i}^{i}$ inductively on $|S>i|\geq0$ for $i\in\ul{n}$ and $S\subseteq\ul{n}\sm\{i\}$. Pick each $\drag_{\emptyset}^{i}$ freely and assume for some $s\geq0$ we chose $\drag_{S>i}^{i}$ for all $|S>i|<s$. Let now $|S'>i|=s$ for some $S'=Sk$ with $k\coloneqq\min\{(S'>i)\cup \{n\}\}$. Then let
\begin{equation}\label{eq:drag-def}
    \drag_{S'{>i}}^{i}(t)\coloneqq\drag_{S{>k}}^{k}(1)^{-1}\circ\drag_{S{>i}}^{i}(t)\circ\drag_{S{>k}}^{k}(1).
\end{equation}
This finishes the definition. Let us check that \eqref{diag:cond-htpies} holds for any fixed $i<k$ and $S\subseteq\ul{n}\sm\{i,k\}$  by induction on $|i<S<k|\geq0$. Firstly, if there is no $i_p\in S$ with $i<i_p<k$, then \eqref{eq:drag-def} becomes precisely \eqref{diag:cond-htpies}. Otherwise, take the smallest such $i_p$; write $S=Ri_p$. Applying \eqref{eq:drag-def} for $S'=(Rk)i_p$ and the induction hypothesis for $\drag^i_{Rk>i}(t)$ gives
\begin{align*}
\drag_{Sk{>i}}^{i}(t)&=
\drag_{Rk{>i_p}}^{i_p}(1)^{-1}\circ
\Big(\drag_{R{>k}}^{k}(1)^{-1}\circ\drag_{R{>i}}^{i}(t)\circ\drag_{R{>k}}^{k}(1)\Big)
\circ\drag_{Rk{>i_p}}^{i_p}(1)
\\
&=\drag_{Ri_p{>k}}^{k}(1)^{-1}\circ
\Big(\drag_{R{>i_p}}^{i_p}(1)^{-1}\circ\drag_{R{>i}}^{i}(t)\circ\drag_{R{>i_p}}^{i_p}(1)\Big)
\circ\drag_{Ri_p{>k}}^{k}(1).
\end{align*}
For the second equality we used that $\drag_{R{>k}}^{k}(1)\circ\drag_{Rk{>i_p}}^{i_p}(1) =\drag_{R{>i_p}}^{i_p}(1)\circ\drag_{R{>k}}^{k}(1)$ again by the induction hypothesis \eqref{diag:cond-htpies}.
Now observe that $\{Ri_p>k\}=\{R>k\}$, so the last expression equals $\drag_{S{>k}}^{k}(1)^{-1}\circ\drag_{S{>i}}^{i}(t)\circ\drag_{S{>k}}^{k}(1)$, finishing the induction step.
\end{proof}

\subsection{The derivative and the retraction}\label{subsec:delooping-final}\label{subsec:retr}

\red{
Let us observe that we can define maps
\[
    \lambda^k_S\colon M_{Sk}\to M_{S}
\]
for $k\notin S$ completely analogously to the maps $\lambda^k_S\colon M_{0Sk}\to M_{0S}$ 
}
from the proof of Lemma~\ref{lem:defn-of-lambda}. Namely, the presence of the index $0$ was irrelevant there. \red{Let us just point out that $\lambda^k_\emptyset\colon M_k=M_\emptyset\sm\ball_{kn+1}\to M_\emptyset$ not only adds the ball $\ball_{kn+1}$, but also deforms $M_\emptyset$ by `dragging', cf.\ Figure~\ref{fig:M-emptyset}. 
}
Similarly as before (see the proof of Theorem~\ref{thm:lambda-cube}), these maps form an $n$-cube 
\begin{equation}\label{eq:M-cube}
    (M_{\bull},\lambda).
\end{equation}
We let $(\Omega\S M_{\bull},\Omega\S\lambda)$ be the corresponding cube of loops on the unit tangent bundles.

\subsubsection*{The derivative}

\begin{proof}[Proof of Theorem~\ref{thm:final-delooping}]
\red{
    We need to prove that for any $n\geq 1$ the homotopy equivalences
\[
    \deriv_S\colon\FF^{n+1}_S\to\Omega\S M_S
\]
    from Corollary~\ref{cor:delooping-vertices}, given as the unit derivatives of embeddings, define a homotopy equivalence
\[
    \deriv_{\bull}\colon(\FF^{n+1}_{\bull},l)\to(\Omega\S M_{\bull},\Omega\S\lambda).
\]
    Indeed, the maps $\deriv_S$ are clearly compatible with the $\lambda$-maps, $\deriv_S\circ l^k_S=(\Omega\S\lambda^k_S)\circ\deriv_{Sk}$
    because the $l$-maps are given as the postcomposition with the $\lambda$-maps.
    
    The second claim of the theorem is that the map of cubes $\S(M_{\bull})\to M_{\bull}$ forgetting the tangent data is a homotopy equivalence on total homotopy fibres. 
    }
    The rows in the commutative diagram
\[\begin{tikzcd}
    \S^{d-1} \arrow{r}{}\arrow[equal]{d}{} & \S(M_{Sk}) \arrow{r}{}\arrow[]{d}{\S\lambda^k_S} & M_{Sk} \arrow[]{d}{\lambda^k_S}\\
    \S^{d-1}\arrow{r}{} & \S(M_{S})\arrow{r}{} & M_{S}
\end{tikzcd}
\]
    are fibre bundles, so \red{the square on the right is a pullback. Therefore, 
    }
    comparing the vertical homotopy fibres gives $\hofib(\S\lambda^k_S)\simeq\hofib(\lambda^k_S)$.
\end{proof}

\begin{remm}\label{rem:straight}
With tangent data now gone, define $\deriv_S\colon\FF^{n+1}_S\to\Omega M_S$ using simply $\wt{\U_0}\coloneqq\U_0$.
\end{remm}

\begin{remm}\label{rem:bcks}
    As mentioned in the introduction, the authors of~\cite{BCKS} use Sinha's cosimplicial model $AM_n(I^3)\coloneqq\holim\Conf'_{\bull}\langle I^3,\partial\rangle$, where $\Conf'_n\langle I^3\rangle$ is a compactified configuration space of $n$ points in $I^3$ together with tangent vectors~\cite{Sinha-cosimplicial}. To compute $\pF_n(I^3)$ they use the cosimplicial identity $\mc{s}^k\circ \mc{d}^k=\Id$, which precisely says that the codegeneracy $\mc{s}^k$ (forgetting the $k$-th point in the configuration) is a strict left inverse for $\mc{d}^k$ (doubling the $k$-th point).
\end{remm}

\begin{remm}\label{rem:no-rho-cube}
Actually, the maps $\deriv_S$ factor through spaces $\Omega M_{0S}$, giving the following diagram
\[\begin{tikzcd}
    \FF^{n+1}_{Sk}\arrow[shift left,dashed]{d}{l^k_{S}}\arrow{rr}{\deriv_{Sk}} &&
    \Omega M_{0Sk}\arrow[shift left,dashed]{d}{\Omega\lambda^k_{S}}\arrow[hook]{rr}{\rho^0} &&
    \Omega M_{Sk}\arrow[dashed]{d}{\Omega\lambda^k_{S}}\\
    \FF^{n+1}_S\arrow[shift left]{u}{r^k_S}\arrow{rr}{\deriv_S} &&
    \Omega M_{0S}\arrow[shift left]{u}{\Omega\rho^k_S}\arrow[hook]{rr}{\rho^0} &&
    \Omega M_S
\end{tikzcd}
\]
in which if either all upward or all downward arrows are omitted, the resulting diagram commutes. By Theorem~\ref{thm:final-delooping} the first and last dashed maps form cubes with homotopy equivalent total fibres. By Theorem~\ref{thm:lambda-cube} there is a cube $(\Omega M_{0\bull},\Omega\rho)$ and its $n$-fold left homotopy inverse $(\Omega M_{0\bull},\Omega\lambda)$.

However, it is important to point out that there are \emph{no $\rho$-maps} between spaces $\Omega M_S$ which would form a cube equivalent to our original cube $(\FF^{n+1}_S,r^k_S)$. The reason lies in the fact that the information about the disjointness of $J_0$ with $\ball_{0i_1}$ is lost when we pass from $M_{0S}$ to $M_S$. More precisely, there are no maps $M_S\to M_{Sk}$ when $k$ is smaller than all indices in $S$ (but otherwise we do have $\rho^k_S$, and the homotopies $h^k_S(t)$ do restrict to $M_{S}$ and also commute with $\deriv_{S}$ by construction).
\end{remm}

\subsubsection*{The retraction}
\red{
For $S=\{i_1,\dots,i_q\}\subseteq\ul{n}$ recall from Notation~\ref{notat:M-0S} and Figure~\ref{fig:M-0S} the spheres $\S_{i_pi_{p+1}}\coloneqq\partial\ball_{i_pi_{p+1}}\subseteq M$ of dimension $d-1$, and the wedge sum $M\vee\S_S\coloneqq M\vee\S_{i_1i_2}\vee\dots\vee\S_{i_qn+1}$. 
}
Let $M\vee\S_{\bull}$ be the $n$-cube with spaces $M\vee\S_{S}$ and maps
  \begin{equation}\label{eq-def:col}
  \coll^k_S\colon M\vee\S_{Sk}\ra M\vee\S_S
  \end{equation}
given by the identity on all wedge factors except on the sphere labelled by $ki_{p+1}$ which is collapsed onto the wedge point. Here $i_p<k<i_{p+1}$ are the closest neighbouring indices.

\begin{lemma}\label{lem:retr}
  For $n\geq 1$ there is a homotopy equivalence $\retr_{\bull}\colon\big(M_{\bull},\lambda\big) \to \big(M\vee\S_{\bull},\coll\big)$.
\end{lemma}
\begin{proof}
    Recall \red{from \eqref{eq-def:M-S} 
    }
    that $M_S\coloneqq M\sm\ball_S$ for $\ball_S\coloneqq\ball_{-\infty0}\sqcup\ball_{i_1i_2}\sqcup\dots\sqcup \ball_{i_qn+1}\sqcup\ball_{n+1\infty}$. Take a thin enough neighbourhood $V\coloneqq\U|_{[L_1,1]}\times\D^{d-1}_\epsilon\subseteq M$ of $\U$ containing all the balls $\ball_{ij}$. Then there is a diffeomorphism
    \[
    V\sm\ball_S\cong I^d\sm\ball_S,
    \]
    \red{The latter space is the $d$-cube $I^d=[0,1]^d$ minus $|S|$ balls and it 
    }
    clearly retracts onto $\S_S$: first project vertically (in the last coordinate of the ) onto the `collection of beads' on $\U$ (see Figure~\ref{fig:M-i}) and then contract $\U_{\wh{S}}$ and some arcs on the spheres to get one wedge point at $L_0$.
    
    Observe that $M=M'\cup V$ where $M'\coloneqq(M\sm V)\cup(\partial M\cap\partial V)$ is diffeomorphic to $M$. Thus, we can define a retraction $\retr_S\colon M_S=M'\cup(V\sm\ball_S)\to M'\vee\S_S$ by applying the above retraction on $V\sm\ball_S$ while also gradually contracting $\partial V\cap \partial M'$ onto the point $L_0\in\partial M'$.
    
    Finally, \red{with our description of the retractions,
    }
    it is not hard to see that they can be chosen so that $\retr_S\circ\lambda^k_S=\coll^k_S\circ\retr_{Sk}$, meaning that $\retr_{\bull}$ is indeed a map of cubes.
\end{proof}
For another proof of the lemma see~\cite[Cor.~5.3]{GWII}. 
The assumption $\partial M\neq\emptyset$ is essential as this does not hold for closed manifolds. 

Let us finish this section by an example for $n=1$.

\begin{example}\label{ex:n=2-part3}
    \red{
    In Examples~~\ref{ex:n=2-part1} and~\ref{ex:n=2-part2} we saw that $\pF_2(M) \simeq \hofib(r^1_\emptyset\colon\FF^2_\emptyset\to \FF^2_{1})$ for $\FF^2_\emptyset\simeq\Omega\S M$ and $\FF^2_1\simeq\Omega\S(M\sm\ball_{12})$.
    }
    Now by Theorem~\ref{thm:delooping-layer}, we have
    \begin{equation*}\begin{tikzcd}
      \pF_2(M) \simeq \hofib\big(r^1_\emptyset\colon\FF^2_\emptyset\to \FF^2_{1}\big) \arrow{r}{\chi}[swap]{\sim} & \Omega\hofib\big(l^1_\emptyset\colon \FF^2_{1}\to \FF^2_\emptyset\big).
    \end{tikzcd}
    \end{equation*}
    The map $l^1_\emptyset$ corresponds to $\lambda^1_\emptyset\colon M_{1}=M\sm\ball_{12}\to M_{\emptyset}\cong M$ which adds back the $d$-dimensional ball  $\ball_{12}\cong\ball^d$ and rescales using the map $\drag$, see the proof of Lemma~\ref{lem:defn-of-lambda}. Next, by Theorem~\ref{thm:final-delooping} the derivative and by Lemma~\ref{lem:retr} the retraction induce equivalences
    \[
    \begin{tikzcd}[column sep=small]
      \arrow{r}{\deriv}[swap]{\sim} & \Omega\hofib\Big(\Omega M_{1}\xrightarrow{\Omega\lambda^1_{\emptyset}} \Omega M_\emptyset\Big)\arrow{r}{\retr}[swap]{\sim} & \Omega\hofib\Big(\Omega(M\vee\S_{1})\xrightarrow{\Omega\coll^1_\emptyset}\Omega M\Big)\simeq\Omega^2\hofib\Big(M\vee\S_1\xrightarrow{\coll^1_\emptyset} M\Big),
    \end{tikzcd}
    \]
    where $\coll^1_\emptyset$ collapses the $(d-1)$-dimensional sphere $\S_1\cong\S^{d-1}$. In the next section we compute the homotopy type of this fibre, see Example~\ref{ex:n=2-part4}.
\end{example}


\section{The homotopy type}
\label{sec:htpy-type}
\red{
    As before, fix $n\geq 1$.
    In this section we study the homotopy type of $\pF_{n+1}(M)$ and prove Theorem~\ref{thm:W-iso}, using the homotopy equivalence $\pF_{n+1}(M)\simeq\Omega^{n+1}\tofib(M\vee\S_{\bull},\coll)$ of Corollary~\ref{cor:gather-h-e} from the previous section. In Section~\ref{subsec:proof-ThmB} we give a description of the homotopy type of $\pF_{n+1}(M)$ in terms of (loop spaces and suspensions of) products of $M$ with itself. In Section~\ref{subsec:gen-maps} we can then compute the first non-trivial homotopy group of $\pF_{n+1}(M)$, and describe its generators as maps to $\tofib\Omega(M\vee\S_{\bull})$ and as maps to $\tofib\Omega(M_{\bull})$. Then in Section~\ref{subsec:strategy} we present our strategy for the main proofs that will appear in Section~\ref{sec:main}. Finally, we give some examples which should clarify this strategy, and the unfortunately cumbersome notation.
}

\subsection{Proof of Theorem~\ref{thm:W-iso}}\label{subsec:proof-ThmB}

We now express the homotopy type of $\pF_{n+1}(M)$ in terms of suspensions of the space of loops on the iterated product of $M$ with itself. Such computations go back to~\cite[Section~5]{GWII} for the case when $M$ has the homotopy type of a suspension, and also~\cite{BCKS} for $M=I^3$. However, as mentioned in the introduction, even then, those results do not suffice for our purposes as we need \emph{a geometric interpretation} of the homotopy classes (see Section~\ref{subsec:gen-maps}): in Section~\ref{sec:main} it will be crucial to determine the class in $\pi_0\pF_{n+1}(M)$ of a geometrically given point $x\in\pF_{n+1}(M)$, for $d=3$.

\subsubsection*{Some preliminaries}
\red{We will use some 
}
classical results which we now recall, referring the reader to Appendix~\ref{app:samelson} for more details. \red{
In what follows $M,X,A$ are well-based spaces.
}

Let $\iota_X\colon X\hra M\vee X$ be the natural inclusion and $\eta_A\colon A\to \Omega\Sigma A$ the unit of the loop-suspension adjunction, taking $a\in A$ to the loop $\theta\mapsto\theta\wedge a\coloneqq[(\theta,a)]$. See Figure~\ref{fig:foliated-sphere}. Consider the composite
\begin{equation}\label{eq:basic-x}
\begin{tikzcd}
    x_{A}\colon A\arrow{r}{\eta_A} & \Omega\Sigma A\arrow[hook]{rr}{\Omega\iota_{\Sigma A}} && \Omega(M\vee\Sigma A)
\end{tikzcd}
\end{equation}
and the map $\Omega\iota_M\colon \Omega M\hra \Omega(M\vee \Sigma A)$.
We form their Samelson product (defined in Appendix~\ref{app:samelson}) $[x_A,\Omega\iota_M]\colon A\wedge\Omega M\to \Omega(M\vee\Sigma A)$, and the sum
\begin{equation}\label{eq-def:mu}
  x_A\vee[x_A,\Omega\iota_M]\colon\; 
  \red{A\wedge(\Omega M)_+= 
  }
  \;A\vee(A\wedge\Omega M)\to \Omega(M\vee\Sigma A),
\end{equation}
\red{where $(\Omega M)_+$ denotes the space $\Omega M$ with a disjoint basepoint added, and we use that the distributivity of the smash product gives a homeomorphism
\begin{align}\label{eq:smash+}
  A\wedge(\Omega M)_+
  &\cong A\wedge (\S^0\vee\Omega M)
  \cong A\vee A\wedge\Omega M
  \cong \Sigma(A\vee(A\wedge\Omega M).
\end{align}
}
\begin{figure}[!htbp]
    \centering
    \includegraphics{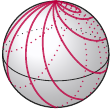}
    \caption[The unit of the loop-suspension adjunction.]{Several values $\eta_{\S^1}(t)$ of the canonical map $\eta_{\S^1}\colon\S^1\to\Omega\S^2$.}
    \label{fig:foliated-sphere}
\end{figure}
The map~\eqref{eq-def:mu} is a map to a loop space, so has a unique multiplicative extension
\[
    \mu_{M,\Sigma A}=\wt{x_A\vee[x_A,\Omega\iota_M]}
\]
to $\Omega\Sigma\big(A\vee(A\wedge\Omega M)\big)$. It maps the loop $\theta\mapsto t_\theta\wedge y_\theta$ to $\theta\to \big(x_A\vee[x_A,\Omega\iota_M]\big)(y_\theta)_{t_\theta}$, see \eqref{eq:mult-ext}.
\begin{lemma}[\cite{Gray,Spencer}]\label{lem:gray-spencer}
    For well-based spaces $M$ and $A$ there is a fibration sequence
\begin{equation}\label{eq:gray-cor}
    \begin{tikzcd}
    \Omega\Sigma\big( A\wedge(\Omega M)_+\big)\arrow{rr}{\mu_{M,\Sigma A}} && \Omega\big(M\vee\Sigma A\big) \arrow[shift left]{rr}{\Omega\coll_{\Sigma A}} && \Omega M \;.\arrow[shift left,dashed]{ll}{\Omega\iota_M}
    \end{tikzcd}
\end{equation}
    Moreover, this fibration of $H$-spaces has a section $\Omega\iota_M$, so it is trivial, \red{i.e.\ there is a homotopy equivalence $\Omega\big(M\vee\Sigma A\big)
\simeq \Omega\Sigma\big(A\wedge(\Omega M)_+\big)\times\Omega M$.
    }
\end{lemma}
\red{
To prove this one shows for any $X$ that $\hofib(\coll_X\colon M\vee X\to M)\simeq X\wedge(\Omega M)_+$, and when $X=\Sigma A$ one uses the associativity of the smash product: $\Sigma A\wedge(\Omega M)_+\simeq\S^1\wedge A\wedge(\Omega M_+)\simeq\Sigma(A\wedge(\Omega M_+))$.
}

\subsubsection*{Computation of the homotopy type}
\red{Fix $S\subseteq\ul{n}$ and recall the wedge of spheres $\S_S$ from \eqref{eq-def:S-S}. Putting $A\coloneqq\bigvee_{i\in S}\S_i^{d-2}$ in Lemma~\ref{lem:gray-spencer}, so that $\S_S=\Sigma A$, 
}
we have
\[
    \Omega(M\vee\S_S)\simeq (\Omega\Sigma Z_S)\times\Omega M
\]
for
\begin{equation*}
    Z_S\coloneqq\big(\bigvee_{i\in S}\S_i^{d-2}\big)\wedge(\Omega M)_+\;\cong\;\bigvee_{i\in S}\big(\S_{i}^{d-2}\wedge\big(\Omega M)_+\big)\cong\bigvee_{i\in S}\Sigma^{d-2}_{i}\big(\Omega M)_+\simeq \bigvee_{i\in S}Z_i.
\end{equation*}
In $Z_i=Z_{\{i\}}\coloneqq\Sigma_i^{d-2}(\Omega M)_+$ \red{the index for the suspension means that it corresponds to smashing with the sphere $\S^{d-2}$ that has index $i$.
}

Now, since $\Omega\Sigma Z_S$ is a loop space on a wedge of suspensions, the Hilton--Milnor Theorem~\ref{thm:hm} (a generalisation of Lemma~\ref{lem:gray-spencer}) applies: there is a weak equivalence
\begin{equation}\label{eq:hm-cor}
    \begin{tikzcd}
    {hm_S}\colon\:\sideset{}{^{\circ}}\prod_{w\in \B(S)} \Omega\Sigma w(Z_i) \arrow[]{r}{\sim} & \Omega\Sigma Z_S
    \end{tikzcd}
\end{equation}
where $\B(S)$ is a \emph{Hall basis} (see Remark~\ref{rem:hall}) for the free Lie algebra $\L(x^i:i\in S)$ and the space $w(Z_i)$ is the iterated smash product of spaces $Z_i=\Sigma^{d-2}_{i}\big(\Omega M)_+$. Using the associativity again, and the identity $X_+\wedge Y_+=(X\times Y)_+$, for a Lie word $w\in\B(S)$ of length $l_w$ we have
\[
  w\big(Z_i\big) \cong \Sigma^{l_w(d-2)}\big((\Omega M)_+\big)^{\wedge l_w}\cong \Sigma^{l_w(d-2)}\big((\Omega M)^{\times l_w}\big)_+\cong \Sigma^{l_w(d-2)}\big(\Omega M^{\times l_w}\big)_+.
\]
Moreover, the Hilton--Milnor map $hm_S$ in \eqref{eq:hm-cor} is analogous to $\mu_{M,\Sigma A}$ (see \eqref{eq:hm-map}): it is the pointwise product of the multiplicative extensions $\wt{w}(z_i)$ of the Samelson products according to $w\in\B(S)$ of the maps $z_i\coloneqq\Omega\iota\circ\eta\colon Z_i\to\Omega\Sigma Z_i\to \Omega\Sigma Z_S$ similarly to the notation \eqref{eq:basic-x}.
\begin{theorem}\label{thm:thmB-1}
  For each $n\geq 1$ there is a weak equivalence
  \begin{equation}\label{eq:tofib-hm}
    \begin{tikzcd}
      \sideset{}{^{\circ}}\prod_{w\in \N\B(\ul{n})}\Omega\Sigma^{1+l_w(d-2)}\big(\Omega M^{\times l_w}\big)_+
      \arrow[]{rr}{\mu\circ hm}[swap]{\sim} && \tofib\big(\Omega(M\vee\S_{\bull}),\Omega\coll\big),
    \end{tikzcd}
  \end{equation}
  where $\N\B(\ul{n})\subseteq\B(\ul{n})$ consists of those words in which every $x^i$ for $i\in\ul{n}$ appears at least once. 
\end{theorem}
\begin{proof}
By the naturality of equivalences $\Omega(M\vee\S_S)\simeq (\Omega\Sigma Z_S)\times\Omega M$ from \eqref{eq:gray-cor} and $hm_S$ from \eqref{eq:hm-cor} there is an equivalence of contravariant $n$-cubes
\[\begin{tikzcd}
  \Big(\Omega M\times \sideset{}{^{\circ}}\prod_{w\in  \B(S)}\Omega\Sigma^{1+l_w(d-2)}\big(\Omega M^{\times l_w}\big)_+\,,
  \;\mathsf{proj}^k_S\Big)\arrow{rrr}{\Omega\iota_M\times\mu\circ hm} &&& \Big(\Omega\big(M\vee\S_{S}\big),\;\Omega\coll^k_S\Big).
\end{tikzcd}
\]
We will show that the total homotopy fibre of the first cube is the desired product over $\N\B(\ul{n})$. The map $\mathsf{proj}^k_S$ for $k\in S$ is a projection onto the factors corresponding to those words $w\in\B(S)$ which also belong to $\B(S\sm k)$. These are precisely the words in which $x^k$ does not appear.

For $T\subseteq S$ let $A_T$ be the product of factors $w\in\B(S)$ in which for each $i\in T$ the letters $x^i$ appears at least once. Now, one clearly has $\hofib(\mathsf{proj}^1_{01}\colon A_0\times A_{01}\to A_0)\simeq A_{01}$ and more generally:
\[
  \tofib_{S\subseteq\ul{n}}\Big(A_0\times\sideset{}{^{\circ}}\prod_{T\subseteq S} A_T, \mathsf{proj}\Big)\simeq A_{\ul{n}}.
\]
This follows by induction using the iterative description of total homotopy fibres \red{from Lemma~\ref{lem:iterated};     see~\cite[Ex.~5.5.5]{MV} for a proof.
}
\end{proof}

\subsubsection*{The proof}
\red{We can now prove Theorem~\ref{thm:W-iso}, which stated that for $n\geq 1$ there is a weak equivalence $\pF_{n+1}(M)\simeq\Omega^n\textstyle\prod^{\circ}_{w\in \N\B(\ul{n})}\Omega\Sigma^{1 + l_w(d-2)}(\Omega M^{\times l_w})_+$, that this space is $(n(d-3)-1)$-connected, and that there are explicit isomorphisms
}
\[\begin{tikzcd}[column sep=large]
    \Lie_{\pi_1M}(n)\arrow[]{r}{W}[swap]{\cong} & \pi_{n(d-2)}\tofib\big(\Omega(M\vee\S_{\bull}),\Omega\coll\big)\arrow[]{rr}{(\retr\circ\deriv\circ\chi)_*^{-1}}[swap]{\cong} && \pi_{n(d-3)}\pF_{n+1}(M).
\end{tikzcd}
\]
\red{
Since Corollary~\ref{cor:gather-h-e} gave the equivalence $\retr\circ\deriv\circ\chi$, and Theorem~\ref{thm:thmB-1} the equivalence $\mu\circ hm$, it remains to prove the following. In the next section we will replace $(\mu\circ hm)_*$ by a  slightly different, more convenient isomorphism~$W$.
}

\begin{prop}\label{prop:thmB-2}
  For each $n\geq 1$ the space $\tofib\Omega\big(M\vee\S_{\bull}\big)$ is $(n(d-2)-1)$-connected and the first non-trivial homotopy group admits an isomorphism
  \[\begin{tikzcd}[column sep=large]
    (\mu\circ hm)_*\colon\Lie_{\pi_1M}(n)\rar{\cong} & \pi_{n(d-2)}\tofib\Omega\big(M\vee\S_{\bull}\big).
\end{tikzcd}
  \]
\end{prop}
\begin{proof}
    Recall from Section~\ref{subsec-prelim:trees} that the group $\Lie_{\pi_1(M)}(n)\cong\Z[(\pi_1M)^{n}]^{(n-1)!}$ is generated by decorated trees $\Gamma^{g_{\ul{n}}}\in\Tree_{\pi_1M}(n)$ consisting of a tree $\Gamma$ together with decorations $g_i\in\pi_1M$, $i\in\ul{n}$.
    
    Since $\Sigma^kX$ is $(k-1)$-connected for any $X$, the group $\pi_{k}\big(\Sigma^{1+l_w(d-2)}X\big)$ is trivial for $k<n(d-2)$, if $l_w\geq n$. Hence, by Theorem~\ref{thm:thmB-1} we immediately have that $\tofib(M\vee\S_{\bull})$ is $(n(d-2)-1)$-connected and the first non-trivial homotopy group is
    \[
      \pi_{n(d-2)}\tofib\Omega\big(M\vee\S_{\bull}\big)\cong
      \bigoplus_{\substack{w\in \N\B(\ul{n}),\\ l_w=n}}\pi_{n(d-2)}\Omega\Sigma^{1+n(d-2)}\big(\Omega M^{\times n}\big)_+.
    \]
    Note that this has precisely $(n-1)!$ summands. For $k\coloneqq1+n(d-2)\geq2$ we have
    \begin{equation}\label{eq:hur-homology}
      \pi_k\big(\Sigma^k\big(\Omega M^{\times n}\big)_+\big)\cong H_k\big(\Sigma^k\big(\Omega M^{\times n}\big)_+\big)\cong
      \wt{H}_0\big(\big(\Omega M^{\times n}\big)_+\big)=H_0\big(\Omega M^{\times n}\big)=\Z\big[(\pi_1M)^{n}\big],
    \end{equation}
    by the Hurewicz theorem and a basic homology computation, finishing the proof.
\end{proof}

\subsection{Explicit generators}\label{subsec:gen-maps}

\subsubsection*{Description of $W$}
In this section we give a set of generators of $\pi_{n(d-2)}\tofib\Omega(M\vee\S_{\bull})$, 
\red{i.e.\ we write down an isomorphism $W$ replacing $(\mu\circ hm)_*$ from Proposition~\ref{prop:thmB-2}, and proving Lemma~\ref{lem:W-iso'}.
}

Firstly, recall from \eqref{eq:htpy-gps-tofib} that the map forgetting homotopies induces an injection
 \[
   \forg_*\colon\pi_{n(d-2)}\tofib\Omega\big(M\vee\S_{\bull}\big)\hra\pi_{n(d-2)}\Omega\big(M\vee\S_{\ul{n}}\big).
 \]
\red{
The desired generators are the homotopy classes of the {generating maps of the group on the right} canonically extended (by null-homotoping its image in all other vertices of the cube) to the total homotopy fibre on the left. We next determine those generating maps $\S^{n(d-2)}\to\Omega(M\vee\S_{\ul{n}})$.

Secondly, 
}
observe that the generator of $\pi_k(\Sigma^k(\Omega M^{\times n})_+)$ corresponding to $(g_1,\dots,g_n)\in(\pi_1M)^n$ under \eqref{eq:hur-homology} is represented by the $k$-fold suspension $\Sigma^k\prod\gamma_i$ of the based map $\prod\gamma_i\colon\S^0\to(\Omega M^{\times n})_+$ which picks out the loop $\gamma_1\times\dots\times\gamma_n\in\Omega M^{\times n}$, where $g_i=[\gamma_i]\in\pi_1M$.

Unravelling the definitions, the composite
\[\begin{tikzcd}
     \sideset{}{^{\circ}}\prod_{w\in \B(\ul{n})}\Omega\Sigma w\big(Z_i\big) \arrow[]{rr}{hm_{\ul{n}}}[swap]{\sim} &&
     \Omega\Sigma Z_{\ul{n}}
     \arrow[]{rr}{\mu_{M,\S_{\ul{n}}}} &&
     \Omega\big(M\vee\S_{\ul{n}}\big)
   \end{tikzcd}
\]
is the pointwise product of the multiplicative extensions of the Samelson products according to $w\in\B(S)$ of the maps $\mu_{M,\S_i}\circ z_i$ which can be equivalently written as the composites
\[\begin{tikzcd}
       Z_i \arrow{rr}{\eta_{Z_i}} &&
       \Omega\Sigma Z_i \arrow{rr}{\mu_{M,\S_i}} &&
       \Omega\big(M\vee\S_S\big).
   \end{tikzcd}
\]
But by definition of $\mu$ this is equal to $x_i\vee[x_i,\Omega\iota_M]$ from \eqref{eq-def:mu}. Recall that the canonical map:
\begin{equation}\label{eq:new-xi}
  x_i\colon\S^{d-2}_i\to\Omega\big(M\vee\S_S\big)
\end{equation}
was defined in \eqref{eq:basic-x} and the Samelson product $[x_i,\Omega\iota_M]\colon \S^{d-2}_i\wedge\Omega M\to\Omega(M\vee\S_S)$ in Appendix~\ref{app:samelson}, and that we take their wedge sum since $Z_i\coloneqq\Sigma^{d-2}_i(\Omega M)_+\cong\S^{d-2}_i\vee(\S^{d-2}_i\wedge\Omega M)$, see~\eqref{eq:smash+}. In more detail, the value of $x_i$
at $\vec{t}\in\S^{d-2}$ is the loop $\eta_{\S_i}(\vec{t})\subseteq\S_i\subseteq M\vee\S_S$ (as in Figure~\ref{fig:foliated-sphere} for $d=3$), while $[x_i,\Omega\iota_M]$ sends $\vec{t}\wedge\gamma\in\S^{d-2}_{i}\wedge\Omega M$ to the commutator of the loops $x_i(\vec{t})$ and $\iota_M\gamma\colon\S^1\to M\hra M\vee\S_S$ which we also simply denote by $\gamma$.

Therefore, the desired generating maps are of the shape
\begin{equation}\label{eq:generators}
    \begin{tikzcd}[column sep=1.4cm]
      \S^{n(d-2)}\arrow[]{r}{\wh{q}_\gamma} &
      \Omega\Sigma w(Z_i)\arrow[]{rr}{\wt{w}(x_i\vee[x_{i},\Omega\iota_M])} && \Omega\big(M\vee\S_{\ul{n}}\big),
\end{tikzcd}
\end{equation}
where we use the adjoint $\wh{q}_\gamma$ of $\Sigma^{1+n(d-2)}(\prod\gamma_i)$ which is for $\theta\in\S^1$ and $\vec{t}_i\in\S^{d-2}_i$ given by
\begin{align*}
  \Sigma^{1+n(d-2)}(\textstyle\prod\gamma_i)\colon\:\S^{1+n(d-2)}\to\Sigma^{1+n(d-2)}\big(\Omega M)^{\times n}_+&\cong\Sigma^{1+n(d-2)}((\Omega M)_+\big)^{\wedge n}\cong\Sigma w(Z_i)\\
  \theta\wedge\bigwedge_{1\leq i\leq n}\vec{t}_i \mapsto \theta\wedge\bigwedge_{1\leq i\leq n}\vec{t}_i\wedge(\gamma_1\times\dots\times\gamma_n) &\equiv \theta\wedge\bigwedge_{1\leq i\leq n}\vec{t}_i\wedge\bigwedge_{1\leq i\leq n}\gamma_i\equiv\theta\wedge\bigwedge_w\vec{t}_i\wedge\gamma_i.
\end{align*}
Thus, we have $\wh{q}_{\gamma}(\wedge\vec{t}_i)\coloneqq\theta\mapsto \theta\wedge\bigwedge_w\vec{t}_i\wedge\gamma_i$, where the wedge is according to the permutation given by the word $w$. Now, by the definition $\wt{w}(\theta\mapsto t_\theta\wedge y_\theta)\coloneqq w(y_\theta)_{t_\theta}$ of a multiplicative extension in \eqref{eq:mult-ext}, the composite \eqref{eq:generators} is given by
\begin{align*}
  \wt{w}(x_i\vee[x_i,\Omega\iota_M])\Big(\wh{q}_{\gamma}(\wedge\vec{t}_i)\Big)
  &=\theta\mapsto w\big(x_i\vee[x_i,\Omega\iota_M]\big)\Big(\bigwedge_w\vec{t}_i\wedge\gamma_i\Big)_\theta=\,\theta\mapsto w\big(x_i(\vec{t}_i)+[x_i(\vec{t}_i),\gamma_i]\big)_\theta\;.
\end{align*}
Let us now simplify these generating maps.
For a decorated tree $\Gamma^{g_{\ul{n}}}\in\Tree_{\pi_1M}(n)$ we define in \eqref{eq:samelson-Gamma-sphere} the corresponding Samelson product
\begin{equation}\label{eq:samelson-gen-x-i}
    {\Gamma}\big(x_i^{\gamma_i}\big)\colon\:\S^{n(d-2)}\to\Omega\big(M\vee\S_{\ul{n}}\big),
  \end{equation}
where $x_i^{\gamma_i}\colon\S^{d-2}_i\to\Omega\big(M\vee\S_{\ul{n}}\big)$ is defined as the pointwise conjugate $\vec{t}\;\mapsto\; x_i^{\gamma_i}(\vec{t})\coloneqq\gamma_i\cdot x_i(\vec{t})\cdot \gamma_i^{-1}$.

The following is basically Lemma~\ref{lem:W-iso'}.
\begin{prop}\label{prop:LemB'}
  The map $W\colon\Lie_{\pi_1M}(n)\to\pi_{n(d-2)}\tofib\Omega(M\vee\S_{\bull})$
  which takes $\Gamma^{g_{\ul{n}}}$ to the canonical extension to the total fibre of the Samelson product ${\Gamma}\big(x_i^{\gamma_i}\big)$ is an isomorphism.
\end{prop}
\begin{proof}
  We define $W$ on decorated trees as the canonical extensions of \eqref{eq:samelson-gen-x-i} 
  \red{to the total homotopy fibre, 
  }
  and then linearly extend to $\Z[\Tree_{\pi_1M}(n)]$. Thanks to the graded antisymmetry and Jacobi relations for Samelson products this vanishes on the relations $AS,IHX$ -- the check is the same as in the proof of Lemma~\ref{lem:lie-tree} -- so $W$ is well defined. We now show it is surjective.

  It is enough to check that any $w(x_i+[x_i,\gamma_i])$ is in the image of $W$. Using the linearity of Samelson products (see Appendix~\ref{app:samelson}), this is equal to the sum $\sum_{\sigma\subseteq \ul{n}}w^\sigma$ of Samelson products $w^\sigma$ according to the word $w$ of the maps $[x_i,\gamma_i]$ for $i\in\sigma$, and $x_i$ for $i\notin\sigma$. Using the homotopy equivalence
  \[
    [x_i,\gamma_i]\simeq x_i-x_i^{\gamma_i}\colon\;\S^{n(d-2)}\to\Omega(M\vee\S_{n})
  \]
  from \eqref{eq:samelson-identity} and the linearity of the Samelson bracket once more, each $w^\sigma$ expands as
  \[
    w^\sigma\simeq\sum_{\sigma'\subseteq \sigma}(-1)^{|\sigma'|}w(x_i^{\gamma'_i})
  \]
  where $\gamma'_i=\gamma_i$ for $i\in\sigma'$, and $\gamma_i=1$ for $i\notin\sigma'$. Now by Lemma~\ref{lem:samelson-tree} the map $w(x_i^{\gamma'_i})$ is homotopic to $\Gamma(x_i^{\gamma'_i})$ for $\Gamma=\omega_2^{-1}(w)$ (see Lemma~\ref{lem:lie-tree}), so all generators are in the image of $W$.

  The inverse $W^{-1}$ is obtained similarly, by defining $[w(x_i^{\gamma'_i})]\mapsto\Gamma^{g'_{\ul{n}}}$, extending linearly and projecting to $\Lie_{\pi_1M}(n)$. We then immediately have $W\circ W^{-1}=\Id$ and $W^{-1}\circ W=\Id$.
\end{proof}

\subsubsection*{Description of $\retr_*^{-1}$}
At this point it is not yet clear what the generating maps $\S^{n(d-3)}\to\pF_{n+1}(M)$ are. We would need to find an explicit inverse of the isomorphism 
\[
    (\retr\circ\deriv\circ\chi)_*\colon\pi_{n(d-3)}\pF_{n+1}(M)\to\pi_{n(d-2)}\tofib\Omega(M\vee\S_{\bull}),
\]
for the equivalences $\retr$, $\deriv$ and $\chi$ from Corollary~\ref{cor:gather-h-e}.
At least for $\retr\colon\tofib\Omega M_{\bull}\to\tofib\Omega(M\vee\S_{\bull})$ \red{we shall now describe the map $\retr_*^{-1}$, i.e.\ the generating maps of $\pi_{n(d-2)}\tofib\Omega M_{\bull}$. For inverting $\deriv_*$ and $\chi_*$ see Remark~\ref{rem:inverting} below.
}

Firstly, we can pick an explicit lift
\begin{equation}\label{eq:m-i}
    m_i\colon\S^{d-2}\to\Omega M_{\ul{n}}
\end{equation}
of the map $x_i\colon\S^{d-2}\to\Omega(M\vee\S_{\ul{n}})$: namely, the $(d-2)$-parameter `swing of a lasso' around the $d$-ball $\ball_i\subseteq M$, as in Figure~\ref{fig:deg-2-m-i} for $d=3$. Indeed, using the definition of $\retr$ in Lemma~\ref{lem:retr} this family of loops covers the $(d-1)$-sphere $\S_1$ exactly once, so $\retr\circ m_i\simeq x_i$ (cf.\ Figure~\ref{fig:foliated-sphere}).
\begin{figure}[!htbp]
    \centering
    \includegraphics[width=0.9\linewidth]{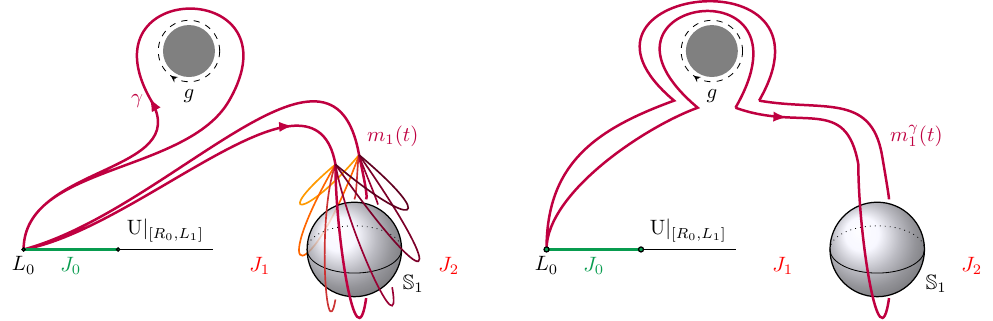}
    \caption[{The map $\protect m_1^\gamma\colon\protect \S^{d-2}\to\protect\Omega M_n$.}]{Both pictures show a part of a $3$-manifold $M$ with the basepoint $L_0$. \textit{Left}: A representative $\gamma$ of $1\neq g\in\pi_1(M)$ is depicted as a loop around the grey `hole in $M$', and the $1$-parameter family $m_1(t)\in\Omega M_1$ for several $t\in\S^1$ is depicted by a gradient of colours. \textit{Right}: One value $m_1^{\gamma}(t)=\gamma m_1(t)\gamma^{-1}\in\Omega M_1$.}
    \label{fig:deg-2-m-i}
\end{figure}

Further, $m_i^{-1}\colon\S^{d-2}\to\Omega M_{\ul{n}}$ can be obtained by reversing orientations of all loops in the family or, equivalently, by performing a twist as in Figure~\ref{fig:deg-2-twist-m-i}. Moreover, any $g_i\in\pi_1M$ can be realised by a loop $\gamma_i$ in $M$ that misses all $\ball_1,\dots,\ball_n$, so defines $\gamma_i\colon\S^0\to\Omega M_{\ul{n}}$. Thus, we can define $m_i^{\varepsilon_i\gamma_i}\colon\S^{d-2}\to\Omega M_{\ul{n}}$, the pointwise conjugate of $m_i^{\varepsilon_i}$ by $\gamma_i$, see the same figure.
\begin{figure}[!htbp]
    \centering
    \includegraphics[width=0.47\linewidth]{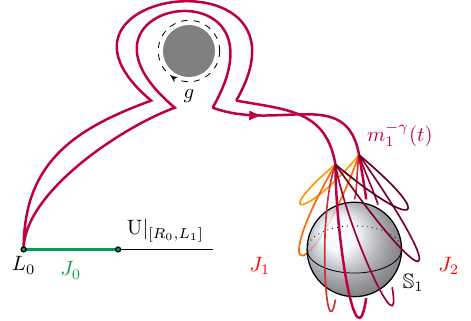}
    \caption[{The family $\protect m_1^{-\gamma}\colon\S^{d-2}\to\Omega M_{n}$.}]{The family $m_1^{-\gamma}\colon\S^1\to\Omega M_1$.}
    \label{fig:deg-2-twist-m-i}
\end{figure}

Finally, as the target is a loop space, we have their Samelson products (see Figure~\ref{fig:deg-3-m-i} for $n=2$):
\begin{equation}\label{eq:samelson-gen-m-i}
  \Gamma(m_i^{\varepsilon_i\gamma_i})\colon\S^{n(d-2)}\to\Omega M_{\ul{n}}\;.
\end{equation}
Now, since $\forg_*\colon\pi_{n(d-2)}\tofib\Omega M_{\bull}\to\pi_{n(d-2)}\Omega M_{\ul{n}}$ is injective in this setting as well, the generators of the source group are represented by the extensions of the maps \eqref{eq:samelson-gen-m-i} to the total fibre using the canonical null-homotopies of $m_i$ -- by `pulling up' through the ball $\ball_i$.
\begin{prop}\label{prop:retr-inverted}
    The isomorphism $\retr_*^{-1}\colon\pi_{n(d-2)}\tofib\Omega(M\vee\S_{\bull})\to\pi_{n(d-2)}\tofib\Omega M_{\bull}$ maps $\Gamma(x_i^{\gamma_i})$ from \eqref{eq:samelson-gen-x-i} to $\Gamma(m_i^{\gamma_i})$ from \eqref{eq:samelson-gen-m-i}. 
    Thus, $\pi_{n(d-2)}\tofib\Omega M_{\bull}$ is generated by $\Gamma(m_i^{\gamma_i})$, where $\Gamma^{g_{\ul{n}}}$ vary over generators of $\Lie_{\pi_1M}(n)$, and $\gamma_i$ is a representative of $g_i$, $i\in\ul{n}$.
\end{prop}
\begin{figure}[!htbp]
    \centering
    \includegraphics[width=0.75\linewidth]{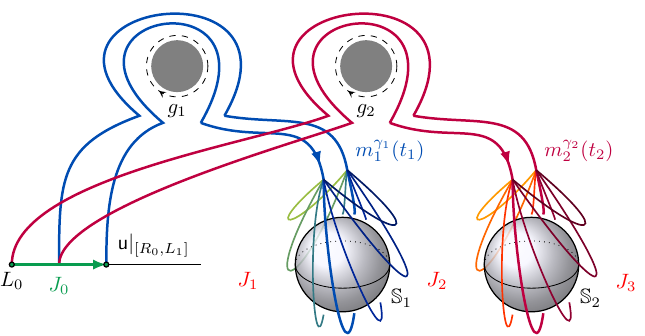}
    \caption[{The map $\protect[m_1^{\gamma_1},m_2^{\gamma_2}]\colon\S^{2(d-2)}\to\protect\Omega M_n$.}]{ The value of $\tree(m_1^{\gamma_1},m_2^{\gamma_2})=[m_1^{\gamma_1},m_2^{\gamma_2}]\colon\S^2\to\Omega M_{12}$ at $(t_1,t_2)\in \S^1\wedge \S^1=\S^2$ is the commutator of the depicted loops $m_1^{\gamma_1}(t_1)$ and $m_2^{\gamma_2}(t_2)$.}
    \label{fig:deg-3-m-i}
\end{figure}

\subsection{The strategy and examples}\label{subsec:strategy}
\subsubsection*{The strategy}
Suppose that we want to check whether the homotopy class of a given map $f\colon\S^{n(d-3)}\to\pF_{n+1}(M)$ corresponds to some class $\upvarepsilon\Gamma^{g_{\ul{n}}}\in\Lie_{\pi_1M}(n)$ under the isomorphism
\[
  W^{-1}(\retr\deriv\chi)_*\colon\pi_{n(d-3)}\pF_{n+1}(M)\to\Lie_{\pi_1M}(n)
\]
from Theorem~\ref{thm:W-iso}, with $\upvarepsilon\in\{\pm1\}$. By the \red{previous section, 
}
this is equivalent to considering
\[
  \deriv\chi(f)\colon\,\S^{n(d-3)}\to\Omega^{n}\tofib\Omega M_{\bull}
\]
and checking that the homotopy class of its adjoint $\S^{n(d-2)}\to\tofib\Omega M_{\bull}$ is the class of the canonical extension of $\upvarepsilon \Gamma(m_i^{g_i})$, see Propositions~\ref{prop:LemB'} and~\ref{prop:retr-inverted}. Actually, we saw that is instead enough to simply check this for the initial vertex (recall that $\forg\deriv\chi(f)\;=\;\big(\deriv\chi f\big)^{\ul{n}} \;=\;\deriv_{\ul{n}}\left(\chi f\right)^{\ul{n}}$):
\[
  \deriv_{\ul{n}}\left(\chi f\right)^{\ul{n}}\;\simeq\; \upvarepsilon \Gamma(m_i^{\gamma_i})\,\colon\quad \S^{n(d-2)}\to\Omega M_{\ul{n}}\;.
\]
We now claim $\upvarepsilon \Gamma(m_i^{\gamma_i})\simeq \Gamma(m_i^{\varepsilon_i\gamma_i})$, for any tuple $\varepsilon_i\in\{\pm1\}$ such that $\bigsqcap_{i=1}^{n}\varepsilon_i=\upvarepsilon$. Namely, define
\[
  m_i^{\varepsilon_i\gamma_i}(\vec{t})\coloneqq\gamma_i\cdot m_i(\vec{t})^{\varepsilon_i}\cdot\gamma_i^{-1},
\]
where $m_i(\vec{t})^{-1}$ is the inverse of the loop $m_i(\vec{t})\in\Omega M_{\ul{n}}$. Then $m_i^{\varepsilon_i\gamma_i}\simeq \varepsilon_i m_i^{\gamma_i}$ are homotopic maps by the Eckmann--Hilton argument, and $\Gamma(m_i^{\varepsilon_i\gamma_i})\simeq\upvarepsilon \Gamma(m_i^{\gamma_i})$ by the bilinearity of Samelson products.

To prove Theorem~\ref{thm:main-thm} in Section~\ref{sec:main} we will show that for $d=3$ and $f_\TG=\emap_{n+1}(\psi(\TG))\in\pF_{n+1}(M)$, the point coming from a thickened grope $\TG\colon\ball_\Gamma\to M$ on $\U$, we have
\begin{equation}\label{eq:final-goal}
  \deriv_{\ul{n}}\left(\chi f_{\TG}\right)^{\ul{n}}\;\simeq\; \Gamma(m_i^{\varepsilon_i\gamma_i})\,\colon\quad
  \S^{n(d-2)}\to\Omega M_{\ul{n}}
\end{equation}
where $(\varepsilon_i, \gamma_i)_{\ul{n}}$ is the signed decoration of $\TG$.

The proof will be based on the fact that both $\psi(\TG)$ and Samelson products are constructed inductively. For the former see Section~\ref{sec:gropes-punc-knots} and for the latter see Lemma~\ref{lem:samelson-inductive}.

Furthermore, for the proof of Theorem~\ref{thm:main-extended} we will use that $\sum_{l=1}^N\upvarepsilon_l\Gamma_l^{g^l_{\ul{n}}}$ is represented by the following pointwise product -- again by the Eckmann--Hilton argument:
\[\prod_{1\leq i\leq N}\Gamma_l(m_i^{\varepsilon_i^l\gamma_i^l})\colon\:
  \S^{n(d-2)}\to\Omega M_{\ul{n}},\quad \vec{t}\mapsto \Gamma_1(m_i^{\varepsilon_i^1\gamma_i^1})(\vec{t})\cdots \Gamma_N(m_i^{\varepsilon_i^N\gamma_i^N})(\vec{t}).
\]
We note that the same strategy should work for appropriate notion of gropes for any $d\geq4$, as mentioned in Remark~\ref{rem:other-d}. Such a generalisation of gropes can be inspired by the following observations.

\begin{remm}\label{rem:inverting}
    Let us try to directly invert the map $\deriv_{\ul{n}}\colon\FF^{n+1}_{\ul{n}}\to\Omega M_{\ul{n}}$, which closes up $J_0\hra M_{\ul{n}}$ into a loop based at $L_0$ (the tangent vectors are forgotten, see Remark~\ref{rem:straight}). \red{More precisely, let us find maps $\Gamma(\varphi_i^{\gamma_i})\colon\S^{n(d-2)}\to\FF^{n+1}_{\ul{n}}$ such that $\deriv_{\ul{n}}\circ\Gamma(\varphi_i^{\gamma_i})$ is homotopic to $\Gamma(m_i^{\gamma_i})\colon\S^{n(d-2)}\to\Omega M_{\ul{n}}$.
    }
\begin{enumerate}
    \item Firstly, there is an obvious lift $\varphi_i\colon\S^{d-2}\to\FF^{n+1}_{\ul{n}}$ of $m_i\colon\S^{d-2}\to\Omega M_{\ul{n}}$ by ensuring that each $m_i(\vec{t})\in\Omega M_{\ul{n}}$ is embedded and changing it to an arc from $L_0$ to $R_0$. See the left hand side of Figure~\ref{fig:deg-2-phi-i}. We can also ensure that different $\varphi_i$ for $i\in\ul{n}$ are mutually disjoint.

    \item Moreover, $\gamma_i$ can be chosen to be embedded in $M_{\ul{n}}$, and we may define $\varphi_i^{\gamma_i}\colon\S^{d-2}\to\FF^{n+1}_{\ul{n}}$ as an \emph{embedded conjugate}, by slightly pushing copies of $\gamma_i$ off of each other. See the right hand side of Figure~\ref{fig:deg-2-phi-i} for $d=3$.
\begin{figure}[!htbp]
    \centering
    \includegraphics[width=0.9\linewidth]{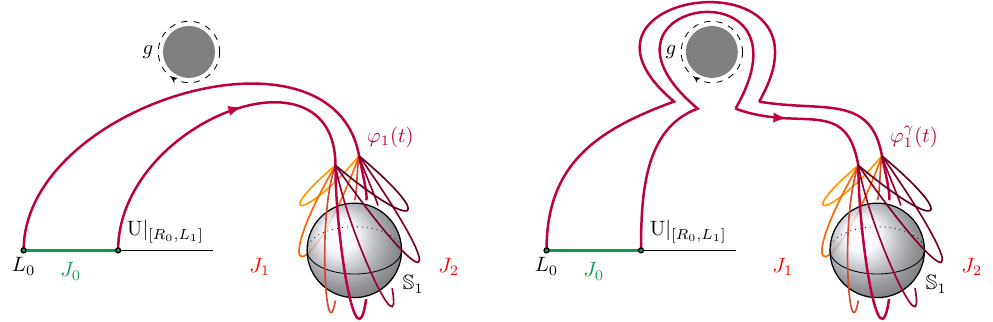}
    \caption[{The map $\protect\varphi_1^{\gamma}\colon\S^{d-2}\to\protect\FF^{n+1}_n$.}]{The $1$-parameter families of arcs $\varphi_1(t)\in\FF^2_1$ and $\varphi_1^{\gamma}(t)\in\FF^2_1$ for several values of $t\in \S^1$.}
    \label{fig:deg-2-phi-i}
\end{figure}

    \item Next, we would need to define Samelson products $\Gamma(\varphi_i^{\gamma_i})$ using some \emph{embedded analogue of commutators}. However, this is not immediate: $\FF^{n+1}_{\ul{n}}$ is not an $H$-space in an obvious way, as concatenation of arcs $J_0\hra M_{\ul{n}}$ might result in a non-embedded arc. 
    
    \item \red{Assuming (3) has been done, 
    }
    the generators of $\pi_{n(d-2)}\tofib(\FF^{n+1}_{\bull},l)$ would be canonical extensions of $\Gamma(\varphi_i^{\gamma_i})$, again by \eqref{eq:htpy-gps-tofib}.
    
    \item \red{Assuming (3) has been done, 
    }
    the map $\chi^{-1}\colon\Omega^{n}\tofib(\FF^{n+1}_{\bull},l)\to\tofib(\FF^{n+1}_{\bull},r)$ would be easy: it is given as the composition of the map $\forg$ which forgets all null-homotopies, and the inclusion $\Omega^{n}\FF^{n+1}_{\ul{n}}\hra\tofib(\FF^{n+1}_{\bull},r)$.
\end{enumerate}
    We do not, however, pursue defining such embedded commutators \red{as in (3) 
    }
    directly, as we will not need them. Namely, for $d=3$ we will in Section~\ref{sec:gropes-punc-knots} directly construct points $\psi(\TG)\in\H_n(M)$ using \emph{gropes} -- which can indeed be seen as embedded commutators, see Remark~\ref{rem:emb-comm-borr} -- and then prove that $\chi\emap_{n+1}(\psi(\TG))\colon\S^{n(d-2)}\to\tofib(\FF^{n+1}_{\bull},l)$ are generators, as explained in the strategy above.
\end{remm}

\subsubsection*{Examples}
\red{We finish this section by 
}
illustrating our computations so far and the strategy of proof of Theorem~\ref{thm:main-thm} on several examples.

\begin{example}[$n=1$]\label{ex:n=2-part4}
    \red{In Examples~\ref{ex:n=2-part1},~\ref{ex:n=2-part2} and~\ref{ex:n=2-part3} we saw that $\pF_2(M)\simeq\Omega\hofib(l^1_\emptyset)$ and $l^1_\emptyset$ corresponds to $\Omega\S\lambda^1_\emptyset$ under $\deriv_S\colon\Embp(J_0,M_{0S})\simeq\Omega\S M_S$ and $\lambda^1_\emptyset$ to $\coll^1_\emptyset\colon M\vee\S_1\to M$ under $M_S\simeq M\vee\S_S$. 
    }
    Now, $\hofib(\Omega\coll^1_\emptyset)\simeq\Omega\Sigma_1^{d-1}(\Omega M)_+\simeq\Omega(\S_1\vee(\S_1\wedge\Omega M))$ by the Grey--Spencer Lemma~\ref{lem:gray-spencer}, so we have
    \begin{align*}
        \pi_{d-3}\pF_2(M)
        &\cong\pi_{d-2}\hofib(l^1_\emptyset)
        \cong\pi_{d-2}\hofib(\Omega\lambda^1_\emptyset)
        \cong\pi_{d-2}\hofib(\Omega\coll^1_\emptyset)
        \cong\pi_{d-2}\Omega(\S^{d-1}\vee(\S^{d-1}\wedge\Omega M)))\\
        &\cong\pi_{d-1}\S^{d-1}\oplus\pi_{d-1}\left(\Sigma^{d-1}\Omega M\right)\\
        &\cong\Z\big\{x_1\big\}
        \oplus\bigoplus_{1\neq g\in\pi_1M}\Z\big\{[x_1,g]\big\}
        \xrightarrow{W^{-1}}\Lie_{\pi_1M}(1)\cong \Z[\pi_1M]
    \end{align*}
    using the isomorphism $W^{-1} x_1=\gchord{1},\;
    W^{-1}[x_1,g_1]=W^{-1} x_1 -W^{-1} x_1^{g_1} =\gchord{1}-\gchord{g_1}$.
    
    The generators for $\hofib(\Omega\coll^1_\emptyset)$ are the canonical extensions to this homotopy fibre of the maps $x_1^\gamma\colon\S^{d-2}\to\Omega(M\vee\S_1)$ (see~\eqref{eq:new-xi} and~\eqref{eq:samelson-gen-x-i}), while for $\hofib(\Omega\lambda^1_\emptyset)$ they are the canonical extensions of $m_1^\gamma\colon\S^{d-2}\to\Omega M_{12}$ (see~\eqref{eq:m-i} and~\eqref{eq:samelson-gen-m-i}), for varying $[\gamma]\in\pi_1M$. See Figure~\ref{fig:deg-2-m-i}. 
    
    Finally, the generators for $\hofib(l^1_\emptyset)$ are the canonical extensions of $\varphi_1^{\gamma}\colon\S^{d-2}\to\FF^2_1$, see Remark~\ref{rem:inverting} and Figure~\ref{fig:deg-2-phi-i}). For example, the extension of $\varphi_1^{\gamma}$ to $\hofib(l^1_\emptyset)$ uses the family of obvious null-homotopies of $l^1_\emptyset(\varphi_1^{\gamma})$ across the $d$-ball $\ball_{1}$.
\end{example}

\begin{example}[$n=2$]\label{ex:f-for-n=3}
For the punctured $4$-cube $\mc{E}^3_{\bull}$, we can draw its top subcube $\mc{E}^3_{\bull\cup 3}$:
\[\begin{tikzcd}[column sep=tiny,row sep=small]
        & \FF^3_{12}\arrow[dashed]{rr} &&
        \Emb_\partial(I\sm J_{123},M) \arrow{rr}\arrow[from=dd] &&
        \Emb_\partial(I\sm J_{0123},M)
        \\
        \FF^3_{2}\arrow[dashed]{rr} &&
        \Emb_\partial(I\sm J_{23},M)\arrow[crossing over]{rr}\arrow{ur} &&
        \Emb_\partial(I\sm J_{023},M) \arrow{ur}
        \\
        & \FF^3_{1}\arrow[dashed]{rr} &&
        \Emb_\partial(I\sm J_{13},M) \arrow{rr} &&
        \Emb_\partial(I\sm J_{013},M)  \arrow{uu}
        \\
        \FF^3_{\emptyset}\arrow[dashed]{rr} &&
        \Emb_\partial(I\sm J_{3},M) \arrow{ur}\arrow{rr}\arrow{uu} &&
        \Emb_\partial(I\sm J_{03},M) \arrow{ur} \arrow[crossing over]{uu} &
  \end{tikzcd}
\]

where the dashed arrows are fibres.

We apply Corollary~\ref{cor:gather-h-e} and Theorem~\ref{thm:thmB-1} to get homotopy equivalences
\[
\pF_3(M)=\tofib\left(
\begin{tikzcd}
    \FF^3_1\arrow{r} & \FF^3_{12}\\
    \FF^3_\emptyset\arrow{r}\arrow{u} & \FF^3_2\arrow{u}
\end{tikzcd}\right)
\xrightarrow{\chi} \Omega^2\tofib\left(\begin{tikzcd}[column sep=scriptsize]
    \FF^3_1\arrow{d} & \FF^3_{12}\arrow{l}\arrow{d}\\
    \FF^3_\emptyset & \FF^3_2\arrow{l}
\end{tikzcd}\right)
\xrightarrow{\deriv}
\Omega^2\tofib\left(\begin{tikzcd}[column sep=scriptsize]
    \Omega M_{1} \arrow{d} & \Omega M_{12} \arrow{d}\arrow{l}\\
    \Omega M & \Omega M_{2}\arrow{l}
\end{tikzcd}\right)
\]
\[
\xrightarrow{\retr}
\Omega^2\tofib\left(\begin{tikzcd}[column sep=tiny]
    \Omega(M\vee\S_1)\arrow{d} & \Omega(M\vee\S_1\vee\S_2)\arrow{d}\arrow{l}\\
    \Omega M & \Omega(M\vee\S_2)\arrow{l}
\end{tikzcd}\right)
    \simeq\Omega^2\sideset{}{^{\circ}}\prod_{w\in \N\B(\{1,2\})} \Omega\Sigma^{1+l_w(d-2)}(\Omega M^{\times l_w})_+
\]

Hence, the first non-trivial homotopy group is $\Lie_{\pi_1M}(2)$ in degree $2(d-2)$, and for the last total fibre in the first row the generating maps are the canonical extensions of the Samelson products
\begin{equation}\label{eq:ex-gen}
    \forg(\retr)^{-1}_*W\left(\gtree{g_1}{g_2}\right)=\tree(m_i^{\gamma_i}) \: = \:\left[m_1^{g_1},m_2^{g_2}\right]\:\in\pi_2\Omega M_{12}.\qedhere
\end{equation}
\end{example}

\begin{example}[Sketch of the proof of Theorem~\ref{thm:main-thm} for $n=2$]\label{ex:proof-main}
    Assume $\ut_2(\TG)=\gtree{g_1}{g_2}$ for a degree $2$ thickened grope $\TG$ in a $3$-manifold $M$, using the underling forest map from Proposition~\ref{prop:uf-map}. The construction in Section~\ref{subsec:grope-points} produces the point
    \[\emap_3\psi(\TG)=\left(
    \begin{tikzcd}[column sep=tiny,row sep=tiny]
        \Psi^\TG(-)_{J_0}^{\{1\}} & \Psi^\TG(-)_{J_0}^{\{12\}}\\
        \TG(a_0^\perp) & \Psi^\TG(-)_{J_0}^{\{2\}}
    \end{tikzcd}\right)\:\in\:\pF_3(M)
    \]
    and so we obtain a class $\forg(\deriv\chi\emap_3\psi(\TG))\in\pi_2\Omega M_{12}$. Then Theorem~\ref{thm:main-thm} asserts that this class agrees with \eqref{eq:ex-gen}. One can visualise this by comparing Figure~\ref{fig:deg-3-m-i} with Figure~\ref{fig:deg-2-grope-decorated}.
    
    In more detail, a close look at the definition of $\chi$ implies that $(\chi\emap_3\psi(\TG))^{\{12\}}$ is obtained from the square-family of loops $\Psi^\TG(-)_{J_0}^{\{12\}}$ by `reflections relative to the balls $\ball_1$ and $\ball_2$', see Proposition~\ref{prop:chi-is-glued}.
    
    In Lemma~\ref{lem:commutator} we will show that $\deriv_{12}\Psi^\TG(-)_{J_0}^{\{12\}}\colon I^2\to\FF^3_{12}\to\Omega M_{12}$ is homotopic to the commutator of certain loops corresponding to degree $1$ gropes out of which $\TG$ is built (the two caps of $\TG$). On the other hand, the Samelson product $[m_{1}^{\gamma_1},m_{2}^{\gamma_2}]\colon\S^2\to\Omega M_{12}$ is also defined inductively in terms of commutators. Hence, we will be able finish the proof by induction.
\end{example}


\section{Gropes and punctured knots}\label{sec:gropes-punc-knots}
Throughout this section $M$ is an oriented $3$-manifold with non-empty boundary.
In Section~\ref{subsec:gropes} we define grope cobordisms and their modifications, and in Section~\ref{subsec:grope-paths} we prove Theorem~\ref{thm:KST}.

\subsection{Grope cobordisms, thickened gropes and grope forests}\label{subsec:gropes}
\subsubsection*{Abstract gropes}
We saw in Definition~\ref{def:trees} the set $\Tree(S)$ of (rooted vertex-oriented uni-trivalent) trees with leaves labelled by a set $S$.
Now we recursively define certain $2$-complexes modelled on such trees, 
\red{
    using the description in terms of grafting: gluing one tree from $\Tree(S_1)$ and one from $\Tree(S_2)$ along their roots and sprouting a new root gives a tree in $\Tree(S_1\bigsqcup S_2)$.
}
\begin{figure}[!htbp]
    \centering
    \includegraphics[width=0.64\linewidth]{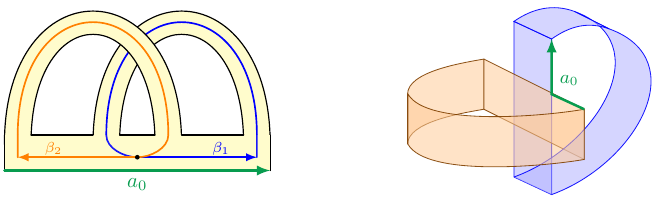}
    \caption{The model of a punctured torus as the gluing of two bands.}
    \label{fig:punctured-torus}
\end{figure}
\begin{defn}\label{def:punc-torus}
    A \textsf{punctured torus} $T$ is a genus one compact oriented surface with one boundary component $\partial T\cong \S^1$. We fix an oriented subarc $a_0\subseteq\partial T$ and 
    \red{
    two simple closed curves $\beta_{\myj}\subseteq T$, $\myj=1,2$, that freely generate the fundamental group $\pi_1T=\langle\beta_1,\beta_2\rangle$. See Figure~\ref{fig:punctured-torus}. 
    
    Moreover, we view $T$ as obtained by gluing together bands $b_1$ and $b_2$ along a square (and straightening the corners), so that $b_{\myj}\cong\beta_{\myj}\times I$ and the arc $a_0$ has one half in $\partial b_1$ and the other in $\partial b_2$. 
    We orient $\beta_1$ in the same manner as the component of $\partial b_1$ which intersects $a_0$. Thus, we have
    \[
        \partial T\simeq[\beta_1,\beta_2]=\beta_1\beta_2\beta_1^{-1}\beta_2^{-1}.
    \]
    }
\end{defn}
\begin{defn}\label{def:abstract-grope}
    An \textsf{abstract (capped) grope} $G_\Gamma$ modelled on $\Gamma\in\Tree(S)$ for a finite non-empty set $S$ is a $2$-complex with  
    \red{
        a distinguished curve $\partial G_\Gamma\cong\S^1$ called the \textsf{boundary}, which comes with a splitting into two arcs
        $\partial G_\Gamma=a_0\cup a_0^\perp$ glued along endpoints, which is defined recursively
    }
    as follows.
    \begin{itemize}[topsep=-1.5em]
    \item
        For $S=\{i\}$ the only tree is $\Gamma=\ichord{i}$ and we let $G_\Gamma$ simply be an oriented disk $G_\Gamma\coloneqq\D^2$, which we call the \textsf{$i$-th cap}; we define its boundary as the boundary of the disk, and we pick an oriented subarc~$a_0\subseteq\partial G_\Gamma$ (and let $a_0^\perp=\partial G_\Gamma\sm a_0$).
    \item
        For $|S|\geq2$ any tree $\Gamma\in\Tree(S)$ is obtained by grafting two trees of lower degrees:
        \[
        \Gamma=\grafted,\quad\quad \Gamma_{\myj}\in\Tree(S_{\myj}),\quad S_1\sqcup S_2=S.
        \]
        Thus, if abstract gropes $G_{\Gamma_1}$ and $G_{\Gamma_2}$ have been defined, we let $G_\Gamma$ be the result of attaching the two of them to a single punctured torus $T$, called the \textsf{bottom stage} of $G_\Gamma$, via orientation-preserving homeomorphisms $\partial G_{\Gamma_{\myj}}\cong\beta_{\myj}\subseteq T$, for $\myj=1,2$. Moreover, we let $\partial G_\Gamma$ be the boundary of the bottom stage $\partial G_\Gamma\coloneqq\partial T$, and $a_0\subseteq \partial G_\Gamma$ be the corresponding oriented subarc of $\partial T$.
    \end{itemize}
\end{defn}
Note that $G_\Gamma$
\red{
    can be viewed as the union of oriented surfaces, each of which is either a punctured torus $T$, which we call a \textsf{stage}, or a disk $\D^2$, called a \textsf{cap}. The caps
}
are labelled bijectively by $S$, and we say that $|S|$ is the \textsf{degree} of $G_\Gamma$. \red{See Figures~\ref{fig:deg-3-grope-abstract} and~\ref{fig:abs-grope-deg-2} for examples of abstract gropes.
}

\begin{figure}[!htbp]
    \centering
    \includegraphics[width=0.3\linewidth]{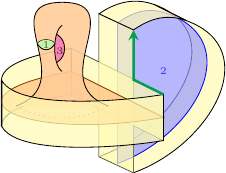}
    
    \vspace{-10pt}
    
    \caption[An abstract grope of degree $3$.]{An abstract grope of degree $3$, modelled on the tree $\begin{tikzpicture}[baseline=0.2ex,scale=0.21,every node/.style={scale=0.78,font=\bfseries}]
        \clip (-2.2,-0.21) rectangle (2.2,4.13);
        \draw[fill=black!30!white]
            (-0.15, -0.2) rectangle ++ (0.3,0.2);
        \draw[thick]
            (0,0) -- (0,1) --
            (-1,2) -- (-2,3) node[pos=1,above]{$1$}
            (-1,2) -- (0,3)    node[pos=1,above]{$3$}
                    (0,1) -- (2,3)  node[pos=1,above]{$2$};
    \end{tikzpicture}$.}
    \label{fig:deg-3-grope-abstract}
\end{figure}
Abstract gropes are \red{
    essentially 
}
combinatorial objects: there is a $1$--$1$ correspondence between them and trees. In fact, the tree $\Gamma$ on which $G_\Gamma$ was modelled can be seen as its subset using the following construction, equivalent to the one given in~\cite[Def.~16 \& Sec.~3.4]{CST}. \red{In fact, by appropriately \emph{framing} the edges of $\ut(\G)$ one obtains the clasper corresponding to $\G$, see~\cite{CT1}.
}
\begin{figure}[!htbp]
    \centering
    \includegraphics{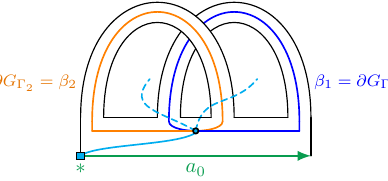}
    \caption[The underlying tree of an abstract grope.]{The underlying tree is obtained by gluing the light blue arcs for each stage.}
    \label{fig:underlying-tree}
\end{figure}
\begin{defn}\label{def:underlying-tree}
    The \textsf{underlying tree} of $G_\Gamma$ is an embedding $\Gamma\subseteq G_\Gamma$, defined as follows. The root of $\Gamma$ maps to the initial point $*$ of the arc $a_0\subseteq\partial G_\Gamma$, and each trivalent vertex of $\Gamma$ maps to the intersection point $\beta_1\cap\beta_2$ in the corresponding grope stage. The leaves of $\Gamma$ are mapped to the centres of caps. Finally, the edges are obtained at each stage as in Figure~\ref{fig:underlying-tree}. The order $(\beta_1,\beta_2)$ agrees with the corresponding vertex orientation in $\Gamma$.
\end{defn}

\subsubsection*{Grope cobordisms}
We consider particular embeddings of abstract gropes into a $3$-manifold $M$ \red{(here by an embedding we mean a homeomorphism onto the image). Recall from the beginning of Section~\ref{sec:punc-knot-model} that $\Knots(M)$ is the space of knots in $M$ and $J_0,J_1,\dots\subseteq I\coloneqq[0,1]$ are fixed subintervals.
}
\begin{defn}\label{def:grope-cob}
    Let $K\in\Knots(M)$ and $\Gamma\in\Tree(S)$ for a finite non-empty set $S$. A \textsf{(simple capped genus one) grope cobordism}\footnote{Non-simple, non-capped and higher genus gropes are also considered in the literature, but will not be needed in our discussion. However, grope forests defined below in Definition~\ref{def:forest} are related to higher genus grope cobordisms.} on~$K$ modelled on $\Gamma$ is an embedding $\G\colon G_\Gamma\to M$ into the complement of $K$ except that:
        \begin{itemize}
            \item $\G(a_0) \subseteq K(J_0)$ and the orientations of these arcs agree;
            \item for each $i\in S$, the $i$-th cap of $\G(G_\Gamma)$ intersects $K_{\wh{0}}\coloneqq K(I\sm J_0)$ transversely in exactly one point $p_i\in K_{\wh{0}}$, which is the centre of the cap and which belongs to $K(J_i)\subseteq K_{\wh{0}}$. 
        \end{itemize}
    We see $\G$ as a cobordism between $K$ and the \textsf{output knot} $\partial^\perp\G\coloneqq (K\sm\G(a_0)) \cup \G(a_0^\perp)$, smoothened at the corners and oriented compatibly with the orientation of $K$.
\end{defn}
In Figures~\ref{fig:grope-deg-1},~\ref{fig:grope-deg-2} and~\ref{fig:deg-2-grope-decorated} are depicted several examples of grope cobordisms of degree $1$ and $2$.
\begin{remm}\label{rem:emb-comm-borr}
  Note that the arc $\G(a_0^\perp)$ is oriented oppositely in $\partial^\perp\G$ than as a subset of $\G$, as usual for oriented cobordisms. The crucial observation is that $\G(a_0^\perp)$ is an `embedded commutator' of the curves $\G(\beta_1)$ and $\G(\beta_2)$, as for the Borromean link, see Figure~\ref{fig:grope-deg-2} and~\cite{Teichner-survey,Teichner-what-is}.
\end{remm}
\begin{figure}[!htbp]
    \centering
    \includegraphics{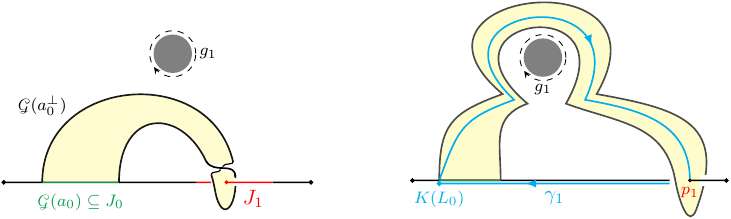}
    \caption[Two grope cobordisms of degree $1$, and a loop $\protect\gamma$ which determines the decoration.]{Two examples of a grope cobordism (shaded yellow) of degree $1$ on a knot $K$ (the horizontal line). In the first example $\G$ is contained in $I^3\subseteq M$ and  $\partial^\perp\G$ is the union of black and red arcs and is isotopic to the trefoil. The signed decoration in the first example is $-1$, and in the second $+g_1=[\gamma_1]$. }
    \label{fig:grope-deg-1}
\end{figure}
\begin{figure}[!htbp]
    \centering
    \includegraphics[width=0.6\linewidth]{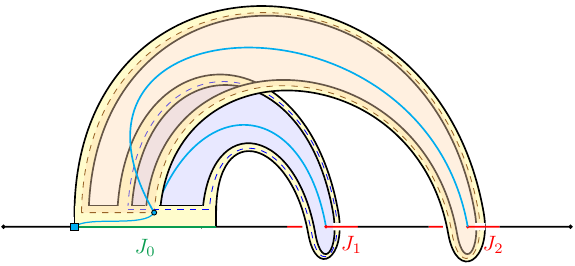}
    \vspace{-7pt}
    \caption[A grope cobordism of degree $2$ and its underlying tree.]{A grope cobordism $\G$ (shaded) in $I^3$ on the unknot (the horizontal line) modelled on $\tree$. The underlying tree is depicted in light blue. The knot $\partial^\perp\G$ is the union of $\U_{\wh{0}}$ and the long black arc $\G(a_0^\perp)$, which is the commutator of $\G(\beta_i)$, the meridian of the arc $K(J_i)$ for $i=1,2$. The signs are $\varepsilon_1=+1$, $\varepsilon_2=-1$. See also Figure~\ref{fig:trefoil}.}
    \label{fig:grope-deg-2}
\end{figure}

\subsubsection*{The underlying decorated tree}
\red{
    The underlying tree $\Gamma\in\Tree(S)$ of a grope cobordism is simply the image of $\Gamma\subseteq G_\Gamma$ from Definition~\ref{def:underlying-tree} under $\G\colon G_\Gamma\to M$. We now decorate this tree by $\pi_1M$-elements; for decorated trees see Definition~\ref{def:decorated-lie}.
}
\begin{defn}\label{def:underlying-decor-tree}
    Let $\G\colon G_\Gamma\to M$ be a grope cobordism on a knot $K\colon I\hra M$. We define a tuple $(\varepsilon_i,\gamma_i)_{i\in S}$, called the \textsf{signed decoration} of $\G$, as follows.
    
    Firstly, et $\varepsilon_i\coloneqq\sgn(p_i)\in\{\pm\}$ be the sign of the intersection of $K(J_i)$ and the $i$-the cap of $\G$. 
    Next, let $\gamma_i'\colon I\to M$ be the path from $K(L_0)$ to $p_i$ obtained as the image under $\G$ of the unique path in the tree $\Gamma\subseteq G_\Gamma$ from the root to its $i$-th leaf. Let $[p_i,K(L_0)]$ be the image of $K$ between $p_i\in K(J_i)\cap\G(G_\Gamma)$ and $K(L_0)$.
    Then we have a loop in $M$ given by $\gamma_i\coloneqq\gamma_i'\cup [p_i,K(L_0)]$.

    Lastly, the \textsf{underlying decorated tree} of $\G$ is 
    \[
        \ut(\G)=\upvarepsilon\Gamma^{g_S}\in\Z[\Tree_{\pi_1(M)}(S)],
    \]
    where $\upvarepsilon\coloneqq\bigsqcap_{i=1}^n\varepsilon_i$ and $g_S(\G)\in(\pi_1M)^{S}$ is the tuple of classes $g_i=[\gamma_i]\in\pi_1M$.
\end{defn}
In other words, $\gamma_i$ is obtained by gluing two different paths from $K(L_0)$ to $p_i$: the obvious one along $K$, and the one that goes `through the grope', following $\G(\Gamma)\subseteq\G(G_\Gamma)\subseteq M$. See Figures~\ref{fig:grope-deg-1}, ~\ref{fig:grope-deg-2},~\ref{fig:deg-2-grope-decorated}.
\begin{figure}[!htbp]
    \centering
    \includegraphics[width=0.75\linewidth]{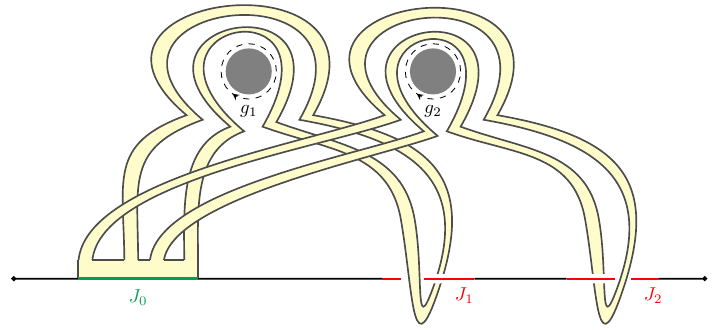}
    \vspace{-15pt}
    \caption[A grope cobordism of degree $2$ with group elements.]{A grope cobordism $\G$ whose underlying decorated tree is $\gtree{g_1}{g_2}$ (with $\varepsilon_1=\varepsilon_2=+1$). If $M=I^3$ and $K$ is the unknot, then $\partial^\perp\G$ is the figure eight knot.}
    \label{fig:deg-2-grope-decorated}
\end{figure}

\subsubsection*{Thickened gropes}
Observe that \red{
    a thickening of the $2$-complex $G_\Gamma$, that is, the union of products of all stages with an interval, is homeomorphic to $\ball^3$. This can be shown inductively, since thickened caps are like $3$-dimensional $2$-handles attached to the thickened punctured torus, which is the union of a ball and two $1$-handles. These handles cancel, see Figure~\ref{fig:thick-grope} and compare to Figure~\ref{fig:abs-grope-deg-2}. 
}

Thus, a regular neighbourhood of a grope cobordism $\G$ is homeomorphic to a $3$-ball $\TG\colon\ball^3\hra M$. This $3$-ball $\TG$ intersects the knot $K$ in the neighbourhoods $\TG(a_i)\subseteq J_i$ of the intersection points $p_i\in K(J_i)$, for some neat arcs $a_i\subseteq\ball^3$, $1\leq i\leq n$. It is convenient to fix a choice of such a neighbourhood $\TG$ as follows; we pick some $\varepsilon>0$.
\begin{figure}[!htbp]
    \centering
    \includegraphics[width=0.35\linewidth]{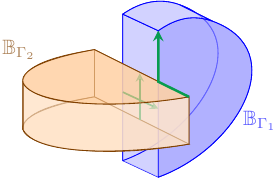}
    \caption[The thickening of an abstract grope is a $3$-ball.]{The thickening of the abstract grope $G_\Gamma$ is a $3$-ball $\ball_\Gamma$, obtained by gluing together the $3$-balls $\ball_{\Gamma_1}$ and $\ball_{\Gamma_2}$ along some distinguished squares in their boundaries.}
    \label{fig:thick-grope}
\end{figure}
\begin{defn}\label{def:thick}
    A \textsf{thickened grope} on $K\in\Knots(M)$ modelled on a tree $\Gamma\in\Tree(S)$ is an embedding
    \[
        \TG\colon\ball_\Gamma\hra M
    \]
    which does not intersect $K$ except that $\TG(a_i)\subseteq K(J_i)$ for certain arcs $a_i\subseteq\ball_\Gamma$, $i\in \{0\}\sqcup S$. Here the \textsf{model ball} $\ball_\Gamma\cong\ball^3$, its set of arcs and a subset $G_\Gamma\subseteq\ball_\Gamma$ are defined inductively on $|S|$ as follows. For the induction base, we have $c=\ichord{i}$ and we define $\ball_c \coloneqq G_c \times I\cong\D^2\times I$ and $G_c\coloneqq G_c\times\{0\}$. We let $a_i$ be the core $(0,0) \times I$ and the arc $a_0$ be the distinguished subarc of $\partial G_c$.
    
    For $\Gamma=\scalebox{0.9}{\grafted}$ define $\ball_\Gamma$ as the \red{gluing together 
    }
    the already defined model balls $\ball_{\Gamma_1}$ and $\ball_{\Gamma_2}$ along the respective squares $a_0\times [-\varepsilon,\varepsilon]\subseteq\partial \ball_{\Gamma_\myj}$ for $\myj=1,2$, using the swap map $a_0\times [-\varepsilon,\varepsilon]\to a_0\times [-\varepsilon,\varepsilon]$, $(x,y)\mapsto (y,x)$. Then let $\{a_i\}_{i\in S}$ be the disjoint union of the sets of arcs for $\ball_{\Gamma_1}$ and $\ball_{\Gamma_2}$.
    
    Define $G_\Gamma\subseteq \ball_{\Gamma}$ as the \red{result of gluing 
    }
    of the bands $\partial G_{\Gamma_{\myj}}\times [-\varepsilon,\varepsilon]\subseteq\partial \ball_{\Gamma_{\myj}}$, \red{
    $\myj=1,2$ 
    }
    along the squares $a_0\times [-\varepsilon,\varepsilon]$. Finally, let $a_0$ for $\ball_\Gamma$ be the distinguished arc in $\partial G_\Gamma$ as in Definition~\ref{def:abstract-grope}.
\end{defn}
Note that $\G\coloneqq\TG|_{G_\Gamma}$ is a grope cobordism on $K$ in the sense of Definition~\ref{def:grope-cob}. We can thus also define an underlying decorated tree $\upvarepsilon(\TG)\Gamma^{g(\TG)}$ of a thickened grope as in Definition~\ref{def:underlying-decor-tree}. Moreover, we define the \textsf{output knot of $\TG$} as $\partial^\perp\TG\coloneqq\partial^\perp\G=K_{\wh{0}}\cup\TG(a_0^\perp)$, as in Definition~\ref{def:grope-cob}. Conversely, for a given grope cobordism $\G$ and a choice of its regular neighbourhood, there is a unique thickened grope $\TG$ whose image is precisely that neighbourhood and $\TG|_{G_\Gamma}=\G$.

\subsubsection*{Grope forests}
Recall from \eqref{eq-def:n-equiv} that two knots are $n$-equivalent if there exists a \emph{sequence} of grope cobordisms between them. An analogue in our setting is a \emph{grope forest}.
\begin{defn}\label{def:forest}
    A \textsf{grope forest} of degree $n$ and cardinality $N\geq 1$ on a knot $K$ is a map
    \[
    \forest\coloneqq\bigsqcup_{l=1}^N\TG_l\colon\bigsqcup_{l=1}^N\ball_{\Gamma_l}\hra M
    \]
    such that $\TG_l\colon\ball_{\Gamma_l}\hra M$ are mutually disjoint thickened gropes on $K$ modelled on some $\Gamma_l\in\Tree(n)$, such that the arcs $\TG_l(a_0)\subseteq K(J_0)$ appear in $K(J_0)$ in the decreasing order of their indices $N\geq l\geq 1$. 
    
    The \textsf{output knot} $\partial^\perp{\forest}$ is obtained from $K$ by replacing each interval $\TG_l(a_0)$ by the arc $\TG(a_0^\perp)$ (the order in which replacements are done is irrelevant by the disjointness assumption).

    \red{The underlying tree of $\forest$ is defined as the linear combination of the corresponding underlying trees:    
    }
    \[
    \ut(\forest)\coloneqq\sum_{l=1}^N\ut(\TG_l)\coloneqq\sum_{l=1}^N\upvarepsilon(\TG_l)\cdot \Gamma_l^{g(\TG_l)}
    \quad\in\Z[\Tree_{\pi_1M}(n)].
    \]
\end{defn}
\red{
    In other words, a grope forest is a \emph{disjoint collection of thickened gropes on the given knot $K$}.
}
For a fixed $i\in\ul{n}$ we allow an arbitrary order of intersections of $K(J_i)$ with the $i$-th caps of different gropes, see Figure~\ref{fig:grope-forest} for an example with $\mathrm{cap}_1(\TG_1)<\mathrm{cap}_1(\TG_2)$, but $\mathrm{cap}_2(\TG_2)<\mathrm{cap}_2(\TG_1)$.
\begin{figure}[!htbp]
    \centering
    \includegraphics[width=0.52\linewidth]{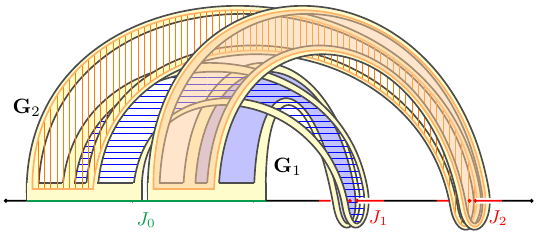}
    \caption[A grope forest of degree $2$.]{A grope forest $\forest=\TG_1\sqcup\TG_2$ of degree $2$ is a thickening of the depicted $2$-complex.}
    \label{fig:grope-forest}
\end{figure}

\red{
\begin{remm}
    Grope forests are suitable for defining `spaces of gropes', namely, as the subspace of $\Emb(\bigsqcup_N\ball^3,M)$ consisting of those embeddings that satisfy the conditions of Definitions~\ref{def:thick} and~\ref{def:forest}; then the map $\psi$ defined in Proposition~\ref{prop:extended-KST} below is in fact continuous. It is an interesting question whether the space of all grope forests can be turned (by adding points) into one which is homotopy equivalent to $\H_n(M)$. We do not pursue such problems here.
\end{remm}
}

\begin{prop}\label{prop:uf-map}
    \red{The underlying tree map is 
    }
    a surjection of sets
    \[
    \begin{tikzcd}
        \ut\colon
        \red{\big\{\text{degree $n$ grope forests}\big\}
        }
        \arrow[two heads]{r}{} & \Z[\Tree_{\pi_1M}(n)].
    \end{tikzcd}
    \]
\end{prop}
\begin{proof}
    Let $\sum_{l=1}^N\upvarepsilon^l \Gamma_l^{g^l}$ be a linear combination of decorated trees, $\upvarepsilon^l\in\{\pm1\}$. Any $g^l_{\ul{n}}\in(\pi_1M)^{n}$ can be represented by a tuple of disjointly embedded loops $\gamma^l_i\subseteq M$, $1\leq l\leq n$. Thus, there is a map $\Gamma_l\to M$ which embeds the edges mutually disjointly, maps the $i$-th leaf to a point $p_i\in K(J_i)$ and has the associated path (from $K(L_0)$ to $p_i$ along $\Gamma$ and then back along $K$) isotopic to $\gamma_i^l$. Thicken this to a ball to get a thickened grope $\TG_l$, introducing a twist to one cap if $\upvarepsilon^l=-1$.

    This can be done so that $\TG_l$ are mutually disjoint (as they are neighbourhoods of $1$-complexes), and that the order $\TG_l(a_0)$ is decreasing with $l$, so this defines a desired grope forest.
\end{proof}

\subsection{Gropes give paths in the Taylor tower}\label{subsec:grope-paths}

\red{In this section we prove Theorem~\ref{thm:KST}: for a thickened grope $\TG$ in $M$ on a knot $K$ modelled on $\Gamma\in\Tree(n)$, there is a path in $\pT_n(M)$ between the evaluation of the output and input knots:
    \[
    \Psi^\TG\colon\ev_n(\partial^\perp\TG)\squig \ev_n(K).
    \]
The idea of proof is based on discussions with Yuqing Shi and Peter Teichner. It is also inspired by Jim Conant's ideas, appearing in \cite{BCKS}.
}

We first reformulate the theorem as the following proposition. Recall that $f\in\pT_n(M)\coloneqq\holim\mc{E}_{\bull}$ is given as a collection $f^S\colon\Delta^S\to\Emb_\partial(I\sm J_S,M)$ for $S\subseteq[n]$, which is compatible under inclusions $\iota_S\colon \Delta^{S}\hra\Delta^{\ul{n}}$, see Section~\ref{subsec-prelim:holims}. Recall also that $K_0$ denotes the restriction of a knot $K$ to $J_0\subseteq I$, while $K_{\wh{S}}$ denotes the restriction to $I\sm J_S$.
\begin{prop}\label{prop:psi-def}
    For $\TG$ as above there is a continuous map
    \[\begin{tikzcd}
    \Path^{\TG}\colon\Delta^{\ul{n}}\arrow[]{r}{} & \Map_\partial\big([0,1],\,\Emb_\partial(J_0,M)\big)
    \end{tikzcd},\; \Path^{\TG}_u(0)=(\partial^\perp\TG)_{J_0}\;, \Path^{\TG}_u(1)=K_0,\,\forall u\in\Delta^{\ul{n}},
    \]
    which gives a well-defined map $\Psi^\TG\colon [0,1]\to\pT_n(M)$ taking $\theta\in[0,1]$ to
\begin{equation}\label{eq:psi-def}
    \Psi^\TG(\theta)^S\colon \Delta^{S} \to \Emb_\partial(I\sm J_S,M),\quad\quad
    \vec{t}\mapsto
\begin{cases}
    K_{\wh{S}}, & \text{if } 0\in S,\\
    K_{\wh{0 S}}\cup \Path^{\TG}_{\iota_S(\vec{t})}(\theta), & \text{if } 0\notin S.
\end{cases}
\end{equation}
\end{prop}
\red{This is clearly equivalent to the statement of Theorem~\ref{thm:KST} recalled above. 
}
Let us outline the proof of the proposition. Let $\Emb_\partial(\D^2,\ball_\Gamma)$ denote the space of embeddings of disks in the model ball with the boundary condition $\partial\D^2=\partial G_\Gamma$. Firstly, in Proposition~\ref{prop:grope-disk-corr} \red{below 
}
we construct a family of disks $\phi_\Gamma\colon\Delta^{\ul{n}}\ra\Emb_\partial\big(\D^2,\ball_\Gamma\big)$ satisfying certain condition \eqref{eq:cond-disks}. Then we choose a homeomorphism $j\colon[0,1]\times J_0\to \D^2$. This gives an isotopy $j_\theta\colon J_0\hra\D^2$, $\theta\in[0,1]$, relative to the endpoints from one half of the boundary circle to the other across $\D^2$. Finally, for $u\in\Delta^{\ul{n}}$, $\theta\in[0,1]$ we define
\begin{equation}\label{eq:proof-psi-def}
\begin{tikzcd}[column sep=1.1cm]
     \Path^{\TG}_u(\theta)\colon\;J_0\arrow[hook]{r}{j_\theta}  & \D^2\arrow[hook]{r}{\phi_\Gamma(u)} & \ball_\Gamma\arrow[hook]{r}{\TG} & M.
\end{tikzcd}
\end{equation}
In other words, the paths of arcs in $M$ are obtained by foliating embedded disks (from one half-circle to the other), that are found in the image of a thickened grope.
We will finish the proof by checking that $\Psi^\TG$ is well defined, \red{using the mentioned
}
conditions~\eqref{eq:cond-disks}. See Section~\ref{subsec:grope-points} for examples.

\subsubsection{The symmetric surgery}
Let us first construct a $1$-parameter family of disks $\D_u\subseteq\ball_\Gamma$, $u\in\Delta^1$, for $\Gamma$ an abstract grope modelled on the unique tree of degree $n=2$ (Figure~\ref{fig:abs-grope-deg-2}). This consists of a punctured torus (yellow) and two caps bounded by its core curves $\beta_1$ (blue) and $\beta_2$ (orange).

There is a classical construction of \emph{ambient surgery} on a punctured torus $T\subseteq M$, using an embedded disk $D$ whose interior is disjoint from $T$  and with boundary a simple closed curve on $T$. Namely, we take out a neighbourhood of the curve $\partial D\subseteq T$ and glue to the newly created boundary two parallel copies of $D$, so that $T$ is turned into an embedded disk.

\begin{figure}[!htbp]
    \centering
    \includegraphics{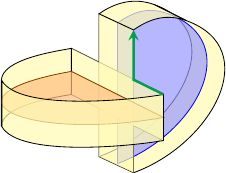}
    \vspace{-10pt}
    \caption[An abstract grope of degree $2$.]{The abstract grope modelled on $\tree$.}
    \label{fig:abs-grope-deg-2}
\end{figure}
Hence, when our abstract grope of degree $2$ is embedded as a grope cobordism we can do two different ambient surgeries on it: on the first (respectively second) cap as depicted in the leftmost (rightmost) part of Figure~\ref{fig:symm-surgery}. Note that the thickened grope specifies concrete push-offs of caps.
\begin{figure}[!htbp]
    \centering
    \includegraphics{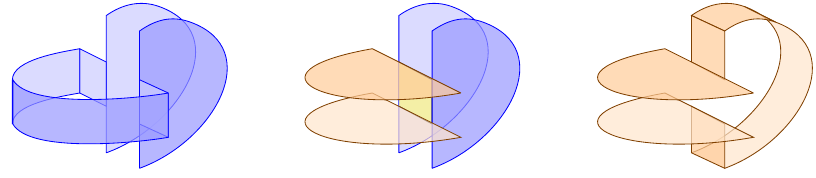}
    \caption[Three surgeries on a capped torus.]{\textit{Left}: The resulting disk $\D_1$ after the surgery along the cap $1$ on $G_\Gamma$ from Figure~\ref{fig:abs-grope-deg-2}. \textit{Middle}: The result of the symmetric surgery on $G_\Gamma$. \textit{Right}: The result $\D_2$ of the surgery along the cap $2$ on $G_\Gamma$.}
    \label{fig:symm-surgery}
\end{figure}

In addition, one can do both surgeries at once, called the \textsf{symmetric surgery} (or contraction), as depicted in the middle part of Figure~\ref{fig:symm-surgery}. The following lemma says that there actually exists a whole $1$-parameter family of disks containing the three disks we have described.
\begin{lemma}[Symmetric Isotopy]\label{lem:sym-surg-isotopy}
For $\Gamma=\tree$ there is an isotopy $\phi_\Gamma\colon[0,1]\to\Emb_\partial(\D^2,\ball_\Gamma)$ such that $\D_t$ for $t\in\{0,1\}$ is the surgery on $G_\Gamma$ using the cap labelled by $1+t$.
\end{lemma}
\begin{proof}
Recall that $\ball_\Gamma$ is the model ball obtained by gluing together $\B\ichord{1}$ and $\B\ichord{2}$. We now specify an isotopy from $\D_1\subseteq\ball_{\Gamma}$ to $\D_2\subseteq\ball_{\Gamma}$, which passes through the symmetric surgery, using Figure~\ref{fig:symm-surgery} as an accurate model of these disks.

First isotope the interior of the blue band of $\D_1$ by pushing it across the interior of the ball $\B\ichord{2}$, until we arrive at the symmetric surgery. In more detail, as $t$ increases from $0$ to $\frac{1}{2}$ we let the blue band `stick more and more to the bottom and top orange disks', as shown in Figure~\ref{fig:disk-u}, so that when $t=\frac{1}{2}$ the band has transformed into the union of the two orange disks and the yellow region.

\begin{figure}[!htbp]
    \centering
    \includegraphics{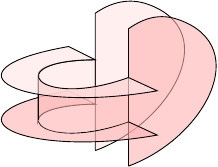}
    \caption[{A fixed time of the symmetric isotopy.}]{Disk $\D_t$ for some $t\in[0,\frac{1}{2}]$.}
    \label{fig:disk-u}
\end{figure}
The two `sticking curves' (inside of the two orange disks, copies of the cap $2$) are specified by an isotopy $j_\theta\colon J_0\hra\D^2$ which we fixed at the beginning of this section (also, smoothen the corners).

Symmetrically, for increasing $t\in[\frac{1}{2},1]$ the isotopy uses the ball $\B\ichord{1}$ to stretch the distinguished yellow region of the symmetric surgery, using the sticking curves on the blue disks as a guide, until reaching the position of the orange band for $t=1$.
\end{proof}

\begin{remm}\label{rem:claspers-gropes}
    It is precisely this isotopy that is a crucial ingredient for the connection between the geometric calculus and the Taylor tower. To construct paths in $\pT_n(M)$ using \text{claspers} instead, it would be necessary to fix a $1$-parameter family of homotopies of Borromean rings whenever one component is erased, but for trees of higher degrees these homotopies will increase in complexity.
    
    Instead, gropes precisely keep track of all necessary homotopies in a canonical way, missing in the clasper picture. Moreover, we will use our exact choice of the isotopy in the crucial Lemma~\ref{lem:commutator}.
\end{remm}

\subsubsection{Families of disks}
We now generalise the Symmetric Isotopy Lemma~\ref{lem:sym-surg-isotopy} to trees of any degree $n\geq 2$. We view $\Delta^S$ as the simplicial set obtained by barycentric subdivision from the standard simplex with the vertex set $S$ (see Figure~\ref{fig:simplex}).

\begin{prop}\label{prop:grope-disk-corr}
    For a finite set $S\neq\emptyset$ and a tree $\Gamma\in\Tree(S)$ there is a continuous map
    \[
        \phi_\Gamma\colon\Delta^S\ra\Emb_\partial\big(\D^2,\ball_\Gamma\big)
    \]
    describing a family $\D_u\coloneqq\im\phi_\Gamma(u)\subseteq\ball_\Gamma$ of neatly embedded disks in the model ball such that
    \begin{align}\label{eq:cond-disks}
        (\forall i\in S)\quad \mathrm{int}(\D_u)\pitchfork a_i\neq\emptyset &\implies i\in|\min(u)|
    \end{align}
    where $\red{\min(u)
    }
    \subseteq\Delta^S$ is the smallest simplex to which $u$ belongs and $|\min(u)|$ is its set of vertices.
\end{prop}
\begin{proof}
    We prove this by induction on $|S|$. For $|S|=1$ we have $\Gamma=\ichord{i}$ and $|\Delta^S|=\Delta^0=\{i\}$, so we need to construct only one disk $\D_{i}\subset\ball_{\Gamma}$ whose boundary is the boundary of the grope and such that $\mathrm{int}(\D_{i})\pitchfork a_1\neq \emptyset$. Clearly, we can just let $\D_{i}\coloneqq G_{\Gamma}$, since in this case the abstract grope is itself a disk, intersecting $a_1$ in one point.
    
    Assume that we have defined the desired family for any tree of degree $<k$ for some $k\geq2$, and consider $S$ with $|S|=k$ and a tree $\Gamma\in\Tree(S)$ such that
    \[
        \Gamma=\grafted, \quad \Gamma_{\myj}\in\Tree(S_{\myj}),\,S=S_1\sqcup S_2.
    \]
    Pick $u\in\Delta^S$ and let us define $\D_u\subseteq\ball_\Gamma$, using the identification $\Delta^S\cong\Delta^{S_1}\star\Delta^{S_2}$, the join of two simplices. Thus, $u$ is given as a linear combination
    \[
    u=(1-t)u_1+t u_2,\quad t\in[0,1],\quad u_{\myj}\in\Delta^{S_{\myj}}.
    \]
    The ball $\ball_\Gamma$ is by definition the gluing of the balls $\ball_{\Gamma_{\myj}}$ for $\myj=1,2$, and since $1\leq|S_{\myj}|\leq |S|-1$, by induction hypothesis we have maps $\phi_{\Gamma_{\myj}}$ satisfying \eqref{eq:cond-disks}. In particular, we have disks $\D_{u_{\myj}}\subseteq\ball_{\Gamma_{\myj}}$.
    
    Let us pick some neat tubular neighbourhoods $\nu\D_{u_{\myj}}\subseteq\ball_{\Gamma_{\myj}}$, so that $\partial(\nu\D_{u_{\myj}})\cap\partial\ball_\Gamma=\partial G_{\Gamma_\myj}\times [-\varepsilon,\varepsilon]$. Then we can plumb $\nu\D_{u_1}$ and $\nu\D_{u_2}$ together along $a_0\times [-\varepsilon,\varepsilon]$ and get a ball $\B\subseteq\ball_\Gamma$ such that $G_\Gamma\subseteq\B$. Now by Lemma~\ref{lem:sym-surg-isotopy} there is an isotopy inside of $\B$ from the disk obtained by surgery on $G_\Gamma$ along $\D_{u_1}$ to the disk obtained by surgery on $G_\Gamma$ along $\D_{u_2}$.
    
    Let $\D_u$ be the time $t$ of that isotopy. Clearly $\partial\D_u=\partial G_\Gamma$. Let us show that the property \eqref{eq:cond-disks} holds. Since $\D_u$ is contained in $\B$, which is a sufficiently small neighbourhood of the disks $\D_{u_1}$ and $\D_{u_2}$, it will intersect an arc $a_i$ only if one of those disks did. Hence, by the induction hypothesis $i$ belongs either to $|\min(u_1)|$ or $|\min(u_2)|$. Since $|\min(u)|=|\min(u_1)|\sqcup|\min(u_2)|$ by the definition of the join, \red{we have $i\in|\min(u)|$ as desired.
    }
\end{proof}

In particular, for $n=2$ we have $u=(1-t)+2t=1+t$ and so $\D_u=\D_{1+t}$ is precisely the isotopy from Lemma~\ref{lem:sym-surg-isotopy}. For an abstract grope of degree $n$ each torus stage gives one independent parameter for the family, so there are $n-1$ parameters in total (remember that $|\Delta^{\ul{n}}|=\Delta^{n-1}$).

\red{
We can now finish the proof of Proposition~\ref{prop:psi-def}: we perform isotopies across disks in the family obtained in Proposition~\ref{prop:grope-disk-corr} to get families of punctured knots.
}

\begin{proof}[Proof of Proposition~\ref{prop:psi-def}]
    \red{As announced in \eqref{eq:proof-psi-def} at the beginning of the section, we use the homeomorphism $j\colon[0,1]\times J_0\to \D^2$, the isotopy $\phi_\Gamma$ of the previous proposition for $S=\ul{n}$, and the given thickened grope $\TG$ to define for $u\in\Delta^{\ul{n}}$ and $\theta\in[0,1]$ an embedding:
    }
    \[\begin{tikzcd}
         \Path^{\TG}_u(\theta)\colon\;J_0\arrow[hook]{r}{j_\theta}  & \D\arrow[hook]{r}{\phi_\Gamma(u)} & \ball_\Gamma\arrow[hook]{r}{\TG} & M.
    \end{tikzcd}
    \]
    We clearly have $\Path^\TG_u(0)=(\partial^\perp\TG)_{J_0}=\TG(a_0^\perp)$ and $\Path^\TG_u(1)=K_0=\TG(a_0)$ for all $u\in\Delta^{\ul{n}}$. We claim that thanks to the condition \eqref{eq:cond-disks}, the map $\Psi^\TG$ as defined in \eqref{eq:psi-def} is well defined, that is:
    \[
        \Psi^\TG(\theta)^S(\vec{t})\;\in\;\Emb_\partial(I\sm J_S,M).
    \]
    This is clear for $S\subseteq[n]$ such that $0\in S$, since we then constantly have the punctured unknot $\U_{\wh{S}}$. On the other hand, for $0\notin S$ we need to check that for each $\vec{t}\in\Delta^S$ and $\theta\in[0,1]$ we have
    \[
        \Path^\TG_u(\theta)\cap K_{\wh{0S}}=\emptyset
    \]
    where $u\coloneqq\iota_S(\vec{t})$. Equivalently, if the interior of $\TG(\D_u)$ intersects some $K(J_i)$, then $i\in S$. Indeed, if $\mathrm{int}\TG(\D_u)\cap K(J_i)\neq\emptyset$, then we must have $\mathrm{int}(\D_u)\cap a_i\neq\emptyset$, since $\TG$ is an embedding. But then \eqref{eq:cond-disks} implies that $i\in|\min(u)|$. As $u$ is obtained by inclusion from the face $\Delta^S$, the maximal simplex that contains it must be contained in $\Delta^S$. Hence, $i\in|\min(u)|\subseteq S$.
\end{proof}

\subsection{Grope forests give points in the layers}\label{subsec:grope-points}
\red{
    In this section we extend the definition of paths $\Psi$ from Proposition~\ref{prop:psi-def} to any grope forest $\forest$. Moreover, we observe that if $\forest$ is a forest on the basepoint knot $\U$, then this gives a point $\psi(\forest)\in\H_n(M)$. Applying the map $\emap_{n+1}\colon\H_n(M)\to\pF_{n+1}(M)$ we obtain the point $\emap_{n+1}\psi(\forest)$, which is the crucial object in our Main Theorems~\ref{thm:main-thm} and~\ref{thm:main-extended}, proven in the next section.
}

\subsubsection*{The extension of Theorem~\ref{thm:KST} to grope forests}
\begin{prop}\label{prop:extended-KST}
    For a grope forest $\forest$ of degree $n$ on $K$ there exists a path $\Psi^{\forest}\colon[0,1]\to\pT_n(M)$ from $\ev_n(\partial^\perp\forest)$ to $\ev_nK$. 
    In particular, if $K=\U$, our basepoint knot, then we have a point
\begin{equation}\label{eq:psi}
    \psi(\forest) \coloneqq (\partial^\perp\forest,\Psi^\forest)\in \H_n(M).
\end{equation}
\end{prop}
\begin{proof}
If $\forest=\bigsqcup_{l=1}^N\TG_l\colon\bigsqcup_{l=1}^N\ball_{\Gamma_l}\hra M$, then each $\TG_l$ can be viewed as a thickened grope on $K$. Indeed, it has $\TG_l(a_0)\subseteq K_0$ and the conditions for all the arcs $a_i$, $i\in\ul{n}$, are satisfied.

Therefore, by Theorem~\ref{thm:KST} we have a path $\Psi^{\TG_l}$ in $\pT_n(M)$ from $\ev_n(\partial^\perp\TG_l)$ to $\ev_nK$, which was constructed in Proposition~\ref{prop:psi-def} using the arcs $\Path^{\TG_l}_u(\theta)\colon J_0\hra M\sm K_{\wh{0S}}$. For a fixed $\theta\in[0,1]$ and $S\subseteq\ul{n}$ these arcs are pairwise disjoint for varying $l=1,\dots,N$, because of the mutual disjointness of $\TG_l$. Hence, we can concatenate them to get an arc
\[
\Psi^{\forest}(\theta)^S_{J_0}\coloneqq\Psi^{\TG_1}(\theta)^S_{J_0}\cdot\Psi^{\TG_2}(\theta)^S_{J_0}\cdots\Psi^{\TG_N}(\theta)^S_{J_0}\;\in\;\Emb(J_0,M\sm K_{\wh{0S}}).
\]
We then define $\Psi^{\forest}$ analogously to the definition of $\Psi^\TG$ in \eqref{eq:psi-def}, by letting
\begin{equation}\label{eq:psi-forest-def}
    \Psi^{\forest}(\theta)^S\colon \Delta^{S} \to \Emb_\partial(I\sm J_S,M),\quad\quad
    \vec{t}\;\mapsto\;
    K_{\wh{0 S}}\cup \Psi^{\forest}(\theta)^S_{J_0}.
\end{equation}
for $\theta\in[0,1]$, $S\subset[n]$. As in the proof of Proposition~\ref{prop:psi-def}, this is indeed a path $\ev_n(\partial^\perp\forest)\squig\ev_nK$.

\red{Finally, if $K\coloneqq\U$ then $\psi(\forest)\coloneqq(\partial^\perp\forest,\Psi^{\forest})$ is indeed a point in $\H_n(M)\coloneqq\hofib_{\ev_n(\U)}(\ev_n)$, by the definition of the homotopy fibre (see Example~\ref{ex:mapping-path-space}).
}
%
\end{proof}

\begin{remm}\label{rem:other-choice-psi-forest}
    A perhaps more obvious choice for the definition of $\Psi^{\forest}$ would simply be
    \[
        \Psi^{\TG_1}\cdot\Psi^{\TG_2}\cdots\Psi^{\TG_N}\colon I\to\pT_n(M),
    \]
    the concatenation of the paths in $\pT_n(M)$. \red{We claim that 
    }
    this would actually give an equivalent point $\emap_n(\psi\forest)\in\pF_n(M)$. In essence this is because $\pF_n(M)$ is an iterated loop space and -- while our definition was concatenation in the $J_0$-direction, this latter definition corresponds to the concatenation in the `diagonal' $\Omega^{n}$ direction. \red{In more detail, the two choices $\deriv\chi\emap_n(\psi\forest)\in\Omega^n\tofib(M_{\bull},\rho)$ can be compared using the description of $\chi\emap_n(\psi\forest)$ in terms of the $h$-reflections of Proposition~\ref{prop:chi-is-glued}; but we omit the complete proof.
    }
    
    This definition implies that concatenation of thickened gropes into a grope forest can be seen as a partially defined $H$-space or $E_1$-structure on the space $\H_n(M)$. However, our definition makes the proof of Theorem~\ref{thm:main-extended} straightforward.
\end{remm}

\begin{example}[degree $1$]\label{ex:grope-path-deg-1} 
    We now demonstrate the map $\Psi$ on an example in the lowest degree. A grope cobordism $\G$ on $K$ of degree $1$ is simply a disk (see Figure~\ref{fig:grope-deg-1} for examples) guiding a crossing change homotopy $K(\theta)$, $\theta\in[0,1]$, from $K(0)=\partial^\perp\TG$ to $K(1)=K$. The corresponding thickened grope $\TG$ is a tubular neighbourhood of this disk and the underlying decorated tree is
    \[
    \gchord{g}
    \]
    for some element $g\in\pi_1(M)$. 
    The disk family \red{from Proposition~\ref{prop:grope-disk-corr} 
    }
    in this case consists of a single disk $\D\subseteq\ball_{\Gamma}$ and $\TG(\D)=\G\subseteq M$. The map $\Path^\TG\colon\Delta^0\to\Map_\partial\big([0,1],\Emb_\partial(J_0,M)\big)$ swings the arc $\TG(a_0^\perp)$ across $\TG(\D)$ to $K_0$. Note that the path through immersions $K_{\wh{0}}\cup\Path^\TG(\theta)$ is precisely $K_\theta$, $\theta\in[0,1]$. The path $\Psi^\TG\colon [0,1]\to\pT_1(M)=\holim_{\PCube[1]}\mc{E}^1_{\bull}$ is hence given by
\begin{align*}
\Psi^\TG(\theta)^{\{0\}}\colon \Delta^0 &\to \Emb(I\sm J_0,M),\quad pt\mapsto K_{\wh{0}},\\
\Psi^\TG(\theta)^{\{1\}}\colon \Delta^0 &\to \Emb(I\sm J_1,M),\quad pt\mapsto (K_\theta)_{\wh{1}},\\
\Psi^\TG(\theta)^{\{01\}}\colon \Delta^1 &\to \Emb(I\sm I_{01},M),\quad t \mapsto K_{\wh{01}},\quad \forall t\in\Delta^1.
\end{align*}
Only $\Psi^\TG(-)^{\{1\}}\colon\Delta^0\to\Emb_\partial(I\sm J_{1},M)$ is not constant with $\theta\in[0,1]$. It is the isotopy between $(\partial^\perp\TG)_{\wh{1}}$ and $K_{\wh{1}}$ -- the crossing change homotopy, now unobstructed since $J_1$ is gone. See also Figure~\ref{fig:deg-1-layer-pt} below for the corresponding points $\psi(\TG)\in\H_1(M)$ and $\emap_2\psi(\TG)\in\pF_2(M)$.
\end{example}
\begin{figure}[!htbp]
    \centering
    \includegraphics{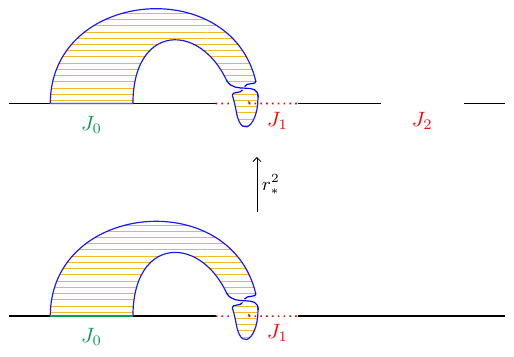}
    \caption[Points $\protect\psi(\TG)\in\H_n(M)$ and $\protect\emap_n(\psi(\TG))\in\pF_n(M)$.]{
        \textit{Bottom}: A point $\psi(\TG)\in\tofib(\FF_S)=\H_1(M)$ consists of the blue arc $\psi(\TG)^\emptyset\coloneqq \TG(a_0^\perp)\in\FF_\emptyset$ and the path $\psi(\TG)^1\coloneqq\Psi^{\TG}(-)^{1}_{J_0}\colon I^1\to\FF_{1}$ across the yellow disk.
        \textit{Top}: In order to get the point $\emap_2(\psi(\TG))\in\tofib(\FF^2_S)=\pF_2(M)$ we simply introduce one more puncture into the ambient space.}
    \label{fig:deg-1-layer-pt}
\end{figure}

\subsubsection*{Points in the layers}
For a grope forest $\forest$ of degree $n$ on $\U$ we obtained in Proposition~\ref{prop:extended-KST} a point $\psi(\forest)\coloneqq(\partial^\perp\forest,\Psi^{\forest})\in\H_n(M)$. Since $\H_n(M)\coloneqq\hofib(\ev_n)\cong\tofib_{S\subseteq[n]}(\mc{E}_S)$ (see Section~\ref{subsec:pn-surj-fib}), in the latter coordinates this point is given by
\[\psi(\forest)^S =
\begin{cases}
    \begin{tikzcd}[column sep=large]
        I^0\arrow[]{r}{\partial^\perp\forest} & \mc{E}_\emptyset,
    \end{tikzcd} &\quad S=\emptyset, \\
    \begin{tikzcd}[column sep=large]
        I^S\arrow[]{r}{\mc{h}^S} & \mathsf{C}^{bar}(\Delta^S)\arrow[]{r}{\Psi^\forest(-)^S} & \mc{E}_S,
        \end{tikzcd}
        &\quad \emptyset\neq S\subseteq\ul{n},
\end{cases}
\]
where $\mc{h}^{\bull}$ is the homeomorphism of cubes \red{from Lemma~\ref{lem:bar-cone-is-cube}, needed for the description of the total homotopy fibre (defined as a homotopy fibre) in terms of maps of cubes (see Lemma~\ref{lem:goo-tofib}).
}

Now, in Section~\ref{subsec:fn-hn} we saw that $\tofib_{S\subseteq[n]}\mc{E}_S$ is homotopy equivalent to its subspace $\tofib_{S\subseteq\ul{n}}\FF_S$.
\begin{lemma}\label{prop:grope-tofib-ffs}
    $\psi(\forest)$ lies in the subspace $\tofib(\FF_S)$ and in the coordinates $\FF_S\cong \Emb_\partial(J_0,M\sm \U_{\wh{0S}})$ it is given simply by restricting $\psi(\forest)^S$ to $J_0\subseteq I$.
\end{lemma}
\begin{proof}
    It is enough to check that for $S\subseteq[n-1]$ with $S\ni 0$ the map $\psi(\forest)^S\colon I^S\to\mc{E}_S$ is constantly equal to $\U_{\wh{S}}$ (since then the only non-trivial part of $\psi(\forest)^S$ for $S\notni0$ is $\psi(\forest)^S|_{J_0}$). However, this is clear from $\psi(\forest)^S\coloneqq\Psi^\forest(-)^S\circ \mc{h}^S$ and the very definition $\Psi^{\forest}(\theta)^S=\U_{\wh{S}}$ in \eqref{eq:psi-def}.
\end{proof}

Recall \red{from Section~\ref{subsec:fn-hn} 
}
that $\pF_{n+1}(M)$ is also homotopy equivalent to its subspace $\tofib_{S\subseteq\ul{n}}(\FF^{n+1}_S)$, and from Definition~\ref{def:FF} that
\[
    \emap_{n+1}\colon\H_n(M)\to\pF_{n+1}(M)
\]
corresponds to the map $r^{n+1}_*\colon\tofib(\FF_S)\to\tofib(\FF^{n+1}_S)$ induced by the postcomposition with
\[
    \rho^{n+1}_S\colon M\sm\U_{\wh{0S}}\hra M\sm\U_{\wh{0Sn+1}}
\]
which is the obvious inclusion map, adding the appropriate neighbourhood of $J_{n+1}$.

Hence, the image under the evaluation map of our grope forest point is (see Figure~\ref{fig:deg-1-layer-pt}):
\begin{equation}\label{eq:final-psi-forest}
    f_{\forest}\coloneqq\emap_n\psi(\forest)=\begin{cases}
    \begin{tikzcd}[column sep=large]
        I^0\arrow[]{r}{(\partial^\perp\forest)_{J_0}} & \FF^{n+1}_\emptyset
    \end{tikzcd}, &\quad S=\emptyset, \\
    \begin{tikzcd}[column sep=large]
        I^S\arrow[]{r}{\mc{h}^S} & \mathsf{C}^{bar}(\Delta^S)\arrow[]{r}{\Psi^\forest(-)^S_{J_0}} & \FF^{n+1}_S
        \end{tikzcd},
        &\quad \emptyset\neq S\subseteq\ul{n}.
\end{cases}
\end{equation}
\red{Note that $\Psi^\forest(-)^S_{J_0}=\Psi^{\TG_1}(-)^S_{J_0}\cdot\dotso\cdot\Psi^{\TG_N}(-)^S_{J_0}$ for the family of arcs $\Psi^{\TG_l}(-)^S_{J_0}\colon\mathsf{C}^{bar}(\Delta^S)\to\FF^{n+1}_S$ concatenated pointwise along their $J_0$ direction. In particular, $(\partial^\perp\TG)_{J_0}=\TG_1(a_0^\perp)\cdot\dotso\cdot\TG_N(a_0^\perp)$.
}

\section{Proofs of main theorems}\label{sec:main}
Let $\TG\colon\ball_\Gamma\hra M$ be a thickened grope on $\U$ with the underlying decorated tree $\upvarepsilon\Gamma^{g_{\ul{n}}}\in\Tree(n)$. In the previous section we have constructed $\psi(\TG)\in\H_n(M)$ and in \eqref{eq:final-psi-forest} we described the point
\begin{align}
    f_{\TG}\;&\coloneqq\;\emap_{n+1}\psi(\TG)\quad\in\pF_{n+1}(M).\nonumber
\\
\intertext{In this section we prove Theorem~\ref{thm:main-thm} -- namely, that}
    [f_{\TG}]\; &=\; \left[\upvarepsilon\Gamma^{g_{\ul{n}}}\right]\quad \in \pi_0\pF_{n+1}(M)\cong\Lie_{\pi_1M}(n).\nonumber
\\
\intertext{In Section~\ref{subsec:strategy} we have reduced this to checking \eqref{eq:final-goal}, that is}
    \deriv_{\ul{n}}\left(\chi f_{\TG}\right)^{\ul{n}}
    \;&\simeq\;
    \Gamma(m_i^{\varepsilon_i\gamma_i})\colon\quad \S^n\to\Omega M_{\ul{n}}.\label{eq:final-goal2}
\end{align}
where $\varepsilon_i\in\{\pm1\}$ and $\gamma_i\in \Omega M$ with $i\in\ul{n}$ are signed decorations of $\TG$, and $g_i=[\gamma_i]$.

\red{
Let us recall all the objects appearing in the last formula. For the purposes of the proof, which is by induction on $n\geq1$, let us write $R\coloneqq\ul{n}$.
\begin{remm}[Recollections]\label{rem:recollections}\hfill
\begin{itemize}
    \item 
The map $\chi\colon\pF_{n+1}(M)\simeq\tofib(\FF^{n+1}_{\bull},r)\to\Omega^{n}\tofib(\FF^{n+1}_{\bull},l)$ was defined in Proposition~\ref{prop:cube-retraction} using left homotopy inverses $l^k_S=(d^k_S(1)\circ e^k_S)\circ-$ (given as `erase then drag') and homotopies $h^k_s=(d^k_S(s)\circ\mathsf{add}_s)\circ-$ constructed in Section~\ref{subsec:delooping-initial}. 
    \item 
The map $\left(\chi f_{\TG}\right)^R\colon\S^n\to\FF^{|R|+1}_R$ is simply the coordinate of $\chi(f_{\TG})$ indexed by the initial vertex $R=\ul{n}$. By Proposition~\ref{prop:chi-is-glued} we have
\[
    \left(\chi f_{\TG}\right)^R\;=\;\glueOp_{S\subseteq R}\big(f_{\TG}^R\big)^{h^S}.
\]
This notation means that $(\chi f_{\TG})^{R}$
is obtained by gluing the $h$-reflections $\big(f_{\TG}^R\big)^{h^S}\colon I^R\to \FF^{|R|+1}_R$, for $S\subseteq R$, along their $0$-faces, so that they form a map out of `a big cube', which is constant on the boundary, so gives a map out of $\S^R=\faktor{I^R}{\partial}$. The $h$-reflections are defined inductively in Definition~\ref{def:iterated-h-reflections} by finding $k=\min S$ and letting:
\begin{equation}\label{eq:hS-reflection}
    \big(f_{\TG}^R\big)^{h^S}\coloneqq\Big(\big(f_{\TG}^R\big)^{h^{S\sm k}}\Big)^{h^k}\coloneqq r^k_{ R\sm k}\left(
h^k\left(\big(f_{\TG}^{ R\sm k}\big)^{h^{S\sm k}}_s\right)\boxbar_k l^k_{ R\sm k}\big(f_{\TG}^R\big)^{h^{S\sm k}}_s
\right).
\end{equation}
Here the notation $F_1\boxbar_kF_2$ means that a map $F_1\colon I^R\to X$ is glued along the face $t_k=1$ to the face $t_k=0$ of a map $F_2\colon I^R\to X$.
    \item 
For $R\subseteq\ul{n}$ the map $\deriv_R\colon\FF^{n+1}_S\to\Omega M_S\coloneqq\Omega( M\sm\ball_S)$ takes an embedding $f\colon J_0\hra M_R$ to the loop obtained by concatenating $\U_0\coloneqq\U|_{J_0}$ and the arc $f$ in reverse (see Remark~\ref{rem:straight}). Since $\deriv_R$ is applied pointwise, we obtain
\begin{equation}\label{eq:glueOp}
    \deriv_R\left(\chi f_{\TG}\right)^R\;=\;\glueOp_{S\subseteq R}\deriv_R\Big(\big(f_{\TG}^R\big)^{h^S}\Big).
\end{equation}
    \item The map $\Gamma(m_i^{\varepsilon_i\gamma_i})$ on the right hand side of \eqref{eq:final-goal2} is the Samelson product $\Gamma(m_i^{\varepsilon_i\gamma_i})$ of the maps $m_i^{\varepsilon_i\gamma_i}\colon\S^1\to\Omega M_{\ul{n}}$. These are given by $m_i^{\varepsilon_i\gamma_i}(\theta)\coloneqq\gamma_i\cdot m_i(\theta)^{\varepsilon_i}\cdot\gamma_i^{-1}$, where $m_i\colon\S^1\to\Omega M_{\ul{n}}$ is the `swing of a lasso' around $\ball_i$. See Section~\ref{subsec:strategy} for details.
\end{itemize}
\end{remm}
}

\begin{proof}[Proof of Theorem~\ref{thm:main-thm}]
We prove \eqref{eq:final-goal2} by induction on $n\geq 1$ (for a proof sketch see Example~\ref{ex:proof-main}).

\textbf{The induction base.}
In this case $R=\{1\}$, $\Gamma=\ichord{1}$ and $\TG$ it a thickening of a disk $\G$. We need to check that $\deriv_{\{1\}}(\chi f_{\TG})^{\{1\}}\colon \S^1\to\Omega M_1$ is homotopic to $m_1^{\varepsilon_1\gamma_1}\colon \S^1\to\Omega M_1$.

\begin{figure}[!htbp]
    \centering
    \includegraphics{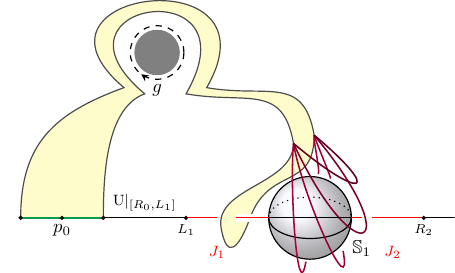}
    \caption[Main proof: the induction base.]{Merging Figures~\ref{fig:grope-deg-1} and~\ref{fig:deg-2-m-i} together: a grope cobordism $\G$ and the family $m_1$.}
    \label{fig:main-induction-base}
\end{figure}
Firstly, the loop $(\chi f_{\TG})^{\{1\}}\coloneqq\big(f_{\TG}^{\{1\}}\big)^{h^1}\cdot f_{\TG}^{\{1\}}$ is the concatenation of the path $f_{\TG}^{\{1\}}\coloneqq\Psi^\TG(t)_{J_0}^{1}$ which is the isotopy across the disk $\G$ (see Example~\ref{ex:grope-path-deg-1}), and the path
\[
\big(f_{\TG}^{\{1\}}\big)^{h^1}_t\coloneqq r^1_\emptyset\Big(\;h^1_\emptyset(s)(\TG(a_0^\perp))\cdot l^1_\emptyset\big(\Psi^\TG(s)_{J_0}^{1}\big)\;\Big)_{t=1-s}
\]
This is obtained by concatenating $h^1_\emptyset(-)(\TG(a_0^\perp))$ with $l^1_\emptyset\big(\Psi^\TG(-)_{J_0}^{1}\big)$ in $\FF^2_\emptyset\coloneqq\Emb_\partial(J_0,M_0)$, then reverse this path and include it into $M_{01}=M_0\cup\ball_{01}$. Recall from Lemma~\ref{lem:defn-of-lambda} that maps $l^1_\emptyset$ and $h^1_\emptyset$ act non-trivially only in the region $[L_1,R_2]\times\D^2\subseteq M$. Also recall that $h^1_\emptyset(s)(\TG(a_0^\perp))\coloneqq d^1_\emptyset(s)\circ\mathsf{add}_s(\TG(a_0^\perp))$ and that the isotopy $d^k_S(s)$ drags the east hemisphere of $\ball_{01}$ to that of $\S_1$.

Therefore, $h^1_\emptyset(s)(\TG(a_0^\perp))$ gradually `drags to the right' the part of $\TG(a_0^\perp)$ inside of this region, and the disk $l^1_\emptyset(\TG)$ agrees with $\G$ except having the tip shifted into $J_2$. So $l^1_\emptyset(\Psi^\TG(s)_{J_0}^{1})$ moves the shifted arc $l^1_\emptyset(\TG(a_0^\perp))=h^1_\emptyset(1)(\TG(a_0^\perp))$ back to $\U_0$ across the shifted disk $l^1_\emptyset(\TG)$. In other words, we use the puncture $J_2$ instead of $J_1$ to isotope $\TG(a_0^\perp)$ back in similar manner.

Applying $\deriv_1$ to this loop $(\chi f_{\TG})^{\{1\}}$ allows us to homotope the part of the disk $\TG$ which is in $M\sm[L_1,R_2]\times\D^2$ (equal to the part of $l^1_\emptyset(\TG)$) onto its core arc (its underlying chord). Hence, we conclude that $\deriv_1(\chi f_{\TG})^{\{1\}}$ is indeed homotopic to $\theta\mapsto\gamma_1\cdot m_1(\theta)^{\varepsilon_1}\cdot\gamma_1^{-1}$.

\subsection{Preliminaries for the induction step}
It will be convenient to consider trees labelled by a finite set $R$.
\begingroup\leqnos
\begin{equation}
    \tag{\emph{setup}}
\begin{minipage}{0.75\textwidth}\label{main-setup}
Let $\TG$ be a thickened grope on $\U$ modelled on a tree $\Gamma\in\Tree(R)$ obtained by grafting together $\Gamma_{1}\in\Tree(R_{1})$ and $\Gamma_2\in\Tree(R_2)$ with $R_1\sqcup R_2=R$. Let $\varepsilon_i\in\{\pm1\}$ and $\gamma_i\in\Omega M$ with $i\in R$ be the signed decorations of $\TG$.
\end{minipage}
\end{equation}
\endgroup
In order to prove \eqref{eq:final-goal2} for $R=\ul{n}$ we first simplify the map $\deriv_{R}\big(\chi f_{\TG}\big)^{R}\colon\S^R\to\Omega M_R$ as follows.

In the \emph{Commutator Lemma}~\ref{lem:commutator} we will show that $\deriv_R(f_{\TG}^R)$ is homotopic to a certain commutator map and in the \emph{Reflections Lemma}~\ref{lem:deriv-h-reflections} generalise this to all $h^S$-reflections $\deriv_R\big(f_{\TG}^R\big)^{h^S}$. Having these homotopies collected in Corollary~\ref{cor:glue-htpies}, we will be able to finish the proof of Theorem~\ref{thm:main-thm}.

\subsubsection*{The Commutator Lemma}
For each $S\subseteq R$ we now study the map
\[
\begin{tikzcd}[row sep=tiny]
    \deriv_Sf_{\TG}^S\colon \quad I^S\arrow{rr}{f_{\TG}^S} && \FF^{|R|+1}_S \arrow{r}{\deriv_S} & \Omega M_S
\end{tikzcd}
\]
given by $\deriv_Sf_{\TG}^S(\vec{t})= (\U_0)_t\cdot (f_{\TG}^S(\vec{t}))_{1-t}$ where $f_{\TG}^S(\vec{t})\coloneqq\Psi^\TG(\theta)^S_{J_0}(u)$ for $\mc{h}^S(\vec{t})=(\theta,u)\in \mathsf{C}^{bar}(\Delta^S)$.

Using the inductive nature of Definition~\ref{def:thick} we can write the thickened grope $\TG$ as the gluing of two thickened gropes $\TG_{\myj}\coloneqq\TG|_{\ball_{\Gamma_{\myj}}}$ modelled on trees $\Gamma_{\myj}$ for $\myj=1,2$ (see Figures~\ref{fig:punctured-torus} and~\ref{fig:thick-grope}).

More precisely, the boundary of the abstract grope $\partial G_{\Gamma_{\myj}}=\beta_{\myj}$ has its corresponding distinguished subarc $\beta_{\myj}^+\subseteq\beta_{\myj}$. Thus, the map $\TG_{\myj}\colon\ball_{\Gamma_{\myj}}\hra M$ can be seen as a thickened grope modelled on $\Gamma_{\myj}\in\Tree(R_{\myj})$ on the knot obtained from $\U$ as follows: replace $\TG(a_0)\subseteq\U_0$ by the arc $\TG(\beta_{\myj}^+)$, together with some arcs connecting their endpoints (the dotted arcs in the model $G_\Gamma\subseteq\ball_\Gamma$ on the left of Figure~\ref{fig:grope-boundaries}). Observe that the newly produced knot is isotopic to $\U$ by an isotopy across the shaded region.

Thus, we can also isotope $\TG_{\myj}$, so that the boundary of its bottom stage is as in the right picture, and hence it is instead a thickened grope from $\U$ to $\partial^\perp\TG_{\myj}\coloneqq(\U\sm\TG(\beta_{\myj}^+))\cup\TG(\beta_{\myj}\sm\beta_{\myj}^+)$. Actually, for $\TG_{\myj}$ to be a grope on $\U$ we also need to reparametrise $\U$ so that punctures indeed have labels $1\leq i\leq |R_{\myj}|$. Also, $\TG_2$ should have the orientation of all stages reversed.

\begin{figure}[!htbp]
    \centering
	\includegraphics{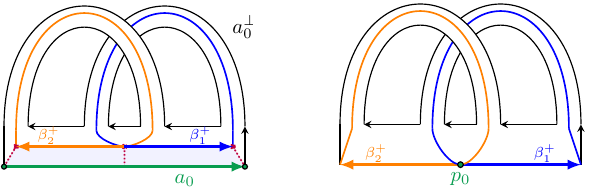}
    \caption{Modifying the bottom stage of a thickened grope.}
    \label{fig:grope-boundaries}
\end{figure}

Thanks to these modifications we have points $\psi(\TG_{\myj})\in\H_{|R_{\myj}|}(M)$ for $\myj=1,2$. However, we now immediately reparametrise back to get the analogous maps $f_{\TG_{\myj}}^{S_{\myj}}\colon I^{S_{\myj}}\to\FF^{|R_{\myj}|+1}_{S_{\myj}}$ with $S_{\myj}\coloneqq S\cap R_{\myj}$. In other words, although formally $\TG_{\myj}$ are not thickened gropes on $\U$, we can easily switch back and forth between the viewpoints. Moreover, we define
\begin{equation}\label{eq-def:deriv}
    \begin{tikzcd}
\deriv_{S}f_{\TG_{\myj}}^{S_{\myj}}\colon\:\:
    I^{S_{\myj}}\arrow{rr}{f_{\TG_{\myj}}^{S_{\myj}}} &&
    \FF^{|R_{\myj}|+1}_{S_{\myj}} \arrow{r}{ \deriv_{S_{\myj}}} &
    \Omega M_{0S_{\myj}}\arrow[]{rr}{\Omega\rho^0\rho} &&
    \Omega M_S,
\end{tikzcd}
\end{equation}
recalling from Remark~\ref{rem:no-rho-cube} that the last map is well defined on the image of $\deriv_{S_{\myj}}$.

By the following result each loop $\deriv_Rf_{\TG}^R(\vec{t})$ is either the commutator of loops $\deriv_{R}f_{\TG_{\myj}}^{R_{\myj}}(\vec{t}_{\myj})$ or some time of a canonical null-homotopy. We use the convention $[0,\frac{1}{2}]^\emptyset=I^\emptyset$ and $[\frac{1}{2},1]^\emptyset=\emptyset$.
\begin{lemma}\label{lem:commutator}
Assume $|R|\geq 2$. The map $\deriv_Rf_{\TG}^R$ is homotopic to the composition of the map $\vartheta_{(R_1,R_2)}\colon I^R\to I^{R_1}\times I^{R_2}$ which permutes the coordinates, and the map $C_\TG^R\colon I^{R_1}\times  I^{R_2}\to\Omega M_R$ given by
\begin{equation}\label{eq:comm-lemma-map}
   C_\TG^R(\vec{t}_1,\vec{t_2})\coloneqq\quad\begin{cases}
   \:\left[\,\deriv_{R}f_{\TG_1}^{R_1}(\vec{t}_1),\;\deriv_{R}f_{\TG_2}^{R_2}(\vec{t_2}) \,\right] &   ,\;(\vec{t}_1,\vec{t_2})\in [0,\frac{1}{2}]^{R_1}\times [0,\frac{1}{2}]^{R_2}\\
    \:\nu_{|\vec{t}_1|}\left(\deriv_{R}f_{\TG_2}^{R_2}(\vec{t}_2)\right) & ,\;(\vec{t}_1,\vec{t_2})\in [\frac{1}{2},1]^{R_1}\times [0,\frac{1}{2}]^{R_2}\\
    \:\nu_{|\vec{t}_2|}\left(\deriv_{R}f_{\TG_1}^{R_1}(\vec{t}_1)\right) & ,\;(\vec{t}_1,\vec{t_2})\in [0,\frac{1}{2}]^{R_1}\times [\frac{1}{2},1]^{R_2}\\
    \:\const_{p_0} &,\;(\vec{t}_1,\vec{t_2})\in [\frac{1}{2},1]^{R_1}\times [\frac{1}{2},1]^{R_2}
    \end{cases}
\end{equation}
where $\nu_{\theta}(x)$ is the time $\theta$ of the canonical null-homotopy $x\cdot x^{-1}\squig \const_{p_0}$ for a loop $x\in\Omega M_R$. In Figure~\ref{fig:schema-comm-lemma-map} these are shown as blue lines, and the subspace on which $C_\TG^R$ is constant is contracted.
\end{lemma}
\begin{figure}[!htbp]
	\centering
	\includegraphics{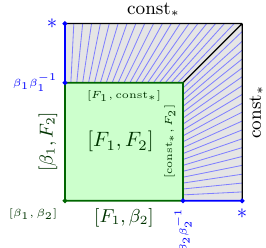}
    \caption[Schematic depiction of the commutator map.]{Schematic depiction of the map $C_\TG^R\colon I^{R_1}\times I^{R_2}\to\Omega M_{R}$, with $F_{\myj}\coloneqq\deriv_{R}f_{\TG_{\myj}}^{R_{\myj}}$ for short.}
    \label{fig:schema-comm-lemma-map}
\end{figure}
\begin{proof}
Assume $R=\emptyset$. We have $f_{\TG}^\emptyset(pt)=(\partial^\perp\TG)_{J_0}=\TG(a_0^\perp)$ and $\deriv_\emptyset f_{\TG}^\emptyset(pt)=(\U_0)_t\cdot\TG(a_0^\perp)_{1-t}$ is exactly the loop $\TG(\partial G_\Gamma)$, the boundary of the bottom stage. Since the bottom stage is a punctured torus, it collapses onto the $1$-skeleton, homotoping the boundary onto the commutator $[\TG(\beta_1),\TG(\beta_2)]\in\Omega M_R$.

 Now each $\TG(\beta_{\myj})=\TG(\beta_{\myj}^+)\cdot\TG(\beta_{\myj}\sm\beta_{\myj}^+)$ is precisely the loop $\deriv_\emptyset f_{\TG_{\myj}}^{\emptyset}(pt)\coloneqq(\U_0)_t\cdot\TG_{\myj}(a_0^\perp)_{1-t}$, so we conclude that $\deriv_\emptyset f_{\TG}^\emptyset(pt)$ is homotopic to $C_\TG^\emptyset(pt)$ as claimed.

Assume now $R\neq\emptyset$ and recall that for $\mc{h}^R(\vec{t})=(\theta,u)\in \mathsf{C}^{bar}(\Delta^R)$ the arc $f_{\TG}^R(\vec{t})\coloneqq\Psi^\TG(\theta)^R_{J_0}(u)$ is the time $\theta$ of an isotopy across the disk $\TG(\D_u)$ (see Proposition~\ref{prop:psi-def}): as $\theta\in[0,1]$ increases, the arc $a_0^\perp$ is being homotoped to $a_0$ across $\D_u$ using a foliation which we are still free to specify.

The disk $\D_u\subseteq\ball_{\Gamma}$ was in turn defined as the time $t\in[0,1]$ of the \emph{symmetric isotopy} (Lemma~\ref{lem:sym-surg-isotopy}) between the two disks obtained by surgery on the bottom stage along $\D_{u_1}\subseteq\ball_{\Gamma_1}$ or $\D_{u_2}\subseteq\ball_{\Gamma_2}$ (see Proposition~\ref{prop:grope-disk-corr}), where $u=(1-t)u_1+tu_2\in\Delta^R=\Delta^{R_1}\star\Delta^{R_2}$, with $u_{\myj}\in\Delta^{R_{\myj}}$. Without loss of generality, assume $t<\frac{1}{2}$, so $\D_{u}$ looks like in the left of Figure~\ref{fig:disk-u-band}.
Here $\vec{t}\in I^R\cong\mathsf{C}^{bar}(\Delta^R)\cong\mathsf{C}^{bar}(\Delta^{R_1}\star\Delta^{R_2})=\mathsf{C}^{bar}(\Delta^{R_1})\times\mathsf{C}^{bar}(\Delta^{R_2})$ precisely gives $\mc{h}^{R_{\myj}}(\vec{t}_{\myj})=(\theta,u_1)$.
\begin{figure}[!htbp]
	\centering
	\includegraphics{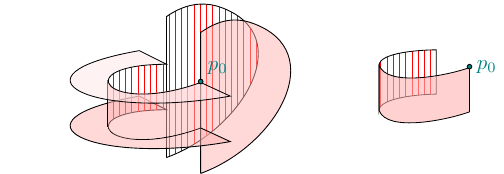}
    \caption[Homotopy across the disk $\D_u$.]{\textit{Left}: Disk $\D_u$ for some $t<\frac{1}{2}$. Two vertical disks are copies of $\D_{u_1}$. \textit{Right}: The band $b(u)$.}
    \label{fig:disk-u-band}
\end{figure}

The loop $\deriv_Rf_{\TG}^R(\vec{t})$ is obtained by closing up the arc $\Psi^\TG(\theta)^R_{J_0}(u)$ using for all $\vec{t}$ the same arc $\U_0$. Thus, we can collapse this $\U_0$ to the basepoint $p_0$ throughout the whole family, so that $\deriv_Rf_{\TG}^R(\vec{t})$ becomes, for a fixed $u$, a basepoint preserving homotopy of the loop $\TG(\partial G_\Gamma)$ to $\const_{p_0}$.

We now \emph{specify the foliation of $\D_u$} in such manner that this homotopy is first done across the two parallel copies of $\D_{u_1}$ (vertical disks in Figure~\ref{fig:disk-u-band}) and the two pieces of $\D_{u_2}$ (two regions lying flat), until for $\theta=\frac{1}{2}$ we have completely exhausted the parts of $\D_u$ which come from the caps. We are then left with a band $b(u)$ as on the right of Figure~\ref{fig:disk-u-band} and we let the \emph{rest of the homotopy} be the 'vertical' contraction onto the vertical line containing $p_0$, followed by its collapse onto $p_0$.

To further simplify these homotopies we collapse throughout the family the remaining pieces of the \emph{surgered torus} in $\D_u$ onto its skeleton. So `parallel copies of curves' get identified similarly as for $R=\emptyset$. The final result is as on the left of Figure~\ref{fig:collapse-u-band}: for any $u\in\Delta^R$, $\theta\leq\frac{1}{2}$ our $\deriv_Rf_{\TG}^R(\vec{t})$ became the commutator of the loops $\deriv_Rf_{\TG_{\myj}}^{R_{\myj}}(\vec{t}_{\myj})= \deriv_R(\Psi^{\TG_{\myj}}(\theta_{\myj})^{R_{\myj}}_{J_0}(u_{\myj}))$, with $\theta_1\in[0,1]$ and $\theta_2\in[0,c]$.
\begin{figure}[!htbp]
	\centering
	\includegraphics{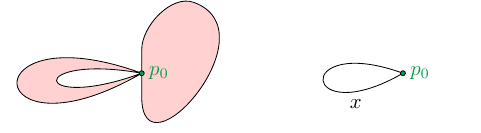}
    \caption[Homotopy after the modification.]{ \textit{Left}: Disk $\D_u$ after collapsing the punctured torus. \textit{Right}: The band $b(u)$ after the collapse.}
    \label{fig:collapse-u-band}
\end{figure}

Here $c$ is such that $\theta=\frac{1}{2}$ corresponds to $(\theta_1,\theta_2)=(1,c)$ and at this moment $\deriv_{R}f_{\TG_1}^{R_1}(\vec{t}_1)=\const_{p_0}$, while $\deriv_{R}f_{\TG_2}^{R_2}(\vec{t_2})$ is the curve $x$ on the right of Figure~\ref{fig:collapse-u-band}. Hence, for $\theta=\frac{1}{2}$ and a fixed $u$ we have $\deriv_Rf_{\TG}^R(\vec{t})=[\const_{p_0},x]$ and our null-homotopy across $b(u)$ for $\theta\geq\frac{1}{2}$ indeed becomes the canonical null-homotopy $t\mapsto x_s|_{[0,t]}\cdot x_{1-s}|_{[1-t,1]}$ after the collapse.
\end{proof}

\subsubsection*{The Reflections Lemma}
We now extend the previous result to describe each $h^S$-reflection
\[\begin{tikzcd}[row sep=tiny]
    \deriv_R\big(f_{\TG}^{R}\big)^{h^{S}}\colon \; I^R\arrow{rr}{\big(f_{\TG}^{R}\big)^{h^{S}}} && \FF^{|R|+1}_R \arrow{r}{\deriv_R} & \Omega M_R.
\end{tikzcd}
\]
\begin{lemma}\label{lem:deriv-h-reflections}
Assume $|R|\geq 2$ and $S\subseteq R$. The map $\deriv_{R}\big(f_{\TG}^{R}\big)^{h^{S}}$ is homotopic to the composition of the map $\vartheta_{(R_1,R_2)}$ as before with the map $\big(C_\TG^R\big)^{h^S}\colon I^{R_1}\times I^{R_2}\to\Omega M_R$ given by
\[\big(C_\TG^R\big)^{h^S}(\vec{t}_1,\vec{t_2})\coloneqq\quad\begin{cases}
    \:\left[\;\deriv_{R}\big(f_{\TG_1}^{R_1}\big)^{h^{S\cap R_1}}(\vec{t}_1),\;
    \deriv_{R}\big(f_{\TG_2}^{R_2}\big)^{h^{S\cap R_2}}(\vec{t}_2)\;\right]
&,\;(\vec{t}_1,\vec{t_2})\in [0,\frac{1}{2}]^{R_1}\times [0,\frac{1}{2}]^{R_2}\\
    \:\nu_{|\vec{t}_1|}\left(\deriv_{R}\big(f_{\TG_2}^{R_2}\big)^{h^{S\cap R_2}}(\vec{t}_2)\right)
& ,\;(\vec{t}_1,\vec{t_2})\in [\frac{1}{2},1]^{R_1}\times [0,\frac{1}{2}]^{R_2}\\
    \:\nu_{|\vec{t}_2|}\left(\deriv_{R}\big(f_{\TG_1}^{R_1}\big)^{h^{S\cap R_1}}(\vec{t}_1)\right)
& ,\;(\vec{t}_1,\vec{t_2})\in [0,\frac{1}{2}]^{R_1}\times [\frac{1}{2},1]^{R_2}\\
    \:\const_{p_0}
&,\;(\vec{t}_1,\vec{t_2})\in [\frac{1}{2},1]^{R_1}\times [\frac{1}{2},1]^{R_2}
    \end{cases}
\]
\end{lemma}
\begin{proof} The statement for $S=\emptyset$ is precisely the previous lemma (put $S\coloneqq R$ there).

Now assume $S\subseteq R$ is non-empty and that the statement is true inductively for all $\hat{R}$ of cardinality $|\hat{R}|<|R|$ and any $\hat{S}\subseteq R$ of cardinality $|\hat{S}|<|S|$. Thus, letting $k\coloneqq\min S$ and $R'\coloneqq R\sm k$ and $S'\coloneqq S\sm k$, the statement is true for the pairs $(R',S')$ and $(R,S')$.

Using the defining formula for an $h^S$-reflection from \eqref{eq:hS-reflection}, we compute
\begin{align}
\deriv_R\big(f_{\TG}^{R}\big)^{h^{S}}&=\deriv_Rr^k_{R'}\left(
h^k\big(f_{\TG}^{ R'}\big)^{h^{S'}}\boxbar_k l^k_{R'}\big(f_{\TG}^R\big)^{h^{S'}}
\right)\nonumber\\
&=\Omega\rho^k_{R'}\circ\deriv_{R'}\left(
h^k\big(f_{\TG}^{ R'}\big)^{h^{S'}}\boxbar_k l^k_{R'}\big(f_{\TG}^R\big)^{h^{S'}}
\right)\nonumber\\
&=\Omega\rho^k_{R'}\left(
\deriv_{R'}\circ h^k\big(f_{\TG}^{ R'}\big)^{h^{S'}}\boxbar_k \deriv_{R'}\circ l^k_{R'}\big(f_{\TG}^R\big)^{h^{S'}}
\right)\nonumber\\
&=\Omega\rho^k_{R'}\left(
h^k\circ\deriv_{R'}\big(f_{\TG}^{ R'}\big)^{h^{S'}}\boxbar_k (\Omega\lambda^k_{R'})\circ\deriv_R\big(f_{\TG}^R\big)^{h^{S'}}
\right).\label{eq:deriv-goes-through}
\end{align}
For the second equality we have used that $\rho^0\circ\deriv_R\circ r^k_{R'}=\rho^0\circ\Omega\rho^k_{R'}\circ\deriv_{R'}$ by Remark~\ref{rem:no-rho-cube} (recall that we omit $\rho^0$ from notation), the third equality holds since $\deriv_R$ is applied pointwise, and for the last see again Remark~\ref{rem:no-rho-cube}. The induction hypothesis now implies
\begin{align}\label{eq:deriv-inductive}
\deriv_R\big(f_{\TG}^{R}\big)^{h^{S}}\simeq\Omega\rho^k_{R'}\left(
h^k\circ\big(C_\TG^{R'}\big)^{h^{S'}}\circ\vartheta_{(R'_1,R'_2)}\;\boxbar_k\; \Omega\lambda^k_{R'}\circ\big(C_\TG^R\big)^{h^{S'}}\circ\vartheta_{(R_1,R_2)}
\right)
\end{align}
and it remains to show that the last expression is equal to $(C_\TG^R)^{h^S}\circ\vartheta_{(R_1,R_2)}$.

To this end, assume without loss of generality that $k\in S\cap R_1$, and denote $R'_1=R_1\sm k$ and $R'_2=R_2$. Then \eqref{eq:deriv-inductive} at $\vec{t}\in I^t$ equals
\[
\Omega\rho^k_{R'}\left(
h^k\big(C_\TG^{R'}\big)^{h^{S'}}(\ul{\vec{t}_1},\vec{t}_2)\;\boxbar_k\; \Omega\lambda^k_{R'}\big(C_\TG^R\big)^{h^{S'}}(\vec{t}_1,\vec{t}_2)
\right).
\]
Let us plug in the formulae for $(C_\TG^{R'})^{h^{S'}}$ and $(C_\TG^R)^{h^{S'}}$ into this. Observe that $h^k$ and $\Omega\lambda^k_{R'}$ act \emph{trivially} on the maps involving the grope $\TG_2$, as those maps interact only with the punctures indexed by $S_2\notni k$. On the other hand, for $\vec{t}_1\in[0,\frac{1}{2}]^{R_1}$ we will get maps
\begin{align*}
    &h^k\circ\deriv_{R'}\big(f_{\TG_1}^{R'_1}\big)^{h^{S\cap R'_1}}(\vec{t}_1)\boxbar_k(\Omega\lambda^k_{R'})\circ\deriv_{R}\big(f_{\TG_1}^{R_1}\big)^{h^{S\cap R'_1}}(\vec{t}_1)\\
    \text{and}\quad & h^k\circ\nu_{|\vec{t}_2|}\left(\deriv_{R}\big(f_{\TG_1}^{R_1}\big)^{h^{S\cap R_1}}(\vec{t}_1)\right)\boxbar_k(\Omega\lambda^k_{R'})\circ\nu_{|\vec{t}_2|}\left(\deriv_{R}\big(f_{\TG_1}^{R_1}\big)^{h^{S\cap R_1}}(\vec{t}_1)\right).
\end{align*}
The second expression is just $\nu_{|\vec{t}_2|}$ applied to the first, which is in turn equal to $\deriv_R\big(f_{\TG_1}^{R}\big)^{h^{S\cap R_1}}(\vec{t}_1)$, by running the equalities from \eqref{eq:deriv-goes-through} in reverse.

Therefore, we indeed have with $\deriv_R\big(f_{\TG}^{R}\big)^{h^{S}}\simeq(C_\TG^R)^{h^S}\circ\vartheta_{(R_1,R_2)}$ as claimed.
\end{proof}
\red{
Using the formula~\eqref{eq:glueOp} we conclude:
}
\begin{cor}\label{cor:glue-htpies}
    The homotopies from the last lemma glue to a homotopy
\[
    \deriv_R\left(\chi f_{\TG}\right)^R\;=\;\glueOp_{S\subseteq R}\deriv_R\big(f_{\TG}^R\big)^{h^S}\:\simeq\:\Big(\glueOp_{S\subseteq R}\big(C_\TG^R\big)^{h^S}\Big)\circ\vartheta_{(R_1,R_2)}.
\]
\end{cor}
\subsection{The end of the proof}
Assume inductively that \eqref{eq:final-goal2} is true for all $\TG$ as in \eqref{main-setup} with $|R|<n$. Let $|R|=n$ and let us prove that $\Gamma(m_i^{\varepsilon_i\gamma_i})$ and $\deriv_{R}(\chi f_{\TG})^{R}\simeq\big(\glueOp_{S\subseteq R}\big(C_\TG^R\big)^{h^S}\big)\circ\vartheta_{(R_1,R_2)}$ are homotopic.

Firstly, $\TG$ is the gluing of thickened gropes $\TG_1$ and $\TG_2$ modelled respectively on $\Gamma_{\myj}\in\Tree(R_{\myj})$ and with signed decorations $(\varepsilon_i,\gamma_i)_{i\in R_{\myj}}$. Since both $|R_{\myj}|<n$ the \emph{induction hypothesis} implies that
\begin{equation}\label{eq:ind-hyp-main}
    \deriv_{R_{\myj}}(\chi f_{\TG_{\myj}})^{R_{\myj}}\;\simeq\; \Gamma_{\myj}(m_i^{\varepsilon_i\gamma_i})\colon\: (I^{R_{\myj}},\partial)\to(\Omega M_{R_{\myj}},\const_*).
\end{equation}

Secondly, in \eqref{eq:samelson-Gamma-sphere} we have defined $\Gamma(m_i^{\varepsilon_i\gamma_i})$ inductively by
\[
    \Gamma(m_i^{\varepsilon_i\gamma_i})\;\coloneqq\;\left[\Gamma_1(m_i^{\varepsilon_i\gamma_i}),\Gamma_2(m_i^{\varepsilon_i\gamma_i})\right]\circ\vartheta_{(R_1,R_2)}.
\]
The first map in the formula is the Samelson product which was shown in Lemma~\ref{lem:samelson-inductive} to be obtained by canonically trivialising all the faces of the map\footnote{More precisely, in that lemma we had $m_{i,R_{\myj}}\coloneqq\Omega\rho^{R_{\myj}\sm i}_i\circ m_i^{\varepsilon_i\gamma_i}$ if $i\in R_{\myj}$, but we have already abused the notation when we decided to write $m_i\coloneqq\Omega\rho^{S\sm i}_i\circ \eta_{\S_i}$ (cf.\ \eqref{eq:new-xi}).}
\[\begin{tikzcd}[column sep=huge]
   I^{R_1}\times I^{R_2}\arrow{rr}{\Gamma_1(m_i^{\varepsilon_i\gamma_i})\times \Gamma_2(m_i^{\varepsilon_i\gamma_i})} &&
    \Omega M_{R_1}\times \Omega M_{R_2} \arrow{rr}{[\Omega\rho_{R_1}^{R\sm R_1},\Omega\rho_{R_2}^{R\sm R_2}]} && \Omega M_{R}.
\end{tikzcd}
\]
Plugging in \eqref{eq:ind-hyp-main} we get $\Gamma(m_i^{\varepsilon_i\gamma_i})\simeq w^{ind}\circ\vartheta_{(R_1,R_2)}$ for the map $w^{ind}$ obtained by trivialising the faces of the map\footnote{Here we denote $\deriv_{R}=\Omega\rho_{R_{\myj}}^{R\sm R_{\myj}}\circ\deriv_{R_{\myj}}$ as in \eqref{eq-def:deriv}.}
\begin{equation}\label{eq:omega-inductive}
    \begin{tikzcd}[column sep=huge]
    I^{R_1}\times I^{R_2}\arrow{rrr}{[\deriv_{R}(\chi f_{\TG_1})^{R_1},\;\deriv_{R}(\chi f_{\TG_2})^{R_2}]} &&& \Omega M_{R}.
\end{tikzcd}
\end{equation}
The map $w^{ind}$ is depicted in Figure~\ref{fig:schema-omega}, with the map \eqref{eq:omega-inductive} given as the green square with the two coordinate axes $\vec{t}_{\myj}\in I^{R_{\myj}}$ (so it is an $n$-cube). Trivialising this on the boundary corresponds to putting the green square into a bigger one and filling in the intermediate region by null-homotopies $x\cdot x^{-1}\squig *$ along straight blue lines. Here $x\in\Omega M_{R}$ is some value of \eqref{eq:omega-inductive} on the boundary of the inner $n$-cube, and so $w^{ind}$ is indeed constant on the boundary of the outer $n$-cube.

We now show that $w^{ind}$ agrees with the map $\glueOp_{S\subseteq R}\big(C_\TG^R\big)^{h^S}$, so Corollary~\ref{cor:glue-htpies} will finish the proof.
\begin{figure}[!htbp]
	\centering
	\includegraphics{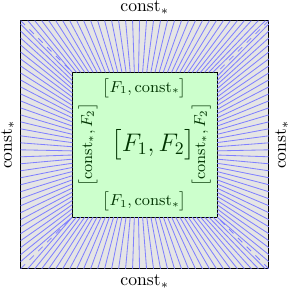}
    \caption[Main proof: the induction step.]{Schematic depiction of the map $w^{ind}$, where $F_{\myj}\coloneqq\deriv_{R}(\chi f_{\TG_{\myj}})^{R_{\myj}}$ for short.}
    \label{fig:schema-omega}
\end{figure}

Let us first show this for the `green parts' of Figures~\ref{fig:schema-comm-lemma-map} and~\ref{fig:schema-omega}, i.e. on $[0,\frac{1}{2}]^{R_1}\times [0,\frac{1}{2}]^{R_2}$ we have
\begin{align*}
w^{ind}
    &\coloneqq\quad \left[\deriv_{R}(\chi f_{\TG_1})^{R_1},\deriv_{R}(\chi f_{\TG_2})^{R_2}\right]\\
    &=\left[\glueOp_{S_1\subseteq R_1}\deriv_{R}\big(f_{\TG_1}^{R_1}\big)^{h^{S_1}},\,
    \glueOp_{S_2\subseteq R_2}\deriv_{R}\big(f_{\TG_2}^{R_2}\big)^{h^{S_2}}\right] \\
    &=\glueOp_{S\subseteq R}\left[\deriv_{R}(f_{\TG_1}^{R_1}\big)^{h^{S\cap R_1}},\,\deriv_{R}(f_{\TG_2}^{R_2}\big)^{h^{S\cap R_2}}\right]
    \quad=:\glueOp_{S\subseteq R}\big(C_\TG^R\big)^{h^S}
\end{align*}
where the second equality is by \eqref{eq:glueOp}, the third holds because the commutator bracket is applied pointwise and the last equality is by definition of $\big(C_\TG^R\big)^{h^S}$ in Lemma~\ref{lem:deriv-h-reflections}.

Similarly, for $\vec{t}_1\in [0,\frac{1}{2}]^{R_1}$ we have a blue line null-homotopy, where $\vec{t_2}$ runs in $[\frac{1}{2},1]^{R_2}$, so
\begin{align*}
w^{ind}(\vec{t}_1,\vec{t_2})
    &\coloneqq\nu_{|\vec{t}_2|}\left(\glueOp_{S_1\subseteq R_1}\deriv_{R}\big(f_{\TG_1}^{R_1}\big)^{h^{S_1}}(\vec{t}_1)\right)\\
    &=\glueOp_{S\subseteq R}\nu_{|\vec{t}_2|}\left(\deriv_{R}(f_{\TG_1}^{R_1}\big)^{h^{S\cap R_1}}(\vec{t}_1)\right) 
    \quad=:\glueOp_{S\subseteq R}\big(C_\TG^R\big)^{h^S}(\vec{t}_1,\vec{t_2}).\qedhere
\end{align*}
\end{proof}
\vspace{7pt}
\subsection{The proof of Theorem~\ref{thm:main-extended}}
\begin{proof}
We now show that for a grope forest $\forest$ of degree $n$ on $\U$ we have $[\emap_{n+1}\psi(\forest)]=[\ut(\forest)]\in\Lie_{\pi_1M}(n)$. \red{The points $\psi(\forest)\in\H_n(M)$ were defined in Proposition~\ref{prop:extended-KST}, and the underlying tree map $\ut(\forest)\in\Z[\Tree_{\pi_1M}(n)]$ in Proposition~\ref{prop:uf-map}.
}

In other words, for $\forest=\bigsqcup_{l=1}^N\TG_l\colon\bigsqcup_{l=1}^N\ball_{\Gamma_l}\hra M$ with $\ut(\TG_l)=\upvarepsilon^l\Gamma_l^{g^l}$, and denoting $f_{\forest}\coloneqq\emap_{n+1}\psi(\forest)$, we need to show
\[
\big[f_{\forest}\big]=\sum_{l=1}^N\left[\upvarepsilon^l\Gamma_l^{g^l}\right] \quad\in\pi_0\pF_{n+1}(M)\cong\Lie_{\pi_1M}(n).
\]
This was reduced in Section~\ref{subsec:strategy} to proving that $\deriv_{R}(\chi f_{\forest})^{R}\colon\S^n\to\Omega M_{R}$ is homotopic to a map realising the class on the right, namely, the pointwise product $\bigsqcap_{l=1}^N \Gamma_l(m_i^{\varepsilon_i^l\gamma_i^l})$
which takes $\vec{t}\in\S^n$ to the concatenation of the loops
\[
\Gamma_l(m_i^{\varepsilon_i^l\gamma_i^l})(\vec{t})\:\in\:\Omega M_{R},\quad  1\leq l\leq N.
\]
Since each $\TG_l$ is a thickened grope on $\U$ with the underlying decorated tree $\ut(\TG_l)=\upvarepsilon^l\Gamma_l^{g^l}$, the maps $\Gamma_l(m_i^{\varepsilon_i^l\gamma_i^l})\simeq\deriv_{R}(\chi f_{\TG_l})^{R}$ are homotopic by Theorem~\ref{thm:main-thm}. Hence, it remains to prove that $\deriv_{R}\big(\chi f_{\forest}\big)^{R}$ is homotopic to the pointwise product $\bigsqcap_{l=1}^N\deriv_{R}\big(\chi f_{\TG_l}\big)^{R}$.

Recall the definition of $f_{\forest}$ in \eqref{eq:final-psi-forest}. Similarly as in the proof of the Commutator Lemma~\ref{lem:commutator}, there is a homotopy between $\deriv_{R}\big(f_{\forest}^{R}\big)$ and the pointwise product of $\deriv_{R}(f_{\TG_l}^{R})$ -- since we can collapse $\U_0$ for all loops in the family. This extends to all $h$-reflections as in the Reflections Lemma~\ref{lem:deriv-h-reflections} -- there we had commutators of loops and here just their pointwise concatenations.

Therefore, using the same arguments as in the proof of Theorem~\ref{thm:main-thm} we can conclude
\begin{align*}
    \deriv_{R}\big(\chi f_{\forest}\big)^{R}
    &=
    \deriv_{R}\Big(\glueOp_{S\subseteq R}(f_{\forest}^{R})^{h^S}\Big)
    =
    \glueOp_{S\subseteq R}\deriv_{R}\big(f_{\forest}^{R}\big)^{h^S}\\
    &\simeq\glueOp_{S\subseteq R}\bigsqcap_{l=1}^N\deriv_{R}\big( f_{\TG_l}^{R}\big)^{h^S}
    =
    \bigsqcap_{l=1}^N\glueOp_{S\subseteq R}\deriv_{R}\big( f_{\TG_l}^{R}\big)^{h^S}=\bigsqcap_{l=1}^N\deriv_{R}\big(\chi f_{\TG_l}\big)^{R}.
    \qedhere
\end{align*}
\end{proof}

\appendix
\addtocontents{toc}{\bigskip\noindent%
  \normalfont\scshape\small{Appendices}\par}

\section{On left homotopy inverses}\label{app:proofs}

In this section we prove Lemmas~\ref{lem:retractions} and~\ref{lem:first-cube-retraction}, and Propositions~\ref{prop:cube-retraction} and~\ref{prop:chi-is-glued}.
\begin{proof}[Proof of Lemma~\ref{lem:retractions}]
Consider the commutative square
\begin{equation}
\begin{tikzcd}
    X \arrow{r}{l\circ r} & X\\
    X \arrow[equals]{u} \arrow{r}{r} & Y\arrow{u}[swap]{l}
\end{tikzcd}
\end{equation}
Its total homotopy fibre $Z$ is according to Corollary~\ref{cor:tofib-pt} given by
\begin{align*}
    \Big\{(x_0,x_s,y_t,x_{s,t})\in 
    & X\times\Path_*X\times\Path_*Y\times \Map(I^2,X)
    \,\Big|
    \\
    & x_s\colon x_0\squig*_X,\,
    y_t\colon r(x_0)\squig*_Y,\,
    x_{s,0}=lr(x_s),\,
    x_{0,t}=l(y_t),\,
    x_{s,1}=x_{1,t}=*_X\Big\},
\end{align*}
\red{
where $\cdot$ denotes the concatenation of paths (from left to right), and by abuse of notation $x_s,y_t,x_{s,t}$ stand for maps $x\colon I\to Y, s\mapsto x_s$ and $y\colon I\to Y, t\mapsto y_t$. See the picture in \eqref{eq:big-retractions-square}.
}

In this description $Z$ can easily be seen as the iterated homotopy fibre in two different ways, that is, in the next diagram all rows and columns are homotopy fibre sequences (cf.\ Lemma~\ref{lem:iterated}):
\begin{equation}\label{eq:big-retractions-square}
\begin{tikzcd}
    \hofib(l\circ r)\arrow[dashed]{r} & {\color{blue}X} \arrow{r}{l\circ r} & {\color{green!70!black}X} \\
    \hofib(r)\arrow[dashed]{u}\arrow[dashed]{r} & {\color{blue}X} \arrow{r}{r} \arrow[equals]{u} & {\color{orange}Y}\arrow{u}[swap]{l} \\
    Z\arrow[dotted]{r}{p_2}\arrow[dotted]{u}{p_1}[swap]{\simeq} & \Path_*X\arrow[dashed]{r}\arrow[dashed]{u} & \hofib(l)\arrow[dashed]{u}
\end{tikzcd}
\qquad\qquad
\begin{tikzpicture}[baseline=2.5ex,scale=0.9]
    \fill[orange,thick,draw]
        (-1,-.8) circle (0.7pt) node[left]{\tiny$r(x_0)$}
        -- (.8,-.8) circle (0.7pt) node[pos=0.5, below]{\footnotesize$y_t$} node[pos=1,right]{\tiny$*$};
    \fill[blue,thick,draw]
        (-2,0) circle (0.7pt) node[below]{\tiny$x_0$} -- (-2,1.8) circle (0.7pt) node[above]{\tiny$*$}
        (-2,0.9)
        node[left]{\footnotesize$x_s$}
        node[right=0.01cm,green!50!black]{\footnotesize$lr(x_s)$}
        node[right=2.5cm,green!50!black]{\footnotesize$\const_*$};
    \fill[green!10!white,text=green!40!black,draw=green!50!black]
        (-1,0) circle (0.7pt) node[below]{\tiny$lr(x_0)$} rectangle (.8,1.8)
        node[pos=0.5]{$x_{s,t}$}
        node[pos=0.5, below=0.8cm]{\footnotesize$l(y_t)$}
        node[pos=0.5, above=0.8cm]{\footnotesize$\const_*$};
\end{tikzpicture}
\end{equation}

The two dotted maps are $p_1(x_0,x_s,y_t,x_{s,t})=(x_0,y_t)$ and $p_2(x_0,x_s,y_t,x_{s,t})=(x_0,x_s)$. Since $l r\simeq\Id_X$, the homotopy fibre $\hofib(lr)$ is contractible, so $p_1$ is a weak equivalence. On the other hand, the path space $\Path_*X=\hofib(\Id_X)$ is also contractible, so we immediately have
\begin{equation*}
\begin{tikzcd}
    \chi^{-1}\colon\quad\Omega\hofib(l)\arrow{r}{d_Z}[swap]{\sim} & Z\arrow{r}{p_1}[swap]{\sim} & \hofib(r),
\end{tikzcd}
\end{equation*}
\red{
where $d_Z$ is the connecting map in the bottom row of the diagram. Therefore, we have
\[
    \chi^{-1}(y_t,x_{s,t})=p_1(*_X,\const_{*_X},y_t,x_{s,t})=(*_X,y_t),
\]
which is simply a forgetful map (here $y_0=y_1,x_{s,0}=x_{s,1}$). However, we will also need a formula for its homotopy inverse $\chi$, which we give next,
}
by finding explicit homotopy inverses $p_1^{-1}$ and $d_Z^{-1}$.

First recall \red{that by assumption we are given a homotopy 
}
$h\colon X\times I\to X$ from $\Id$ to $lr$. Thus, for $(x_0,y_t)\in\hofib(r)$ we have a path $h_t(x_0)\colon x_0\squig lr(x_0)$, and since $y_0=r(x_0)$, this can be concatenated with $l(y_t)$. Also, for each $t\in[0,1]$ we have a path $h_s(ly_t)\colon l(y_t)\squig lrl(y_t)$, so we let
\[
    p^{-1}_1(x_0,y_t) = \Big(\: x_0,\quad h_s(x_0)\cdot l(y_{s}),\quad y_t,\quad h_s(h_t(x_0))\wt{\boxbar} h_s(ly_t)\:\Big)\:\in Z,
\]
where $h_s(h_t(x_0))\wt{\boxbar} h_s(ly_t)$ \red{denotes the following square in $X$:
}
identify the two edges $h_t(x_0)\equiv h_s(x_0)$ of the square $(s,t)\mapsto h_s(h_t(x_0))$ to get a bigon, then glue to it $(s,t)\mapsto h_s(ly_t)$ along their common edge $h_s(lr(x_0))$; reshape this into a square with $0$-faces $lr\big(h_s(x_0)\cdot l(y_s)\big)$ and $ ly_t$ and both $1$-faces $\const_*$. \red{Equivalently, in the following rectangle glue the edges $h_t(x_0)$ and $h_s(x_0)$ together, then bend the face $\const_*$ into two and rescale to get a square:
}
\begin{equation}\label{eq-def:boxbartilde}
    \begin{tikzpicture}[scale=0.9,baseline=1.3cm,font=\small]
    \fill[green!20!white,text=green!40!black,draw=black]
        (-2,1)
        -- (0,1) node[pos=0.5, below]{$h_tx_0$}
        -- (0,3)
        -- (-2,3) node[pos=0.5, above]{$lr(h_tx_0)$} 
        -- (-2,1) node[pos=0.5, left]{$h_sx_0$} 
        node[pos=0.6, right=0.2cm,black]{\bfseries$h_s(h_tx_0)$} 
        ;
    \fill[green!10!white,text=green!40!black,draw=black]
        (0,1) 
        -- (2,1) node[pos=0.5, below]{$ly_t$} 
        -- (2,3) node[pos=0.5,right]{$\const_*$}
        -- (0,3) node[pos=0.5, above]{$lrly_t$} 
        -- (0,1) node[pos=0.3]{$h_s(lrx_0)$}
        (1,1) node[above=0.4cm,black,align=center]{ \\ $h_s(ly_t)$}
        ;
    \draw
        (-2,1)  circle (0.8pt) node[below]{\footnotesize$x_0$}
        (-2,3) circle (0.8pt) node[above]{\footnotesize$lrx_0$}
        (2,1) circle (0.7pt) node[below]{\footnotesize$*$}
        (2,3)  circle (0.8pt) node[above]{\footnotesize$*$}
        (0,3)  circle (0.8pt) node[above]{\tiny$lrlrx_0$}
        (0,1) circle (0.8pt) node[below]{\footnotesize$lrx_0$}
        ;
    \end{tikzpicture}
    \;
    \text{becomes}
    \qquad
    \begin{tikzpicture}[scale=0.9,baseline=1.3cm,font=\footnotesize]
    \fill[green!20!white,text=green!40!black,draw=black]
        (-2,1)
        -- (0,1) node[pos=0.5, below]{$ly_t$}
        -- (0,3) node[pos=0.5,right]{$\const_*$}
        -- (-2,3) node[pos=0.5, above]{$\const_*$} 
        -- (-2,1) node[pos=0.5, left]{$lr(h_tx_0)\cdot lrly_t$} 
        node[pos=0.6, right=0.22cm,black,align=center]{\bfseries$h_sh_tx_0$ \\ $\wt{\boxbar} h_sly_t$} 
        ;
    \end{tikzpicture}
\end{equation}
Since clearly $p_1\circ p^{-1}_1=\Id$ and $p_1$ is a weak equivalence, $p^{-1}_1$ is a homotopy inverse for $p_1$ (alternatively, there is a homotopy $p^{-1}_1\circ p_1\simeq\Id$ by gradually introducing back the coordinate $x_s$).

The map $d_Z^{-1}\colon Z\to\Omega\hofib(l)$ comes from comparing the bottom raw in the diagram \eqref{eq:big-retractions-square} to the fibre sequence $\Omega\hofib(l)\to\Path_*\hofib(l)\to\hofib(l)$. We define it by
\[
    d_Z^{-1}(x_0,x_s,y_t,x_{s,t})=\big\{\:t\:\mapsto\:\big(r(x_{1-t})\cdot y_t,\; lr(x_{s\geq t})\boxbar x_{s,t}\big)\:\big\},
\]
where $x_{1-t}$ is the inverse of the path $x_t$ and $x_{s\geq t}$ is the square $(s,t)\mapsto x_{1-t}|_{t\in[0,1-s]}$. Since $d_Z^{-1}\circ d_Z(y_t,x_{s,t})=(r(\const_*)\cdot y_t, lr(\const_*)_{\leq t}\cdot x_{s,t})=(y_t,x_{s,t})$, the map $d_Z^{-1}$ is indeed a homotopy inverse for $d_Z$, by the same argument as above.

Hence, the map $\chi\coloneqq d_Z^{-1}\circ p_1^{-1}\colon\hofib(r)\to\Omega\hofib(l)$ is given by
\[
    \chi(x_0,y_t)\coloneqq\big\{\:t\:\mapsto\:
    \Big( r\Big(h_s(x_0)\cdot l(y_s)\Big)_{s=1-t}\cdot y_t,\quad lr\Big(h_s(x_0)\cdot l(y_s)\Big)_{s\geq t}\boxbar \big(h_s(h_tx_0)\wt{\boxbar} h_s(ly_t)\big) \Big)\:\big\}
\]
and is the desired homotopy inverse to $\chi^{-1}$, where the notation $\boxbar$ means glue the two squares along the common edge, namely:
\begin{equation}\label{eq-def:boxbar}
    \begin{tikzpicture}[scale=0.9,font=\footnotesize,baseline=1.3cm]
        \fill[green!20!white,text=green!40!black,draw=black]
        (-2,1) circle (0.8pt) node[below left]{\tiny$*$}
        -- (0,1) node[pos=0.5, below,align=right]{$lr(h_tx_0)\cdot\quad$ \\ $ lrly_t$}
        -- (0,3) node[pos=0.5,right]{$\const_*$}
        -- (-2,3) node[pos=0.5, above]{$\const_*$}
        -- (-2,1) node[pos=0.5, left]{$\const_*$} 
        node[pos=0.5, right=0cm,black,align=right]{\bfseries$lr\Big(h_s(x_0)\cdot\quad$ \\ $ l(y_s)\Big)_{s\geq t}$} ;
    \fill[green!10!white,text=green!40!black,draw=black]
        (0,1) circle (0.8pt) node[below]{\tiny$lrx_0$}
        -- (2,1) node[pos=0.5, below]{$ly_t$} circle (0.7pt) node[below]{\tiny$*$}
        -- (2,3) node[pos=0.5,right]{$\const_*$}
        -- (0,3) node[pos=0.5, above]{$\const_*$}
        -- (0,1) node[pos=0.3,above]{$lr(h_tx_0)\cdot lrly_t$}
        (1,1) node[above=0.4cm,black,align=center]{\bfseries$h_sh_tx_0$ \\ $\wt{\boxbar} h_sly_t$} ;
    \end{tikzpicture}
    \;
    \text{becomes}
    \quad
     lr\Big(h_s(x_0)\cdot l(y_s)\Big)_{s\geq t}\boxbar \big(h_s(h_tx_0)\wt{\boxbar} h_s(ly_t)\big) \Big).
\end{equation}
\red{One should also rescale the resulting rectangle into a square, but \emph{for simplicity we draw rectangles}.
}
\end{proof}

We simplify the expression for $\chi$ using the following notation.
\begin{defn}\label{def:h-reflection}
    In the situation of the previous proof let us define the \textsf{$h$-reflection of $y\colon I\to Y$} by
    \[
    y^h_t\coloneqq r\big(h_s(x_0)\cdot ly_s\big)_{s=1-t}.
    \]
\end{defn}
Note that $y^h_t$ is a path from $*$ to $rx_0\in Y$. We now rewrite and redraw $\chi$ as
\begin{equation}\label{eq:chi-formula}
\chi(x_0,y_t)=\Big(\;y^h_t\cdot y_t,\: ly^h_{s\geq t}\boxbar \big(h_sh_tx_0\wt{\boxbar} h_sly_t\big)\;\Big).
\end{equation}
\begin{equation}\label{eq:pictorial-chi}
\begin{gathered}
	\includegraphics[width=0.8\linewidth]{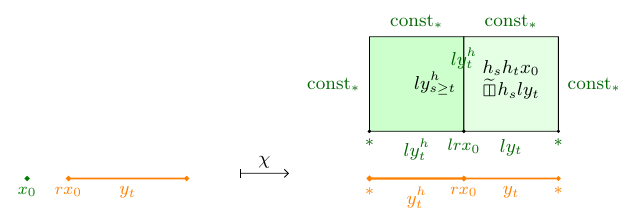}
\end{gathered}
\end{equation}
\red{
Observe that $y_t^h$ is obtained from  the path $y_t$ by `reflecting' across the starting point, 
since in $y_t^h$ we use $y_t$ reversed (and concatenate with the homotopy $h$ applied to $x_0$). We will use this language later on, see Section~\ref{subsec:chi-as-reflections}.
}

\begin{remm}
All triangles in the following diagram commute
\begin{equation}
\begin{tikzcd}[column sep=large, row sep=large]
    \{t\mapsto y_t\}\in &[-3em] \Omega Y\arrow[hook]{rr}
    &&
    \hofib(r)
    \arrow[shift left,red]{dll}{\chi}
    \arrow[shift right,red]{d}[swap]{p_1^{-1}}
    &[-3em] \ni (*,y_t)
    \\
    \{t\mapsto (y_t,x_{s,t})\}\in &[-3em]
    \Omega\hofib(l)
    \arrow[shift right,hook]{rr}[swap]{d_Z}
    \arrow[shift left]{urr}{\chi^{-1}}
    \arrow{u}{\forg}
    &&
    Z\arrow[shift right]{u}[swap]{p_1}
    \arrow[shift right,near start,red]{ll}[swap]{d_Z^{-1}}
    &[-3em] \ni (*,\const_*,y_t,x_{s,t})
\end{tikzcd}
\end{equation}
\end{remm}
As a word of caution, note that the composite map $\Omega Y\hra \hofib(r)\xrightarrow{\chi} \Omega\hofib(l)$ is not a loop map: $\chi(*,y_t)=(rly_{1-t}\cdot y_t,\; lrly_{s\geq t}\boxbar h_sly_t )$. When attempting to prove Proposition~\ref{prop:cube-retraction} by induction, one might run into needing that a similar map is a loop map, which is not the case. However, instead of delooping from `below' (first delooping $\FF^{n+1}_S$ for example), we need to start delooping from `above', using the following lemma.

\begin{proof}[Proof of Lemma~\ref{lem:first-cube-retraction}]
We have two $1$-cubes of $(m-1)$-cubes, namely the original cube $R_{\bull}$ which uses maps $r^m_S$ and the cube $L_{\bull}$ which uses $l^m_S$ instead, so we can write
\[\begin{tikzcd}
    R_{\bull}\colon & C_{\bull\notni m}\arrow[shift left]{r}{r^m} & C_{\bull\ni m}\arrow[shift left]{l}{l^m} & :L_{\bull}\;.
\end{tikzcd}
\]
Their total fibres can be computed as the homotopy fibres of the induced maps
\[\begin{tikzcd}
    \tofib(R_{\bull})\arrow[dashed]{r} & \tofib(C_{\bull\notni m})\arrow[shift left]{rr}{r^m_*} && \tofib(C_{\bull\ni m})\arrow[shift left]{ll}{l^m_*} & \tofib(L_{\bull}).\arrow[dashed]{l}
\end{tikzcd}
\]
We claim that $l^m_*\circ r^m_*=(l^m\circ r^m)_*$ is homotopic to $\Id$. Indeed, by the condition $(2)$ of Definition~\ref{def:left-inverse-cube} for each $t\in[0,1]$ the homotopies $h^m_S(t)\colon l_S^m\circ r_S^m\squig\Id$ assemble into a map of cubes
$
    h^m_{\bull}(t)\colon C_{\bull\notni m}\to C_{\bull\notni m}\;.
$
Hence, the induced map $h^m_*(t)$ on the total fibres is precisely a homotopy $(l^m\circ r^m)_*\squig \Id$. Therefore, we can simply apply the preceding Lemma~\ref{lem:retractions} to the left homotopy inverse $l^m_*$ and the homotopy $h^m_*$ to obtain the desired homotopy equivalence
\[\begin{tikzcd}
    \chi_m\colon\tofib(R_{\bull})\cong\hofib\big(r^m_*\big) \arrow{r}{\simeq} & \Omega\hofib\big(l^m_*\big)\cong\Omega\tofib(L_{\bull}).
 \end{tikzcd}\qedhere
\]
\end{proof}

\begin{proof}[Proof of Proposition~\ref{prop:cube-retraction}]
We will construct for each $0\leq k\leq m-1$ a homotopy equivalence $\chi_{m-k}\colon\tofib(D^{m-k})\to\Omega\tofib(D^{m-k-1})$, with the case $k=0$ covered in the previous lemma. The argument is actually the same: conditions \eqref{cond:ll} and \eqref{cond:lr} in Definition~\ref{def:m-left-inverse-cube} of an $m$-fold homotopy inverse ensure that for each $k$ we have two $1$-cubes of $(m-1)$-cubes:
\[\begin{tikzcd}[column sep=large]
    D^{m-k}\colon & D^{m-k}_{\bull\notni m-k}\arrow[shift left]{r}{r^{m-k}} & D^{m-k}_{\bull\ni m-k}\arrow[shift left]{l}{l^{m-k}} & :D^{m-k-1}\;.
\end{tikzcd}
\]
Moreover, the condition \eqref{cond:htpies-of-diagrams} ensures that $h^{m-k}_{\bull}(t)$ is for each $t\in[0,1]$ a map of cubes, so $h^{m-k}_*\colon l^{m-k}_*\circ r^{m-k}_*\squig\Id_{D^{m-k}}$ witnesses that $l^{m-k}_*$ is a left homotopy inverse. Therefore, by Lemma~\ref{lem:retractions}
\[\begin{tikzcd}
    \chi_{m-k}\colon \tofib(D^{m-k})\cong\hofib\big(r^{m-k}_*\big) \arrow{r}{\sim} & \Omega\hofib\big(l^{m-k}_*\big)\cong\Omega\tofib(D^{m-k-1}),
 \end{tikzcd}
\]
and the composite $\chi\coloneqq\chi_{1}\circ\dots\circ\chi_{m}\colon\tofib(C_{\bull},r)\xrightarrow{\sim}\Omega^m\tofib(C_{\bull},l)$ is the desired equivalence.
\end{proof}

\subsection{A description of $\chi$ in terms of reflections of cubes}\label{subsec:chi-as-reflections}
\subsubsection*{The first delooping}
Let us first  describe $\chi_m f\in\Omega\tofib(D^{m-1})$ for a point $f\in\tofib(R_{\bull})_{S\subseteq\ul{m}}$. The case $m=1$ is Lemma~\ref{lem:retractions} and picture \eqref{eq:pictorial-chi}. For $m=2$ we have
\begin{equation}\label{eq:2-cube}
R_{\bull}=D^2_{\notni 2}\xrightarrow[]{r^2} D^2_{\ni 2}\coloneqq\begin{tikzcd}
    {\color{blue}C^1} \arrow{r}{r^2} & {\color{green!70!black}C^{12}}\\
    {\color{red}C^\emptyset} \arrow{u}{r^1} \arrow{r}{r^2} & {\color{orange}C^2}\arrow{u}[swap]{r^1}
\end{tikzcd}
\quad\quad\quad
L_{\bull}=D^2_{\notni 2}\xleftarrow[]{l^2} D^2_{\ni 2}\coloneqq\begin{tikzcd}
    {\color{blue}C^1}  & {\color{green!70!black}C^{12}}\arrow{l}[swap]{l^2}\\
    {\color{red}C^\emptyset} \arrow{u}{r^1} & {\color{orange}C^2} \arrow{l}{l^2}\arrow{u}[swap]{r^1}
\end{tikzcd}
\end{equation}
and $\chi_2\colon\hofib(r^2_*)\to\Omega\hofib(l^2_*)$ is depicted in \eqref{eq:pictorial-chi-m}, with the colours indicating the ambient spaces from \eqref{eq:2-cube}. Note that $f^1$ is discarded, but used in the rest of the diagram. For example
\[
    (f^{12})^{h^2}\coloneqq r^2\big(h^2_s(f^1)\boxbar_2 l^2f^{12}\big),
\]
where $\boxbar_2$ denotes the concatenation $\cdot$ \red{as in \eqref{eq-def:boxbar} 
}
along the second coordinate.
\begin{equation}\label{eq:pictorial-chi-m}
\begin{gathered}
	\includegraphics[width=0.9\linewidth]{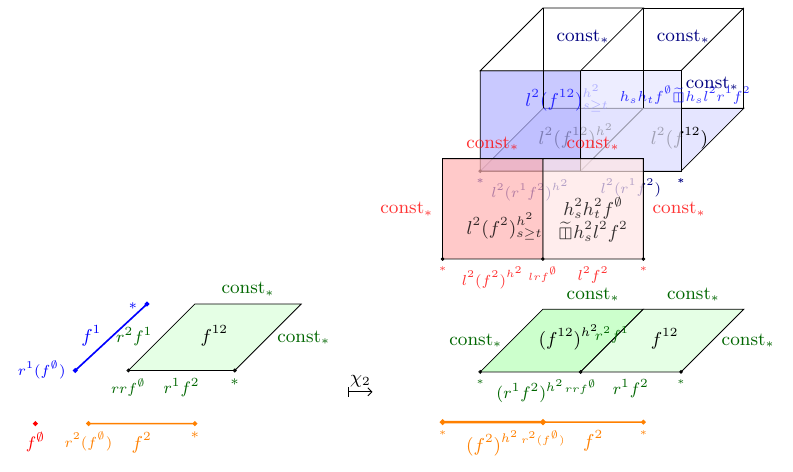}
\end{gathered}
\end{equation}
\begin{remm}
Observe that $\chi_2(f)$ is a well-defined point in $\Omega\tofib(L_{\bull})$ thanks to the conditions of Definition~\ref{def:left-inverse-cube}. For example, the orange line is indeed mapped to the bottom green line:
\[
r^1(f^2)^{h^2}=r^1r^2\big(h^2_s(f^\emptyset)\cdot l^2f^{2}\big)=r^2r^1\big(h^2_s(f^\emptyset)\cdot l^2f^{2}\big)=r^2\big(h^2_s(r^1f^\emptyset)\cdot l^2(r^1f^{2})\big)=(r^1f^2)^{h^2},
\]
since $r^1$ commutes with $l^2$ by condition $(1)$ and with $h^2_s$ by condition $(2)$ of the definition.
\end{remm}
For a general $m\geq1$, let us denote $f^{Pm}_{t_m}(t_{p\in P})\coloneqq f^{P m}(t_{p\in P},t_m)\in C_{P m}$ and rewrite $f$ as
\[
    \Big(\;\big\{f^P\big\}_{P\subseteq\ul{m-1}},\;\big\{f^{Pm}_{t_m}\big\}_{P\subseteq\ul{m-1}}\;\Big)\in\hofib(r^m_P).
\]
Then $\chi_m(f)\in\Omega\hofib(l^{m})$ is given by
\begin{equation}\label{eq:chi-m-fla}
    \chi_m(f^P,f^{Pm}_{t_m})=\Big(\; (f^{Pm})^{h^m}_{t_m}\boxbar_m f^{Pm}_{t_m},\quad l^m(f^{Pm})^{h^m}_{s\geq t_m}\boxbar h^m_sh^m_{t_m}f^P\boxbar h^m_sl^mf^{Pm}_{t_m} \;\Big),
\end{equation}
\red{
using the following notation.
}
\begin{defn}\label{def:boxbar,h-m-refl}
    For two maps $F_1,F_2$ from a cube $I^R$ to some space $X$, such that $F_1|_{t_m=1}=F_2|_{t_m=0}$ let $F_1\boxbar_mF_2\colon I^R\to X$ be the map obtained as the concatenation in the $t_m$-direction, plus rescaling.
    
    Furthermore, let $(f^{Pm})^{h^m}_{t_m}$ be the \textsf{$h^m$-reflection} as in Definition~\ref{def:h-reflection}, namely reflection of $f^{Pm}$ across the wall $I^P\times\{0\}$ in $I^{P m}$. Explicitly,
\begin{equation}\label{eq:h-m-reflection}
    (f^{Pm})^{h^m}_{t_m}\coloneqq r^m\Big(\,h^m_s(f^{P})\boxbar_m l^m(f^{Pm}_s)\,\Big)_{s=1-t_m}.
\end{equation}
\end{defn}
Therefore, we see that $\chi_m$ discards $f^P$ for $P\subseteq\ul{m-1}$ but incorporates it into its first coordinate. The second coordinate of $\chi$ is another `higher' layer of loops, in the spaces $C^{P}$.

\subsubsection*{The second delooping}
Consider again $D^2=R_{\bull}$ from \eqref{eq:2-cube} and $\chi_1\chi_2(f)\in\Omega^2\tofib(D^0)$. The right part of \eqref{eq:pictorial-iterated-chi} depicts its coordinates $S=\{1\}$ and $S=\{12\}$ (omitting $S=\emptyset,{\{2\}}$). The large green square $\forg\circ\chi (f)=\chi_1\chi_2(f)^{\{12\}}\in\Omega^2 C_{\ul{2}}$ is obtained by gluing reflections of $f^{12}$.
\begin{equation}\label{eq:pictorial-iterated-chi}
\begin{gathered}
	\includegraphics[width=0.9\linewidth]{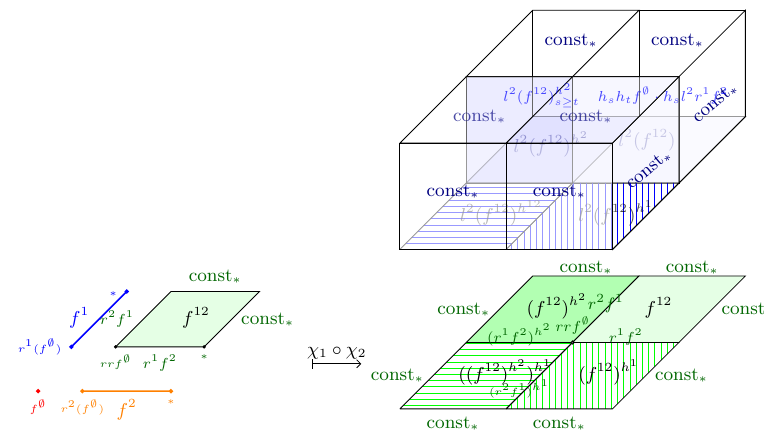}
\end{gathered}
\end{equation}

In order to generalise this observation we consider $m\geq2$ and
\begin{multline}\label{eq:double-chi}
\chi_{m-1}\chi_m(f)=\chi_{m-1}\Big((f^{Pm})^{h^m}_{t_m}\boxbar_m f^{Pm}_{t_m}\;,\;x^{P}\Big)_{P\subseteq\ul{m-1}}\\
=\Big(\;\Big((f^{Rm-1m})^{h^m}_{t_m}\boxbar_m f^{Rm-1m}_{t_m}\Big)^{h^{m-1}}_{t_{m-1}}\boxbar_{m-1} \Big((f^{Rm-1m})^{h^m}_{t_m}\boxbar_m f^{Rm-1m}_{t_m}\Big)_{t_{m-1}},\\
            (x^{Rm-1})^{h^{m-1}}\boxbar_{m-1} x^{Rm-1},\;y^{Rm}\;\Big)_{R\subseteq\ul{m-2}}.
\end{multline}
Here we have applied \eqref{eq:chi-m-fla} twice and denoted by $x$ and $y$ the remaining irrelevant coordinates. The first coordinate is $(\chi_{m-1}\chi_m(f))^{Rm-1m}\in\Omega^2C_{Rm-1m}$ where $R$ runs through subsets of $\ul{m-2}$.
\begin{lemma}\label{lem:iterate-chi-h}
$\big((f^{Rm-1m})^{h^m}_{t_m}\boxbar_m f^{Rm-1m}_{t_m}\big)^{h^{m-1}}=\big(\big(f^{Rm-1m})^{h^m}_{t_m}\big)^{h^{m-1}}\boxbar_m \big(f^{Rm-1m}_{t_m}\big)^{h^{m-1}}.$
\end{lemma}
\begin{proof}
The left hand side is by \eqref{eq:h-m-reflection} equal to $r^{m-1}$ applied to the map (denoting $s\coloneqq t_{m-1}$)
\begin{align*}
    &= h^{m-1}_s\Big((f^{Rm})^{h^m}_{t_m}\boxbar_m f^{Rm}_{t_m}\Big)\boxbar_{m-1} l^{m-1}\Big((f^{Rm-1m})^{h^m}_{t_m}\boxbar_m f^{Rm-1m}_{t_m}\Big)_s\\
    &= \Big(h^{m-1}_s(f^{Rm})^{h^m}_{t_m}\boxbar_m h^{m-1}_sf^{Rm}_{t_m}\Big)\boxbar_{m-1} \Big(l^{m-1}(f^{Rm-1m})^{h^m}_{t_m,s}\boxbar_m l^{m-1}f^{Rm-1m}_{t_m,s}\Big).
\end{align*}
The concatenation in the $t_{m-1}$-direction can be interchanged with the one in the $t_m$-direction (this is another manifestation of the Eckmann--Hilton principle), so we obtain
\begin{align*}
    &= \Big(h^{m-1}_s(f^{Rm})^{h^m}_{t_m}\boxbar_{m-1}l^{m-1}(f^{Rm-1m})^{h^m}_{t_m,s}\Big)\boxbar_m \Big(h^{m-1}_sf^{Rm}_{t_m}\boxbar_{m-1} l^{m-1}f^{Rm-1m}_{t_m,s}\Big).
\end{align*}
Finally, applying $r^{m-1}$ gives the desired right hand side in the statement of the lemma.
\end{proof}

Therefore, we make the following definition.

\begin{defn}\label{def:iterated-h-reflections}
Fix $P\subseteq\ul{m}$ and let $f^P\colon I^{P}\to C_{P}$ for an $m$-cube $C_{\bull}$. For $S\subseteq P$ we define a map $(f^P)^{h^S}\colon I^{P}\to C_{P}$ inductively on $|S|\leq |P|$ and call it the \textsf{$h^S$-reflection} of $f^P$.

For $S=\emptyset$ we let $(f^P)^{h^\emptyset}\coloneqq f^P$ and otherwise let $k=\min S$ and define
\[
\big(f^P\big)^{h^S}\coloneqq\Big(\big(f^P\big)^{h^{S\sm k}}\Big)^{h^k}
\]
using the definition of $h^k$-reflection from \eqref{eq:h-m-reflection}.
\end{defn}
Note that $(f^P)^{h^S}$ is indeed a kind of a reflection of $f^P$ across the $0$-faces $\{0\}^S\times I^{P\sm S}\subseteq\partial_0 I^P$. Indeed, the value of $f^P$ on them agrees with the value of $(f^P)^{h^S}$ on the corresponding faces $\{1\}^S\times I^{P\sm S}\subseteq\partial_1 I^{P}$. On the other hand, $(f^P)^{h^S}$ is constant on $\{0\}^S\times I^{P\sm S}\subseteq\partial_0 I^P$.

\begin{prop}\label{prop:chi-is-glued}
    For $D^m$ as above, $\forg\circ\chi\colon\tofib(D^m)\to\Omega^m C_{\ul{m}}$ maps $f\in\tofib(D^m)$ to the map 
    obtained by gluing together all $h$-reflections $(f^{\ul{m}})^{h^S}$ of $f^{\ul{m}}\colon I^m\to C_{\ul{m}}$ along the $0$-faces and $1$-faces of $I^m$. \red{We will denote this by
    \begin{equation}\label{eq-def:glueOp}
        \forg\circ\chi(f)=(\chi f)^{\ul{m}}\;=\;\glueOp_{S\subseteq\ul{m}}(f^{\ul{m}})^{h^S}.
    \end{equation}
    }
\end{prop}
\begin{proof}[Proof of Proposition~\ref{prop:chi-is-glued}]
From \eqref{eq:double-chi} and Lemma~\ref{lem:iterate-chi-h} we conclude
\begin{align*}
    \chi_{m-1}\chi_m(f)^{Rm-1m}&=\Big(\big(f^{Rm-1m})^{h^m}\big)^{h^{m-1}}\boxbar_m(f^{Rm-1m})^{h^{m-1}}\Big)\boxbar_{m-1}\Big((f^{Rm-1m})^{h^m}\boxbar_m f^{Rm-1m}\Big)\\
    &=(f^{Rm-1m})^{h^{\{m-1,m\}}}\boxbar(f^{Rm-1m})^{h^{m-1}}\boxbar(f^{Rm-1m})^{h^m}\boxbar f^{Rm-1m}.
\end{align*}
In the second line we simply omitted the brackets and used the symbol $\boxplus$ instead, since the operation $\boxbar_k$ glues unambiguously any two maps which have matching $0$- and $1$-faces indexed by $k$. Continuing in the manner of \eqref{eq:double-chi}, we find
\[
    \chi_k\circ\dots\circ\chi_{m-1}\chi_m(f)^{Rk\dots m-1m}=\glueOp_{S\subseteq\{k,\dots,m-1,m\}}(f^{Rk\dots m-1m})^{h^S}
\]
so the coordinate of $\chi_1\circ\dots\circ\chi_m(f)$ corresponding to $\ul{m}$ is given as claimed by
\[
(\chi f)^{\ul{m}}\;=\;\glueOp_{S\subseteq\ul{m}}(f^{\ul{m}})^{h^S}\,.\qedhere
\]
\end{proof}


\section{Samelson products}\label{app:samelson}
In this section we provide a reminder on Samelson products and some of their properties; the main references are~\red{\cite[Ch.~X~\&~XI]{Whitehead} and~\cite[Sec.~4~\&~6]{Neisendorfer}.
}

Let $G$ be a grouplike $H$-space, that is, an $H$-space whose multiplication $\cdot$ is homotopy associative and which has homotopy inverses, and we assume \emph{these homotopies are specified as part of the data}. 
In our applications $G\coloneqq\Omega(M\vee\bigvee_S\S^{d-1})$ with an inverse for $\gamma\in G$ given as the inverse loop $(\gamma^{-1})_t=\gamma_{1-t}$ and the canonical homotopy $\gamma\cdot\gamma^{-1}\squig\const_*$ given by $t\mapsto\gamma|_{[0,1-t]}\cdot\gamma^{-1}|_{[t,1]}$.

The homotopy groups $\pi_*G$ can be equipped with two associative operations -- on one hand, the standard multiplication on homotopy groups of a space using the co-$H$-space structure on spheres (additive except on the fundamental group), and on the other hand, using the $H$-space structure on $G$ to define the pointwise multiplication $f_1\cdot f_2\colon X_1\times X_2\to G\times G\to G$ of maps $f_{\myj}\colon X_{\myj}\to G$, and then if $X_1=X_2=\S^n$, precompose with the diagonal $\S^n\to \S^n\times \S^n$.

These operations give equivalent additive structure on $\pi_{>0}G$ (by the Eckmann--Hilton argument), but the latter also gives a group structure on $\pi_0G$. Moreover, $\gamma\in G$ acts on a map $f\colon X\to G$ by pointwise conjugation
\[
  f^\gamma(x)\coloneqq\gamma\cdot f(x)\cdot\gamma^{-1},
\]
and this defines an action of the group $\pi_0G$ on the graded group $\pi_*G$. When $G$ is a loop space this corresponds to the standard action of the fundamental group on the higher homotopy groups.

The \textsf{Samelson product} is a non-associative product of maps $f_{\myj}\colon X_{\myj}\to G$ for $\myj=1,2$, given by
\begin{equation}\label{eq:samelson-def}
    \begin{tikzcd}[column sep =large]
        [f_1,f_2]\colon\: X_1\wedge X_2 \arrow{r}{f_1\wedge f_2} &
        G\wedge G \arrow{r}{[\cdot,\cdot]} & G.
    \end{tikzcd}
\end{equation}
Here the commutator map $[\cdot,\cdot]\colon G\times G\to G$, given by $(x,y)\mapsto x\cdot y\cdot x^{-1}\cdot y^{-1}$, is null-homotopic on the wedge $G\vee G$, so factors through the smash product $G\wedge G\coloneqq\faktor{G\times G}{G\vee G}$. More precisely, since on the wedge one of the coordinates is equal to the basepoint, the word $[x,y]$ becomes of the shape $x\cdot x^{-1}$, for which there is a specified null-homotopy by the definition of a homotopy inverse.

Applying this to the case when each $X_{\myj}=\S^{n_{\myj}}$ is a sphere, and using homeomorphisms
\begin{equation}\label{eq:vartheta}
    \vartheta_{(n_1,n_2)}\colon\S^{n_1+n_2}\to \S^{n_1}\wedge\S^{n_2}
\end{equation}
we get an operation $[\cdot,\cdot]:\pi_{n_1}G\times\pi_{n_2}G\to\pi_{n_1+n_2}G$. On $\pi_0G$ this is just the group commutator, and there is an identity\footnote{To see \eqref{eq:samelson-identity}, plug $f_1=f$ and $f_2\colon\S^0\to G$ into \eqref{eq:samelson-def}. The latter simply picks out a point $\gamma\in G$, so $[f,\gamma]=f\cdot(f^{-1})^\gamma$. Here $f^{-1}$ is the pointwise inverse, homotopic to $-f$ by the mentioned Eckmann--Hilton principle.}
\begin{equation}\label{eq:samelson-identity}
    [f,\gamma]\simeq f-f^\gamma,\quad\quad f\in\pi_{>0}G\;,\;\gamma\in\pi_0G.
\end{equation}
On the abelian group $\pi_{>0}G$ the Samelson bracket is bilinear and satisfies \emph{graded} antisymmetry and Jacobi relations, making it into a graded Lie algebra over $\Z$. The origin of graded signs is in the fact that the coordinate exchange $\theta\colon\S^{n_1}\wedge\S^{n_2}\to\S^{n_2}\wedge\S^{n_1}$ induces the self-map $\vartheta_{(n_2,n_1)}^{-1}\circ\theta\circ\vartheta_{(n_1,n_2)}$ of $\S^{n_1+n_2}$ which has degree $(-1)^{n_1n_2}$.

Actually, the action of $\pi_0G$ on $\pi_{>0}G$ respects the grading and the Lie bracket
\[
    [f,g]^\gamma=[f^\gamma,g^\gamma],\quad\quad f,g\in\pi_{>0}G\;,\;\gamma\in\pi_0G,
\]
so all this structure is encapsulated by saying that $\pi_{>0}G$ is a \emph{graded Lie algebra over $\Z[\pi_0G]$}.

Furthermore, the Hurewicz homomorphism $h\colon\pi_*G\to H_*(G;\Z)$ takes the Samelson bracket to the graded commutator in the Pontrjagin Hopf algebra $H_*(G;\Z)$.

One can iteratively form the Samelson product of maps $f_i\colon X_i\to G$ for $i\in S$ according to a Lie word (a non-associative bracketing) $w(x^i)$ in the letters $x^i$, $i\in S$. To spell out this explicitly, we recursively define the space $w(X_i)$ and the map $w(f_i)\colon w(X_i)\to G$.

Firstly, for $k\in S$ let $x^k(X_i)\coloneqq X_k$ and $x^k(f_i)\coloneqq f_k$. Further, for $w=[w_1,w_2]$ let $w(X_i)=w_1(X_i)\bigwedge w_2(X_i)$ and define
\begin{equation}\label{eq:general-samelson}
    \begin{tikzcd}[column sep=large]
    w(f_i)\colon\:\: w(X_i) \arrow{rr}{w_1(f_i)\bigwedge w_2(f_i)} &&
    G\wedge G \arrow{r}{[\cdot,\cdot]} &
    G.
    \end{tikzcd}
\end{equation}
In particular, if $X_i=\S^{d-2}$ for all $i\in S$, to get a map from a sphere we need to precompose with a homeomorphism $\vartheta_{\sigma_w}\colon\S^{l_w(d-2)}\to w(\S^{d-2})$ similarly as in \eqref{eq:vartheta} ($l_w$ is the word length of~$w$).

\subsection{The Hilton--Milnor theorem}\label{subsec:HM}
Consider now $G=\Omega\bigvee_{i\in S}\Sigma X_i$. As in Section~\ref{subsec:proof-ThmB} we have the inclusion $\iota_{\Sigma X_i}\colon \Sigma X_i\hra\bigvee_{i\in S}\Sigma X_i$ and the canonical map $\eta_{X_i}\colon X_i\to \Omega\Sigma X_i$. Plugging into \eqref{eq:general-samelson} the composite maps
\[\begin{tikzcd}
    x_i\colon\: X_i\arrow{r}{\eta_{X_i}} & \Omega\Sigma X_i\arrow[hook]{rr}{\Omega\iota_i} && \Omega\bigvee_{i\in S}\Sigma X_i
\end{tikzcd}
\]
gives the Samelson product
\begin{equation}\label{eq:samelson-suspensions}
 w(x_i)\colon\: w(X_i) \to \Omega\bigvee_{i\in S}\Sigma X_i.
\end{equation}

\begin{remm}
For $G=\Omega Y$ the Samelson bracket on $\pi_*(\Omega Y)\cong\pi_{*+1}(Y)$ is adjoint to the Whitehead bracket on the latter group. The (generalised) Whitehead product of $f_i\colon\Sigma X_i\to Y$ is defined via
\[\begin{tikzcd}[]
[f_1,f_2]_W\colon\: \Sigma(X_1\wedge X_2)\arrow{r}{} & \Sigma X_1\vee\Sigma X_2\arrow{rr}{f_1\vee f_2} && Y\vee Y\to Y,
\end{tikzcd}
\]
where the last map is the fold and the first map can be explicitly defined, see~\cite{Whitehead}. It is precisely the adjoint of the Samelson product $[x_1,x_2]\colon X_1\wedge X_2\to\Omega(\Sigma X_1\vee\Sigma X_2)$ from \eqref{eq:samelson-suspensions}. If $X_i=\S^{n_i-1}$ are spheres, this is the attaching map $\S^{n_1+n_2-1}\to\S^{n_1}\vee\S^{n_2}$ of the top cell in $\S^{n_1}\times\S^{n_2}$.
\end{remm}

Actually, for $G=\Omega\bigvee_{i\in S}\Sigma X_i$ the Samelson products $w(x_i)$ `generate the homotopy type of $G$'. A more precise statement is the Hilton--Milnor theorem below, for which we need a bit more notation. Firstly, since \eqref{eq:samelson-suspensions} is a map into a loop space, there is a unique \textsf{multiplicative extension}
\begin{equation}\label{eq:mult-ext}
    \begin{tikzcd}[column sep=large]
    \wt{ w}(x_i)\colon\: \Omega\Sigma w(X_i)\arrow{r} &
    \Omega\bigvee_{i\in S}\Sigma X_i.
    \end{tikzcd}
\end{equation}
Namely, for any space $X$ the map $\eta_X\colon X\to \Omega\Sigma X$ is initial among all maps from $X$ to a loop space, so any $f\colon X\to \Omega Z$ factors as the composition of $\eta_X$ with $\wt{f}\coloneqq\Omega(ev\circ\Sigma f)\colon\Omega\Sigma X\to\Omega\Sigma\Omega Z\to\Omega Z$. Explicitly, $\wt{f}\big(\theta\mapsto t_\theta\wedge a_\theta\big)=\theta\mapsto f(a_\theta)_{t_\theta}$ for $t_\theta\in\S^1$ and $a_\theta\in X$, when $\theta$ ranges $\S^1$.

Moreover, given Lie words $w_1$ and $w_2$ we can take the pointwise product $\wt{w}_1(x_i)\cdot\wt{w}_1(x_i)$ (pointwise concatenate loops) as we saw above. Therefore, if $\B(S)$ denotes a Hall basis for the free Lie algebra $\L(x^i:i\in S)$ \red{(see~\cite[Sec.~XI.6]{Whitehead}), 
}
we can define the map
\begin{equation}\label{eq:hm-map}
\begin{tikzcd}
hm\coloneqq\prod_{ w}\wt{ w}\colon\quad \sideset{}{^{\circ}}\prod_{w\in \B(S)}\Omega\Sigma w(X_i) \arrow{rr} && \Omega\bigvee_{i\in S}\Sigma X_i,
\end{tikzcd}
\end{equation}
where the source is the weak product, defined as the filtered colimit of products over the finite subsets of $\B(S)$. Thus, points in it have all but finitely many coordinates equal to the basepoint.

\begin{theorem}[\cite{Hilton,Milnor,Gray,Spencer}]\label{thm:hm}
    If for each $i\in S$ the space $X_i$ is well-based and path-connected, then the map \eqref{eq:hm-map} is a weak homotopy equivalence.
\end{theorem}
This can be proven by first iterating Gray--Spencer Lemma~\ref{lem:gray-spencer} to get a weak homotopy equivalence
\[\begin{tikzcd}
\Omega{\iota_{A}} \times\wt{\bigvee_{i\geq0}[x_A,[x_A,\dots,[x_A,x_B]\dots]]}\colon\:
\Omega\Sigma A \times \Omega\Sigma \bigvee_{i\geq0} \Big(A^{\wedge i}\wedge B\Big) \arrow[]{rr}{} && \Omega\big(\Sigma A\vee\Sigma B\big),
\end{tikzcd}
\]
and then using an inductive argument on the word length. See~\cite[Thm.~4]{Milnor}.
\begin{remm}\label{rem:hall}
    \red{For an exposition on Hall bases see Theorem (6.3) in \cite{Whitehead}.
    }

    Note that the set $\B(S)$ is a Hall basis for the usual, \emph{ungraded}, free Lie algebra. This should not be confused with the fact that, if we put $X_i=\S^{n_i}$ with $n_i\geq2$, then the theorem implies that $\pi_*(\Omega\bigvee_S\S^{n_i+1})\otimes\Q\cong\L(x^{n_i}:i\in S)\otimes\Q$ is the \emph{free graded Lie algebra}, with $x^{n_i}$ having degree $n_i$.
\end{remm}

\subsection{Samelson products for trees}\label{subsec:samelson-trees}

Let us now consider Samelson products for Lie words $w(x^i)$ in which each letter $x^i$, $i\in R$, appears exactly once, for a finite ordered set $R$. This is the ungraded case for now, with $|x^i|=0$. Recall from Section~\ref{subsec-prelim:trees} there is an isomorphism $\omega_2\colon\Lie(R)\to\Lie_2(R)$ (see also next subsection).

As mentioned in \eqref{eq:general-samelson}, given maps $f_i\colon\S^{d-2}\to G$ one obtains a map out of a sphere by precomposing the Samelson product $w(f_{i})\colon w(\S^{d-2})\to G$ with $\vartheta_{\sigma_w}\colon\S^{l_w(d-2)}\to w(X_i)$, which permutes the factors according to the permutation $\sigma_w$ corresponding to the word $w$. However, in the case when letters do not repeat we can instead define this directly by induction. This was used in Section~\ref{subsec:gen-maps}.
\begin{lemma}\label{lem:samelson-tree}
    The map $w(f_{i})\circ\vartheta_{\sigma_w}$ for $w=\omega_2(\Gamma)$ is homotopic to the map $\Gamma(f_{i})$ defined inductively by $\ichord{k}(f_i)=f_k$ and for $\Gamma_{\myj}\in\Tree(R_{\myj})$ with $R_1\sqcup R_2=R$ by
    \begin{equation}\label{eq:samelson-Gamma-sphere}
        \grafted\hspace{-1em}\begin{tikzcd}
        (f_i)\colon\: (\S^{d-2})^{\wedge R}\arrow[]{rr}{\vartheta_{(R_1,R_2)}} &&
        (\S^{d-2})^{\wedge R_1}\wedge(\S^{d-2})^{\wedge R_2}\arrow[]{rrr}{\left[{\Gamma_1}(f_i),\;{\Gamma_2}(f_i)\right]} &&&
        G,
        \end{tikzcd}
    \end{equation}
    where $(R_1,R_2)$ permutes the ordered set $R$ into $R_1\sqcup R_2\coloneqq$ first all indices of $R_1$, then of $R_2$.
\end{lemma}
\begin{proof}
For $w=[w_1,w_2]$ with $w_{\myj}$ on letters in $R_{\myj}$ we have homotopies
\[
w(f_{i})\circ\vartheta_{\sigma_w}\simeq(\sgn\sigma_w)^{d-2} w(f_{i}),\quad\quad \Gamma(f_i)\:\simeq\: (-1)^{(1|2)_d}[{\Gamma_1}(f_i),{\Gamma_2}(f_i)],
\]

since $\deg\vartheta_{(R_1,R_2)}=(-1)^{(1|2)_d}$, where $(1|2)_d\coloneqq(d-2)\cdot(1|2)$ for $(1|2)\coloneqq|\{(i_1,i_2)\in R_1\times R_2: i_1>i_2\}|$
as in Lemma~\ref{lem:lie-tree} from Section~\ref{subsec-prelim:trees} (see below for its proof). The proof now follows by induction using that the sign of permutation satisfies the recursive formula \[
    \sgn\sigma_w=(-1)^{(1|2)}\sgn\sigma_{w_1}\cdot\sgn\sigma_{w_2}.\qedhere
\]
\end{proof}
In the proof of Theorem~\ref{thm:main-thm} in Section~\ref{sec:main} we will need the following observation. Let $(M_S,\rho)_{S\subseteq R}$ be a cube with maps $\rho^k_S\colon M_S\to M_{Sk}$ and assume we are given maps $f_i\colon\S^{d-2}\to \Omega M_i$ for $i\in R$. For $S\subseteq R$ with $i\in S$ let us denote
\[
\begin{tikzcd}
    f_{i,S}\coloneqq\Omega\rho^{S\sm i}_i\circ f_i\colon\;\S^{d-2}\arrow{r} & \Omega M_i\arrow{r} & \Omega M_S.
\end{tikzcd}
\]
\begin{lemma}\label{lem:samelson-inductive}
The map $[{\Gamma_1}(f_{i,R}), {\Gamma_2}(f_{i,R})]$ as in \eqref{eq:samelson-Gamma-sphere} is obtained by \emph{canonically} trivialising on the boundary the map
\[
    \begin{tikzcd}[column sep=huge]
    x\colon (I^{d-2})^{ R_1}\times (I^{d-2})^{ R_2}\arrow{rr}{{\Gamma_1}(f_{i,R_1})\times{\Gamma_2}(f_{i,R_1})} &&
    \Omega M_{R_1}\times \Omega M_{R_1}\arrow{rr}{[\Omega\rho^{R\sm R_{1}}_{R_{1}},\Omega\rho^{R\sm R_{2}}_{R_{2}}]} &&
    \Omega M_R.
\end{tikzcd}
\]
More precisely, for each $\vec{t}\in\partial (I^{d-2})^{ R}$ we glue in the standard null-homotopy $x(\vec{t})\cdot x(\vec{t})^{-1}\squig *$ of loops in $\Omega M_R$ to extend $x$ to a bigger cube on whose boundary it is constant.
\end{lemma}
Here we defined a map on $(\S^{d-2})^{\wedge R}\cong\faktor{(I^{d-2})^{R}}{\partial}$ by giving it on the cube so that it is constant on the boundary. The proof of the lemma is clear from definitions. See Section~\ref{subsec:strategy} for how it is used.

\subsection{Proof of Lemma~\ref{lem:lie-tree}}\label{subsec:lem:lie-tree}
    We now prove Lemma~\ref{lem:lie-tree}, which says that the map $\omega_d\colon\Lie(n)\to\Lie_d(n)$ given by $\omega_d(\ichord{i})=x^i$, $\omega_d(\Gamma)=\omega_\Gamma\coloneqq(-1)^{(1|2)_d}[\omega_d(\Gamma_1),\omega_d(\Gamma_2)]$ is an isomorphism. Here $(1|2)_d\coloneqq(d-2)\cdot(1|2)$ and $(1|2)\coloneqq|\{(i_1,i_2)\in R_1\times R_2: i_1>i_2\}|$.
    
    Note that $\Lie(n)$ \emph{extends} to an $\mc{S}_{n+1}$-representation by permuting all the labels $\{0,1,\dots,n\}$; this is related to the cyclic operad structure on $\Lie(\bull)$. Then Lemma~\ref{lem:lie-tree} precisely gives $\Lie_d(n)\cong\Lie(n)\otimes\sgn_{n+1}$ as representations of $\mc{S}_{n+1}$; see~\cite[Prop.~3.4]{Robinson-part-cx} for details.
    \begin{proof}[Proof of Lemma~\ref{lem:lie-tree}]
    We define $\omega_d$ by linearly extending the definition in the lemma and check it descends to the quotient by~\eqref{eq-def:AS-IHX}. We write $\omega_\Gamma\coloneqq\omega_d(\Gamma)$ for short.
    
    To this end, let $\omega_{AS}$ and $\omega_{IHX}$ be the images under $\omega_d$ of the linear combinations as in~\eqref{eq-def:AS-IHX}, but with roots instead of dots. It suffices to show that these are trivial, since then $\omega_d$ will also vanish on any tree in which $AS$ or $IHX$ appears as a subtree.
    
    Firstly, using $|\omega_{\Gamma_{\myj}}|=|S_{\myj}|(d-2)$ and the obvious identity $(1|2)_d + (2|1)_d=|S_1||S_2|(d-2)$ we obtain:
    \begin{align*}
    \omega_{AS}&=(-1)^{(1|2)_d}\big[\omega_{\Gamma_1},\omega_{\Gamma_2}\big]+(-1)^{(2|1)_d}\big[\omega_{\Gamma_2},\omega_{\Gamma_1}\big]\\
    &=(-1)^{(1|2)_d}\Big(
    \big[\omega_{\Gamma_1},\omega_{\Gamma_2}\big]+(-1)^{|S_1||S_2|(d-2)}\big[\omega_{\Gamma_2},\omega_{\Gamma_1}\big]
    \Big)\\
    &=(-1)^{(1|2)_d}\Big(
    \big[\omega_{\Gamma_1},\omega_{\Gamma_2}\big]+(-1)^{|\omega_1||\omega_2|}\big[\omega_{\Gamma_2},\omega_{\Gamma_1}\big]
    \Big)
    \end{align*}
    which vanishes by the graded antisymmetry in $\Lie_d(S)$. Secondly, recall \red{
    from Definition~\ref{def:trees} 
    }
    that
    \[IHX:\begin{tikzpicture}[baseline=1ex,scale=0.35,every node/.style={scale=0.85}]
            \clip (-2.5,-0.6) rectangle (2.5,4.45);
            \draw (0,0) node[]{$\vdots$};
            \draw[thick]
                (0,0) -- (0,1) -- (-1,2) -- (-2,3) node[pos=1,above]{$\Gamma_3$}
                                  (-1,2) -- (0,3) node[pos=1,above]{$\Gamma_2$}
                        (0,1) -- (2,3)  node[pos=1,above]{$\Gamma_1$};
        \end{tikzpicture} -
        \begin{tikzpicture}[baseline=1ex,scale=0.35,every node/.style={scale=0.85}]
            \clip (-2.5,-0.6) rectangle (2.5,4.45);
            \draw (0,0) node[]{$\vdots$};
            \draw[thick]
                (0,0) -- (0,1) -- (-1,2) -- (-2,3) node[pos=1,above]{$\Gamma_3$}
                        (1,2) -- (0,3) node[pos=1,above]{$\Gamma_2$}
                        (0,1) -- (2,3)  node[pos=1,above]{$\Gamma_1$};
        \end{tikzpicture} + \begin{tikzpicture}[baseline=1ex,scale=0.35,every node/.style={scale=0.85}]
            \clip (-2.5,-0.6) rectangle (2.5,4.45);
            \draw (0,0) node[]{$\vdots$};
            \draw[thick]
                (0,0) -- (0,1) -- (-1,2) -- (-2,3) node[pos=1,above]{$\Gamma_1$}
                                  (-1,2) -- (0,3) node[pos=1,above]{$\Gamma_3$}
                        (0,1) -- (2,3)  node[pos=1,above]{$\Gamma_2$};
        \end{tikzpicture} =\: 0.
    \]
    and note that by $AS$ the last tree is equal to
    \[-\begin{tikzpicture}[baseline=1ex,scale=0.35,every node/.style={scale=0.85}]
            \clip (-2.5,-0.6) rectangle (2.5,4.45);
            \draw (0,0) node[]{$\vdots$};
            \draw[thick]
                (0,0) -- (0,1) -- (-1,2) -- (-2,3) node[pos=1,above]{$\Gamma_3$}
                                  (-1,2) -- (0,3) node[pos=1,above]{$\Gamma_1$}
                        (0,1) -- (2,3)  node[pos=1,above]{$\Gamma_2$};
        \end{tikzpicture}\;.
    \]
    Therefore, $\omega_{IHX}$ is equal to
\begin{align*}
    &(-1)^{(1|23)_d+(2|3)_d}
    \big[\omega_{\Gamma_1},[\omega_{\Gamma_2},\omega_{\Gamma_3}]\big]
    -
    (-1)^{(12|3)_d+(1|2)_d} \big[[\omega_{\Gamma_1},\omega_{\Gamma_2}],\omega_{\Gamma_3}\big]
    -
    (-1)^{(2|13)_d+(1|3)_d}
    \big[\omega_{\Gamma_2},[\omega_{\Gamma_1},\omega_{\Gamma_3}]\big]
    \\
    &=(-1)^{(1|2)_d+(1|3)_d+(2|3)_d}
    \Big(\big[\omega_{\Gamma_1},[\omega_{\Gamma_2},\omega_{\Gamma_3}]\big]
    -\big[[\omega_{\Gamma_1},\omega_{\Gamma_2}],\omega_{\Gamma_3}\big]-
    (-1)^{(2|1)_d+(1|2)_d}
    \big[\omega_{\Gamma_2},[\omega_{\Gamma_1},\omega_{\Gamma_3}]\big]\Big),
\end{align*}
    where we have used the identities
    \begin{align*}
        (1|23)+(2|3)&=(1|2)+(1|3)+(2|3)=(12|3)+(1|2),\\
        (2|13)+(1|3)&=(2|1)+(2|3)+(1|3).
    \end{align*}
    
    Now again plugging in $(2|13)_d+(1|3)_d=|\omega_{\Gamma_1}||\omega_{\Gamma_2}|$ we get that the terms in the parenthesis are precisely those of the graded Jacobi relation \eqref{eq:graded-jacobi}, which holds in $\Lie_d(S)$.
    
    Finally, $\omega_d$ is clearly a surjection and an inverse $\omega_d^{-1}$ can be constructed in an analogous way -- $AS$ and $IHX$ will imply it is well defined modulo graded antisymmetry and Jacobi relations.
\end{proof}

\section{Finite type knot invariants}
\label{app:finite-type}

Let us briefly review \emph{classical and geometric approaches} to finite type theory; for book treatments see~\cite{Ohtsuki,CDM}. We restrict to the case of classical (long) knots $\KK(I^3)\coloneqq\pi_0\Knots(I^3)$; \red{this is 
}
a monoid under the connected sum $\#$, with unit the unknot $\U$. For general $3$-manifolds, see~\cite{K-thesis}.

\subsection{Classical and geometric approach to finite type theory}\label{subsec:classical-finite-type-theory}

\subsubsection{Jacobi diagrams and Jacobi trees}\label{subsec:jacobi-trees}
In the theory of finite type knot invariants for $M=I^3$ one considers general \emph{uni-trivalent graphs}\red{: this is a $1$-complex with vertices of valence either one (called univalent vertices or leaves) or three (called trivalent vertices). Additionally, leaves come equipped with labels $0,\dots,n-1$, and trivalent vertices come with vertex-orientation (as in Definition~\ref{def:trees}); we draw graphs immersed in the plane (so that the vertex-orientation is the one induced from the orientation of the plane, as we had in Figure~\ref{fig:Lie-trees}), with leaves attached to the $x$-axis in the order of their label. 
}
The \textsf{degree} of a graph is one half of the total number of vertices.
\begin{defn}\label{def:jac-diagrams}
    For $n\geq0$ define the abelian group $\A_n$ of Jacobi diagrams as the quotient of the $\Z$-linear span of the set of degree $n$ \red{uni-trivalent 
    }
    graphs by the linear combinations which locally look like
\begin{equation}\label{eq-def:stu-1t}
\begin{aligned}
    \includegraphics{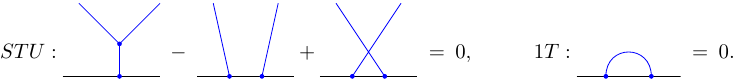}
\end{aligned}
\end{equation}
\end{defn}
\red{We will see in Section~\ref{subsec:geom-finite-type-theory} below that in the geometric approach one can reduce to considering trees only. Recall the group of Lie trees from Definition~\ref{def:trees}.
}
\begin{defn}[\cite{Conant}]\label{def:jac-trees}
Define the abelian group of \textsf{Jacobi trees} by
    \[
    \A^T_1\,\coloneqq\,\faktor{\Lie(1)}{ \scalebox{0.82}{\:\:\gchord{1}} }=0\quad\quad\text{and for}\;n\geq2\;\text{by}\quad\quad
    \A^T_n\,\coloneqq\,\faktor{\Lie(n)}{\stusq}
    \]
where the $\stusq$ relation is given by applying $STU$ in two different ways:
  \begin{equation}\label{eq:stusq-blue}
      \stusq:\quad STU(D,v_k) \:=\: STU(D,v_0).
  \end{equation}
Here a Jacobi diagram $D$ has degree $n$ and exactly one loop (i.e. the first Betti number $\beta_1(D)=1$) and $v_0$ and $v_k$ lie on the loop of $D$ and are neighbours of the leaf $0$ or $k$ respectively (see Figure~\ref{fig:stu2}).
\end{defn}
\begin{figure}[!htbp]
    \centering
    \includegraphics{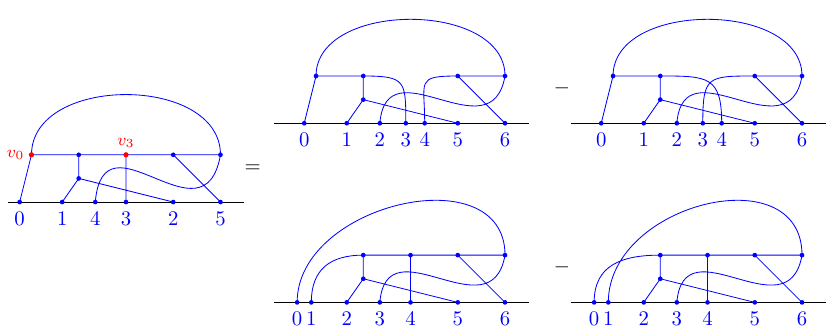}
    \caption[The \protect$\stusq$ relation.]{\textit{Left}: A $1$-loop graph $D$ with $n=6$ and $k=3$. \textit{Right}: The corresponding $\stusq$ relation.}
    \label{fig:stu2}
\end{figure}
The natural inclusion $\A^T_n\hra \A_n$ sends a Jacobi tree to itself viewed as a Jacobi diagram modulo $STU$, and lands in the primitives of the Hopf algebra $\A\coloneqq\oplus_{n\geq0}\A_n$, see~\cite{Conant,K-thesis}.
\begin{remm}\label{rem:stusq-compare}
Actually, Conant~\cite{Conant} more generally has $STU(D,v_k)=STU(D,v_j)$ for vertices $v_j,v_k$ which lie on the loop of $D$ and are neighbours of the leaves $j,k\in\{0,1,\dots,n\}$ respectively.

But this follows from \eqref{eq:stusq-blue} via $STU(D,v_j)=STU(D,v_0)=STU(D,v_k)$, since we can assume that the distance of the root to the unique loop of $D$ is one. Indeed, in the $IHX$ relation~\eqref{eq-def:AS-IHX} let $v_0$ be the vertex joining $\Gamma_2$ and $\Gamma_3$, with the root $0$ in $\Gamma_1$. Then two of the terms have distance from $0$ to the loop smaller than the third, so we can proceed by induction.
\end{remm}

\subsubsection{Classical approach: Vassiliev knot invariants}\label{subsec:vassiliev}

A \emph{singular knot} is an immersion $\sigma\colon I\imra I^3$ with finitely many transverse double points, which agrees with $\U$ near boundary. Each double point can be resolved by pushing the two strands off each other in two different ways, and all possible resolutions of $\sigma$ give $K_\sigma\in\Z[\KK(I^3)]$, a linear combination of knots. Now depending on the minimal number $n\geq 1$ of double points this defines a decreasing filtration \red{by ideals 
}
$\VV_n\subseteq\Z[\KK(I^3)]$ of the monoid ring, with $\VV_1$ precisely the augmentation ideal\red{, $\VV_1=\ker(\Z[\KK(I^3)]\ni\sum\epsilon_iK_i\mapsto\sum\epsilon_i\in\Z)$.}

The associated graded of this filtration is related to the Hopf algebra of \textsf{chord diagrams}: those Jacobi diagrams from Definition~\ref{def:jac-diagrams} which have no trivalent vertices. Namely, for a singular knot $\sigma$ with $n$ double points one has $n$ pairs of points on the source interval $I$ which are identified by $\sigma$; we record each pair by a chord to get a chord diagram $D_\sigma$ on $I$ of $\deg(D_\sigma)=n$.

This assignment is surjective, but far from being injective. However, there is a well-defined map $\realmap_n$ which takes a chord diagram $D$ to the class $[K_{\sigma(D)}]\in\faktor{\VV_n}{\VV_{n+1}}$, where $\sigma(D)$ is any singular knot with the chord diagram $D_{\sigma(D)}=D$. One can check that this is well defined and vanishes on diagrams which have an `isolated chord', as $1T$ from Figure~\ref{eq-def:stu-1t} (since our knots are not framed) and on the $4T$ relations, certain linear combinations of four chord diagrams (coming from triple points). Actually, $4T$ is a consequence of the relation $STU$ from Definition~\ref{def:jac-diagrams}, so there is a linear map
\begin{equation}\label{eq:chord-jacobi}
    \faktor{\Z\left[\textit{chord diagrams of deg }n\right]}{4T,1T}\to\:\faktor{\Z\left[\textit{Jacobi diagrams of deg }n\right]}{STU,1T}=:\A_n.
\end{equation}

Moreover, by Bar-Natan this is an \emph{isomorphism}~\cite{Bar-Natan} (the proof was given over $\Q$, but it actually applies integrally\red{, see also~\cite[Sec.~5.3]{CDM}%
}
). Combining this with the previous paragraph gives a surjection of finitely generated abelian groups
\begin{equation}\label{eq:classical-real-map}
    \begin{tikzcd}
        \realmap_n\colon\; \A_n\arrow[two heads]{r}{}&\faktor{\VV_n}{\VV_{n+1}} & D\mapsto[K_{\sigma(D)}].
    \end{tikzcd}
\end{equation}
called \textsf{the realisation map}. It is an open problem if its kernel is non-trivial, and a potential inverse is classically called a universal Vassiliev knot invariant of type $\leq n$ over $\Z$.

Namely, a knot invariant $v\colon\KK(I^3)\to T$ is \textsf{of type $\leq n$} if its linear extension $\ol{v}\colon\Z[\KK(I^3)]\to T$ vanishes on $\VV_{n+1}$. Here $T$ is an abelian group and $v$ is just a map of sets.
\begin{defn}\label{def:universal}
Let $R$ be a ring, $A$ a graded $R$-module and $\wh{A}\coloneqq\prod_{n\geq1}A_n$ its completion. A map $\zeta\colon\KK(I^3)\to\wh{A}$ is a \textsf{universal Vassiliev invariant over $R$} if the linear extension $\ol{\zeta}\colon R[\KK(I^3)]\to\wh{A}$ is a filtered $R$-linear map inducing an isomorphism of the associated graded $R$-modules. Equivalently,
\begin{enumerate}
    \item $\zeta=\prod_{n\geq1}\zeta_n$ and for each $n\geq1$ the map $\zeta_n$ is an invariant of type $\leq n$,
    \item the restriction $\ol{\zeta}_n|_{\VV_n}\colon \faktor{\VV_n}{\VV_{n+1}}\otimes R\to A_n$ is an isomorphism.
\end{enumerate}
 We say that $\zeta$ is \textsf{classical} if the composite $\ol{\zeta}_n|_{\VV_n}\circ(\realmap_n\otimes R)$ is the identity (so $\A_n\otimes R=A_n$).
\end{defn}
\begin{lemma} [justifying the `universality']
If $\zeta$ is a universal Vassiliev invariant over $R$, then any invariant $v\colon\KK(I^3)\to T$ of type $\leq n$ with values in an $R$-module $T$ can be written as a sum $\sum_{k=1}^nv_k\circ\ol{\zeta}_k$, where $v_k\coloneqq\ol{v}|_{\VV_k}\circ(\ol{\zeta}_k|_{\VV_k})^{-1}\colon A_k\to T$, called the \textsf{$k$-th symbol of $v$}.
\end{lemma}
\begin{proof}
  Indeed, $v-v_n\circ\ol{\zeta}_n$ vanishes on $\VV_n$, so it is an invariant of type $\leq n-1$ whose $(n-1)$-st symbol is equal to $v_{n-1}$, so we can proceed by induction.
\end{proof}
The Kontsevich integral, as well as the Bott--Taubes configuration space integrals~\cite{Bott-Taubes,Altschuler-Freidel}, are classical universal Vassiliev invariant \emph{over $\Q$}. It is an open problem if they agree, but some progress was made in~\cite{Lescop02} (note that there may be several universal invariants over the same coefficient ring since only the `bottom part' $\ol{\zeta}_n|_{\VV_n}$ is determined). As a consequence,
\begin{equation}\label{eq:universal-over-Q}
    \begin{tikzcd}
    \realmap_n\otimes\Q\colon\; \A_n\otimes\Q\arrow[two heads]{r}{\cong}&\faktor{\VV_n}{\VV_{n+1}}\otimes\Q.
    \end{tikzcd}
\end{equation}
Therefore, the kernel of $\realmap_n$ consists of torsion elements, but it is unknown if $\A_n$ has any.

The Bar-Natan's isomorphism \eqref{eq:chord-jacobi} gives more power to the theory as it is relatively easy to construct interesting invariants of Jacobi diagrams, called \textsf{weight systems}. For example, invariants are obtained by interpreting each trivalent vertex as the Lie bracket in a fixed semisimple Lie algebra and the horizontal line as its representation. Actually, symbols of quantum invariants of knots are precisely these weight systems, but by~\cite{Vogel} this is a strict subset of all of them.

However, introducing trivalent vertices raises the question of their \emph{geometric} interpretation, as we had for chords. Several different answers are summarised in the following theorem.
\begin{theorem}\label{thm:vassiliev}
For $K,K'\in\KK(I^3)$ and $n\geq 1$ the following are equivalent:
\begin{enumerate}[topsep=-0.2em]
    \item $K-K'\in\VV_n$ or, equivalently, $K$ and $K'$ are not distinguished by any invariant of type $\leq n-1$;
    \item $K'$ can be obtained from $K$ by a surgery on a \textit{simple strict forest clasper} of degree $n$;
    \item $K'$ can be related to $K$ by a finite sequence of \textit{simple capped genus one grope cobordisms} of degree $n$. In this case we say that $K$ and $K'$ are \textsf{$n$-equivalent} and write $K\sim_n K'$.
\end{enumerate}
\end{theorem}
The equivalence 
$(1)\Leftrightarrow(2)$ is independently by~\cite{Gusarov} and~\cite[Thm.~3.17 \&~6.18]{Habiro} and $(2)\Leftrightarrow(3)$ by~\cite[Thm.~4]{CT1}.
The idea behind \red{both of 
}
these descriptions is to view a crossing change as the simplest move, of degree one, in a whole \emph{family of moves}. Namely, a chord guides a crossing change (a homotopy passing through the corresponding singular knot), while moves of higher degrees are certain iterations of the `trivalent' move: grab three strands of the knot and tie them into the Borromean rings. We make this precise in Section~\ref{subsec:gropes} using the last approach of the theorem\red{, in terms of gropes 
}
(see Remark~\ref{rem:emb-comm-borr}). \red{For claspers we refer the reader to \cite{Habiro,CT1}.
}

\subsubsection{Geometric approach: Gusarov--Habiro filtration}
\label{subsec:geom-finite-type-theory}
Let us define the \textsf{Gusarov--Habiro filtration} by \red{submonoids 
}
$\KK_n(I^3)\coloneqq\{K\in\KK(I^3):K\sim_n\U\}\subseteq\KK(I^3)$. Then the theorem implies that it maps to the Vassiliev--Gusarov filtration:
\begin{equation}\label{eq:GH-fil}
\begin{tikzcd}
    & \KK_n(I^3)\arrow[hook]{d}{}\arrow[]{rr}{}  && \VV_n\arrow[hook]{d}{} & \\
    \pi_0\Knots(I^3)= & \KK(I^3)\arrow[]{rr}{K\mapsto K-\U} && \Z[\KK(I^3)] & =H_0(\Knots(I^3);\Z).
\end{tikzcd}
\end{equation}

This is what we call the \emph{geometric approach}, as we are back to working with knots, instead of their linear combinations -- or dually, their invariants $H^0(\Knots(I^3);T)$. In terms of invariants of finite type, the next lemma shows that we are restricting to the study of those which are \textsf{additive}, that is, monoid maps from $\KK(I^3)$ to abelian groups.
\begin{lemma}\label{lem:additive-invts}
    An additive invariant has type $\leq n$ if and only if it vanishes on $\KK_{n+1}(I^3)$. That is, $v\colon \KK(I^3)\to A$ is a monoid map vanishing on $\KK_{n+1}(I^3)$ if and only if its linear extension $\ol{v}\colon\Z[\KK(I^3)]\to A$ vanishes on $\VV_1\cdot\VV_1+\VV_{n+1}$ \red{(the product and sum of ideals).
    }
\end{lemma}
\begin{proof}
  Since $v(K_1\#K_2)-v(K_1)-v(K_2)=\ol{v}\big((K_1-\U)\#(K_2-\U)\big)$, \red{we see that 
  }
  $v$ is a monoid map if and only if its linear extension \red{$\ol{v}$ 
  }
  vanishes on $\VV_1\cdot\VV_1\subseteq\Z[\KK(I^3)]$. On one hand, by Theorem~\ref{thm:vassiliev} we have 
  $\{K-\U:K\in\KK_{n+1}(I^3)\}\subseteq\VV_{n+1}$
  and on the other, $\VV_{n+1}\subseteq\VV_1\cdot\VV_1+\{K-\U:K\in\KK_{n+1}(I^3)\}$ by a result of Habiro~\cite[Thm.~6.17]{Habiro}. Since $v(K)=\ol{v}(K-\U)$, the claim follows.
\end{proof}
In this setting we pass from Jacobi diagrams $\A_n$ to its subgroup $\A^T_n\subseteq\A_n$ of Jacobi trees, and from the realisation map $\realmap_n$ to its `tree part' $\realmap^T_n$, defined as the unique map completing the diagram
\begin{equation*}
    \begin{tikzcd}[row sep=1cm]
    \{\text{degree $n$ grope forests}\}\arrow[two heads]{d}[swap]{\partial^\perp}\arrow[two heads]{rr}{\ut} && \Z[\Tree(n)]\arrow[two heads]{d}{\:\realmap^T_n}\\
    \KK_n(I^3)\arrow[two heads]{rr}{\textrm{mod }\sim_{n+1}}  && \faktor{\KK_n(I^3)}{\sim_{n+1}}
\end{tikzcd}
\end{equation*}
Grope forests, the output knot map $\partial^\perp$ and  the underlying tree map $\ut$ were defined in Section~\ref{subsec:gropes}. The following describes the exact relation between the two realisation maps.
\begin{theorem}[\cite{Habiro,CT2,Ohtsuki}, see also~\cite{K-thesis}]\label{thm:real}
There is a structure of an abelian group on the set $\,\faktor{\KK_n(I^3)}{\sim_{n+1}}$ so that $\realmap^T_n$ is map of abelian groups.

Moreover, $\realmap^T_n$ vanishes on relations $AS$, $IHX$ and $\stusq$, and fits into the commutative diagram of abelian groups
\[\begin{tikzcd}[column sep=large]
    \A^T_n
    \arrow[hook]{rr}
    \arrow[two heads]{d}[swap]{\realmap^T_n}
    &&
    \A_n\arrow[two heads]{d}{\realmap_n}
    \\
    \faktor{\KK_n(I^3)}{\sim_{n+1}}\arrow[]{rr}{{K\mapsto K-\U}}
    &&
    \faktor{\VV_n}{\VV_{n+1}}
    \end{tikzcd}
\]
\end{theorem}
There is also a similar notion of a universal additive invariant.
\begin{defn}\label{def:additive-universal}
Let $A=F_0\supseteq F_1\supseteq\cdots$ be a filtered $R$-module and $\wh{A}\coloneqq\lim\faktor{A}{F_n}$. A \textsf{universal additive Vassiliev knot invariant over $R$} is a map of \emph{filtered monoids} $\zeta\colon\KK(I^3)\to\wh{A}$ which induces an isomorphism on the associated graded $R$-modules, that is
\[\begin{tikzcd}
\zeta_n|_{\KK_n(I^3)}\otimes R\colon\; \faktor{\KK_n(I^3)}{\sim_{n+1}}\otimes R\arrow[hook, two heads]{r}{} & \faktor{F_n}{F_{n+1}}
\end{tikzcd}\quad\forall n\geq0.
\]
We say that $\zeta$ is \textsf{classical} if $(\ol{\zeta}_n|_{\KK_n(I^3)}\otimes R) \circ(\realmap^T_n\otimes R)$ is the identity (so $\A^T_n\otimes R=\faktor{F_n}{F_{n+1}}$).
\end{defn}
\begin{remm}\label{rem:completions}
  Here we of course consider $\KK(I^3)$ is a filtered monoid with the Gusarov--Habiro filtration $\KK(I^3)=\KK_0\supseteq\KK_1\supseteq\cdots$. We can define the completion of $\KK$ over $R$ with respect to the Gusarov--Habiro filtration as the inverse limit
\[
    \wh{\KK(I^3)}_R\coloneqq\lim\big(\faktor{\KK(I^3)}{\sim_n}\otimes R\big).
\]
  Then for $\zeta$ a universal additive Vassiliev knot invariant over $R$, the induced map $\wh{\zeta}\colon\wh{\KK(I^3)}_R\to\wh{A}$ is an isomorphism of complete filtered $R$-modules.
\end{remm}
Similarly as before, a universal additive invariant $\zeta$ indeed satisfies a \emph{universality property}: any additive invariant $v\colon\KK(I^3)\to T$ of type $\leq n$ with values in an $R$-module $T$ is a sum $\sum_{k=1}^nv_k\circ\ol{\zeta}_k$.

Note however that this definition is more flexible than Definition~\ref{def:universal}, since the completion on $\wh{A}$ is with respect to a filtration instead of a grading. For instance, we could take $R=\Z$ and $A=\KK(I^3)$, so the obvious map $\KK(I^3)\to\lim\faktor{\KK(I^3)}{\sim_n}$ satisfies the conditions. This is precisely universal additive invariant $\nu$ of Habiro~\cite[Thm.~6.17]{Habiro}.

A universal additive Vassiliev invariant over $\Q$ can be constructed as the logarithm either of the Kontsevich or the Bott-Taubes integral $Z$ from \eqref{eq:universal-over-Q}, which are both grouplike~\cite{Kont-Vassiliev,Altschuler-Freidel}:
\[
  z\coloneqq\log(Z)\colon\KK(I^3)\to\prim(\wh{\A\otimes\Q})\cong\wh{\A^T\otimes\Q}.
\]
Recall that Conjecture~\ref{conj:universal} asserts that the evaluation map $\pi_0\ev_n$ factors through an isomorphism
\[
  \pi_0\ev_n\colon \faktor{\KK(I^3)}{\sim_n}\to\pi_0\pT_n(I^3),\quad\forall n\geq1.
\]
This is equivalent to the claim that
\[
  \lim\pi_0\ev_n\colon \KK(I^3)\to\lim\pi_0\pT_n(I^3)
\]
\emph{is a universal additive Vassiliev invariant over $\Z$} in the sense of Definition~\ref{def:additive-universal}, where the filtration on the target is by abelian groups $F_k\coloneqq\ker\big((\lim\pi_0\pT_n(I^3))\twoheadrightarrow\pi_0\pT_k(I^3)\big)$. See also \eqref{eq:limits}.
\begin{remm}\label{conj:habiro-vassiliev}
    The previous discussion largely generalises to knots in any oriented $3$-manifold~$M$. In particular, the equivalence of the definitions of the $n$-equivalence relation analogous to \textsl{(2, 3, 4)} of Theorem~\ref{thm:vassiliev} still hold. However, their equivalence to \textsl{(1)} remains open, see~\cite[Sec.~6]{Habiro}.
\end{remm}
\subsection{Proofs of corollaries}\label{subsec:further-cor}
As mentioned in the introduction, an important result is the identification due to Conant of the diagonal of the second page of spectral sequence $E^*_{-n,t}(M)$ for the tower of fibrations $p_{n+1}\colon\pT_{n+1}(M)\to\pT_n(M)$ in embedding calculus. \red{This is the standard spectral sequence for homotopy groups of a tower of fibrations, which we recalled in \eqref{eq-def:sp-seq}.
}
\begin{theorem}[\cite{Conant}]\label{thm:conant-d1}
  There is an isomorphism $E^2_{-(n+1),n+1}(I^3)\cong\A^T_n\coloneqq\faktor{\Lie(n)}{\stusq}$.
\end{theorem}
We saw in Corollary~\ref{cor:ss} that the first page has $\Lie(n)$ on the diagonal, so Conant identifies the image of the first differential $d^1$ as $\stusq\subseteq\Lie(n)$. \red{Roughly speaking, the reason behind this is that $d^1$ which hits the diagonal is evaluated on a homotopy class of a configuration space which corresponds to a collision of exactly two points (which can be graphically presented as on the left of Figure~\ref{fig:stu2}), and $d^1$ resolves this collision at two possible trivalent vertices in two possible ways (as on the right of Figure~\ref{fig:stu2}). 
}
See also~\cite{Shi} \red{and \cite[App.C]{K-thesis}.
}


We now summarise the discussion so far in the following diagram:
\begin{equation}\label{diag:big-diagram}
    \begin{tikzcd}
    \{\text{degree $n$ grope forests}\} \arrow[bend right=2cm]{ddd}[swap]{\partial^\perp}
    \arrow[dashed]{d}{\psi} \arrow[two heads]{rr}{\ut}
    &&
    \Z[\Tree(n)]\arrow[two heads]{d}{\text{mod}~(AS,IHX)}\arrow[dashed]{dll}[swap]{\rho_n}
    &
    \\
    \pi_0\H_n(I^3)\arrow[dashed]{rr}{\pi_0\emap_{n+1}}
    &&
    \pi_0\pF_{n+1}(I^3)
    \arrow[two heads,dashed,bend left=2cm]{dd}{\substack{\text{mod }\\\im \pi_0\delta}}
    \\
    & \A^T_n  \arrow[two heads]{ld}[swap]{\realmap^T_n}\arrow[two heads,dashed]{dr}{\substack{\text{mod}\\d^{*>1}}}  
    & \Lie(n)
    \arrow[two heads]{l}[swap]{\stusq}
    \arrow{u}{\cong}
    \\
    \faktor{\KK_n(I^3)}{\sim_{n+1}} \arrow[dashed]{rr}{\ol{e}_{n+1}}
    &&
    \ker(\pi_0p_{n+1})
    &
    \subseteq\pi_0\pT_{n+1}(I^3).
\end{tikzcd}
\end{equation}
All objects here are abelian groups except that $\{\text{degree $n$ grope forests}\}$ and $\pi_0\H_n(I^3)$ are only sets. The vertical dashed map on the right takes the quotient by the image on $\pi_0$ of the connecting map $\delta\colon\Omega\pT_n(I^3)\to\pF_{n+1}(I^3)$ for $p_{n+1}$. This factors through $\A^T_n\cong E^2_{-(n+1),n+1}$ by Theorem~\ref{thm:conant-d1}, as $\Lie(n)\cong E^1_{-(n+1),n+1}$ and the quotient by higher differentials is $\ker(\pi_0p_{n+1})=E^\infty_{-(n+1),n+1}(I^3)$.

The bottom horizontal map $\ol{e}_{n+1}$ is defined as as the map induced on the (set-theoretic) kernels
\begin{equation}\label{eq:two-ses}
    \begin{tikzcd}
    \faktor{\KK_n(I^3)}{\sim_{n+1}} \arrow[hook]{r}{}\arrow[densely dashed]{d}[swap]{\ol{e}_{n+1}} & \faktor{\KK(I^3)}{\sim_{n+1}}\arrow[two heads]{r}\arrow[]{d}{\pi_0\ev_{n+1}} & \faktor{\KK(I^3)}{\sim_n}\arrow[]{d}{\pi_0\ev_n}\\
    \ker(\pi_0p_{n+1}) \arrow[hook]{r}{} & \pi_0\pT_{n+1}(I^3)\arrow[two heads]{r}{} & \pi_0\pT_n(I^3)
\end{tikzcd}
\end{equation}
using that $\pi_0\ev_n$ factors though the quotient by $\sim_n$, which is a corollary of Theorem~\ref{thm:KST}, see \eqref{eq:kst-factor}.
\begin{remm}
  By Corollary~\eqref{eq:kst-factor} $\pi_0\ev_n$ vanishes on $\KK_n(I^3)$, so $\KK_n(I^3)\subseteq\im\big(\pi_0\H_n(I^3)\to\KK(I^3)\big)$, and Conjecture~\ref{conj:universal} claims that this inclusion is an equality. Thus, on the left side of \eqref{diag:big-diagram} there is a priori no map \red{$\pi_0\H_n(I^3)\to{\KK_n(I^3)}/{\sim_{n+1}}$, but the conjecture would give one.
  }
\end{remm}
\begin{cor}\label{cor:detect-tree}
  The diagram \eqref{diag:big-diagram} commutes.
\end{cor}
\begin{proof}
  The subdiagram comprised of solid arrows commutes by Theorem~\ref{thm:real}. The upper rectangle commutes by Theorem~\ref{thm:main-extended}, see \eqref{eq:main-extended}. It remains to check that the triangle on the right commutes.

  To this end, let $F\in\Z[\Tree(n)]$ and $[F]\in\A^T_n$ its class. By the surjectivity of $\ut$ we can find a degree $n$ grope forests $\forest$ on $\U$ in $I^3$ with $\ut(\forest)=F$. Let $K=\partial^\perp(\forest)$ and note that $[K]=\realmap^T_n[F]$ by the solid diagram. Then
  \[
    \ol{e}_{n+1}\big(\realmap^T_n[F]\big)=[\ev_{n+1}(K)]=[\emap_{n+1}\psi(\forest)]=[F]\pmod{\im \delta_*}.
  \]
  Here the second equality holds since $\ev_{n+1}(K)=i\circ\emap_{n+1}\circ\psi(\forest)$, where $i\colon\pF_{n+1}(I^3)\to\pT_{n+1}(I^3)$, by definition, see \eqref{diag:hofib-hn}. The last equality follows from the commutativity of the upper rectangle.
\end{proof}


\begin{proof}[Proof of Corollary~\ref{cor:collapse-implies-universal}]
    From the bottom triangle in the diagram~\eqref{diag:big-diagram} we deduce: if for some $n\geq1$ the map $(\mathrm{mod}~d^{*>1})$ is an isomorphism, then the other two maps are isomorphisms as well. More generally, if $A$ is a torsion-free abelian group and the map $(\mathrm{mod}~ d^{*>1})\colon\A^T_n\otimes A\to \ker(\pi_0p_{n+1})\otimes A$ is an isomorphism, then both $\realmap^T_n\otimes A$ and $\ol{e}_{n+1}\otimes A$ are isomorphisms.
    
    If for some $A$ this is the case for all degrees below a fixed degree $n$, then $\pi_0\ev_{n}$ is a universal additive Vassiliev invariant of type $\leq n-1$ over $A$, meaning that there is an isomorphism
    \begin{equation}\label{eq:ev-over-A}
        \begin{tikzcd}
    \pi_0\ev_{n}\otimes A\colon\:\faktor{\KK(I^3)}{\sim_{n}}\otimes A\arrow[]{r}{\cong} & \pi_0\pT_{n}(I^3)\otimes A.
        \end{tikzcd}
    \end{equation}
    This follows by induction by tensoring the sequences in \eqref{eq:two-ses} by $A$ and using its flatness.
\end{proof}
\begin{proof}[Proof of Corollary~\ref{cor:rational-universal}]
    The isomorphism \eqref{eq:ev-over-A} holds for $A=\Q$ for all $n\geq1$, and the $p$-adics $A=\Z_p$ in the range $n\leq p+2$, using Corollary~\ref{cor:collapse-implies-universal} and the results of~\cite{BH} (see also Remark~\ref{remm:rational-collapse}). 
    
    Furthermore, they also show that for $n\leq p+2$ there is a splitting
    \[
     \pi_0\pT_{n}(I^3)\otimes \Z_p\cong\bigoplus_{1\leq k\leq n-1}\A^T_k\otimes \Z_p.
    \]
    To deduce the integral result, we use that the kernels of both $\realmap^T_n$ and $\ol{e}_{n+1}$ must consist of torsion elements, and by~\cite{Gusarov-main} and~\cite[Sec.~3.5]{Manathunga} the group $\A^T_n$ is torsion-free for $n\leq 6$.
\end{proof}

Note that conversely, if there is a non-trivial higher differential hitting the diagonal, then not both $\realmap^T_n\otimes A$ and $\ol{e}_{n+1}\otimes A$ can be injective.
\begin{remm}\label{rem:integrals-factor}
There exists an inverse $z_n$ to $\realmap^T_n\otimes\Q$ obtained as the logarithm of either the Kontsevich integral~\cite{Kont-Vassiliev} or the Bott--Taubes integrals~\cite{Bott-Taubes,Altschuler-Freidel} (see the end of Section~\ref{subsec:classical-finite-type-theory}). Hence, $\ol{e}_{n}\otimes\Q$ agrees with these invariants, implying that the configuration space integrals factor through the embedding calculus tower:
\[\begin{tikzcd}
  & \bigoplus_{1\leq k\leq n-1}\A^T_k\otimes\Q\arrow[hook',two heads,shift left,near start]{ld}{\realmap^T_{\leq n-1}} & \\
  \faktor{\KK(I^3)}{\sim_{n}}\otimes \Q\arrow[hook,two heads]{rr}[swap]{\pi_0\ev_{n}\otimes\Q}\arrow[hook,bend left=1cm]{ur}{z_{\leq n-1}}  && \pi_0\pT_{n}(I^3)\otimes \Q\arrow[hook',two heads]{ul}{}
\end{tikzcd}
\]
where the map on the right is given by some splittings over $\Q$, making the diagram commute.
\end{remm}
\begin{remm}\label{remm:rational-collapse}
Let us remark that the collapse of the spectral sequence $E^*_{-n,t}(I^d)\otimes\Q$, converging to the rational homotopy groups of $\Knots(I^d)$, was shown earlier by~\cite{ALTV} but only for $d\geq 4$.

The collapse of the corresponding \emph{homology} spectral sequences for the Taylor tower of $\Knots(I^d)$ for any $d\geq3$ was shown by~\cite{LTV-homology,Tsopmene, Moriya}. However, it is not clear if those arguments can be extended to show that the homotopy collapse for $d=3$. This follows from the results of~\cite{FTW}, and more directly from~\cite{BH}.
\end{remm}

\red{
\begin{remm}\label{rem:graph-complexes}
    Graph complexes are chain complexes whose generators are given as graphs of some sort,  and they arise in the study of configuration spaces of points in $\R^d$ for $d\geq2$, more precisely, in proofs of their (co)formality (as operads). For example, Lambrechts and Turchin \cite{LT-GC}, and \v{S}evera and Willwacher~\cite{Severa-Willwacher} used Kontsevich's formality theorem to obtain a graph complex $ICG_d^{\bull}$ (a cosimplicial $L_\infty$-algebra)\footnote{This is the complex of `internally connected graphs', denoted by $P_{\bull}$ in \cite{LT-GC}, and the corresponding operad by $CG(\bull)$ in \cite{Severa-Willwacher}.} whose rational homology is isomorphic to $\mathfrak{t}_d^{\bull}\otimes\Q$, where $\mathfrak{t}_d^{\bull}=\pi_*\Conf_{\bull}(\R^d)$ is the Drinfeld--Kohno cosimplicial Lie algebra (in dimension $d$). Moreover, one can show that this implies that the total complex $\mathrm{Tot}(ICG_d^{\bull})$ (aka Moore complex of normalised chains) has the same rational homology as $\mathrm{Tot}(\mathfrak{t}_d^{\bull})$, see \cite{LT-GC}.
    
    On the other hand, as mentioned in Remark~\ref{rem:bcks}, there is a model due to Dev Sinha \cite{Sinha-cosimplicial} for the Taylor tower $\pT_n(I^d)$ of $\Knots(I^d)$, given as partial totalisations of a cosimplicial space built out of configuration spaces of $I^d$, with tangent vectors. For simplicity, in order to get rid of those tangent vectors, one considers the analogous Taylor tower $\ol{\pT}_n(I^d)$ for $\ol{\Knots}(I^d)\coloneqq \hofib(\Knots(I^d)\to\Imm_\partial(I,I^d))$, so that $\ol{\pT}_n(I^d)\simeq\mathrm{Tot}_n(\Conf_{\bull}(\R^d))$. Then the first page $\ol{E}_{*,*}^1(I^d)$, of the Bousfield--Kan homotopy spectral sequence for the tower $\ol{\pT}_n(I^d)$, is equal to $\mathrm{Tot}(\mathfrak{t}_d^{\bull})$, implying
    \[
    \ol{E}_{*,*}^2(I^d)\cong H_*\mathrm{Tot}(\mathfrak{t}_d^{\bull})\cong H_*\mathrm{Tot}(ICG_d^{\bull}).
    \]
    Lambrechts and Turchin point out in \cite[Sec.~7]{LT-GC} that on the diagonal this precisely corresponds to the result of Conant (Theorem~\ref{thm:conant-d1} above), and to the analogous result of Bar-Natan in homology (equation \eqref{eq:chord-jacobi} above). This explains our comment in Introduction~\ref{sec:intro} that Conant's groups of Jacobi trees $\A^T_n$ are closely related to graph complexes.
    
    Note that since this spectral sequence collapses over $\Q$ at the second page (for any $d\geq3$, see the comment after Corollary~\ref{cor:rational-universal}), we have $\pi_*\ol{\pT}_n(I^d)\otimes\Q\cong H_*\mathrm{Tot}_n(ICG_d^{\bull}\otimes\Q)$.
    Moreover, for $d\geq4$ the Goodwillie--Klein convergence theorem implies (see \cite[Thm.~9.4]{LT-GC}):
    \[
    \pi_*\ol{\Knots}(I^d)\otimes\Q\cong H_*\mathrm{Tot}(ICG_d^{\bull}\otimes\Q).
    \]
    See also \cite{Turchin-Hodge} for another graph complex (of `hairy graphs') which computes these groups and generalises another result of Bar-Natan (saying that the group of chord diagrams is isomorphic to the group of unitrivalent graphs, but not attached to a line).
\end{remm}
}


\makeatletter
\newcommand\setcurrentname[1]{\def\@currentlabelname{#1}}
\makeatother

\section*{Notation}\setcurrentname{Notation}\label{Notation}
\addcontentsline{toc}{section}{Notation} 

See also Notation~\ref{notat:punctured-model} for the notation related to the Taylor tower and Notation~\ref{notat:M-0S} for manifolds $M_{0S}$.

\begin{notation}[Basic objects]
\label{notat:basic}\hfill
\begin{longtable}{ m{1.4cm}|m{11.1cm}|m{1.4cm} }
\THICKhline
    $[n]$ & $[n]\coloneqq\{0,1,2,\dots,n\}$ for some $n\geq1$ \\ \thinhline
    $\ul{n}$ & $\ul{n}\coloneqq\{1,2,\dots,n\}$ for some $n\geq1$ \\ \thinhline
    $S$ & a subset of $\ul{n}$ or $[n]$, of cardinality $|S|$
    &\\ \thinhline
    $Sk$ & the union $Sk\coloneqq S\cup\{k\}$ for some $k\notin S$
    &\\ \thinhline
    $\Delta^S$ & the standard simplex on the vertex set $S$, of dimension $|S|-1$; we think of it both as a simplicial set and its geometric realisation
    &\\ \thinhline
    $I^S$ & the standard cube of dimension $|S|$ whose coordinates are labelled by the elements of the set $S$
    &\\ \thinhline
    $\mathsf{C}(X)$ & the cone on a simplicial set $X$, so $\mathsf{C}(X)\coloneqq X\times[0,1]/X\times\{1\}$
    &\\ \thinhline
    $X^{bar}$ & the barycentric subdivision of a simplicial set $X$
    &\\ 
\THICKhline 
    $\PCube(S)$ & the category whose objects are non-empty subsets of $S$ and morphisms are inclusions of sets, called the \textsf{punctured cube} category, of size $|S|$\\ \thinhline
    $\displaystyle\holim_{\PCube(S)}X_{\bull}$ & the homotopy limit of a punctured $|S|$-cube $X_{\bull}\colon\PCube(S)\to\mathsf{Top}_*$
    & Eq.~\eqref{eq-def:holim-punc-cube} \\ \thinhline
    $\Cube(S)$ & the category whose objects are subsets of $S$ and morphisms are inclusions of sets, called the \textsf{cube} category, of size $|S|$\\ \thinhline
    $\tofib X_{\bull}$ & the total homotopy fibre of an $|S|$-cube $X_{\bull}\colon\Cube(S)\to\mathsf{Top}_*$
    & Eq.~\eqref{eq-def:holim-cube} \\ 
\THICKhline 
    %
    $I$ & the unit interval $I\coloneqq[0,1]$\\ \thinhline
    $J_i$, $i=0,1,\dots$ & a fixed collection $J_i=[L_i,R_i]\subseteq I$ with $0<L_i<R_i<L_{i+1}<1$
    & Sec.~\ref{subsec:stages}\\ \thinhline
    $J_S$ & the union of $J_i$ for $i\in S$ 
    & Sec.~\ref{subsec:stages}\\ 
\THICKhline
    $\Path X$  & the space of all paths $I\to X$ in a space $X$ \\ \thinhline
    $\Path_*X$ & the space of paths $I\to X$ that start at the basepoint $\gamma(0)=*\in X$ \\ \thinhline
    $\Omega X$ & the space of based loops $X$, given by $\Omega X\coloneqq \{\gamma\colon I\to X: \gamma(0)=\gamma(1)=*\}$ \\ \thinhline
    $\Sigma X$ & the reduced suspension of $X$, given by 
    $\Sigma X\coloneqq {X\times I}/{X\times\{0,1\}\cup \{*\}\times I}$\\ \thinhline
    $X_+$ & the disjoint union of a space $X$ and a point, which is taken as the basepoint
    &\\ \thinhline
    $X^{\times k}$ & the $k$-fold product $X\times\dots\times X$ of a space $X$ with itself     
    &\\ \thinhline
    $\prod^{\circ}_{n\in \N}X_n$ & the weak product of spaces $X_n$, i.e.\ the topology is defined as the colimit of finite products
    &\\
\THICKhline
    $\gamma\colon a\squig b$ & a path $\gamma\colon I\to X$ from $a=\gamma(0)$ to $b=\gamma(1)$ \\ \thinhline
    $\gamma^{-1}$ & also $\gamma_{1-t}$, the inverse path, given by mapping $t\in I$ to $\gamma(1-t)$ \\ \thinhline
    $\gamma\cdot\eta$ & the concatenation of paths, given by mapping $t\in[0,1/2]$ to $\gamma(2t)$ and $t\in[1/2,1]$ to $\eta(2t-1)$ \\ \thinhline
    $[\gamma,\eta]$ & the commutator of loops $\gamma$ and $\eta$, given by $[\gamma,\eta]\coloneqq\gamma\cdot\eta\cdot\gamma^{-1}\cdot\eta^{-1}$ \\ 
\THICKhline
    $\B(\ul{n})$ & the Hall basis, an additive basis for the free Lie algebra $\L(x^i:i\in\ul{n})$
    & Rem.~\ref{rem:hall} \\\thinhline
    $\N\B(\ul{n})$ & the subset of Lie words in $\B(\ul{n})$ such that each letter $x^i$ for $i\in\ul{n}$ appears at least once
    & Thm.~\ref{thm:thmB-1} \\\thinhline
    $l_w$ & the length (the number of letters) of a word $w\in\B(\ul{n})$
    &\\\thinhline
    $\Gamma(x_i)$ & the Samelson product of maps $x_i$ according to a tree $\Gamma$
    & Sec.~\ref{subsec:samelson-trees}\\
\THICKhline
\end{longtable}
\end{notation}

\begin{notation}[Knots]
\label{notat:knots}\hfill

\begin{longtable}{ m{1.3cm}|m{11.1cm}|m{1.5cm} }   
\THICKhline
    $M$ & a fixed $d$-dimensional manifold, for some $d\geq3$ \\ \thinhline
    $\S M$ & the unit sphere subbundle of the tangent bundle of $M$ \\ \thinhline
    $-\infty,\infty$ & two fixed points in the boundary $\partial M$ \\ \thinhline
    $\hra$ & a smooth embedding\\ 
\THICKhline
    $\Knots(M)$ & the space of $K\colon I\hra M$ such that $K(I)\cap\partial M=K(\partial I)$, $K(0)=-\infty$, $K(1)=\infty$, called \textsf{knots}
    & Eq.~\eqref{eq-def:knots}\\ \thinhline
    $\U$ & an arbitrarily chosen basepoint in $\Knots(M)$, called the \textsf{unknot}
    & \\ \thinhline
    $K_{S}$ & the restriction of a knot $K\in\Knots(M)$ to $J_S$
    & \\ \thinhline
    $K_{\wh{S}}$ & the restriction of a knot $K\in\Knots(M)$ to $I\sm J_S$, so $K$ \textsf{punctured} at $J_i$, $i\in S$
    & Eq.~\eqref{eq-def:K-punctured}\\ 
\THICKhline
    $\ball_{ij}$ & a fixed neighbourhood of $\U|_{{(R_i,L_{j+1}})}$, homeomorphic to a $d$-dimensional ball
    & \\ \thinhline
    $\ball_S$ & the disjoint union $\ball_S\coloneqq\ball_{i_1i_2}\sqcup\dots\sqcup \ball_{i_qn+1}$ for $S=\{i_1<i_2<\dots<i_q\}\subseteq\ul{n}$
    & \\ \thinhline
    $\S_{ij}$ & the boundary $\S_{ij}\coloneqq\partial\ball_{ij}$ (a $(d-1)$-dimensional sphere embedded in $M$)
    & Eq.~\eqref{eq-def:S-ii}\\ \thinhline
    $\S_i$ & the sphere $\S_i\coloneqq\S_{ii+1}$
    & Fig.~\ref{fig:punctured-U}\\ \thinhline
    $\S_S$ & the wedge sum $\S_S\coloneqq\S_{i_1i_2}\vee\dots\vee\S_{i_qn+1}$ for $S=\{i_1<i_2<\dots<i_q\}\subseteq\ul{n}$
    & Eq.~\eqref{eq-def:S-S}\\ \thinhline
    $M_{0S}$ & the manifold $M$ minus the fixed tubular neighbourhood $\ball_{0S}$ of $\nu(\U_{\wh{0S}})$; so $M_{0S}$ is the union of $M\sm\nu(\U)$ and $\nu(\U_i)$ for $i\in \{0\}\sqcup S$
    & Eq.~\eqref{eq-def:M-0S}\\ \thinhline
    $M_S$ & the manifold $M_{0S}$ union the ball $\ball_{0i_1}$
    & Eq.~\eqref{eq-def:M-S}\\ 
\THICKhline
\end{longtable}
\end{notation}

\begin{notation}[Homotopy theoretic objects] 
\label{notat:htpy-th}\hfill

\begin{longtable}{ m{1.3cm}|m{11.1cm}|m{1.5cm} }
\THICKhline
    $\pT_n(M)$ & the $n$-th stage of the \textsf{punctured knots model} of the Taylor tower for $\Knots(M)$
    & Sec.~\ref{subsec:stages}\\ \thinhline
    $\mc{E}_{S}$ & the space of $I\sm J_S\hra M$ which agree on $\partial I$ with $\U$
    & Eq.~\eqref{eq-def:E-S}\\ \thinhline
    $(\mc{E}^n_{\bull},r)$ & the punctured cube with spaces $\mc{E}_S$ for $\emptyset\neq S\subseteq[n]$ and maps $r^k_S\colon\mc{E}_S\to\mc{E}_{Sk}$ for $k\notin S$ that restrict to more punctures; we have $\pT_n(M)=\holim\mc{E}^n_{\bull}$
    & Eq.~\eqref{eq-def:E-cube}\\ 
\THICKhline
    $p_{n+1}$ & the map between stages $p_{n+1}\colon\pT_{n+1}(M)\to\pT_n(M)$ (a fibration)
    & Eq.~\eqref{eq-def:p-n+1}\\ \thinhline
    $\pF_{n+1}(M)$ & the fibre of $p_{n+1}$ over $\ev_n(\U)$ (a subspace $\pF_{n+1}(M)\subseteq\pT_{n+1}(M)$)
    & Eq.~\eqref{eq-def:F-n+1}\\ \thinhline
    $\FF^{n+1}_S$ & the space of $J_0\hra M_{0Sn+1}$ which on $\partial J_0$ agree with $\U_0$
    & Eq.~\eqref{eq-def:FF-n+1-S}\\ \thinhline
    $(\FF^{n+1}_{\bull},r)$ & the cube with spaces $\FF^{n+1}_S$ for $S\subseteq\ul{n}$ and maps $r^k_{Sn+1}$ for $k\notin S$ that add back $\nu\U_k$; we have $\pF_{n+1}(M) \simeq\tofib(\FF^{n+1}_{\bull},r)$
    & Sec.~\ref{subsec:fn-hn}\\ 
\THICKhline
    $\ev_n$ & the canonical map $\ev_n\colon\Knots(M)\to\pT_n(M)$, called the \textsf{evaluation map}
    & Eq.~\eqref{eq-def:ev-n}\\ \thinhline
    $\H_n(M)$ & the homotopy fibre of $\ev_n$ over $\ev_n(\U)$, so a point $(K,\gamma)\in\H_n(M)$ is a knot $K\in\Knots(M)$ and a path $\gamma\colon I\to\pT_n(M)$ from $\gamma(0)=\ev_n(K)$ to $\gamma(1)=\ev_n(\U)$
    & Eq.~\eqref{eq-def:H-n}\\ \thinhline
    $\FF_S$ & the space of $J_0\hra M_{0S}$ which on $\partial J_0$ agree with $\U_0$
    & Eq.~\eqref{eq-def:FF-S}\\ \thinhline
    $(\FF_{\bull},r)$ & the cube with spaces $\FF_S$ for $S\subseteq\ul{n}$ and maps $r^k_S$ for $k\notin S$ that add back $\nu\U_k$; we have $\H_n(M)\simeq\tofib(\FF_{\bull},r)$
    & Sec.~\ref{subsec:fn-hn}\\ \thinhline
    $\emap_{n+1}$ & the map $\emap_{n+1}\colon\H_n(M)\to\pF_{n+1}(M)$ induced by $\ev_{n+1}$; it is also the map on total homotopy fibres induced by the map of $n$-cubes $r^{n+1}_S\colon\FF_S\to\FF^{n+1}_S$
    & Eq.~\eqref{eq-def:emap-n+1}\\ 
\THICKhline
    $(M_{0\bull},\rho)$ & the cube with spaces $M_{0S}$ for $S\subseteq\ul{n}$ and maps $\rho^k_S\colon M_{0S}\to M_{0Sk}$ that add back the material that `corresponds to the $k$-th puncture'; we have that $(\FF^{n+1}_{\bull},r)$ is homeomorphic to $\Embp(J_0,-)$ applied to $(M_{0\bull},\rho)$
    & Cor.~\ref{cor:r-rho}\\ \thinhline
    $(M_{0\bull},\lambda)$ & the cube with spaces $M_{0S}$ for $S\subseteq\ul{n}$ and maps $\lambda^k_{0S}\colon M_{0Sk}\to M_{0S}$ that `erase the $k$-th ball and drag the ball preceding it to the right'; we have that $(\FF^{n+1}_{\bull},l)$ is homeomorphic to $\Embp(J_0,-)$ applied to $(M_{0\bull},\lambda)$
    & Sec.~\ref{subsec:l-h-i-for-layer}\\ \thinhline
    $(M_{\bull},\lambda)$ & the cube with spaces $M_S$ for $S\subseteq\ul{n}$ and maps $\lambda^k_{S}\colon M_{Sk}\to M_{S}$ as above; we have that $(\FF^{n+1}_{\bull},l)$ is homotopy equivalent to $\Omega\S$ (the loop space of the unit sphere bundle) applied to $(M_{\bull},\lambda)$
    & Sec.~\ref{subsec:l-h-i-for-layer}\\ \thinhline
    $(M\vee\S_{\bull},\coll)$ & the cube with spaces $M\vee\S_S$ for $S\subseteq\ul{n}$ and maps $\coll^k_{Sk}$ that collapse the $k$-th sphere; we have that $(M_{\bull},\lambda)$ is homotopy equivalent to $(M\vee\S_{\bull},\coll)$
    & Eq.~\eqref{eq-def:col}\\ 
\THICKhline
    $x_i$ & the natural map $x_i\colon\S^{d-2}\to\Omega(M\vee\S_{S})$ that sends $p\in\S^{d-2}$ to the loop $\S^1\ni t\mapsto [t,p]\in\Sigma\S^{d-2}=\S^{d-1}\hra M\vee\S_S$, where we include the sphere as the $i$-th wedge factor in $\S_S$
    & Eq.~\eqref{eq:samelson-gen-x-i} \\ \thinhline
    $m_i$ & the map $m_i\colon\S^{d-2}\to\Omega M_{\ul{n}}$ sends $p\in\S^{d-2}$ to the loop in $M_{\ul{n}}$ that goes around the $i$-th ball $\ball_{ii+1}\subseteq M$ and corresponds to $x_i(p)$
    & Eq.~\eqref{eq:m-i} \\ \thinhline
    $\Gamma(x_i^{\gamma_i})$ &
    the Samelson product $\Gamma(x_i^{\gamma_i})\colon\S^{n(d-2)}\to \Omega(M\vee\S_{\ul{n}})$
    & Eq.~\eqref{eq:samelson-gen-x-i} \\ \thinhline
    $\Gamma(m_i^{\gamma_i})$ &
    the Samelson product $\Gamma(m_i^{\gamma_i})\colon\S^{n(d-2)}\to \Omega(M_{\ul{n}})$
    & Eq.~\eqref{eq:samelson-gen-m-i} \\ 
\THICKhline
    $\chi$ & the homotopy equivalence $\pF_{n+1}(M)\simeq\tofib(\FF^{n+1}_{\bull},r)\simeq\Omega^n\tofib(\FF^{n+1}_{\bull},l)$
    & Thm.~\ref{thm:delooping-layer}\\ \thinhline
    $\deriv_{\bull}$ & the derivative map $\deriv_S\colon\FF^{n+1}_S\to\Omega\S M_S$; it induces a homotopy equivalence $\deriv\colon\tofib\big(\FF^{n+1}_{\bull},l\big)\xrightarrow{\sim} \Omega\tofib\big( M_{\bull},\lambda\big)$.
    & Thm.~\ref{thm:final-delooping} \\ \thinhline
    $\retr_{\bull}$ & the retraction $\retr_S\colon M_S\to M\vee\S_S$; it induces a homotopy equivalence $\retr\colon \tofib(M_{\bull},\lambda)\to\tofib(M\vee\S_{\bull},\coll)$
    & Lem.~\ref{lem:retr} \\ \thinhline
    $W$ & our chosen isomorphism from $\Lie_{\pi_1M}(n)$ to $\pi_{n(d-2)}\tofib\Omega(M\vee\S_{\bull})$, sending a decorated Lie tree $\Gamma^{g_{\ul{n}}}$ to the canonical extension of  to the total homotopy fibre
    & Prop.~\ref{prop:LemB'}\\ 
\THICKhline
    $\boxbar$ & the operation of concatenating maps out of cubes along one face
    & Eq.~\eqref{eq-def:boxbar} \\ \thinhline
    $\glueOp$ & the operation of concatenating maps out of cubes along appropriate faces
    & Eq.~\eqref{eq-def:glueOp} \\
\THICKhline
\end{longtable}
\end{notation}

\begin{notation}[Geometric objects] 
\label{notat:geometric}\hfill

\begin{longtable}{ m{1.25cm}|m{11.1cm}|m{1.6cm} } 
\THICKhline
    $\Gamma$ & a connected simply connected graph with vertices of valence one (leaves, which are labelled by elements of a set $S$) or three (internal vertices, which have ordering of incident edges), called a \textsf{tree} of degree $|S|$
    & Def.~\ref{def:trees} \\ \thinhline
    $\Gamma^{g_S}$ & a tree together with an assignment of an element $g_i\in\pi$ to each vertex $i\in S$, called a \textsf{$\pi$-decorated tree}
    & Def.~\ref{def:decorated-lie}\\ \thinhline
    $\Tree_\pi(n)$ & the set of $\pi$-decorated trees of degree $n$
    & Def.~\ref{def:decorated-lie}\\ \thinhline
    $\Lie_\pi(n)$ & the quotient of the group $\Z[\Tree_\pi(n)]$ by $AS$ and $IHX$, with elements called \textsf{Lie trees}
    & Def.~\ref{def:decorated-lie}\\ \thinhline
    $\A^T_n$ & the group of degree $n$ Lie trees modulo $\stusq$ relations, with elements called \textsf{Jacobi trees}
    & Def.~\ref{def:jac-trees}\\ \thinhline
    $\A_n$ & the group of degree $n$ uni-trivalent graphs modulo $1T$ and $STU$ relations, with elements called \textsf{Jacobi diagrams}
    & Def.~\ref{def:jac-diagrams}\\ \thinhline
    chord diagram & a Jacobi diagram without trivalent vertices\\ 
\THICKhline
    $G_\Gamma$ & the abstract grope modelled on $\Gamma\in\Tree(S)$; it has boundary $\partial G_\Gamma\cong\S^1$ split into two arcs $a_0$ and $a_0^\perp$
    & Def.~\ref{def:abstract-grope}\\ \thinhline
    $\G$ & a grope cobordism $\G\colon G_\Gamma\to M$ on a knot $K$; it has $\G(a_0)\subseteq K(J_0)$
    & Def.~\ref{def:grope-cob}\\ \thinhline
    $\partial^\perp\G$ & the output knot of $\G$, given by $\partial^\perp\G\coloneqq(K\setminus \G(a_0))\cup \G(a_0^\perp)$
    & Def.~\ref{def:grope-cob}\\ \thinhline
    $\TG$ & a thickened grope $\TG\colon\ball_\Gamma\to M$ on a knot $K$
    & Def.~\ref{def:thick}\\ \thinhline
    $\forest$ & a grope forest $\forest\colon\bigsqcup_{l=1}^N\ball_{\Gamma_l}\to M$ on a knot $K$
    & Def.~\ref{def:forest}\\ \thinhline
    $(\varepsilon_i,\gamma_i)_{i\in S}$ & the signed decoration of a grope cobordism (or thickened grope)
    & Def.~\ref{def:underlying-decor-tree}\\ \thinhline
    $\ut$ & the underlying tree of a grope cobordism, or a thickened grope, or a grope forest, defined using the signed decoration $(\varepsilon_i,\gamma_i)_{i\in S}$ to decorate the model tree
    & Def.~\ref{def:underlying-decor-tree}\\ 
\THICKhline
    $\Psi^\TG$ & a path $\Psi^\TG\colon [0,1]\to\pT_n(M)$ from $\ev_n(\partial^\perp\TG)$ to $\ev_n(K)$
    & Prop.~\ref{prop:psi-def}\\ \thinhline
    $\Path^\TG$ & a family of paths in $\Embp(J_0,M)$ each going from $(\partial^\perp\TG)|_{J_0}$ and to $K|_{J_0}$
    & Prop.~\ref{prop:psi-def}\\ \thinhline
    $\phi_\Gamma$ & the family of embedded $2$-disks in the model ball $\ball_\Gamma$ for $\Gamma\in\Tree(S)$, parameterised by $\Delta^S\cong\Delta^{|S|-1}$
    & Prop.~\ref{prop:grope-disk-corr}\\ \thinhline
    $\Psi^\forest$ & the path $\Psi^\forest\colon [0,1]\to\pT_n(M)$ obtained by concatenating all $\Psi^{\TG_l}$ along their $J_0$-directions
    & Prop.~\ref{prop:extended-KST}\\ \thinhline
    $\psi(\forest)$ & the point $\psi(\forest)\coloneqq(\partial^\perp\forest,\Psi^\forest)\in\H_n(M)$, where $\partial^\perp\forest$ is  the output knot of a grope forest $\forest$ on $\U$, and the path $\Psi^\forest$ in $\pT_n(M)$ from $\ev_n(\partial^\perp\forest)$ to $\ev_n(\U)$
    & Prop.~\ref{prop:extended-KST}\\ 
\THICKhline
    $\KK(I^3)$ & the monoid of classical knots $\KK(I^3)=\pi_0\Knots(I^3)$
    & \\ \thinhline
    $\VV_n$ & the Vassiliev--Gusarov filtration of the monoid ring $\VV_n\subseteq\Z[\KK(I^3)]$, given by resolving singular knots with $n$ double points
    & Sec.~\ref{subsec:vassiliev}\\ \thinhline
    $\sim_n$ & the $n$-equivalence relation on $\KK(I^3)$
    & Thm.~\ref{thm:vassiliev}\\ \thinhline
    $\KK_n(I^3)$ & the filtration by submonoids $\KK_n(I^3)\subseteq\KK(I^3)$ consisting of those knots $K$ which are $n$-equivalent to the unknot $\U$, $K\sim_n\U$
    & Eq.~\eqref{eq:GH-fil}\\ \thinhline
    $\realmap_n$ & the classical realisation map $\realmap_n\colon\A_n\to\VV_n/\VV_{n+1}$ given by taking the linear combination of knots obtained by resolving a singular knot realising the given chord diagram 
    & Eq.~\eqref{eq:classical-real-map}\\ \thinhline
    $\realmap^T_n$ & the realisation map $\realmap^T_n\colon\A_n^T\to\KK_n(I^3)/\sim_{n+1}$ given by taking the output knot of a grope forest modelled on the given Jacobi tree
    & Thm.~\ref{thm:real}\\
\THICKhline
\end{longtable}
\end{notation}

\thispagestyle{plain}

\vspace{10pt}

\printbibliography[heading=bibintoc]

\vspace{10pt}
\hrule

\end{document}